\documentclass[11pt,letterpaper,leqno]{article}
\usepackage{amsmath,amssymb,amsfonts,amsthm}
\usepackage{enumitem}
\usepackage{bbm}
\usepackage{xcolor}
\usepackage[ruled,vlined,linesnumbered]{algorithm2e}
\usepackage[margin=1in]{geometry}
\usepackage[colorlinks, citecolor = blue, linkcolor = magenta]{hyperref}
\usepackage{cleveref}
\usepackage{float}
\usepackage{caption}
\usepackage{mathrsfs}
\usepackage{centernot}
\usepackage{booktabs}
\usepackage{tikz}
\usetikzlibrary{shapes,arrows,chains}
\usetikzlibrary[calc]
\tikzstyle{plan}=[draw, rounded corners,align=center]
\tikzstyle{line} = [draw, -latex']

\numberwithin{equation}{section}

\DeclareMathOperator*{\PP}{\mathbb{P}}
\DeclareMathOperator*{\WW}{\mathbb{W}}

\newcommand{\I}{\mathcal{I}}
\newcommand{\V}{\mathcal{V}}
\newcommand{\B}{\mathrm{B}}

\newcommand{\D}{\mathrm{D}}
\newcommand{\A}{\mathrm{A}}
\newcommand{\W}{\mathrm{W}}
\newcommand{\Y}{\mathrm{Y}}
\newcommand{\Z}{\mathrm{Z}}
\newcommand{\K}{\mathrm{K}}
\newcommand{\X}{\mathrm{X}}

\newtheorem{theorem}{Theorem}[section]
\newtheorem{lemma}[theorem]{Lemma}
\theoremstyle{definition}
\newtheorem{remark}[theorem]{Remark}
\newtheorem{proposition}[theorem]{Proposition}
\newtheorem{corollary}[theorem]{Corollary}
\newtheorem{definition}[theorem]{Definition}

\newtheorem{conjecture}[theorem]{Conjecture}
\allowdisplaybreaks

\newcommand{\nlr}[4]{#3\mathrel{\mathop{\centernot\longleftrightarrow}\limits_{#1}^{#2}} #4}

\makeatletter
\newcommand\xleftrightarrow[2][]{%
  \ext@arrow 9999{\longleftrightarrowfill@}{#1}{#2}}
\newcommand\longleftrightarrowfill@{%
  \arrowfill@\leftarrow\relbar\rightarrow}
  \newcommand*{\rom}[1]{\expandafter\@slowromancap\romannumeral #1@}
\makeatother

\newlist{steps}{enumerate}{1}
\setlist[steps, 1]{label = Step \arabic*:}
\title{Sharp asymptotics of disconnection time of large cylinders by simple and biased random walks}

\author{Xinyi Li\\Peking University\\xinyili@bicmr.pku.edu.cn \and Yu Liu\\Peking University\\liuyu8300@stu.pku.edu.cn \and Yuanzheng Wang\\MIT\\yuanz210@mit.edu}

\begin{document}

\maketitle

\begin{abstract}
	We investigate the asymptotic disconnection time of a large discrete cylinder $(\mathbb{Z}/N\mathbb{Z})^{d}\times \mathbb{Z}$, $d\geq 2$, by simple and biased random walks. For simple random walk, we derive a sharp asymptotic lower bound that matches the upper bound from \cite{Szn09a}. For biased walks, we obtain bounds that asymptotically match in the principal order when the bias is not too strong, which greatly improves non-matching bounds from \cite{Win08}. As a crucial tool in the proof, we also obtain a ``very strong'' coupling between the trace of random walk on the cylinder and random interlacements, which is of independent interest. 
\end{abstract}
\setcounter{page}{1}
\setcounter{section}{-1}
\tableofcontents
\section{Introduction}\label{sec:intro}

This paper studies the disconnection time $T_{N}$ of a large discrete cylinder $(\mathbb{Z}/N\mathbb{Z})^{d}\times \mathbb{Z}$, $d\geq 2$ by the trace of a random walk. This ``termite in the wooden beam" problem was first considered by Dembo and Sznitman in \cite{DS06} in which they proved that $T_N$ asymptotically grows like $N^{2d+o(1)}$. Later on, with the introduction of the model of random interlacements and the discovery of its connection with the trace of random walk in cylinders and tori, the asymptotics of disconnection time have been greatly improved and various phenomena regarding disconnection are better understood; see \cite{belius2011cover,DS08,Szn08,Szn09b,sznitman2009random,Szn09a,Windisch10random}.
In particular, Sznitman obtained in \cite{Szn09a} a conjecturally tight asymptotic upper bound for disconnection by a simple random walk. The asymptotic disconnection time of a biased walk (with an upward drift along the $\mathbb{Z}$-direction of strength $N^{-d\alpha}$ with $\alpha>0$) was first investigated by Windisch, who showed in \cite{Win08} that for $d\geq 2$, when $\alpha>1$, the disconnection time $T_{N}$ is still of order $N^{2d+o(1)}$, while for $\alpha<1$, $T_{N}$ becomes (stretched) exponential in $N$. For the latter case Windisch also gave upper and lower bounds that do not match (see \eqref{eq:Windisch} and \eqref{eq:Windisch d=2} for precise statements).

In this paper, we derive a sharp asymptotic lower bound in the simple random walk case that matches the upper bound in \cite{Szn09a} and (when combined with the upper bound) give the limiting law of $T_N/N^{2d}$. For biased random walk with $d\geq 2$, we significantly improve \cite{Win08} by providing precise asymptotics when $\alpha \geq 1$ and offering bounds that asymptotically match in the principal order when the bias is not too strong, that is, when $1/d<\alpha<1$. 
\smallskip

Before stating our main results, let us first present the model and notation in a more precise fashion and briefly introduce the model of random interlacements, which plays a crucial role in the analysis of this problem. 

For $d \geq 2$ and $N \geq 1$, we consider the discrete cylinder 
\begin{equation}
\mathbb E=\mathbb{T}\times\mathbb{Z},
\end{equation}
where $\mathbb{T}$ denotes the $d$-dimensional discrete torus $(\mathbb{Z}/N\mathbb{Z})^{d}$. The cylinder $\mathbb E$ is equipped with the $\ell^{\infty}$-distance $|\cdot|_{\infty}$ and the natural product graph structure, and all vertices $x_1,x_2\in \mathbb E$ with $|x_1-x_2|_{\infty}=1$ are connected with an edge. The (discrete-time) random walk with upward drift $\Delta\in[0,1)$ along the $\mathbb{Z}$-direction is the (discrete-time) Markov chain $(X_{n})_{n \geq 0}$ on $\mathbb E$ with transition probability
\begin{equation}\label{eq:def of transition probability}
p(x_1,x_2)=\frac{1+\Delta \cdot \pi_{\mathbb{Z}}(x_2-x_1)}{2d+2}\mathbbm{1}_{|x_1-x_2|_{\infty}=1},
\end{equation}
where $\pi_{\mathbb{Z}}$ denotes the projection from $\mathbb{E}$ to $\mathbb{Z}$. Note that when $\Delta=0$, the random walk is exactly the simple random walk on $\mathbb{T}$. A finite subset $S$ of $\mathbb E$ is said to disconnect $\mathbb E$ if for large $M$, $\mathbb{T}\times [M,\infty)$ and $\mathbb{T}\times (-\infty, -M]$ are contained in distinct connected components of $\mathbb E\setminus S$. The central object of interest is the disconnection time 
\begin{equation}\label{eq:def of T_N}
T_{N}:=\inf\left\{n\geq 0:X_{[0,n]}\text{ disconnects }\mathbb E\right\}.   
\end{equation}
Let $P_{0}^{N,\alpha}$ denote the law of the random walk on $(\mathbb{Z}/N\mathbb{Z})^{d}\times \mathbb{Z}$ started from the origin $(0,0,\dots,0)$ with drift $\Delta=N^{-d\alpha},\alpha\geq 0$ (see \Cref{subsec:prelim} regarding conventions on notation). When $\alpha=0$, we simply write $P_{0}^{N}$ for short. We use $\mathbb{W}$ to denote the Wiener measure and write
\begin{equation}\label{eq:def of first achiving time}
\zeta^{\mu}(u):=\inf\left\{t \geq 0:\sup_{v\in\mathbb{R}}L^{\mu}(v,t)\geq u\right\},\quad u \geq 0,\, \mu \in \mathbb{R},
\end{equation}
where for every $\mu\in\mathbb{R}$, $L^{\mu}(v,t)$ is a jointly continuous version of the local time of a Brownian motion with drift $\mu$ and we call $\zeta^\mu(u)$ the ``record-breaking time", that is,  the first time the maximum of the local time reaches $u$. When $\mu=0$, we simply write $\zeta^{\mu}$ and $L^{\mu}$ as $\zeta$ and $L$, in which case the explicit distribution of $\zeta$ is known; see \eqref{eq:laplace of zeta} for details.

We now turn to random interlacements. This model, first introduced by Sznitman in \cite{Szn10}, plays a central role in the study of the percolative properties of the trace of random walks. 
Heuristically, the interlacement set $\I^{u}$ is the trace of a Poissonian cloud of bi-infinite transient $\mathbb{Z}^{d+1}$-valued trajectories modulo time-shift whose intensity measure is governed by the level parameter $u>0$ (here $d+1$, with $d \geq 2$, plays the role of $d \geq 3$ in \cite{Szn10}). 
The complement of $\I^{u}$, denoted by $\V^{u}$, is called the vacant set at level $u$. 
It is known that there exist several positive critical thresholds $0<\overline{u}\leq u_{*} \leq u_{**}<\infty$ regarding the percolation phase transition of the vacant set, which delimit the strongly super-critical regime, existence of infinite cluster and the strongly sub-critical regime, respectively. 
Recently in a series of extraordinary works \cite{DCGR+23c,DCGR+23a,DCGR+23b} by Duminil-Copin, Goswami, Rodriguez, Severo and Teixeira, it is proved that these three critical parameters are actually equal, that is,
\begin{equation}\label{eq:unique critial parameter}
\overline{u}=u_{*}=u_{**}.
\end{equation}
In fact, all results in this paper involve expressions of $\overline{u}$ only for the lower bound and expressions of $u_{**}$ only for the upper bound, and we apply \eqref{eq:unique critial parameter} to show that these bounds do match. 
We refer readers to \Cref{subsec:interlacements} for a more detailed introduction of random interlacements. We also mention that \eqref{eq:unique critial parameter} also plays a key role in three intimately related problems, namely sharp asymptotics of fragmentation time of large tori (see \cite{TW11,10.1214/15-AAP1165,DCGR+23b}), sharp asymptotic probability of the disconnection of a macroscopic body (see \cite{li2017lower,Szn17,nitzschner2020solidification}) and bulk deviations for random walk and random interlacements \cite{AM23,Szn21,Szn23a,Szn23b}.

\subsection{Main results}
In the simple random walk case, we derive an asymptotic lower bound on the tail distribution of the disconnection time $T_{N}$.
\begin{theorem}\label{thm:sharp lower bound}
For all $s>0$, we have 
\begin{equation}\label{eq:sharp lower bound}
\liminf_{N\to \infty}P_{0}^{N}\left[\frac{T_{N}}{N^{2d}}\geq s\right]\geq \WW\left[\zeta\left(\frac{\overline{u}}{\sqrt{d+1}}\right)\geq s\right].
\end{equation}
\end{theorem}
In \cite[Corollary 4.6]{Szn09a}, Sznitman proved the following upper bound.
\begin{equation}\label{eq:sharp upper bound}
\limsup_{N\to \infty}P_{0}^{N}\left[\frac{T_{N}}{N^{2d}}\geq s\right]\leq \WW\left[\zeta\left(\frac{u_{**}}{\sqrt{d+1}}\right)\geq s\right].
\end{equation}
Thanks to \eqref{eq:unique critial parameter}, we can combine \eqref{eq:sharp lower bound} and \eqref{eq:sharp upper bound}, and obtain the precise weak limit for the renormalized disconnection time. 

\begin{theorem}\label{thm:sharp asymptotics}
Under the law $P_0^{N}$, as $N\rightarrow \infty$, we have
\begin{equation}\label{eq:sharp aymptotics}
\frac{T_{N}}{N^{2d}}\Longrightarrow \zeta\left(\frac{u_{*}}{\sqrt{d+1}}\right)\overset{d}{=}\frac{u_{*}^2}{d+1}\zeta(1).
\end{equation}
\end{theorem}

We now turn to the biased walk. 
Recall the bounds obtained by Windisch (see \cite[Theorem 1.1]{Win08}): for $d\geq 3$ and every $\delta > 0$, with $P_{0}^{N,\alpha}$-probability tending to $1$ as $N\to\infty$, 
\begin{equation}\label{eq:Windisch}
\begin{split}
2d-\delta\leq &\;\frac{\log T_{N}}{\log N} \leq 2d+\delta,\quad \alpha>1; \\
d(1-\alpha-\varphi(\alpha))-\delta\leq &\;\frac{\log\log T_{N}}{\log N} \leq d(1-\alpha)+\delta,\quad 0<\alpha<1,
\end{split}
\end{equation}
where $\varphi(\alpha)$ is a {\it strictly positive} piece-wise linear function on $(0,1)$; see \cite[(1.5)]{Win08} for its precise definition. In addition, for $d=2$, this work provides a similar upper bound as in $d\geq 3$ case, and a lower bound that as $N\rightarrow \infty$, with $P_{0}^{N,\alpha}$-probability tending to $1$,
\begin{equation}\label{eq:Windisch d=2}
    T_N\geq \exp(cN^{2({1-2\alpha})}), \quad 0 < \alpha < 1/2.
\end{equation}

In the next theorem, we obtain  the limiting distribution of $T_{N}/N^{2d}$, for $\alpha \geq 1$ (which is exactly the same as in \eqref{eq:sharp aymptotics} when $\alpha>1$ and in a similar form but with local time of drifted Brownian motion involved when $\alpha=1$), while for $1/d<\alpha<1$ we obtain bounds that match in the principal order, in line with the upper bound in the second line of \eqref{eq:Windisch} (but more precise). These asymptotics are valid for all $d\geq 2$. 
This is a significant improvement of \eqref{eq:Windisch} when the drift is not too strong.

\begin{theorem} \label{thm:biased walk aymptotics}
Under the law $P_{0}^{N,\alpha}$, as $N\rightarrow \infty$, we have
\begin{align}
    \frac{T_{N}}{N^{2d}}&\Longrightarrow \zeta\left(\frac{u_{*}}{\sqrt{d+1}}\right),\quad \alpha>1;\label{eq:distribution limit large alpha}\\ 
\frac{T_{N}}{N^{2d}}&\Longrightarrow \zeta^{\frac{1}{\sqrt{d+1}}}\left(\frac{u_{*}}{\sqrt{d+1}}\right),\quad \alpha=1;\label{eq:distribution limit alpha=1}\\
\frac{\log T_{N}}{N^{d(1-\alpha)}}&\Longrightarrow \frac{u_{*}}{d+1},\quad 1/d<\alpha<1.\label{eq:distribution limit small alpha}
\end{align}
\end{theorem}

Let us remark that within the present framework we are not able to give better bounds than \eqref{eq:Windisch} in the presence of an extremely strong bias (i.e., when $\alpha\leq 1/d$) which  radically changes the nature of this problem and poses new and essential difficulties; see the short discussion at the end of \Cref{subsec:sketch} and \Cref{rem:finalremark} for more details. 

We also mention that our approach is capable of deriving sharp bounds for more general drifts, say, of the form $ CN^{-d\alpha}(1+o(1))$ for $1/d<\alpha \leq 1$ or $N^{-d}\log^\beta N$ with $\beta$ a fixed positive constant. (The interest in probing the latter case is to observe in a quantitative fashion the transition of the local time from being polynomial in $N$ to (stretched) exponential in $N$.) We leave details to curious readers. 

\subsection{Sketch of proofs}\label{subsec:sketch}

We start by going through the intuitions behind the results \Cref{thm:sharp lower bound,thm:sharp asymptotics,thm:biased walk aymptotics}. After that, we  explain the outline of the proof for the lower bound in the simple random walk case that is contained in \Cref{sec:Geometric,sec:bad(beta-gamma),sec:bad local time}, and then move on to the biased case lower bound by discussing necessary additional technical details that are incorporated in \Cref{sec:lower bound in biased case}, and finally sketch the proof of the upper bounds in both simple and biased walks that is detailed in \Cref{sec:upper bound}. 

We now explain the ideas behind the simple random walk results \eqref{eq:sharp lower bound}-\eqref{eq:sharp aymptotics}. 
The main underlying intuition for $T_{N}$ ``living in scale $N^{2d}$" is that it takes about $N^{2d}$ steps to cover a positive proportion of the torus at some height $z$, whose cardinality is of order $N^{d}$. 
To rigorously formalize this intuition, one compares the random walk trajectories with random interlacements.
As demonstrated in \cite[Theorem 1.1]{Szn09b} and \cite[Section 4]{Szn09a}, for a small box $B$ with side-length $N^{\psi}$ ($0<\psi<1$) centered at height $z$, the strong mixing property of random walk implies the trace left by the walk in $B$
resembles a sample of interlacements, with intensity proportional to the ``average local time'' of the small box. 
In addition, it also suggests that the average local time in a different box $B'$ at same height $z$ is approximately the same as that of $B$, a phenomenon commonly referred to as ``spatial regularity''. 

In light of this, the percolative property for the complement of the trace left in $B$ can then be indicated by 
the local time on $\mathbb{T}\times \{z\}$. 
For example, 
when the average local time is smaller than $\overline{u}$ (resp.~larger than $u_{**}$), the intensity of the corresponding interlacements in box $B$ is also smaller than $\overline{u}$ (resp.~larger than $u_{**}$), meaning the vacant set of interlacements lies in the ``strongly percolative" (resp.~``strongly non-percolative") regime. 
This results in the complement of the random walk trajectory being ``well-connected'' (resp.~``well-fragmented"). In short, disconnection should happen when the average local time of some level $\mathbb{T}\times\{z\}$ exceeds $u_{*}$ (which equals $\overline{u}$ and $u_{**}$ thanks to \eqref{eq:unique critial parameter}).
Therefore, estimating the asymptotics of disconnection time $T_N$ boils down to 
analyzing the local time profile of a simple random walk in one dimension, which, after appropriate scaling, is further associated with that of a one-dimensional standard Brownian motion. 

The above intuition also applies to biased walks and hints at the results \eqref{eq:distribution limit large alpha}-\eqref{eq:distribution limit small alpha}.
(However, we are only able to verify it for $\alpha > 1/d$; see the short discussion at the end of this subsection as well as \Cref{rem:finalremark}.) This time we need to analyze a one-dimensional biased random walk, and a major phase transition happens when $\alpha=1$ such that $N^{-d\alpha}$ is the reciprocal of $N^{d}$, the size of the base $(\mathbb{Z}/N\mathbb{Z})^{d}$.

We remark here that although Sznitman has discussed the domination between simple random walk trajectories on a cylinder and interlacements both in the lower bound case (see \cite[Theorem 1.1]{Szn09b}) and the upper bound case (see \cite[Section 4]{Szn09a}), such domination can only help derive a sharp upper bound. The reason is that the ``strongly percolative" property is not monotone. 
To be more specific, given a set $V \subseteq \mathbb Z^{d+1}$ that satisfies this property, the ``strongly-percolative'' property may not hold for a larger set $V'$ containing $V$.
However, the ``strongly non-percolative'' property used in obtaining the upper bound of $T_N$ is actually monotone, making the coupling in \cite[Section 5]{Szn09a} sufficient for deducing \eqref{eq:sharp upper bound}.\\ 

Let us now sketch the proof of lower bound on the disconnection time $T_N$ in the simple random walk case. We remark that the proofs for biased walks (see equations \eqref{eq:lower bound large alpha} to \eqref{eq:lower bound small alpha} below) with $\alpha > 1/d$ will subsequently follow a similar strategy with certain technical adjustments. Many of the techniques below have drawn inspirations from the approaches presented in \cite{Szn09b,Szn09a,Szn17}. The formal definitions of several key ingredients (such as $\underline S_N$, $\text{good}(\beta, \gamma)$, $\text{fine}(\gamma)$, etc.) will be provided in subsequent sections.

As discussed in \Cref{sec:Geometric,sec:bad(beta-gamma),sec:bad local time,sec:lower bound in biased case}, we essentially want to show that, if for every height $z \in \mathbb Z$, the number of distinct visits of the simple random walk $X_{\cdot}$ to $\mathbb{T}\times \{z\}$ is no more than a certain level, then with high probability disconnection cannot happen. More precisely, for $\delta > 0$, we define a record-breaking time (see \Cref{sec:one-dimensional walks} for details)
\begin{equation}
\underline S_N := \inf\left\{n: \mbox{ there exists }z\in\mathbb{Z}, \mbox{ s.t. }X_{[0, n]} \mbox{ has more than }\frac{\overline{u}-\delta}{d+1}\cdot N^{d} \mbox{ distinct visits to }\mathbb{T}\times \{z\}\right\},
\end{equation}
we claim that with high probability, the disconnection time $T_{N}$ is no less than $\underline{S}_{N}$. Therefore the problem can be reduced to analyzing a one-dimensional lazy random walk and its local time. 

To achieve this, one conducts a coarse-graining framework and considers disjoint boxes $B$ of side-length $N^{\psi}$ (with $\psi\in(1/d,1\land \alpha)$) in the cylinder $\mathbb E$. For each box we consider a slightly larger box $D$ and a much larger box $U$ with $B \subseteq D \subseteq U$; see \eqref{eq:def of boxes at origin}-\eqref{eq:def of translated boxes} for precise definitions.
Thanks to the recurrence of this walk, for each box $B$ one can consider an infinite number of successive excursions $W_{\ell}^{D},\ell\geq 1$ from $D$ to $\partial U$. 
For two fixed constants $\beta>\gamma$ in $(\overline{u}-\delta,\overline{u})$, we define the following events (see  \Cref{def:good(beta-gamma),def:bad local time} for formal definition):
\begin{equation}
    \begin{split}
{\rm good}(\beta, \gamma)&\; := \left\{D\setminus \cup_{\ell\leq \gamma\cdot\text{cap}(D)}W_{\ell}^{D} \mbox{ is strongly percolative}\right\}, \text{ and }\\
{\rm fine}(\gamma)&\;:= \left\{X_{[0,\underline{S}_{N}]} \mbox{ contains no more than }\gamma\cdot\text{cap}(D) \mbox{ excursions }W_{\ell}^D\right\},    
    \end{split}
\end{equation}
where $\text{cap}(D)$ refers to the capacity of set $D$ (see  \eqref{eq:equilibrium measure} for definition). 
Roughly, the first event encapsulates a ``strongly percolative'' property for the complement of the first $\gamma\cdot\text{cap}(D)$ excursions of the random walk, drawing inspiration from a similar definition for random interlacements in sub-critical regime, while the second event requires that the local time of $(X_n)_{n\geq 0}$ at box $B$ before time $\underline{S}_{N}$ is not too large. We refer to the complement of these two events as $\text{bad}(\beta,\gamma)$ and $\text{poor}(\gamma)$. With these events in mind, we define (see \Cref{def:bad boxes}) for box $B$ a property named
\begin{equation}
\text{normal}(\beta, \gamma):= \text{good}(\beta,\gamma) \cap \text{fine}(\gamma).
\end{equation}
A box with this property is favorable for us because, on one hand, it occurs with high probability, and on the other hand, it facilitates the construction of a connected path in the complement of $X_{[0,\underline{S}_{N}]}$ in $\mathbb E$ between the box $B$ and an adjacent box $B'$ of the same size.

In view of the definition of ``normal" boxes, through a geometric argument \Cref{prop:geometric argument}, the event $T_{N}\leq \underline{S}_{N}$ can happen only if one can find a box $\mathsf{B}$ of side-length $[N/\log^{3}N]$ containing a ``$d$-dimensional'' coarse-grained surface of abnormal boxes. We will prove that the occurrence of any of such surfaces has an extremely low probability (see \Cref{prop:bad(beta-gamma),prop:bad local time} respectively). On the other hand, the combinatorial complexity of such a surface can also be bounded from above. Indeed, the choices of the box $\mathsf{B}$ is polynomial in $N$ despite the fact that the cylinder $\mathbb{E}$ is an infinite set, since $X_{[0,\underline{S}_{N}]}$ can be confined in a cylinder of height $O(N^{2d})$ with high probability. The restriction $\psi>1/d$ then helps to bound the combinatorial complexity of the $d$-dimensional set given $\mathsf{B}$. 

In the proof of \Cref{prop:bad(beta-gamma)}, the unlikeliness of a $d$-dimensional surface of ``bad'' boxes follows from a similar argument as in \cite{Szn17}, leveraging the ``good decoupling" properties of the excursions $W_{\ell}^{D}$ when the boxes are sufficiently far apart. 
Combining also the soft local time techniques in \cite{PT15}, especially in the form developed in \cite[Section 2]{CGP13}, one can couple excursions of random walk with independent collections of i.i.d.~excursions (see \Cref{prop:coupling W and widetildeZ,prop:coupling between widehatZ and W}), and then subsequently with excursions constructed by random interlacements, see \Cref{{lem:coupling between widetildeZ and Z}}. 

On the other hand, proving  \Cref{prop:bad local time}, the unlikeliness of ``poor'' boxes $B_x$, with $x$ indexed by a $d$-dimensional set $\mathcal C_1$, requires a ``very strong'' coupling. 
Here, the excursions of $X_{[0,\underline{S}_{N}]}$ going from $\mathsf{B}$ to a larger concentric box $\mathsf{D}$ with side-length $[N/20]$ are dominated by excursions in $\I^{u'}\cap\mathsf{D}$ for a certain $u'\in (\overline{u} -\delta, \gamma)$, as listed in \Cref{prop:very strong coupling}. 
Under this coupling, the event that $\left\{N_{\underline{S}_{N}}(B_x)>\gamma\cdot \text{cap}(D_x)\right\}$ holds for each $x\in \mathcal C_1$ indicates an excessive number of excursions in the global set $\I^{u'}\cap\mathsf{B}$. The probability of this event is of order $\exp(-c\text{cap}(\mathsf B))$ (which is $\exp(-N^{d-1}/\log^{c}N))$ here) as determined by the exponential Chebyshev's inequality for the occupation time of continuous-time random interlacements in \cite{Szn17}.

It is important to note that this coupling differs from those of \cite[Theorem 1.1]{Szn09b} and \cite{Szn09a}, as it requires the domination of excursions rather than just the range of excursions, and has a much smaller error term. 
Nevertheless, the proof remains quite similar, with significant optimization of the error terms inspired by \cite{Bel13}.
We also remark that the proof will be incorporated in \Cref{sec:couplings}, where we develop a more general version of the coupling. \\

We then sketch the proof of lower bound on the disconnection time $T_N$ of biased walks. We shall prove that, for every $\delta > 0$, the lower bound of $T_N$ satisfies
\begin{align}
\liminf_{N\to \infty}\,&P_{0}^{N,\alpha}\left[\frac{T_{N}}{N^{2d}}\geq s\right]\geq \WW\left[\zeta\left(\frac{\overline{u}-\delta}{\sqrt{d+1}}\right)\geq s\right],\quad &\alpha>1; \label{eq:lower bound large alpha}\\
\liminf_{N\to \infty}\,&P_{0}^{N,\alpha}\left[\frac{T_{N}}{N^{2d}}\geq s\right]\geq \WW\left[\zeta^{\frac{1}{\sqrt{d+1}}}\left(\frac{\overline{u}-\delta}{\sqrt{d+1}}\right)\geq s\right],\quad &\alpha=1;\label{eq:lower bound alpha=1}\\
\lim_{N\to\infty}&P_{0}^{N,\alpha}\left[\log T_{N}\geq \frac{\overline{u}-2\delta}{d+1}\cdot N^{d(1-\alpha)}\right]=1, \quad &\frac{1}{d}< \alpha < 1,\label{eq:lower bound small alpha}
\end{align}
after which we can send $\delta$ to zero and use the continuity of Brownian local time to get the desired bound. The proof resembles that of \eqref{eq:sharp lower bound}, and we now briefly explain how to adapt the proof of \eqref{eq:sharp lower bound} to prove \eqref{eq:lower bound large alpha}-\eqref{eq:lower bound small alpha}. The details are contained in \Cref{sec:lower bound in biased case}. 

First, since the biased walk is no longer recurrent, we need to introduce infinitely many auxiliary biased walks started from a uniform distribution on $\mathbb{T}\times\{z\}$ where $z$ is sufficiently negative, and define the excursions $W_{\ell}^{D}, \ell\geq 1$ by extracting the excursions from both the original walk and the supplementary walks in order. Second, the excursions $W_{\ell}^{D}, {\ell\geq 1}$ are now excursions of biased random walks. This requires addressing some technicalities when comparing $W_{\ell}^{D}, \ell\geq 1$ with
i.i.d.~biased excursions and further with i.i.d.~unbiased excursions (see \Cref{prop:coupling W and widetildeW,lem:coupling widetildeW with widetildeZ}). The condition $N^{\psi}<N^{d\alpha}$ is important here so that typically the walk cannot ``feel" the drift inside small box $B$ (recall the side-length of $B$ is $N^{\psi}$). 
Third, when $\alpha<1$, the random time $\underline{S}_{N}$ is now of order $\exp(\frac{\overline{u}-\delta}{d+1}\cdot N^{d(1-\alpha)})$, and typically the trace of $X_{[0, \underline S_N]}$ can no longer be restricted in a cylinder with height polynomial in $N$. 
In this case, we will lose an exponential factor in the combinatorial complexity of box $\mathsf{B}$, which can still be absorbed, thanks to the requirement that $\psi<\alpha$. Note that the assumption $\alpha>1/d$ is crucial to allow us to pick a suitable $\psi$.\\

We now turn to the upper bound (including the simple random walk case). We shall prove that for every $\delta>0$, 
\begin{align}
    \limsup_{N\to \infty}\,&P_{0}^{N,\alpha}\left[\frac{T_{N}}{N^{2d}}\geq s\right]\leq \WW\left[\zeta\left(\frac{u_{**}+\delta}{\sqrt{d+1}}\right)\geq s\right], &\quad \alpha>1; \label{eq:upper bound large alpha}\\
    \limsup_{N\to \infty}\,&P_{0}^{N,\alpha}\left[\frac{T_{N}}{N^{2d}}\geq s\right]\leq \WW\left[\zeta^{\frac{1}{\sqrt{d+1}}}\left(\frac{u_{**}+\delta}{\sqrt{d+1}}\right)\geq s\right], &\quad \alpha=1; \label{eq:upper bound alpha=1}\\
    \lim_{N\to\infty}&P_{0}^{N,\alpha}\left[\log T_{N}\leq \frac{u_{**}+2\delta}{d+1}\cdot N^{d(1-\alpha)}\right]=1, &\quad \frac{1}{d} < \alpha < 1,\label{eq:upper bound small alpha}
\end{align}
after which we again send $\delta$ to zero and use the continuity of Brownian local time to obtain the desired bound. This time, the key step is to show that, if for some level $\mathbb{T}\times \{z\}$, the number of distinct visits of the random walk $X$ to $\mathbb{T}\times \{z\}$ exceeds a certain level, then with high probability disconnection will happen. In other words, writing
\begin{equation}
\overline S_N(z) = \inf\left\{n: X_{[0, n]} \mbox{ has more than }\frac{u_{**}+\delta}{d+1}\cdot N^{d} \mbox{ distinct visits to }\mathbb{T}\times \{z\} \right\},
\end{equation}
then we claim that with high probability, the disconnection time $T_{N}$ is no more than $\overline{S}_{N}(z)$ for any $z\in\mathbb{Z}$ (see \Cref{prop:overlineS and T} and \Cref{cor:overlineS and Tz}), thus again reducing the problem to the analysis of a one-dimensional lazy random walk.  

To illustrate this idea, we fix $z = 0$ as an example. 
The key observation is that, conditioned on the event $\left\{\overline S_N(0) < \infty\right\}$, the biased walk $(X_n)_{n\geq 0}$ which originally has upward drift now has a drift towards $\mathbb T \times \{0\}$ before time $\overline{S}_{N}(0)$, and after that it again has the same upward drift (see \Cref{lem:conditional law}). 
One then use the ``strong'' coupling (see \Cref{{prop:strong coupling}}), which states that for every box $\overline{B}$ with side-length $[N^{1/3}]$ around $\mathbb T\times \{0\}$, the trace of $X_{[0,\overline{S}_{N}(0)]}$ in $\overline{B}$ dominate $\I^{u_{**}+\delta/2}\cap \overline{B}$ with high probability (larger than $1-N^{-10d}$), to prove the ``strongly non-percolative" property by definition of $u_{**}$. 
In other words, one shows that with high conditional probability, every small box $\overline{B}$ around height $0$ is ``strongly non-percolative", thus creating a flat ``fence" for the complement of $X_{[0,\overline{S}_{N}(0)]}$. Note that the ``strong" coupling here, which is similar to that in \cite{Szn09a} and \cite[Theorem 1.1]{Szn09b}, has a similar statement as the ``very strong" coupling in the proof of lower bounds, except that we prove stochastic domination in the opposite direction, and this coupling will also be treated as a special case of the general coupling in \Cref{sec:couplings}. That said, unlike the ``very strong" coupling in the proof of general lower bounds, it suffices for the error term of the ``strong" coupling here to be a sufficiently small negative polynomial of $N$, say $N^{-10d}$, since we only need to give a union bound on polynomially many (in $N$) boxes at the same height.  

Let us quickly remark that there are two places in the proof where the assumption $\alpha>1/d$ is crucially used. The first one is \Cref{sec:lower bound in biased case} where one needs to carefully pick the size of the mesoscopic boxes $\mathsf{B}$, as explained earlier in this subsection. The second one is to ensure spatial regularity in the derivation of the very strong coupling between the trace of biased walk and interlacement sets (see \Cref{lem:horizontal mixing}). See \Cref{rem:finalremark} for more discussions.

Before ending this subsection, we also remark here that there is actually an alternative approach inspired by \cite{RS13} to obtain the lower bound \eqref{eq:sharp lower bound} for simple random walk. Unfortunately this approach does not extend to the small $\alpha$'s that the approach in the main body of text is able to treat. Hence we do not pursue it in detail but rather sketch it in \Cref{sec:appendix}.

\subsection{Organization of the article}

We set up notation and preliminaries in \Cref{sec:preliminary}. Then in \Cref{sec:one-dimensional walks} we analyze one-dimensional simple or biased lazy random walk and prove asymptotics with respect to $\underline{S}_{N}$ and $\overline{S}_{N}$. 

\Cref{sec:Geometric,sec:bad(beta-gamma),sec:bad local time,sec:lower bound in biased case} are devoted to proving $\underline S_N \leq T_N$. 
In \Cref{sec:Geometric} we first introduce the coarse-graining setup and then with the help of a geometric argument (see \Cref{prop:geometric argument}) reduce the event $\underline S_N \geq T_N$ to two cases, \Cref{prop:bad(beta-gamma)} (surface of $\text{bad}(\beta,\gamma)$ boxes) and \Cref{prop:bad local time} (surface of $\text{poor}(\gamma)$ boxes). 
\Cref{sec:bad(beta-gamma),sec:bad local time} are devoted to the two situations respectively. We construct a chain of couplings in \Cref{subsec:soft local time} that incorporate various ingredients from \cite{PT15,Szn17}, and prove \Cref{prop:bad(beta-gamma)} in \Cref{subsec:proof of prop bad(beta-gamma)}. We state the ``very strong" coupling in \Cref{prop:very strong coupling}, and use occupation-time bounds from \cite{Szn17} to prove \Cref{prop:bad local time} in \Cref{subsec:continuous-time interlacements}. Then in \Cref{sec:lower bound in biased case}, the above procedures in \Cref{sec:Geometric,sec:bad(beta-gamma),sec:bad local time} are generalized to the biased walk case respectively in \Cref{subsec:adapting geometric,subsec:adapting bad(beta-gamma),subsec:adapting bad local time}.

In \Cref{sec:upper bound}, we show that for all $z\in\mathbb{Z}$, the estimate $T_N\leq \overline S_N(z)$ holds with high probability. This inequality relies on the ``strong" coupling; see \Cref{prop:strong coupling}. 

\Cref{sec:couplings} gives a generalized version of the ``very strong" coupling between random walks on the cylinder and random interlacements (see \Cref{thm:very strong coupling S,thm:very strong coupling excursion}). The outline for proving \Cref{thm:very strong coupling S,thm:very strong coupling excursion}, the proof details, and their modifications for \Cref{prop:very strong coupling,prop:very strong coupling bias,prop:strong coupling} are respectively given in \Cref{subsec:proof of very strong coupling theorem,subsec:proof of horizontal indep etc.,subsec:adapted proof very strong coupling}. 

\Cref{sec:denouement} concludes the proofs of our main theorems, where some remarks on the disconnection time of the case $0<\alpha\leq 1/d$ and our result are also incorporated. 

We have included various notation tables in \Cref{sec:symbol} for readers' convenience. Finally, \Cref{sec:appendix} provides the sketch of simple proof of \Cref{thm:sharp lower bound}.

We now explain the convention concerning constants. Throughout the text $c,c',\widetilde{c},C,C',\widetilde{C}\cdots$ denote positive constants changing from place to place that only depend on the dimension $d$. Numbered constants $c_0, c_1,\dots$ refer to constants whose value is fixed at their first appearance. Dependence on additional parameters appears at the first appearance (unless otherwise specified). For instance $c(\delta)$ will stand for a positive constant depending on both $d$ and $\delta$. \\

{\bf Acknowledgements}: The authors thank Zhan Shi for inspiring discussions and Maximilian Nitzschner for helpful comments on an earlier draft. The authors are supported by National Key R\&D Program of China (No.\ 2021YFA1002700 and No.\ 2020YFA0712900) and NSFC (No.\ 12071012).

\section{Notation and Preliminaries}\label{sec:preliminary}
In this section we introduce the basic notation and collect various facts concerning simple random walks, potential theory, and random interlacements. Throughout, we tacitly assume that $d\geq 2$.

\subsection{Notation}\label{subsec:prelim}
We write $\mathbb{N}=\{0,1,2,\dots\}$ for the set of natural numbers. Given a non-negative real number $a$, we denote by $[a]$ the integer part of $a$, and for real numbers $b$ and $c$, we write $b\land c $ or $b\lor c$ for the respective minimum and maximum between $b$ and $c$.

The $d$-dimensional torus $\mathbb{T}=(\mathbb{Z}/N\mathbb{Z})^{d}$ can be embedded into $\mathbb{Z}^{d}$ so that there is a one-to-one correspondence between $\mathbb{T}$ and the box $\{0,1,\dots,N-1\}^{d}$. In the rest of this paper, we arbitrarily choose such an embedding, and any point $x$ in the cylinder can be denoted by the coordinate $x= (u, v) = (u_1,u_2,\dots,u_{d}, v)$ with $u_1,u_2,\dots,u_{d}\in \{0,1,\dots,N-1\}$ and $v\in\mathbb{Z}$. Without causing ambiguity, we will simply write $0$ for the origin $(0,0,\dots,0)$. The projections $\pi_{i}$, $i=1,2,\dots,d+1$, from $\mathbb E$ to the $d$-dimensional hyperplanes of $\mathbb E$ are the mappings from $\mathbb E$ to $(\mathbb{Z}/N\mathbb{Z})^{d-1}\times\mathbb{Z}$ when $i=1,2,\dots,d$, or to $(\mathbb{Z}/N\mathbb{Z})^{d}$ when $i=d+1$. Specifically, the projection $\pi_i$ are defined by omitting the $i$-th component of the coordinate $(u_1,u_2,\dots,u_{d}, v)$, and will play an important role in \Cref{sec:Geometric}. Additionally, we write $\pi_{\mathbb T}$ and $\pi_{\mathbb Z}$ for the respective canonical projections from ${\mathbb E} = \mathbb T \times \mathbb Z$ onto $\mathbb T$
and $\mathbb Z$. Note that $\pi_{\mathbb T}$ is equivalent to $\pi_{d+1}$ indeed, while we use these two symbols in different contexts for clarity. 

We let $|\cdot|$ and $|\cdot|_{\infty}$ respectively stand for Euclidean and $\ell^{\infty}$-distances on $\mathbb{Z}^{d+1}$ or for the corresponding distances induced on the cylinder $\mathbb E$. We say that two points $x,y$ of $\mathbb{Z}^{d+1}$ or $\mathbb E$ are neighbors, if $|x-y|=1$.
We denote  by $B(x,r)$ and $S(x,r)$ the closed $|\cdot|_{\infty}$-ball and $|\cdot|_{\infty}$-sphere with radius $r\geq 0$ and center $x$ in $\mathbb{Z}^{d+1}$ or $\mathbb E$. 
For a subset $A$ of $\mathbb{Z}^{d+1}$ or $\mathbb{E}$, we write $A\subset\subset\mathbb{Z}^{d+1}$ or $A\subset\subset \mathbb E$ to indicate that $A$ is a finite subset of $\mathbb{Z}^{d+1}$ or $\mathbb E$.

For $A,B$ subsets of $\mathbb{Z}^{d+1}$ or $\mathbb E$, we write $A+B$ for the set of elements $x+y$ with $x$ in $A$ and $y$ in $B$, and $d(A,B):=\inf\{|x-y|_{\infty}: x\in A, y\in B\}$ for the natural $\ell^{\infty}$-distance between $A$ and $B$. When $A=\{x\}$ is a singleton, we simply write $d(x,B)$ for short. 
If we are further given $A\subseteq B$, then we denote with $\partial_{B}A$ the relative outer boundary of $A$ in $B$ and $\partial_{B}^{\text{int}}A$ the relative internal boundary of $A$ in $B$: 
\begin{equation}\label{def:relative boundary}
    \partial_{B}A:=\{x\in B\setminus A:\exists x'\in A, |x-x'|=1\},\quad \partial_{B}^{\text{int}}A:=\{x\in A:\exists x'\in B\setminus A, |x-x'|=1\}.
\end{equation}
When $B=\mathbb{Z}^{d+1}$ or $B=\mathbb{E}$, we simply write $\partial A$ and $\partial^{\text{int}}A$. 
Given $A,B,U$ subsets of $\mathbb{Z}^{d+1}$ or $\mathbb E$, we say that $A$ and $B$ are connected in $U$ and write $A\overset{U}{\longleftrightarrow}B$ when there exists a nearest-neighbour path with values in $U$ which starts in $A$ and ends in $B$. If there exists no such path, we say that $A$ and $B$ are not connected in $U$, denoted by $\nlr{}{U}{A}{B}$.  

Denote by $\text{supp}(\mu)$ the support of a point measure $\mu$ so that
\begin{equation}\label{eq:def of support}
\mu=\sum_{x\in\text{supp}(\mu)}\delta_{x}.
\end{equation}
Note that $\text{supp}(\mu)$ may be a multiset. In the rest of this paper, we will often consider $\mu$ as a $\sigma$-finite delta measure on the space of excursions from $K$ to the boundary of $U$, where $K\subseteq U$ are two finite subsets of $\mathbb{E}$ or $\mathbb{Z}^{d+1}$. In this case, $\text{supp}(\mu)$ is a multiset of excursions from $K$ to $\partial U$. We will also introduce some Poisson point processes as random measures supported on these $\sigma$-finite delta measures on the space of excursions. 

\subsection{Random walk on cylinders and lattices}\label{subsec:rwcylinder}
Throughout this work, the simple or biased random walk on $\mathbb E$ is denoted by $(X_n)_{n\geq 0}$. We write $(Y_{n})_{n\geq 0}:=(\pi_{\mathbb{T}}(X_{n}))_{n\geq 0}$ and $(Z_{n})_{n\geq 0}:=(\pi_{\mathbb{Z}}(X_{n}))_{n\geq 0}$ for the respective $\mathbb{T}$- and $\mathbb{Z}$-projections for this walk. For each $x$ in $\mathbb E$, we denote by $P_{x}^{N,\alpha}$ the law on $\mathbb E^{\mathbb{N}}$ of a random walk with upward drift $N^{-d\alpha}$ along the $\mathbb Z$-direction started at $x$, and write $E_{x}^{N,\alpha}$ for the corresponding expectation. Moreover, when $\mu$ is a measure on $\mathbb E$, we write
$P_{\mu}^{N,\alpha}$ and $E_{\mu}^{N,\alpha}$ for the measure $\sum_{x\in \mathbb E}\mu(x)P_{x}^{N,\alpha}$ (which is not necessarily a probability measure) and
its corresponding expectation (that is, the integral with respect to the measure $P_{\mu}^{N,\alpha}$). When $\alpha=\infty$, that is, when there is no drift, we simply write $P_{x}^{N},E_{x}^{N},P_{\mu}^{N},E_{\mu}^{N}$ for short.

Let $\mathscr T_{\mathbb{E}}$ and $\mathscr T_{\mathbb{Z}^{d+1}}$ respectively stand for the set of nearest-neighbor $\mathbb E$-valued and $\mathbb{Z}^{d+1}$-valued trajectories with time indexed by $\mathbb N$. 
When $F$ is a subset of $\mathbb E$, or of $\mathbb Z^{d+1}$, we denote by $\mathscr T_{F}$ the countable set of nearest neighbor $(F\cup \partial F)$-valued trajectories which remain constant after a finite time. 
The canonical shift on $\mathbb E^{\mathbb N}$ or $(\mathbb{Z}^{d+1})^{\mathbb{N}}$ is denoted by $(\theta_{n})_{n\geq 0}$, that is, $\theta_{n}$ stands for the map from $\mathbb E^{\mathbb N}$ into $\mathbb E^{\mathbb N}$ or from $(\mathbb{Z}^{d+1})^{\mathbb{N}}$ to $(\mathbb{Z}^{d+1})^{\mathbb{N}}$ such that $\theta_{n}(w)(\cdot )=w(\cdot+n)$ for $w \in \mathbb E ^{\mathbb N}$ or $(\mathbb{Z}^{d+1})^{\mathbb{N}}$. 

Given a subset $K$ of $\mathbb E$, we denote with $H_{K}$, $\widetilde{H}_{K}$ and $T_{K}$,  the entrance time, hitting time, and exit time from $K$, that is,
\begin{equation}\label{eq:hitting time}
H_{K}:=\inf\{n\geq 0:X_{n}\in K\},\ \widetilde{H}_{K}:=\inf\{n\geq 1:X_{n}\in K\},\ \text{and}\ T_{K}:=\inf\{n\geq 0:X_{n}\notin K\}.
\end{equation}
In the case of a singleton $K=\{x\}$, we simply write $H_{x}$, $\widetilde{H}_{x}$ and $T_{x}$. 

For two sets $K\subseteq U$ in $\mathbb E$, we then define the successive times of return to $K$ and departure from $U$ for a simple or biased random walk $(X_n)_{n\geq 0}$ as 
\begin{equation}\label{eq:def of return and departure}
\begin{split}
R_{1}^{K,U}=H_{K},\quad  &R_{k}^{K,U}=D_{k-1}^{K.U}+H_{K}\circ\theta_{D_{k-1}^{K,U}} \quad \mbox{ for }k\geq 2,\\ 
&D_{k}^{K,U}=R_{k}^{K,U}+T_{U}\circ\theta_{R_{k}^{K,U}} \quad \mbox{ for }k\geq 1.
\end{split}
\end{equation}
The number of excursions of the walk from $K$ to the complement of $U$ is defined as 
\begin{equation}\label{eq:def of N_K,U}
N^{K,U}:=\sup\{k\geq 1:D_{k}^{K,U} <\infty\}.
\end{equation} 
When considering simple or biased random walk on $\mathbb Z^{d+1}$, we keep the same notation as in \eqref{eq:hitting time}-\eqref{eq:def of N_K,U}.

We now discuss some potential theory with respect to the simple random walk on $\mathbb{Z}^{d+1}$ or $\mathbb{E}$. For each $x\in\mathbb{Z}^{d+1}$ and $\Delta\in [0,1]$, we denote by $P_{x}^{\Delta}$ and $E_{x}^{\Delta}$ the respective law and expectation of biased random walk on $\mathbb{Z}^{d+1}$ with upward drift $\Delta$ along the $(d+1)$-th direction. Moreover, when $\mu$ is a measure on $\mathbb{Z}^{d+1}$, we also write $P_{\mu}$ and $E_{\mu}$ for the measure $\Sigma_{x\in\mathbb{Z}}\mu(x)P_{x}$ (which is not necessarily a probability measure) and its corresponding expectation. When $\Delta=0$, that is, there is no drift, we simply write $P_{x},E_{x},P_{\mu},E_{\mu}$ for short. Note that the notation $P_{x}^{N,\alpha}$ refers to the random walk on the cylinder with upward drift $N^{-d\alpha}$, while $P_{x}^{\Delta}$ denotes the law of biased walk on $\mathbb Z^{d+1}$ with drift $\Delta$.  

Given $\varnothing \neq K\subset\subset\mathbb{Z}^{d+1}$ and $U\supseteq K$, the equilibrium measure and capacity of $K$ relative to $U$ are defined by 
\begin{equation}\label{eq:equilibrium measure}
e_{K,U}(x)=
\begin{cases}
P_{x}\left[\widetilde{H}_{K}>T_{U}\right],\quad &x\in K,\\
0, \quad &x\notin K,
\end{cases}
\quad\text{and}\quad \text{cap}_{U}(K)=\sum_{x\in K}e_{K,U}(x) (\leq |K|).
\end{equation}
The normalized equilibrium measure of a non-epmty $K$ with respect to $U$ is defined as 
\begin{equation}\label{eq:def of normalized equilibrium measure}
\overline{e}_{K,U}(x)=\frac{e_{K,U}(x)}{\text{cap}_{U}(K)},\quad x\in \mathbb{Z}^{d+1}.
\end{equation}
In addition, The Green's function of the walk on $\mathbb{Z}^{d+1}$ killed outside $U$ is defined as 
\begin{equation}\label{eq:def of green function}
g_{U}(x,x')=E_{x}\left[\sum_{n\geq 0}\mathbbm{1}\{X_{n}=x',n<T_{U}\}\right],\quad \text{for }x,x'\in\mathbb{Z}^{d+1}.
\end{equation}
When $U=\mathbb{Z}^{d+1}$, we drop $U$ from the notation in \eqref{eq:equilibrium measure}-\eqref{eq:def of green function}.

In the special case $K=[0,L)^{d+1}$ is a box with side-length $L$, it holds that (see \cite[(2.16)]{lawler2012intersections}) 
\begin{equation}\label{eq:equilibrium measure of a box}
\overline{e}_{K}(x) \geq \frac{c}{L^{d}},\quad\text{for any }x\in \partial^{\text{int}}K,   
\end{equation}
and that 
\begin{equation}\label{eq:capacity of a box}
cL^{d-1}\leq \text{cap}(K) \leq c'L^{d-1}.    
\end{equation}
Moreover, there exists a constant $c_{0}=c_{0}(d)$ such that (see \cite[Theorem 1.5.4]{lawler2012intersections})
\begin{equation}\label{eq:decay of green function}
g(x,y)\sim c_{0}|x-y|^{1-d},\quad \text{as }|x-y|\to\infty.
\end{equation}

Furthermore, one has a variational characterization of the capacity, 
which is convenient for deriving lower bounds on capacity:
\begin{equation}\label{eq:variational characterization}
\text{cap}(K)=\left(\inf_{\nu}E(\nu)\right)^{-1}, \mbox{ for } 
E(\nu):=\sum_{x,y}\nu(x)\nu(y)g(x,y),
\end{equation}
where $\nu$ ranges over all probability measure supported on $K$.

\subsection{Random interlacements}\label{subsec:interlacements}

We now recall some notation and results concerning random interlacements, and refer to \cite{DRS14} for a more detailed introduction. We denote by $W$ the space of doubly infinite, nearest-neighbor $\mathbb{Z}^{d+1}$-valued trajectories which tend to infinity at positive and negative infinite times.
We further denote by $W^{*}$ the space of equivalence classes of trajectories in $W$ modulo time-shift. That is, $W^*=W/\sim$, where for $w,w'\in W$, $w\sim w'$ means that $w(\cdot)=w'(\cdot+k)$ for some $k\in\mathbb{Z}$. The canonical projection from $W$ onto $W^{*}$ is denoted by $\pi^{*}$. We also write with  $\mathcal{W}$ its respective canonical $\sigma$-algebra and denote by $(X_n^{\pm})_{n\in \mathbb Z}$ the canonical coordinates.

We also consider $W_{+}$ the space of nearest-neighbor $\mathbb Z^{d+1}$-valued
trajectories defined for non-negative times and tending to infinity and let $\mathcal W^+$ stand for the canonical $\sigma$-algebra. 
For $K\subset\subset\mathbb{Z}^{d+1}$,
we denote by $W_{K}$ (resp.~$W_{K}^{*}$) the subset of $W$ (resp.~$W^{*}$) of trajectories modulo time-shift that intersect $K$. That is, 
\begin{equation}
    W_K := \{w \in W: \mbox{for some } n \in \mathbb Z, X_n^{\pm}(\omega)\in K\}\quad\text{and}\quad W_K^* := \pi^*(W_K).
\end{equation}

We then consider the space of point measures on the product space $W^* \times \mathbb R_+$:
\begin{equation}\label{eq:def of Omega}
\begin{split}
\Omega=\bigg\{&\omega=\sum_{i\geq 0}\delta_{(w_{i}^{*},u_{i})}:\text{with }(w_{i}^{*}, u_i) \in W^{*}\times \mathbb R_+\text{ for each }i\geq 0,\text{ and }\\
&\quad \omega(W_{K}^{*}\times\mathbb{R}^{+})<\infty, \text{ for any non-empty }K\subset\subset\mathbb{Z}^{d+1}\text{ and }u\geq 0\bigg\},\\
\end{split}
\end{equation}
where the space $\Omega$ is endowed with the canonical $\sigma$-algebra. The random interlacements can be constructed as a Poisson point process on the space $W^{*}\times\mathbb{R}^{+}$ supported on $\Omega$ with intensity measure $\nu(dw^{*})du$, where $du$ denotes the Lebesgue measure and $\nu$ is a certain translation-invariant $\sigma$-finite measure on $W^{*}$ (see \cite[Theorem 1.1]{Szn10}). 
We denote by $\PP$ the law of random interlacements.

Given $\omega\in\Omega$, $K\subset\subset\mathbb{Z}^{d+1}$ and $u\geq 0$, we define the point measure $\mu_{K,u}(\omega)$ on $W_{+}$ collecting the onward part of trajectories $w_{i}^{*}$ with label $u_{i}\leq u$ that enter $K$ in the cloud $\omega$, i.e.,
\begin{equation}
\mu_{K,u}(\omega)=\sum_{i\geq 0}\mathbbm{1}\{w_{i}^{*}\in W_{K}^{*}, u_{i}\leq u\}\delta_{(w^{*}_i)_{K,+}},\quad\text{if }\omega=\sum_{i\geq 0}\delta_{(w_{i}^{*},u_{i})}.
\end{equation}
Here, for a $w^{*}\in W_{K}^{*}$, the onward part is denoted by $(w^{*})_{K,+}$, which is the unique element of $W_{+}$ that follows $w^{*}$ step by step from the first time it enters $K$. The key property of these point measures is that for any $K\subset\subset\mathbb{Z}^{d+1}$ and $u\geq 0$,
under $\PP$, $\mu_{K,u}$ is a Poisson point process on $(W_{+}, \mathcal W^+)$ with intensity measure $u P_{e_{K}}$. Then, given $\omega \in \Omega$ and $u\geq 0$, the random interlacements and its vacant set at level $u$ are now defined as the random subsets of $\mathbb{Z}^{d+1}$:
\begin{equation}\label{eq:def of interlacement}
\I^{u}(\omega)=\bigcup_{u_{i}\leq u}\text{range}(w_{i}^{*}),\quad \text{and}\quad \V^{u}(w)=\mathbb{Z}^{d+1}\setminus \I^{u}(\omega), \quad \text{for }\omega = \sum_{i\geq 0}\delta_{(w_{i}^{*},u_{i})}.
\end{equation}

We now turn to some facts concerning the percolative properties of the vacant set $\V^{u}$ as parameter $u$ varies, and explain the critical parameters $\overline{u}$ and $u_{**}$ in detail. Here we refer to \cite{DCGR+23b} for the definition of these quantities. Given $u>v>0$ and $R>0$, we introduce two events, namely,
\begin{equation}\label{eq:def of existence}
\text{Exist}^{\mathcal V}(R,u):=\big\{\text{there exists a cluster in }\V^{u}\cap B(0,R)\text{ with diameter at least }\frac{R}{5}\big\},  
\end{equation}
\begin{equation}\label{eq:def of uniqueness}
\begin{split}
\text{Unique}^{\mathcal V}(R,u,v):=\big\{&\text{any two clusters in }\V^{u}\cap B(0,R)\text{ having diameter at least }\frac{R}{10}\text{ are }\\ &\text{connected to each other in }\V^{v}\cap B(0,2R)\big\}.
\end{split}
\end{equation}
Note that $\text{Exist}^{\mathcal V}(R,u)$ is monotone in $u$, and $\text{Unique}^{\mathcal V}(R,u,v)$ is monotone in $v$, but a priori we do not know whether $\text{Unique}^{\mathcal V}(R,u,v)$ is monotone in $u$. We say the vacant set of random interlacements \textit{strongly percolates} at levels $u,v$, if there exist constants $c=c(u,v,d)$ and $C=C(u,v,d)$ in $(0,\infty)$ such that for every $R\geq 1$, 
\begin{equation}\label{eq:def of strongly percolate}
\PP\left[\text{Exist}^{\mathcal V}(R,u)\right]\geq 1-Ce^{-R^{c}},\quad\text{and}\quad \PP\left[\text{Unique}^{\mathcal V}(R,u,v)\right]\geq 1-Ce^{-R^{c}}.
\end{equation}
We then define the critical value 
\begin{equation}\label{eq:def of ubar}
\begin{split}
\overline{u}=\sup\big\{s>0:&\text{ the vacant set of random interlacements strongly percolates }\\ &\text{ at levels }u,v\text{ for every }u>v\text{ in }(0,s)\big\},
\end{split}
\end{equation}
and refer to $(0,\overline{u})$ the strongly percolative regime of the vacant set of random interlacements. 

On the other side, we say the vacant set of random interlacements is \textit{strongly non-percolative} at level $u$, if there exist constants $c=c(u,d)$ and $C=C(u,d)$ in $(0,\infty)$ such that for every $R\geq 1$,
\begin{equation}\label{eq:def of strongly non-percolate}
\PP\left[0\overset{\V^{u}}{\longleftrightarrow}S(0,R)\right]\leq Ce^{-R^{c}}.    
\end{equation}
The strongly non-percolative property is monotone in $u$, and we can then define the critical value 
\begin{equation}\label{eq:def of u**}
u_{**}=\inf\big\{u>0: \mathcal V^u\text{ is strongly non-percolative at level }u\big\},    
\end{equation}
and also refer to $(u_{**}, \infty)$ the strongly non-percolative regime of the vacant set. We also remark that the definition of $u_{**}$ was later relaxed in \cite{PT15}, which shows that it is sufficient to require the infimum of annulus-crossing probability to fall below an explicit constant.

We end this subsection with the definition of the excursions in random interlacements, which will be a primary focus of study in the following sections. We consider a box $U$ in $\mathbb{Z}^{d+1}$ and a non-empty set $D\subseteq U$. By \eqref{eq:def of Omega}, we know
that for all $\omega\in \Omega$ and all $u \geq 0$, $\omega(W_{D}^{*}\times[0,u])<\infty$ and $\omega(W_{D}^{*}\times\mathbb{R}^{+})=\infty$ holds. Moreover, almost surely, the labels $u_{i}$ that appear in the point measure $\omega$ are all distinct, and each $w_{i}^{*}$
that belongs to $W_{D}^{*}$ only contains finitely many excursions from $D$ to $\partial U$, since the bilateral trajectory of an element of $W$ only spends
finite time in any finite subset of $\mathbb{Z}^{d+1}$. Thus, given $\omega=\sum_{i>0}\delta_{(w_{i}^{*},u_{i})}$ in $\Omega$, we can sort the infinite sequence of excursions
from $D$ to $\partial U$ by lexicographical order, i.e.~first by value of $u_{i}$ in increasing order, and then by the order of appearance inside a given trajectory $w_{i}^{*}\in W_{D}^{*}$. In this way we obtain a sequence of 
random variables on $\Omega$:
\begin{equation}\label{eq:def of interlacement excursions}
\begin{split}
\big(Z_{\ell}^{D,U}(\omega)\big)_{\ell\geq 1}=&\Big(w_{1}^{*}\big[R_1(w_{1}^{*}),D_1(w_{1}^{*})\big),\dots,\big.\\
&\quad \big.w_{1}^{*}\big[R_{N^{D,U}(w_{1}^{*})}(w_{1}^{*}),D_{N^{D,U}(w_{1}^{*})}(w_{1}^{*})\big), \, w_{2}^{*}\big[R_1(w_{2}^{*}),D_1(w_{2}^{*})\big),\dots\Big).
\end{split}
\end{equation} 
For every $u>0$, we denote by $N_{u}^{D,U}$ the total number of excursions from $D$ to $\partial U$, that is,
\begin{equation}\label{eq:number of excursions in Iu}
N_{u}^{D,U}=\sum_{i:u_i\leq u}N^{D,U}(w_{i}^{*}).
\end{equation}
We will also use $N_u(D)$ in place of $N_{u}^{D,U}$ for short if there is no ambiguity.

\subsection{Radon-Nikodym derivatives}\label{subsec:RN}
In these subsection we collect some properties on the Radon-Nikodym derivatives of biased random walks with respect to simple random walks that will be frequently used in this paper. For an excursion $e=(x_{0},x_{1},\cdots,x_{n})$ in $\mathbb{E}$ or $\mathbb{Z}^{d+1}$, we define its length and its height as
\begin{equation}\label{eq:def of length and height}
\ell(e)=n\quad\text{and}\quad h(e)=|x_{0,d+1}-x_{n,d+1}|,    
\end{equation}
where we recall that for a point $x\in\mathbb{Z}^{d+1}$, $x_{d+1}$ stands for $(d+1)$-th coordinate of $x$.
We use $\text{up}(e)$ and $\text{down}(e)$ to respectively denote the number of upward steps and downward steps in $e$, that is, 
\begin{equation}\label{eq:count up and down}
    \text{up}(e) = \sum_{i =0}^{n-1}\mathbbm 1\{x_{i,d+1} < x_{i+1,d+1}\}\quad\text{and}\quad \text{down}(e) = \sum_{i =0}^{n-1}\mathbbm 1\{x_{i,d+1} > x_{i+1,d+1}\}.
\end{equation}
Then we have the following relations regarding the number of upward and downward steps:
\begin{equation}\label{eq:relationships between up and down}
\text{up}(e)+\text{down}(e)\leq\ell(e)\quad\text{and}\quad\big|\text{up}(e)-\text{down}(e)\big|=h(e).
\end{equation}  
We also write
\begin{align}
p(e):=P_{x_{0}}^{N}[X_{[0,n]}=e],&\quad\text{and}\quad p^{\text{bias}}(e):=P_{x_{0}}^{N,\text{bias}=\Delta}[X_{[0,n]}=e],&&\quad\text{for }e\subseteq\mathbb{E}; \label{eq:p and pbias cylinder}\\
p(e):=P_{x_{0}}[X_{[0,n]}=e],&\quad\text{and}\quad p^{\text{bias}}(e):=P_{x_{0}}^{\Delta}[X_{[0,n]}=e],&&\quad\text{for }e\subseteq\mathbb{Z}^{d+1}.  \label{eq:p and pbias Zd}
\end{align}

By the standard Radon-Nykodym derivative, \eqref{eq:relationships between up and down} and the fact that $1-\Delta^{2}\leq 1$, the ratio between $p$ and $p^{\text{bias}}$ (no matter whether $e$ belongs to $\mathbb{E}$ or $\mathbb{Z}^{d+1}$) satisfies
\begin{equation}\label{eq:RN bias to unbias bound}
    (1-\Delta^{2})^{\frac{\ell(e)}{2}}\cdot \left(\frac{1-\Delta}{1+\Delta}\right)^{\frac{h(e)}{2}} \leq \frac{p^{\rm bias}(e)}{p(e)} \leq \left(\frac{1+\Delta}{1-\Delta}\right)^{\frac{h(e)}{2}}. 
\end{equation}

\subsection{Properties of random walk}\label{subsec:random walk preliminary}

In this subsection, we gather several important properties of the one-dimensional biased random walk, which are useful in \Cref{sec:one-dimensional walks,sec:upper bound,sec:denouement}. The notation used here will differ from other parts for clarity. For $\Delta\in [0,1)$ and $x\in\mathbb{Z}$, let $\mathsf P^{\Delta}_x$ be the law of a one-dimensional biased random walk, defined as the discrete-time Markov chain $(w_{n})_{n\geq 0}$ on $\mathbb Z$ with initial position $x$ and transition probability 
\begin{equation}\label{def:general biased 1drw}
    p(x_1, x_2) = \frac{1 + \Delta \cdot (x_2 - x_1)}{2}\mathbbm{1}_{|x_2-x_1|_{\infty}=1}.
\end{equation}
When $\Delta=0$, we may write $\mathsf{P}_{x}$ for short. For each trajectory $w = (w_0, w_1, \cdots)$ in $\mathbb Z$, let $L_k^z$ represent the local time at time $k$ and position $z\in \mathbb Z$, that is, 
\begin{equation}
    L_k^z = L_k^z(w) = \sum_{i = 0}^k\mathbbm{1}{\{w_i = z\}}.
\end{equation}
For every $z\in\mathbb{Z}$ and positive integers $k$ and $\ell$, following the strong Markov property on the first hitting time of $z$, the local time satisfies
\begin{equation}\label{eq:probability of multiple hits}
\mathsf{P}_{0}^{\Delta}\left[L_{k}^{z}\geq \ell\right]\leq \mathsf{P}_{0}^{\Delta}\left[L^{0}_{\infty}\geq \ell\right]=(1-\Delta)^{\ell-1}.  
\end{equation}
This distribution will be useful in estimates regarding $\underline{S}_{N}$ and $\overline{S}_{N}$, especially in the case $\alpha<1$ (see e.g.~the proof of \Cref{prop:asymptotics_alpha<1} and \Cref{sec:denouement}).
We further define for a positive constant $\ell$ and trajectory $w$ the first time when the local time at $z$ of the trajectory $w$ reaches $\ell$. 
\begin{equation}\label{def:1d record breaking time}
    S(\ell,z) = S(\ell, z, w) := \inf\{k\geq 0: L_k^z (w) \geq \ell\},
\end{equation}
and call $\inf_{z\in\mathbb Z}S(\ell, z)$ the one-dimensional {\it record-breaking time} with parameter $\ell$.

In the following context, we show that the record-breaking time of random walk converges weakly to its Brownian local time counterpart when $\Delta = N^{-1}$. This convergence is of great importance for concluding our final results (for case $\alpha=1$) once we have compared $T_N$ with $\overline S_N$ and $\underline S_N(z)$, since the latter can be seen as the record-breaking time of lazy one-dimensional random walk; we refer to \Cref{prop:asymptotics_alpha=1} and \Cref{sec:denouement} for details.

Recall that $\mathbb W, \mathbb E_{W}$ denote the law and the expectation of standard Wiener process $\{W_t\}_{t\geq 0}$. For every $\mu \in \mathbb R$, 
$L^{\mu}(v, t)$ and $\zeta^{\mu}(u)$ are the jointly continuous version of the local time of $W_t + \mu t$ and its record-breaking time respectively (see \eqref{eq:def of first achiving time}).

\begin{lemma}\label{lem:weak converge 1d rw}
    For every $s > 0$ and $0 <  \widetilde{u} < u$, the random times $S(u,z)$ and $\zeta^{1}(u)$ satisfy
    \begin{align}
        \lim_{N\rightarrow \infty}\mathsf P^{N^{-1}}_0\left[\inf_{z\in \mathbb Z}S(uN, z) > sN^2\right] &= \WW\left[\zeta^{1}(u) > s\right]; \quad\label{eq:weak convergence}\\
        \limsup_{L\rightarrow \infty}\limsup_{N\rightarrow \infty}\mathsf P_0^{N^{-1}}\left[\inf\limits_{z = \lfloor \ell N/L \rfloor, |\ell|\leq L^2} S(\widetilde u N, z)> sN^2 \right]&\leq \mathbb W[\zeta^1(u)> s].\label{eq:weak convergence truncated}
    \end{align}
\end{lemma}
    \begin{proof}[Proof of \Cref{lem:weak converge 1d rw}]
        Let $\theta = \theta_N $ satisfy $\tanh \theta_N =  N^{-1}$. By a standard calculation, we see that $\theta N$ tends $1$ as $N$ goes to infinity. 
        Let $U_n,{n\geq 0}$ be the simple random walk with law $\mathsf P_0$. Write $\mathsf{E}$ for the  expectation, we define $(V_n)_{n\geq 1}$ as a series of random variables such that for every measurable function $g:\mathbb R \rightarrow \mathbb R_+$, 
        \begin{equation}\label{eq:def of transformed iid}
            \mathsf E[g(V_n)] = \frac{\mathsf E[g(U_n)e^{\theta U_n}]}{\mathsf E[e^{\theta U_n}]}. 
        \end{equation}
        By induction and conditional expectation, under $\mathsf P_{0}$, $(V_n)_{n \geq 1}$ has the same law as the biased walk under $\mathsf P_0^{N^{-1}}$. 
        Then for all $s > 0$ and $u > 0$, \begin{equation}\label{eq:compare rw bm local time1}
            \mathsf P^{N^{-1}}_0\left[\sup_{z\in \mathbb Z}L_{sN^2}^z < uN\right] 
            \overset{\eqref{eq:def of transformed iid}}{=}
            \mathsf E\Big[e^{\theta U_{sN^2}}\mathbbm{1}_{\left\{\sup\limits_{z\in\mathbb Z} L^z_{sN^2}(U) < uN\right\}}\Big]\cdot\mathsf E\Big[e^{\theta U_{sN^2}}\Big]^{-1}.
        \end{equation}
        The denominator satisfies
        \begin{equation}\label{eq:compare rw bm local time2}
            \mathsf E[e^{\theta U_{sN^2}}] = (\cosh \theta)^{sN^2} = \left(1 + \frac{1}{N^2 - 1}\right)^{sN^2/2} \xrightarrow[]{N\rightarrow \infty}e^{s/2}.
        \end{equation}
        Then by a more general version of invariance principle (see \cite{reversz81}),
        \begin{equation}\label{eq:compare rw bm local time4}
            \left(\theta U_{sN^2},\, \frac{1}{N}\sup\limits_{z\in \mathbb Z}L^z_{sN^2}(U)\right) \text{ converges in distribution to }\left(W_s,\, \sup\limits_{v\in \mathbb R}L(v, s)\right).
        \end{equation}
        Note that the function $e^{\theta U_{sN^2}}\mathbbm{1}_{\left\{\sup_{z\in\mathbb Z} L^z_{sN^2}(U) < uN\right\}}$ is uniformly integrable since $\mathsf{E}\left[e^{2\theta U_{sN^{2}}}\right]$ converges to $e^{2s}$ as $N$ tends to infinity. 
        Combining \eqref{eq:compare rw bm local time2} and \eqref{eq:compare rw bm local time4}, we have 
        \begin{equation}\label{eq:limit result}
            \lim_{N\rightarrow \infty}
            \mathsf E\Big[e^{\theta U_{sN^2}}\mathbbm{1}_{\left\{\sup\limits_{z\in\mathbb Z} L^z_{sN^2}(U) < uN\right\}}\Big]\cdot {\mathsf E[e^{\theta U_{sN^2}}]}^{-1}
            {=} e^{-s/2}\cdot
            \mathbb E_{W}\Big[e^{W_s}\mathbbm{1}_{\left\{\sup\limits_{v\in\mathbb R}L(v, s) < u\right\}}\Big].
        \end{equation} 
        By \eqref{eq:compare rw bm local time1}, \eqref{eq:limit result} and Girsanov's theorem,
        \begin{equation}\label{eq:compare rw bm local time6}
            \lim_{N\to\infty}\mathsf P^{N^{-1}}_0\Big[\sup_{z\in \mathbb Z}L_{sN^2}^z < uN\Big] =
            \mathbb E_{W} \Big[\mathbbm{1}_{\left\{\sup\limits_{v\in \mathbb R}L(v,s) < u\right\}}e^{W_s - \frac{1}{2}s} \Big] = 
            \mathbb W\left[\sup_{v\in \mathbb R}L^1(v,s) < u\right].
        \end{equation}
    Recalling the definition \eqref{eq:def of first achiving time} of $\zeta^{\mu}(u)$, \eqref{eq:weak convergence} then follows.

    For the second claim \eqref{eq:weak convergence truncated}, combining \eqref{eq:def of transformed iid} and \eqref{eq:compare rw bm local time2} yields
    \begin{equation}\label{eq:tranfer to non bias}
    \mathsf P_0^{N^{-1}}\left[\inf\limits_{z = \lfloor \ell N/L \rfloor, |\ell|\leq L^2} S(\widetilde uN, z)> sN^2 \right]=e^{-s/2}\mathsf E\Big[e^{\theta U_{sN^2}}\mathbbm{1}_{\big\{ \inf\limits_{z = \lfloor \ell N/L \rfloor, |\ell|\leq L^2}S( \widetilde u N, z)>sN^2\big\}}\Big].    
    \end{equation}
    The right hand side of the above equation can be bounded by \begin{equation}\label{eq:comparison between truncated and normal}
    \begin{split}
    e^{-s/2}\mathsf E\Big[e^{\theta U_{sN^2}}\mathbbm{1}_{\big\{\inf\limits_{\footnotesize\substack{z = \lfloor \ell N/L \rfloor,\\ |\ell|\leq L^2}}S( \widetilde u N, z)>\inf\limits_{z\in \mathbb Z}S(u,z)\big\}}\Big]+e^{-s/2}\mathsf E\Big[e^{\theta U_{sN^2}}\mathbbm{1}_{\big\{\inf\limits_{z\in \mathbb Z}S(uN,z)> sN^2\big\}}\Big]:=\text{\rom{1} + \rom{2}}.
    \end{split}   
    \end{equation}
For the first term, by (a slightly modified version of) (4.31) in \cite{Szn09a} (see also \Cref{lem:overline S confine region small bias}), we have
\begin{equation}\label{eq:1 without bias vanish}
\limsup_{L\to\infty}\limsup_{N\to\infty}\mathsf{P}\left[\inf\limits_{z = \lfloor \ell N/L \rfloor, |\ell|\leq L^2}S( \widetilde uN, z)>\inf\limits_{z\in \mathbb Z}S(uN,z)\right]=0;  \end{equation}
Then thanks to Cauchy-Schwartz inequality and the fact that $\mathsf{E}\left[e^{2\theta U_{sN^{2}}}\right]$ is uniformly bounded, 
\begin{equation}\label{eq:1 vanish}
\limsup_{L\to\infty}\limsup_{N\to\infty}\text{\rom{1}}=0.    
\end{equation}
As for the second term, it follows from \eqref{eq:weak convergence} and \eqref{eq:def of transformed iid}-\eqref{eq:compare rw bm local time2} that
\begin{equation}\label{eq:compare rw bm local time7}
    \lim_{N\rightarrow \infty} \text{\rom{2}}
    =e^{-s/2}\cdot\mathsf E\Big[e^{\theta U_{sN^2}}\Big]\cdot\mathsf P^{N^{-1}}_0\left[\sup_{z\in \mathbb Z}L_{sN^2}^z < uN\right]=\mathbb W[\zeta^1(u)>s]. 
\end{equation}
The second equation \eqref{eq:weak convergence truncated} then follows from plugging \eqref{eq:comparison between truncated and normal},\eqref{eq:1 vanish} and \eqref{eq:compare rw bm local time7} into \eqref{eq:tranfer to non bias}.
\end{proof}

At the end of this section, let us also mention the distribution for $\zeta^{\mu}(u)$ when $\mu = 0$. 
By \cite{B84} and \cite[Proposition~5]{E90}, the Laplace transform of $\zeta(u)$ defined in \eqref{eq:def of zeta} can be written as 
\begin{equation}\label{eq:laplace of zeta}
    \mathbb E_W\left[\exp\left({-\frac{\theta^2 }{2}\zeta(u)}\right)\right] = \frac{\theta u}{[\sinh (\theta u / 2)]^2}\cdot\frac{I_1(\theta u / 2)}{I_0(\theta u / 2)},\quad\mbox{ for }\theta, u > 0,
\end{equation}
where the function $I_{\nu}$ is the modified Bessel function of index $\nu$; cf.~\cite[page 60]{Olver74}.

\subsection{Hitting distribution estimates}\label{subsec:hitting distribution}
In this subsection, we provide some estimates on hitting probabilities of simple and biased random walks. These estimates will play important roles in \Cref{sec:bad(beta-gamma),sec:lower bound in biased case,sec:couplings}. 
\begin{lemma}\label{prop:hitting distribution two boxes simple}
For any $\eta\in (0,1)$, if $L\geq 1$ and $K\geq c_{1}(\eta)\geq 2$, then for any non-empty $A\subseteq B(0,L)$, finite $U\supseteq B(0,KL)$, $x\in\partial U$ and $y\in \partial^{\text{int}} A$, if 
\begin{equation}\label{eq:positive hitting probability simple}
P_{x}\left[H_{A}<\widetilde{H}_{\partial U}\right]>0,
\end{equation}
then we have
\begin{equation}\label{eq:hitting distribution two boxes simple}
\left(1-\frac{\eta}{10}\right)\overline{e}_{A}(y)\leq P_{x}\left[X_{H_{A}}=y\mid H_{A}<\widetilde{H}_{\partial U}\right]\leq \left(1+\frac{\eta}{10}\right)\overline{e}_{A}(y). \end{equation}
\end{lemma}
\begin{proof}
This is a direct corollary of \cite[Lemma 2.3]{Szn17}.
\end{proof}

By considering the last visit point (to $D$), we obtain the following corollary.
\begin{corollary}\label{lem:hitting distribution three boxes simple}
For any $\eta\in(0,1)$, if $L\geq 1$ and $K\geq c_{1}(\eta)\geq 2$, then for any non-empty $A\subseteq B(0,L)$, $B(0,KL)\subseteq D\subseteq B(0,2KL)$ and finite $U\supseteq B(0,10KL)$, $x\in \partial D$ and $y\in \partial^{\text{int}} A$, we have
\begin{equation}\label{eq:hitting distribution three boxes simple}
\left(1-\frac{\eta}{10}\right)\overline{e}_{A}(y)\leq P_{x}\left[X_{H_{A}}=y\mid H_{A}<H_{\partial U}\right]\leq \left(1+\frac{\eta}{10}\right)\overline{e}_{A}(y). \end{equation}
\end{corollary}

We then deduce the biased version of \Cref{lem:hitting distribution three boxes simple} using estimates of Radon-Nikodym derivatives.

\begin{lemma}\label{lem:hitting distribution three boxes bias}
Consider $A=B(0,L)$, $B(0,KL)\subseteq D\subseteq B(0,2KL)$ and $B(0,10KL)\subseteq U\subseteq B(0,20KL)$. For any $\eta\in (0,1)$, if $L\land K\geq c_{2}(\eta)>100$ and $\Delta^{-1}\geq KL(K+L)$, then for every $x\in\partial D$ and $y\in\partial^{\text{int}} A$, we have 
\begin{equation}\label{eq:hitting distribution three boxes bias}
\left(1-\frac{\eta}{10}\right)\overline{e}_{A}(y)\leq P_{x}^{\Delta}\left[X_{H_{A}}=y\mid H_{A}<H_{\partial U}\right]\leq \left(1+\frac{\eta}{10}\right)\overline{e}_{A}(y).    
\end{equation}
\end{lemma}
\begin{proof}
Recall \Cref{subsec:RN} for the notation $\ell(e),h(e),\text{up}(e),\text{down}(e),p(e), p^{\text{bias}}(e)$. 
Fix $x\in \partial D$. We define the set of excursions from $x$ to $\partial^{\text{int}} A$ that does not touch $\partial U$ as 
\begin{equation}\label{eq:def of Se three boxes}
\begin{split}
\Sigma_{\text{excur}}:&=\{e= (x_{0},x_{1},\cdots,x_{n}):\mbox{ for each }0\leq i\leq n-1,|x_{i}-x_{i+1}|_{\infty}=1, \\
&\quad \quad x_{0}=x,x_{n}\in\partial^{\text{int}} A, \mbox{ and for }1\leq i\leq n-1,  x_{i}\in U\setminus A\}. \end{split}
\end{equation}
For every $y\in \partial^{\text{int}} A$, we further define $\Sigma_{\text{excur}}(y)$ as the set of excursions in $\Sigma_{\text{excur}}$ that ends at $y$. Since the excursions in $\Sigma_{\text{excur}}$ are fully contained in $U \subset B(0, 20KL)$,
\begin{equation}\label{eq:bound on h and ell three boxes}
\quad h(e)\leq CKL,\quad\text{ for }e\in\Sigma_{\text{excur}}.
\end{equation}  

According to the length of excursions, we further divide $\Sigma_{\text{excur}}$ into
\begin{equation}\label{eq:def of Sshort and Slong three boxes}
\Sigma_{\rm short}:=\left\{e\in \Sigma_{\text{excur}}:\ell(e)\leq \frac{KL}{\Delta}\right\},\quad\text{and}\quad
\Sigma_{\rm long}:=\left\{e\in \Sigma_{\text{excur}}:\ell(e)>\frac{KL}{\Delta}\right\},  
\end{equation}
and also define $\Sigma_{\text{short}}(y)$ and $\Sigma_{\text{long}}(y)$ as the intersection of $\Sigma_{\text{excur}}(y)$ with $\Sigma_{\text{short}}$ and $\Sigma_{\text{long}}$ respectively.

We now control the Radon-Nikodym derivative of some excursion under $P_{x}^{\Delta}$ with respect to $P_{x}$ using properties stated in \Cref{subsec:RN}. Note that since $L\land K\geq c_{2}(\eta)$ and $\Delta^{-1}\geq KL(K+L)$, we have $(\Delta KL)^{-1}\geq K\geq c_{2}(\eta)$. 
Combining \eqref{eq:RN bias to unbias bound} and \eqref{eq:bound on h and ell three boxes} then yields
\begin{equation}\label{eq:RN bias to unbias upper three boxes}
\frac{p^{\rm bias}(e)}{p(e)}\leq \left(\frac{1+\Delta}{1-\Delta}\right)^{CKL} 
\leq 1+CKL\Delta\leq 1+\frac{C}{c_{2}(\eta)},\quad \mbox{ for all }e\in \Sigma_{\text{excur}}.
\end{equation}
Similarly, combining \eqref{eq:RN bias to unbias bound}, \eqref{eq:bound on h and ell three boxes} and \eqref{eq:def of Sshort and Slong three boxes} yields that for all $e\in\Sigma_{\text{short}}$,
\begin{equation}\label{eq:RN bias to unbias lower Sshort three boxes}
\frac{p(e)}{p^{\rm bias}(e)}\leq \left(1-\Delta^{2}\right)^{-\frac{CKL}{\Delta}}\left(1+\frac{C}{c_{2}(\eta)}\right)\leq 1+CKL\Delta+\frac{C}{c_{2}(\eta)}\leq 1+\frac{C}{c_{2}(\eta)}. 
\end{equation}

The next ingredient is to give an upper bound of $P_{x}[\Sigma_{\rm long}]$. Note that under $P_{x}$, in each direction, the simple random walk makes a  $+1,0,-1$ move with probability $\frac{1}{2d+2},\frac{d}{d+1}.\frac{1}{2d+2}$ respectively. It then follows from \cite[Lemma 1.1]{Szn99} that
\begin{equation}\label{eq:uniform second moment three boxes}
\sup_{x\in \partial D}E_{x}[T_{U}]\leq C(KL)^{2}.  
\end{equation}
Consequently, by Kha\'{s}minskii's lemma (see \cite{Khas59}), we have
\begin{equation}\label{eq:khasminskii1}
\sup_{x\in\partial D}E_{x}\left[\exp\left(\frac{cT_{U}}{(KL)^{2}}\right)\right]\leq C.   
\end{equation}
It then follows from exponential Chebyshev's inequality  and the fact that $\Delta^{-1}\geq KL(K+L)$ that
\begin{equation}\label{eq:Slong measure upper bound three boxes}
\sum_{e\in \Sigma_{\text{long}}}p(e)=P_{x}\left[\frac{KL}{\Delta}<H_{A}<T_{U}\right]\leq P_{x}\left[\frac{KL}{\Delta}<T_{U}\right]\leq\frac{E_{x}\left[\exp\left(\frac{cT_{U}}{(KL)^{2}}\right)\right]}{\exp\left(c(\Delta KL)^{-1}\right)}\leq Ce^{-c(K+L)}.  
\end{equation}
For the biased case, combining also the estimate of Radon-Nikodym derivative in \eqref{eq:RN bias to unbias upper three boxes}, 
\begin{equation}\label{eq:Slong measure upper bound bias three boxes}
\sum_{e\in \Sigma_{\text{long}}}p^{\text{bias}}(e)\leq \left(1+\frac{C}{c_{2}(\eta)}\right)\cdot e^{-c(K+L)}\leq Ce^{-c(K+L)}.  
\end{equation}

We then move on towards bounding $P_{x}[X_{H_{A}}=y,H_{A}<H_{\partial U}]$ from below. Recall that $A=B(0,L)$, $B(0,KL)\subseteq D\subseteq B(0,2KL)$ and $B(0,10KL)\subseteq U\subseteq B(0,20KL)$. Therefore, when $K\geq c_{2}(\eta)$ is large, we have
\begin{equation}\label{eq:hitting probability three boxes}
P_{x}[H_{A}<H_{\partial U}]\geq \frac{c}{K^{d+1-2}}=\frac{c}{K^{d-1}}.   
\end{equation}
Then by \Cref{lem:hitting distribution three boxes simple} and \eqref{eq:equilibrium measure of a box}, when $K\geq c_{2}(\eta)\geq c_{1}(1/2)$, for any $y\in\partial^{\text{int}} A$, we have
\begin{equation}\label{eq:hitting probability lower bound three boxes}
\sum_{e\in\Sigma_{\text{excur}}(y)}p(e)=P_{x}[X_{H_{A}}=y,H_{A}<H_{\partial U}]\geq \frac{c\overline{e}_{A}(y)}{K^{d-1}}\geq \frac{c}{(KL)^{d}}.   
\end{equation}

We finally prove \eqref{eq:hitting distribution three boxes bias} by comparing it with \eqref{eq:hitting distribution three boxes simple}. Indeed, when $K\geq c_{2}(\eta)\geq c_{1}(\eta/2)$, 
\begin{equation}\label{eq:hitting distribution three boxes simple stronger}
\left(1-\frac{\eta}{20}\right)\overline{e}_{A}(y)\leq P_{x}\left[X_{H_{A}}=y\mid H_{A}<H_{\partial U}\right]\leq \left(1+\frac{\eta}{20}\right)\overline{e}_{A}(y).    
\end{equation} 
\
It then follows that
\begin{equation}\label{eq:hitting distribution step1}
\begin{split}
&\frac{P_{x}^{\Delta}[X_{H_{A}}=y\mid H_{A}<H_{\partial U}]}{P_{x}[X_{H_{A}}=y\mid H_{A}<H_{\partial U}]}
\geq \frac{\sum_{e\in \Sigma_{\text{short}}(y)}p^{\text{bias}}(e)}{\sum_{e\in \Sigma_{\text{excur}}(y)}p(e)}\cdot\frac{\sum_{e\in \Sigma_{\text{excur}}}p(e)}{\sum_{e\in \Sigma_{\text{excur}}}p^{\text{bias}}(e)}\\
=&\frac{\sum_{e\in \Sigma_{\text{short}}(y)}p^{\text{bias}}(e)}{\sum_{e\in \Sigma_{\text{short}}(y)}p(e)}\cdot\frac{\sum_{e\in \Sigma_{\text{short}}(y)}p(e)}{\sum_{e\in \Sigma_{\text{excur}}(y)}p(e)}\cdot\frac{\sum_{e\in \Sigma_{\text{excur}}}p(e)}{\sum_{e\in \Sigma_{\text{excur}}}p^{\text{bias}}(e)} = {\rm \rom{1}} \cdot {\rm \rom{2} }\cdot {\rm \rom{3}}.
\end{split}
\end{equation}
By \eqref{eq:RN bias to unbias lower Sshort three boxes} and \eqref{eq:RN bias to unbias upper three boxes} respectively, \rom{1} $\geq 1-C/c_{2}(\eta)$ and \rom{3} $\geq 1-C/c_{2}(\eta)$. Using \eqref{eq:Slong measure upper bound three boxes} and \eqref{eq:hitting probability lower bound three boxes}, with $L\land K\geq c_{2}(\eta)$,
\begin{equation}\label{eq:proportion of Sshort}
{\rm \rom{2}} \geq1-\frac{\sum_{e\in\Sigma_{\text{long}}(y)}p(e)}{\sum_{e\in\Sigma_{\text{excur}}(y)}p(e)}\geq 1-\frac{c(KL)^{d}}{\text{exp}(c(K+L))}\geq 1-\frac{C}{c_{2}(\eta)}. 
\end{equation}
The lower bound in \eqref{eq:hitting distribution three boxes bias} follows from combining \eqref{eq:hitting distribution three boxes simple stronger}, \eqref{eq:hitting distribution step1} and lower bounds for \rom{1}, \rom{2} and \rom{3} above. 
Similarly, by using \eqref{eq:RN bias to unbias upper three boxes}, \eqref{eq:RN bias to unbias lower Sshort three boxes} and \eqref{eq:proportion of Sshort}, the upper bound in \eqref{eq:hitting distribution three boxes bias} follows from the upper bound in \eqref{eq:hitting distribution three boxes simple stronger}.
\end{proof}

We finally come to the biased version of \Cref{prop:hitting distribution two boxes simple}. 
By considering the first visit point (to a medium-size box $D$), it is essentially the corollary of \Cref{lem:hitting distribution three boxes bias}. We omit the proof. 
\begin{proposition}\label{prop:hitting distribution two boxes bias}
Suppose the drift of biased random walk on the $(d+1)$-th direction is $\Delta\in[0,1]$. Consider $A=B(0,L)$ and $B(0,KL)\subseteq U\subseteq B(0,2KL)$. For any $\eta\in (0,1)$, if $L\land K\geq c_{2}(\eta)\geq 100$ and $\Delta^{-1}\geq KL(K+L)$, then for every $x\in\partial U$ such that 
\begin{equation}\label{eq:positive hitting probability bias}
P_{x}^{\Delta}\left[H_{A}<\widetilde{H}_{\partial U}\right]>0,
\end{equation}
we have
\begin{equation}\label{eq:hitting distribution two boxes bias}
\left(1-\frac{\eta}{10}\right)\overline{e}_{A}(y)\leq P_{x}^{\Delta}\left[X_{H_{A}}=y\mid H_{A}<\widetilde{H}_{\partial U}\right]\leq \left(1+\frac{\eta}{10}\right)\overline{e}_{A}(y). \end{equation}    
\end{proposition}

\section{Analysis of record-breaking times of biased random walk}\label{sec:one-dimensional walks}

In this section, we prove the asymptotics of the local time profile of one-dimensional random walk which is of great importance to the analysis of distribution of disconnection time $T_N$. 
Precisely, we are going to provide the asymptotics for $\underline{S}_{N}$ and $\overline{S}_{N}$ already mentioned in the sketch of proof (also see formal definition \eqref{eq:def of underlineSz}), for the case when bias $N^{-d\alpha}$ satisfies $\alpha>1$ (weak bias case),  $\alpha=1$, and $1/d<\alpha<1$ (strong bias case), which corresponds to \Cref{prop:asymptotics_alpha>1,prop:asymptotics_alpha=1,prop:asymptotics_alpha<1} respectively.

Let us first clarify the two random times $\overline S_N$ and $\underline S_N$ appearing in \Cref{subsec:sketch}. 
For the vertical component $(Z_n)_{n\geq 0}$ of random walk $(X_n)_{n\geq 0}$, we write $\rho_{k}$, $k\geq 0$ for the times of successive shifts of $(Z_n)_{n\geq 0}$, and use $(\widehat{Z}_n)_{n\geq 0}$ for its time-changed version, that is, 
\begin{equation}\label{eq:def of Zhat}
\begin{split}
\rho_{0}=0,\quad &\rho_{k}=\inf\left\{n>\rho_{k-1}:Z_{n}\neq Z_{\rho_{k-1}}\right\},\,k\geq 1,\\
&\widehat{Z}_{k}:=Z_{\rho_{k}},\,k\geq 0.
\end{split}
\end{equation}
It is not difficult to observe that, for any $\alpha > 0$, under law $P_{0}^{N,\alpha}$, the time-changed process $(\widehat{Z}_n)_{n\geq 0}$ has the same distribution as the canonical process under $\mathsf P_0^{N^{-d\alpha}}$ defined in \Cref{subsec:random walk preliminary}, that is, a biased non-lazy one-dimensional random walk with drift $N^{-d\alpha}$.

For large $N$, the sequence $\rho_k$ defined in \eqref{eq:def of Zhat} has the same distribution as the partial sums of i.i.d.~geometric variables on $\{1,2,\cdots\}$ with success probability $1/(d+1)$. The strong law of large numbers then gives that 
\begin{equation}\label{eq:limit of rho_k}
    \lim\limits_{k\rightarrow \infty}\frac{\rho_k}{k} = d +1,\quad \text{a.s.}
\end{equation}
The local time of $\widehat{Z}$ is defined as 
\begin{equation}\label{eq:def of local time of Zhat}
\widehat{L}_{k}^{z}:=\sum_{0\leq m\leq k}\mathbbm{1}\left\{\widehat{Z}_{m}=z\right\},\quad \text{for every }k\geq 0\text{ and }z\in\mathbb{Z}.
\end{equation}
With this, we then consider a specific record-breaking time, that is, the first time when the number of distinct visits of the walk 
$X$ to height $z$ in the cylinder, i.e., $\mathbb T \times \{z\}$, reaches an amount $uN^d/(d+1)$ for $u>0$:
\begin{equation}\label{eq:def of Sz}
{S}_{N}(\omega,u,z):=\inf\Big\{\rho_k\geq 0:\widehat{L}_{k}^{z}\geq\frac{u}{d+1}N^{d}\Big\}.
\end{equation}
Recall the critical value $\overline{u}$ in \eqref{eq:def of ubar}. For $\delta\in(0,\overline{u})$, we introduce 
\begin{align}
&\underline{S}_{N}(z)=\underline{S}_{N}(\omega, \delta,z):=S_{N}(\omega,\overline{u}-\delta,z);  \label{eq:def of underlineSz}\\
&\overline{S}_{N}(z)=\overline{S}_{N}(\omega, \delta,z):=S_{N}(\omega,u_{**}+\delta,z). \label{eq:def of overlineSz}
\end{align}
We further introduce the respective infimum of $\underline{S}_{N}(z)$ and $\overline S_N(z)$ over all $z \in \mathbb Z$:
\begin{align}\label{eq:def of underlineS and overlineS}
\underline{S}_{N}=\underline{S}_{N}(\omega, \delta):=\inf_{z\in\mathbb{Z}}\underline{S}_{N}(\omega,\delta,z),\quad  \overline{S}_{N}=\overline{S}_{N}(\omega, \delta):=\inf_{z\in\mathbb{Z}}\overline{S}_{N}(\omega,\delta,z).
\end{align}
Note that these variables all depend on a priori fixed parameter $\delta$. However, since we consistently work with a fixed value of $\delta$ throughout this paper except in \Cref{sec:denouement}, where we will compare $\overline S_N$ with a truncated version with parameter $\delta$ slightly changed and will send $\delta$ to zero, this dependency is not explicitly expressed in the notation.

\subsection{The weak bias case}\label{subsec:large alpha}

We first deal with the weak bias case, i.e., $\alpha > 1$.
\begin{proposition}\label{prop:asymptotics_alpha>1}
    For every fixed $\delta > 0$ and $1 < \alpha \leq \infty$,
    \begin{align}
    \limsup_{N\rightarrow \infty}P_{0}^{N,\alpha}\left[\frac{\overline{S}_{N}}{N^{2d}} \geq s \right]&\leq \WW\left[\zeta\left(\frac{u_{**}+\delta}{\sqrt{d+1}}\right)\geq s\right];\label{eq:asymp_for_overlineS_alpha>1}\\
    \liminf_{N\rightarrow \infty}P_{0}^{N,\alpha}\left[\frac{{\underline S}_{N}}{N^{2d}} \geq s \right]&\geq \WW\left[\zeta\left(\frac{\overline{u}-\delta}{\sqrt{d+1}}\right)\geq s\right]. \label{eq:asymp_for_underlineS_alpha>1}
    \end{align}
\end{proposition}

We remark that based on the strong invariance of simple walk local time in \cite{CR83}, \cite[Corollary 4.6]{Szn09a} offered the proof for the first inequality \eqref{eq:asymp_for_overlineS_alpha>1} for simple random walk, which corresponds to the $\alpha = \infty$ case here. A similar procedure (or \cite[Proposition~7.1]{Szn09b} with minor adaptation) yields \eqref{eq:asymp_for_underlineS_alpha>1} for simple random walk. Therefore, the conclusion follows from the next lemma.

\begin{lemma}\label{lem:asympt_srw_to_biasedwalk} 
For every $s >0$, $\delta > 0$ and $1 < \alpha \leq \infty$, 
\begin{align}
    \limsup_{N\rightarrow \infty} P_{0}^{N,\alpha}\left[\frac{\overline{S}_{N}}{N^{2d}} \geq s \right]&\leq \limsup_{N\rightarrow \infty}P_{0}^{N}\left[\frac{\overline{S}_{N}}{N^{2d}} \geq s \right],\label{eq:RN_derivative_overlineS}\\
    \liminf_{N\rightarrow \infty} P_{0}^{N,\alpha}\left[\frac{\underline{S}_{N}}{N^{2d}} \geq s \right]&\geq \liminf_{N\rightarrow \infty}P_{0}^{N}\left[\frac{\underline{S}_{N}}{N^{2d}} \geq s  \right].\label{eq:RN_derivative_underlineS}
\end{align}
\end{lemma}

\begin{proof}
    We are going to estimate the Radon-Nykodim derivative for typical trajectories between measures $P_{0}^{N,\alpha}$ and $P_{0}^{N}$, using the estimate in \Cref{subsec:RN}. We only provide the proof of \eqref{eq:RN_derivative_overlineS} here, and \eqref{eq:RN_derivative_underlineS} follows similarly. 
    
    For every  fixed $s > 0$ and $\delta >0$ appeared in the definition of $\overline S_N$, and every $0<\varepsilon <d(\alpha - 1)$, 
    \begin{equation}\label{eq:separate typical trajectories}
        \begin{split}
        P_{0}^{N}\left[\overline{S}_{N} < sN^{2d}\right] &\leq P_{0}^{N}\left[\overline{S}_{N} < s N^{2d}, |Z_{s N^{2d}}| < N^{d + \varepsilon}\right]+P_{0}^{N}\left[|Z_{s N^{2d}}| \geq N^{d + \varepsilon}\right]=\text{\rom{1}}+\text{\rom{2}}.
        \end{split}
    \end{equation}
    By classical large deviation bounds, \rom{2} converges to 0 as $N$ approaches infinity. 
    Note that the event in \rom{1} only depends on the trajectory of $(X_n)_{n\geq 0}$ before time $sN^{2d}$. Recall \Cref{subsec:RN} for the notation $\ell(e),h(e),\text{up}(e),\text{down}(e),p(e),p^{\text{bias}}(e)$. We slightly abuse the notation and define
    \begin{equation}\label{eq:def of Se big alpha}
    \Sigma_{\text{excur}}:=\left\{e=(0,e_{1},e_{2},\dots,e_{sN^{2d}})\in\mathbb{E}^{sN^{2d}}:\overline{S}_{N}<sN^{2d},\left|\widehat{Z}_{sN^{2d}}\right|<N^{d+\varepsilon}\right\}.    
    \end{equation}
    Then for any $e\in\Sigma_{\text{excur}}$, $\ell(e)=sN^{2d}$ and $h(e)=|\widehat{Z}_{sN^{2d}}|<N^{d+\varepsilon}$. 
    Using \eqref{eq:RN bias to unbias bound}, we have \begin{equation}\label{eq:radon nykodim upper bound}
        \frac{p(e)}{p^{\rm bias}(e)}
        \leq \left(1 - \frac{1}{N^{2d\alpha}}\right)^{-\frac{1}{2} s N^{2d}}\left(1 + \frac{1}{N^{d\alpha}}\right)^{\frac{1}{2}N^{d + \varepsilon}}\left(1  - \frac{1}{N^{d\alpha}}\right)^{-\frac{1}{2}N^{d + \varepsilon}}.
    \end{equation}
    When $\alpha > 1$, by choosing $\varepsilon < d(\alpha - 1)$, the upper bound converges to 1 uniformly in $e$ as $N$ tends to infinity. 
    Then \eqref{eq:radon nykodim upper bound} yields 
    \begin{equation}\label{eq:get result asymp S_N}
        \begin{split}
        \liminf_{N\to\infty}\text{\rom{1}}= \liminf_{N\rightarrow \infty} \sum_{e\in \Sigma_{\text{excur}}}p(e)\leq \liminf_{N\rightarrow \infty} \sum_{e\in \Sigma_{\text{excur}}}p^{\text{bias}}(e)\leq \liminf_{N\rightarrow \infty} P_{0}^{N,\alpha}\left[\frac{\overline{S}_{N}}{N^{2d}} < s \right],
        \end{split}
    \end{equation}
    which, combined with \eqref{eq:separate typical trajectories}, immediately implies \eqref{eq:RN_derivative_overlineS}.
\end{proof}

\subsection{The \texorpdfstring{$\alpha=1$}{alpha=1} case}\label{subsec:alpha=1}

For the case $\alpha = 1$, the asymptotics for $\overline S_N$ and $\underline S_N$ is the same as that of \Cref{prop:asymptotics_alpha>1}, except that
the $\zeta$ in the right-hand side is replaced by the drifted version $\zeta^{1/\sqrt{d+1}}$ defined in \eqref{eq:def of first achiving time}. We recall the results in \Cref{lem:weak converge 1d rw} about the biased random walk in \Cref{subsec:random walk preliminary}.

\begin{proposition}\label{prop:asymptotics_alpha=1}
    For every $s >0$, $\delta > 0$,
    \begin{align}
        \limsup_{N\rightarrow \infty}P_{0}^{N,1}\left[\frac{\overline{S}_{N}}{N^{2d}} \geq s \right]&\leq \WW\left[\zeta^{1/\sqrt{d+1}}\left(\frac{u_{**} + \delta}{\sqrt{d+1}}\right)\geq s\right],\label{eq:asymp_for_overlineS_alpha=1}\\
        \liminf_{N\rightarrow \infty}P_{0}^{N,1}\left[\frac{\underline{S}_{N}}{N^{2d}} \geq s \right]&\geq
        \WW\left[\zeta^{1/\sqrt{d+1}}\left(\frac{\overline{u}-\delta}{\sqrt{d+1}}\right)\geq s\right]\label{eq:asymp_for_underlineS_alpha=1}.
    \end{align}
\end{proposition}

\begin{proof}
    We only prove \eqref{eq:asymp_for_overlineS_alpha=1} and \eqref{eq:asymp_for_underlineS_alpha=1} follows from similar arguments. 
    For every $0 < \widetilde{s} < s$ and $\delta >0$, we have
    \begin{equation}\label{eq:weak conv proof1}
    P_{0}^{N,1}\left[\overline{S}_{N} \geq s N^{2d}\right]
        \leq  P_{0}^{N,1}\left[\rho_{\widetilde sN^{2d}/(d+1)} > sN^{2d}\right] + P_{0}^{N,1}\left[\overline{S}_{N} \geq \rho_{\widetilde sN^{2d}/(d+1)}\right],
    \end{equation}
    By \eqref{eq:limit of rho_k}, the first term on the right side tends to zero as $N$ approaches infinity. In addition, following the same notation as in \Cref{lem:weak converge 1d rw} with $N$ replaced by $N^d$ as well as the scaling property of drifted Brownian motion, we have 
    \begin{equation}\label{eq:weak conv proof2}
    \begin{split}
        &\limsup_{N\to \infty}P_{0}^{N,1}\left[\overline{S}_{N} \geq \rho_{\widetilde sN^{2d}/(d+1)}\right]= \limsup_{N\rightarrow \infty} \mathsf P^{N^{-d}}\left[\inf_{z\in \mathbb Z}S\Big(\frac{u_{**} + \delta}{d + 1}N^{d}, z\Big)\geq \frac{\widetilde s}{d + 1}N^{2d}\right] \\
        {=} &\WW\left[\zeta^{1}\left(\frac{u_{**} + \delta}{d + 1}\right)\geq \frac{\widetilde s}{d + 1}\right]=\WW\left[\zeta^{\frac{1}{\sqrt{d+1}}}\left(\frac{u_{**}+\delta}{\sqrt{d+1}}\right)\geq \widetilde{s}\right].
        \end{split}
    \end{equation}
    The last term in \eqref{eq:weak conv proof2} is continuous in $\widetilde{s}$ thanks to the continuity of local time of drifted Brownian motion, and therefore taking $\widetilde s \rightarrow s$ then concludes the proof.
    
\end{proof}

\subsection{The strong bias case}\label{subsec:small alpha}
Finally, for random walk with law $P_{0}^{N,\alpha}$ where $\alpha \in (1/d, 1)$, we estimate $\underline S_N$ with the help of \eqref{eq:probability of multiple hits}.

\begin{proposition}\label{prop:asymptotics_alpha<1}
    For every $\delta > 0$ and $\alpha > 0$, 
    \begin{equation}\label{eq:asymptotics_alpha<1}
    \lim_{N\to\infty}P_{0}^{N,\alpha}\left[\log\underline{S}_{N}\geq 
    \frac{\overline{u}-2\delta}{d+1}\cdot N^{d(1-\alpha)}    \right]=1.
    \end{equation}
\end{proposition}

\begin{proof}
    We write
    \begin{equation}
    \overline K = \frac{2}{d + 1}\cdot\exp\left(\frac{\overline{u}-2\delta}{d+1}\cdot N^{d(1-\alpha)}\right)\quad \mbox{and}\quad\overline{\ell} = \frac{\overline u - \delta}{d + 1}N^d.
    \end{equation}
    Since the random walk starts from origin, all points $z$ with $\lvert\pi_{\mathbb{Z}}(z)\rvert > \overline K$ have local time $\widehat{L}^z_{\overline K} = 0$. Therefore, 
    \begin{equation}\label{eq:union bound underline SN}
    \begin{split}
        P_{0}^{N,\alpha}\left[\log\underline{S}_{N}< 
        \frac{\overline{u}-2\delta}{d+1} N^{d(1-\alpha)}\right]
        P_{0}^{N,\alpha}\left[\sup_{z\in \mathbb{Z}\cap [-\overline K, \overline K]}\widehat{L}^z_{\overline K}\geq \overline{\ell}\right] +
        P_{0}^{N,\alpha}\left[\frac{\rho_{\overline K}}{\overline K} < \frac{d+1}{2}\right].
        \end{split}
    \end{equation}
    By \eqref{eq:limit of rho_k}, the second term converges to zero as $N$ goes to infinity. While for the first term,
    it follows from \eqref{eq:probability of multiple hits} that
    \begin{equation}\label{eq:union bound for alpha<1}
    \begin{split}
        P_{0}^{N,\alpha}\left[\sup_{z\in \mathbb{Z}\cap [-\overline K, \overline K]}\widehat{L}^z_{\overline K}\geq \overline{\ell}\right]
        &\leq (2\overline K + 1)\sup_{z \in \mathbb Z \cap [-\overline K, \overline K]}P_{0}^{N,\alpha}\left[\widehat L_{\overline K}^z\geq \overline{\ell}\right]\\
        &\leq C\exp\left(\frac{\overline{u}-2\delta}{d+1}N^{d(1-\alpha)} - \left(\frac{\overline{u}-\delta}{d+1}N^d-1\right)\log \left(1 - \frac{1}{N^{d\alpha}}\right)\right)\\
        &=C\exp\left(-\frac{\delta}{d+1}N^{d(1-\alpha)}+o\left(N^{d(1-\alpha)}\right)\right),
        \end{split}
    \end{equation}
    and the conclusion hence follows. 
    \end{proof}

\section{A geometric argument}\label{sec:Geometric}

In this section, we prove for the random time $\underline S_N$ defined in \eqref{eq:def of underlineS and overlineS}, with probability tending to $1$ as $N\rightarrow \infty$, the event $T_{N}\geq \underline{S}_{N}$ holds  (see \Cref{prop:underlineS and T}). We first address the case of simple random walk and will extend this to the biased walk case in \Cref{subsec:adapting geometric}. 

We begin with the coarse-graining setup. We introduce a series of concentric boxes in \eqref{eq:def of boxes at origin}, and give the definitions of $\text{good}(\beta,\gamma)$, $\text{fine}(\gamma)$ and their intersection, normal$(\beta, \gamma)$ (whose complements are respectively $\text{bad}(\beta,\gamma)$, $\text{poor}(\gamma)$ and abnormal$(\beta, \gamma)$) respecitively in \Cref{def:good(beta-gamma),def:bad local time,def:bad boxes}. We then show in \Cref{lem:good connectivity} that the vacant sets in neighbouring normal$(\beta,\gamma)$ boxes have good connectivity property. Based on that, we prove in \Cref{prop:geometric argument} that on the event $T_{N}\leq\underline{S}_{N}$, there must exist a ``$d$-dimensional coarse-grained surface'' of abnormal$(\beta, \gamma)$ boxes with the help of a geometric argument. We remark that this geometric argument, based a discrete intermediate value trick and isoperimetric inequality, is similar to the ones in \cite[Section 5]{Win08} whose root can be traced back to \cite[Appendix A]{JA96}.

In \Cref{prop:bad(beta-gamma),prop:bad local time}, we further split the event in \Cref{prop:geometric argument} into two cases: having a $d$-dimensional surface of $\text{bad}(\beta, \gamma)$ boxes or $\text{poor}(\gamma)$ boxes. The proofs are deferred to \Cref{sec:bad(beta-gamma),sec:bad local time}, respectively, and to \Cref{sec:lower bound in biased case} for the biased walk case. We end this section with the proof of \Cref{prop:underlineS and T} assuming \Cref{prop:bad(beta-gamma),prop:bad local time}.\\

Our main result of this section is as follows.

\begin{proposition}\label{prop:underlineS and T}
For every $\delta>0$, we have 
\begin{equation}\label{eq:underlineS and T}
\lim_{N\to \infty}P_{0}^{N}[T_{N}\geq \underline{S}_{N}]=1.
\end{equation}
\end{proposition}

We now introduce the system of boxes that will play an important role in the subsequent analysis. Choose an arbitrary $\psi\in(1/d,1)$ and let $L=[N^{\psi}]$. We will consider boxes with side-length $L$ on the cylinder $\mathbb{E}$, each of which is associated with a series of concentric boxes. More precisely, we write
\begin{equation}\label{eq:def of boxes at origin}
\begin{split}
&B_{0}=[0,L)^{d+1},\quad D_{0}=[-3L,4L)^{d+1},\quad\check{D}_{0}=[-4L,5L)^{d+1},\\ 
&U_{0}=\left[-L[\log N]+1,L[\log N]-1\right)^{d+1},\,\check{U}_{0}=\left[-L([\log N]+1)+1,L([\log N]+1)-1\right)^{d+1}.
\end{split}
\end{equation}
The above series of concentric boxes then satisfy the following inclusion relationship
\begin{equation}\label{eq:inclusion of concentric boxes}
B_{0}\subseteq D_{0}\subseteq \check{D}_{0}\subseteq U_{0}\subseteq \check{U}_{0}.
\end{equation}
We also consider the translates of these boxes (i.e., $x+B_{0},x+D_{0}$, etc.)
\begin{equation}\label{eq:def of translated boxes}
B_{x}\subseteq D_{x}\subseteq \check{D}_{x}\subseteq U_{x}\subseteq \check{U}_{x},\quad\text{for every }x\in \mathbb E.  
\end{equation}
Very often, for convenience, we will refer to the boxes $B_{x}$, $x\in \mathbb E$, as $L$-boxes, and write $B,D,\check{D},U,\check{U}$ as $B_{x},D_{x},\check{D}_{x}, U_{x},\check{U}_{x}$ with $x\in \mathbb E$ for short, when no confusion arises.
We remark that the construction of series of boxes is quite similar to that in Section 3 of \cite{Szn17} (see \cite[(3.9) and (3.10)]{Szn17}), while we adapt the box size $L$, and replace the large constant $K$ in \cite{Szn17} with $[\log N]$ in \eqref{eq:def of boxes at origin} to simplify some technicalities  here.

Given an $L$-box $B$, we write the successive times of return to $D$ and departure from $U$ as (recall \eqref{eq:def of return and departure} for definitions of return times and departure times and notation) 
\begin{equation}\label{eq:def of return and departure for boxes}
R_{1}^{D,U}<D_{1}^{D,U}<R_{2}^{D,U}<D_{2}^{D,U}<\cdots<R_{k}^{D,U}<D_{k}^{D,U}<\cdots.   
\end{equation}
We write $R_{k}^{D}$ and $D_{k}^{D}$ for short. Note that in this section we only work on the recurrent simple random walk on $\mathbb E$, and thus the times $R_{k}^{D},D_{k}^{D},{k\geq 1}$ are all $P_{0}^{N}$-a.s.~finite.
The successive excursions from $D$ to $\partial U$ in the random walk $(X_n)_{n\geq 0}$ are then defined as
\begin{equation}\label{eq:def of simple random walk excursions}
\{W_{\ell}^{D,U}\}_{\ell\geq 1}:=
\left\{X_{[R_\ell^{D},D_\ell^{D})}\right\}_{\ell\geq 1}. 
\end{equation}
We also write $W_{\ell}^{D}$, $\ell\geq 1$ as a shorthand for $W_{\ell}^{D,U}$, $\ell\geq 1$. Moreover, for a real number $t\geq 1$, we define $W_{t}^{D}$ to be $W_{\ell}^{D}$ with $\ell=[t]$. We write the number of excursions from $D$ to $\partial U$ in the simple random walk $(X_n)_{n\geq 0}$ before time $\underline{S}_{N}$ as
\begin{equation}\label{eq:def of N_SN}
N_{\underline{S}_{N}}(D):=\sup\{k\geq 0:D_{k}^{D}\leq \underline{S}_{N}+1\}.
\end{equation}
We then consider two favorable events with respect to different types of excursions. 
For a fixed series of boxes $B\subseteq D\subseteq U$ and two constants $a, b$, let $X = \{X_1^{D, U}, \dots, X_{\ell}^{D, U}, \dots\}$ stand for an arbitrary type of excursions from $D$ to $\partial U$, we define 
\begin{equation}\label{def:exist event}
    \begin{split}
       \text{Exist}(B, X, a) := &\Big\{\mbox{There exists a connected subset with diameter at least }\frac{L}{5}\Big.\\
       &\quad \text{ in }B\setminus \Big(\text{range}\big(X_1^{D, U}\big)\cup\cdots\cup\text{range}\big(X_{a\cdot \text{cap}(D)}^{D, U}\big)\Big)\Big\}, \quad \text{ and}
    \end{split}
\end{equation}
\begin{equation}\label{def:unique event}
\begin{split}
        \text{Unique}(B,X,a,b) := &\Big
        \{\mbox{For any }L\text{-box } B'=Le + B, |e|=1, \mbox{ any two connected sets with }\\
        &\ \ \mbox{diameter at least }\frac{L}{10}\mbox{ in }B\setminus \Big(\text{range}\big(X_1^{D, U}\big)\cup\cdots\cup\text{range}\big(X_{a\cdot \text{cap}(D)}^{D, U}\big)\Big)
        \\ 
        &\ \ \mbox{and }B'\setminus \Big(\text{range}\big(X_1^{D', U'}\big)\cup\cdots\cup\text{range}\big(X_{a\cdot \text{cap}(D)}^{D', U'}\big)\Big)\text{ are connected in }\\ 
        &\ \  D\setminus \Big(\text{range}\big(X_1^{D, U}\big)\cup\cdots\cup\text{range}\big(X_{b\cdot \text{cap}(D)}^{D, U}\big)\Big) \Big\},
\end{split}
\end{equation}
where $\text{range}(X_{\ell}^{D, U})$ denotes the set of points visited by $X_{\ell}^{D, U}$, and is seen as an empty set when $\ell<1$, and $D',U'$ are concentric boxes of $B'$. We then denote the complements of these events by
\begin{equation}\label{def:fail event}
    \text{fail}_1(B, X, a) = \text{Exist}(B, X, a)^c, \quad \text{and}\quad\text{fail}_2(B, X, a, b) = \text{Unique}(B, X, a, b)^c.
\end{equation}
In the following, we fix constants $\beta>\gamma$ in $(\overline{u}-\delta,\overline{u})$.
\begin{definition}[Good boxes]\label{def:good(beta-gamma)}
    Given an $L$-box $B$, we say $B$ is $\text{good}(\beta,\gamma)$ if both events $\text{Exist}(B, W, \beta)$ and $\text{Unique}(B, W, \beta, \gamma)$ hold.
    If $B$ is not $\text{good}(\beta,\gamma)$, we say that it is $\text{bad}(\beta,\gamma)$.
\end{definition}

\begin{definition}[Fine boxes]\label{def:bad local time}
Given an $L$-box $B$ and its associated concentric boxes $D$ and $U$, 
we say $B$ is $\text{fine}(\gamma)$ if 
\begin{equation}
    N_{\underline{S}_{N}}(D)\leq \gamma\cdot\text{cap}(D), \text{ that is,}\quad D_{\gamma\cdot\text{cap}(D)}^{D}\leq \underline{S}_{N}+1,
\end{equation}
where we recall \eqref{eq:def of N_SN} that $N_{\underline S_N}(D)$ denotes the number of excursions before time $\underline S_N$. Otherwise, we say that it is $\text{poor}(\gamma)$.
\end{definition}

\begin{definition}[Normal boxes]\label{def:bad boxes}
We say $B$ is normal$(\beta, \gamma)$ if $B$ is both $\text{good}(\beta,\gamma)$ and $\text{fine}(\gamma)$. Otherwise, the box $B$ is abnormal$(\beta, \gamma)$.
\end{definition}

Let us briefly discuss the above three events. Here, the events $\text{good}(\beta,\gamma)$ and $\text{fine}(\gamma)$ play similar roles to the notions of $\text{good}(\alpha,\beta,\gamma)$ and the event $N_{u}(D)<\gamma\cdot\text{cap}(D)$ in \cite{Szn17} respectively. In addition, in the definition of good$(\beta, \gamma)$, the events $\text{Exist}(B, W, \beta)$ and $\text{Unique}(B, W, \beta, \gamma)$ in \eqref{def:exist event} and \eqref{def:unique event} resemble the events $\text{Exist}^{\mathcal V}(R,u)$ and $\text{Unique}^{\mathcal V}(R,u,v)$ in \eqref{eq:def of existence} and \eqref{eq:def of uniqueness} respectively. Similarly, the ``existence'' condition is monotone in $\beta$, while the ``uniqueness'' condition  is only monotone in $\gamma$ but not in $\beta$.

The next lemma shows that the events $\text{good}(\beta,\gamma)$ and $\text{fine}(\gamma)$ lead to good connectivity of the vacant set of the simple random walk $X_{\cdot}$ on $\mathbb E$.

\begin{lemma}\label{lem:good connectivity}
Let $B^{i}$, $0\leq i\leq n$, be a sequence of neighbouring $L$-boxes, that is, for each $0\leq i\leq n-1$, there exist coordinate vectors $e_{i}$ such that $B^{i+1}= Le^{i}+B^{i}$, and denote by $D^{i}$ the $D$-type box attached to $B^{i}$. 
If for all $0\leq i\leq n$, $B^i$ is $\text{normal}(\beta,\gamma)$, then there exists a path in $\big(\bigcup_{i=0}^{n}D^{i}\big)\cap \big(\mathbb{E}\setminus X_{[0,\underline{S}_{N}]}\big)$ starting in $B^{0}$ and ending in $B^{n}$. 
\end{lemma}
\begin{proof}
For every $0\leq i\leq n$, since $B^{i}$ is $\text{good}(\beta,\gamma)$, 
$B^i\setminus\big(\text{range}(W^{D^i}_1)\cup\cdots\cup\text{range}(W^{D^i}_{\beta\cdot\text{cap}(D^{i})})\big)$
contains a connected subset $C^{i}$ with diameter at least $L/5$. Again by the uniqueness property of $\text{good}(\beta,\gamma)$, for every $0\leq i<n$, $C^{i}$ and $C^{i+1}$ are connected in 
$D^i\setminus\big(\text{range}(W^{D^i}_1)\cup\cdots\cup\text{range}(W^{D^i}_{\gamma\cdot\text{cap}(D^{i})})\big)$, which is a subset of $\mathbb{E}\setminus X_{[0,\underline{S}_{N}]}$ using the definition of $\text{fine}(\gamma)$. Therefore, $\mathbb{E}\setminus X_{[0,\underline{S}_{N}]}$ contains a connected component which further contains $\bigcup_{i=0}^{n}C^{i}$.
\end{proof}

Given \Cref{lem:good connectivity}, we now come to the main geometric proposition of this section. Recall that $\pi_{i}$, $1\leq i \leq d+1$, stand for the projection from $\mathbb E$ to its corresponding $d$-dimensional hyperplane in $\mathbb E$ and $\pi_{\mathbb{Z}}$ denotes the projection to $\mathbb{Z}$ (see  \Cref{subsec:prelim} for formal definition).

\begin{proposition}\label{prop:geometric argument}
There exists a constant $c_{3}=c_{3}(\psi)>0$ such that for all $N\geq c_{3}$, on the event $\{T_{N}\leq\underline{S}_{N}\}$, there exists a box $\mathsf{B}$ with side-length $[N/\log^{3}N]$ and $\mathcal C$ which is a subset of $\mathbb E$ such that for some $\pi_{*}\in\{\pi_{i}\}_{1\leq i\leq d+1}$ and a positive constant $c_{4}=c_{4}(d)$,
\begin{align}
    d(0,\mathsf{B})&\geq N/10;\label{eq:large distance from mathsfB}\\
    |\pi_{*}(\mathcal{C})|&\geq c_{4}\left(\frac{N}{L\log^{3}N}\right)^{d};\label{eq:d-dim surface of bad boxes}\\
    \{B_{x}\}_{x\in \mathcal{C}} \text{ are disjoint } &  \text{abnormal}(\beta, \gamma) \text{ boxes contained in }\mathsf{B}.\label{eq:bad boxes property}
\end{align}
Moreover, on the event $\{T_{N}\leq\underline{S}_{N}\leq N^{5d}\}$, we further have
\begin{equation}\label{eq:vertical component of mathsfB}
\pi_{\mathbb{Z}}(\mathsf{B})\subseteq [-N^{10d},N^{10d}].
\end{equation}
\end{proposition}

From now on we call $\mathcal C$ the set of base points. We remark here that \eqref{eq:large distance from mathsfB} is a technical condition, which will be useful in estimating the hitting distribution on a box $\check D \subseteq \mathsf B$ for a random walk starting from a point far away in $\mathsf B$ (see e.g.~\Cref{prop:coupling W and widetildeZ}). 
Moreover, as already stated in the sketch of proof \Cref{subsec:sketch}, the condition \eqref{eq:vertical component of mathsfB} is necessary to control the combinatorial complexity of selecting the box $\mathsf{B}$ (see the discussions below \eqref{eq:union bound on C1}). 

\begin{proof}
We define
\begin{equation}\label{eq:def of M, TM and EM}
    M=\left[\frac{N}{10L\log^{3}N}\right],\quad v^* = \Big(\Big[\frac{N}{4L}\Big],0,\dots,0\Big)  \in \mathbb Z^{d},\quad T_{M}=B(v^*, M),\quad E_{M}=T_{M}\times\mathbb{Z}.
\end{equation}
For the small cylinder $E_M$, we still use $\pi_{i},i=1,\dots,d+1$ to denote the projection of $E_{M}$ onto the $i$-th coordinate hyperplane. We focus on the set of $L$-boxes $\{B_{x}:x\in LE_{M}\}$, whose disjoint union is $\big([-ML,(M+1)L)^{d}+ Lv^*\big)\times\mathbb{Z}$, and therefore there exists $c_3(\psi)>0$ such that when $N\geq c_{3}(\psi)$, the distance between the origin and the union of these boxes are larger than $N/6$.   
For each $x\in LE_{M}$, we say $x$ is normal (resp., abnormal) if its corresponding box $B_{x}$ is normal$(\beta, \gamma)$ (resp., abnormal$(\beta, \gamma)$) in $\mathbb{E}$. We also say $x,y\in LE_{M}$ are neighbouring if $B_{x}$ and $B_{y}$ are neighbouring $L$-boxes.

By \Cref{prop:asymptotics_alpha>1}, $\underline{S}_{N}$ is finite almost surely. Therefore, there exists a (random) large positive constant $\Gamma>100\max\{\underline{S}_{N},L\}$, such that for all $x\in T_{M}\times\big((-\infty,-\Gamma]\cup[\Gamma,\infty)\big)$, $B_{x}$ is not visited by $X_{[0,\underline{S}_{N}]}$. We then say $x\in T_{M}\times\big((-\infty,-2\Gamma]\cup[2\Gamma,\infty)\big)$ are empty vertices and say corresponding $L$-boxes $B_{x}$ are empty $L$-boxes. We respectively denote the connected component of normal vertices in $E_{M}$ that contains $T_{M}\times[2\Gamma,\infty)$ or $T_{M}\times(-\infty,-2\Gamma]$ as $\mathcal{C}_{\rm Top}$ or $\mathcal{C}_{\rm Bottom}$, where $T_{M}\times[2\Gamma,\infty)$ and $T_{M}\times(-\infty,-2\Gamma]$ themselves are both seen as connected components. 

Now by \Cref{lem:good connectivity}, if a sequence of neighbouring $L$-boxes $B^{i},0\leq i\leq n$ satisfies that $B^{i}$ is normal if $j\leq i\leq k$ and $B^{i}$ is empty if $i\leq j$ or $i\geq k$, then there exists a path in $\big(\bigcup_{i=j}^{k}D^{i}\big)\cap \big(\mathbb{E}\setminus X_{[0,\underline{S}_{N}]}\big)$ starting in $B^{j}$ and ending in $B^{k}$, which can further be extended into a path starting in $B^{0}$ and ending in $B^{n}$ by emptiness of $B^{i}$, $i\leq j$ and $i\geq k$. Therefore, on the event $\{T_{N}\leq\underline{S}_{N}\}$, $\mathcal{C}_{\rm Top}$ and $\mathcal{C}_{\rm Bottom}$ cannot belong to the same infinite connected component.
We then define the function
\begin{equation}\label{eq:def of function p}
p(x)=\frac{|\mathcal{C}_{\rm Top}\cap B(x,M)|}{|B(x,M)|},\quad x\in \{v^*\}\times\mathbb{Z}. 
\end{equation}
Then $p(x)$ equals $1$ for $\pi_{\mathbb Z}(x)\geq 2\Gamma+M$ and equals $0$ for $\pi_{\mathbb Z}(x)\leq -2\Gamma-M$. Moreover, when $|\pi_{\mathbb Z}(x)-\pi_{\mathbb Z}(y)|=1$, with $\Delta$ standing for the symmetric difference, we have
\begin{equation}\label{eq:difference of p}
|p(x)-p(y)|\leq \frac{|B(x,M)\Delta B(y,M)|}{|B(0,M)|}\leq \frac{c}{M}.
\end{equation}
Thus by a discrete intermediate value trick, for $N\geq c_{3}(\psi)$, there exists $x_{*}\in \{v^{*}\}\times[-2\Gamma-M,2\Gamma+M]$ such that
\begin{equation}\label{eq:volume of CTop in box x*}
\left|p(x_{*})-\frac{1}{2}\right|\leq \frac{c}{M}\leq \frac{1}{4}.    
\end{equation}

Recall in \Cref{sec:preliminary} that we denote by $\partial_{B(x_{*},M)}A$ the relative outer boundary of a set $A\subseteq B(x_{*},M)$ in the box $B(x_{*},M)$, then all the vertices in $\partial_{B(x_{*},M)}(\mathcal{C}_{\rm Top}\cap B(x_{*},M))$ must be abnormoal. Combining \eqref{eq:volume of CTop in box x*} with the isoperimetric inequality (A.3)-(A.6) in \cite{JA96} implies that there exists $\pi_{*}\in\{\pi_{i}\}_{1\leq i\leq d+1}$ satisfying 
\begin{equation}
\pi_{*}\left(\partial_{B(x_{*},M)}(\mathcal{C}_{\rm Top}\cap B(x_{*},M))\right)\geq c'|\mathcal{C}_{\rm Top}\cap B(x_{*},M)|^{\frac{d}{d+1}}\geq c'|B(x_{*},M)|^{\frac{d}{d+1}}=c'\cdot M^{d}.
\end{equation}
We conclude the proof of \eqref{eq:large distance from mathsfB}-\eqref{eq:bad boxes property} by choosing a box $\mathsf{B}$ containing $LB(x_{*},M)$ and $\mathcal{C}$ the set of vertices in $\partial_{B(x_{*},M)}(\mathcal{C}_{\rm Top}\cap B(x_{*},M))$ such that its corresponding $L$-boxes $B_x, x\in\mathcal C$ do not intersect. If we further have $\underline{S}_{N}\leq N^{5d}$, then we may take $\Gamma$ as $N^{7d}$, yielding the last claim \eqref{eq:vertical component of mathsfB}.
\end{proof}

With this proposition, the proof of \Cref{prop:underlineS and T} can be reduced to the following two parts: 

\begin{proposition}\label{prop:bad(beta-gamma)}
For two fixed constants $\beta > \gamma$ in $(\overline{u}- \delta, \overline{u})$, consider a fixed box $\mathsf{B}$ with side-length $[N/\log^{3}N]$ that satisfies $d(0,\mathsf{B})\geq N/10$, a set of base points $\mathcal C$ such that $\{B_{x}\}_{x\in\mathcal{C}}$ are disjoint abnormal$(\beta, \gamma)$ boxes contained in $\mathsf B$ and a projection $\pi_{*}\in\{\pi_{i}\}_{1\leq i\leq d+1}$ which satisfies \eqref{eq:d-dim surface of bad boxes}. Then for any subset $\mathcal{C}_{1}$ of $\mathcal{C}$ with
\begin{equation}\label{eq:size of C1}
|\mathcal{C}_{1}|=\left[\frac{1}{3}c_{4}\left(\frac{N}{L\log^{3}N}\right)^{d}\right],
\end{equation}
we have 
\begin{equation}\label{eq:bad(beta-gamma)}
\lim_{N\to\infty}\frac{1}{|\mathcal{C}_{1}|\log N}\log P_{0}^{N}\Big[\bigcap_{x\in\mathcal{C}_{1}}\big\{B_x\text{ is bad}(\beta,\gamma)\big\}\Big]=-\infty. 
\end{equation}
\end{proposition}
As discussed at this beginning of this section, we will give the proof in \Cref{sec:bad(beta-gamma)}.

\begin{proposition}\label{prop:bad local time}
For two fixed constants $\beta > \gamma$ in $(\overline{u}- \delta, \overline{u})$, we still consider a fixed box $\mathsf{B}$ with side-length $[N/\log^{3}N]$ that satisfies $d(0, \mathsf B)\geq N/10$, a set of base points $\mathcal C$ such that $\{B_{x}\}_{x\in\mathcal{C}}$ are disjoint abnormal$(\beta, \gamma)$ boxes in $\mathsf B$ and a projection $\pi_{*}\in\{\pi_{i}\}_{1\leq i\leq d+1}$ satisfying \eqref{eq:d-dim surface of bad boxes}. Then for any subset $\mathcal{C}_{1}$ of $\mathcal{C}$ with
\begin{equation}\label{eq:C1 disjoint}
|\mathcal{C}_{1}|=\left[\frac{1}{3}c_{4}\left(\frac{N}{L\log^{3}N}\right)^{d}\right], \mbox{ and }\pi_{*}(x)\neq \pi_{*}(y)\text{\ for all different }x,y\in\mathcal{C}_{1}.
\end{equation}
we have 
\begin{equation}\label{eq:bad local time}
\lim_{N\to\infty}\frac{1}{|\mathcal{C}_{1}|\log N}\log P_{0}^{N}\Big[\bigcap_{x\in \mathcal C_1}\big\{B_{x}\text{ is poor}(\gamma)\big\}\Big]=-\infty.  
\end{equation}
\end{proposition}
As discussed at this beginning of this section, we will give the proof in \Cref{sec:bad local time}.
At the end of this section, we complete the proof of \Cref{prop:underlineS and T} assuming the two propositions above.
\begin{proof}[Proof of \Cref{prop:underlineS and T} assuming \Cref{prop:bad(beta-gamma),prop:bad local time}] 
Note that 
\begin{equation}\label{eq:underlineS and T and N5d}
    \lim_{N\rightarrow \infty} P_{0}^{N}[T_N \leq \underline{S}_N] \leq \lim_{N\to\infty}P_{0}^{N}[\underline{S}_{N} > N^{5d}] + \lim_{N\to\infty}P_{0}^{N}[T_{N}\leq\underline{S}_{N}\leq N^{5d}] = {\rm \rom{1}} + {\rm \rom{2}}.
\end{equation}
By 
\Cref{prop:asymptotics_alpha>1}, \rom{1}$=0$. Therefore, it suffices to prove \rom{2}$=0$. 

By \Cref{prop:geometric argument}, for $N\geq c_{3}(\psi)$, on the event $\{T_{N}\leq\underline{S}_{N}\leq N^{5d}\}$, there exists a box $\mathsf{B}$ with side-length $[N/\log^{3}N]$, a set of base points $\mathcal C$ and a projection $\pi_{*}\in\{\pi_{i}\}_{1\leq i\leq d+1}$ satisfying \eqref{eq:large distance from mathsfB}-\eqref{eq:vertical component of mathsfB}. Then by the definition of abnormal$(\beta,\gamma)$ and \eqref{eq:d-dim surface of bad boxes}, there must exist a subset $\mathcal{C}_{1}$ of $\mathcal{C}$ such that \eqref{eq:C1 disjoint} holds, and the corresponding $L$-boxes $\{B_{x}\}_{x\in\mathcal{C}_{1}}$ are all $\text{bad}(\beta,\gamma)$ or all $\text{poor}(\gamma)$. Using a union bound gives
\begin{equation}\label{eq:union bound on C1}
P_{0}^{N}[T_{N}\leq\underline{S}_{N}\leq N^{5d}]\leq\sum_{\mathcal{C}_{1}}P_{0}^{N}\Big[\bigcap_{x\in\mathcal{C}_{1}}\big\{B_x\text{ is bad}(\beta,\gamma)\big\}\Big]+\sum_{\mathcal{C}_{1}}P_{0}^{N}\Big[\bigcap_{x\in\mathcal{C}_{1}}\big\{B_x\text{ is poor}(\gamma)\big\}\Big].
\end{equation}

The number of choices of a box $\mathsf{B}$ that satisfies \eqref{eq:vertical component of mathsfB} is no more than $CN^{20d}$. In addition, given $\mathsf{B}$, the possible ways of selecting a set $\mathcal{C}_{1}$ with fixed cardinality as in \eqref{eq:C1 disjoint} is no more than $\left(N^{d+1}\right)^{|\mathcal{C}_{1}|}$.
Therefore, the total possible ways of choosing a base points set $\mathcal{C}_{1}$ satisfying \eqref{eq:C1 disjoint} is no more than
\begin{equation}\label{eq:number of C1}
CN^{20d}\cdot N^{(d+1)|\mathcal{C}_{1}|}=CN^{20d}\cdot e^{C|\mathcal{C}_{1}|\log N}=e^{C|\mathcal{C}_{1}|\log N}.
\end{equation}
Plugging \eqref{eq:number of C1} as well as \eqref{eq:bad(beta-gamma)} and \eqref{eq:bad local time} on the probabilities of two atypical events into the union bound \eqref{eq:union bound on C1} yields \rom{2}$=0$, which concludes the proof given \eqref{eq:underlineS and T and N5d}.
\end{proof}

\section{Unlikeliness of surfaces of \texorpdfstring{$\text{bad}(\beta,\gamma)$}{bad(beta, gamma)} boxes}\label{sec:bad(beta-gamma)}
The main goal of this section is to prove \Cref{prop:bad(beta-gamma)}, which states that for a simple random walk, the probability that there exists a ``$d$-dimensional'' coarse-grained surface of bad$(\beta, \gamma)$ boxes decays rapidly as described in \eqref{eq:bad(beta-gamma)}. 
We remark here that, this result will be extended to biased random walk for establishing an analogue of \Cref{prop:underlineS and T} for the biased walk case. The necessary adaptations will be detailed in  \Cref{subsec:adapting bad(beta-gamma)}.  

Before providing an in-depth introduction of this section, let us recall the definition of concentric boxes $B_{x}\subseteq\cdots\subseteq\check{U}_{x}$ for $x\in\mathbb{E}$ in \eqref{eq:def of boxes at origin}-\eqref{eq:def of translated boxes}, and the definition of random walk excursions in \eqref{eq:def of simple random walk excursions}.
We also recall the definition of the set $\mathcal{C}_{1}$ of base points introduced in \Cref{prop:bad(beta-gamma)}.
For any $x \in \mathcal C_2$, $x'=Le + x$ for some coordinate vector $e$, we abbreviate boxes such as $D_x,D'_x$ as $D,D'$ respectively in the following introduction part, and we call $D, D'$ or $U,U'$ a pair of neighbouring boxes.
Additionally, we will focus on a subset $\mathcal{C}_{2}$ of $\mathcal{C}_{1}$, which will be defined formally in \eqref{eq:sparsity of C2} to make its corresponding $L$-boxes far from each other. In the following, we will develop a sequence of couplings, which comprises the following ingredients.
\begin{itemize}
    \item \Cref{prop:coupling W and widetildeZ} employs the soft local time techniques developed in \cite{PT15,CGP13} to ``decouple" the excursions centered around different $L$-boxes $B_x, x\in \mathcal C_2$. 
    More specifically, we introduce a coupling between the excursions $W_{\ell}^{\check{D}}$ from $\check{D}$ to $\partial\check{U}$ in the trajectory of the simple random walk $(X_{n})_{n\geq 0}$ and a certain collection of i.i.d.~random walk excursions $\widetilde Z_{\ell}^{\check D}$ from $\check{D}$ to $\partial\check{U}$, where $\check{D},\check{U}$ runs over a collection $\check{D}_{x},\check{U}_{x}, {x\in\mathcal{C}_{2}}$. The latter collection of $\widetilde{Z}$-type excursions are independent from each other as $x$ varies. 
    The reason for introducing the sequence of excursions 
    $W_{\ell}^{\check{D}},{\ell\geq 1}$ from $\check{D}$ to $\partial \check{U}$ is that this sequence contains the information of excursions within smaller boxes $D,U$ and their neighboring boxes $D',U'$. 
    That is, 
    both excursions $W_{\ell}^{D}$ (from $D$ to $\partial U$) and $W_{\ell}^{D'}$ (from $D'$ to $\partial U'$) involved in the  definition of $\text{good}(\beta,\gamma)$ boxes (see \eqref{def:unique event}) are contained in $W_{\ell}^{\check{D}},{\ell\geq 1}$. 
    \item \Cref{prop:coupling between widehatZ and W} then exploits the above coupling to gather information on the excursions going through the smaller boxes $D, U$ and $D', U'$.
    More specifically, this second coupling will connect the excursions of random walk $W^{D}_{\ell}, W^{D'}_{\ell}$ (from $D$ to $\partial U$ and $D'$ to $\partial U'$) extracted from the same excursions $W_{\ell}^{\check{D}},{\ell\geq 1}$ with successive excursions $\widetilde{Z}_{\ell}^{D}$ and $\widetilde{Z}_{\ell}^{D'}$ (from $D$ to $\partial U$ and $D'$ to $\partial U'$), extracted from the same sequence  $\widetilde{Z}_{\ell}^{\check{D}},{\ell\geq 1}$.
    \item \Cref{lem:coupling between widetildeZ and Z} compares the excursions  $\widetilde{Z}_{\ell}^{D}$ and $\widetilde{Z}_{\ell}^{D'}, {\ell\geq 1}$ with excursions of random interlacements $Z_{\ell}^{D}$ and $Z_{\ell}^{D'}$ (from $D$ to $\partial U$ and $D'$ to $\partial U'$). This follows from a similar procedure to the combination of \Cref{prop:coupling W and widetildeZ,prop:coupling between widehatZ and W}, where we first compare the excursions of $\widetilde Z_{\ell}^{\check D}$ with $Z_{\ell}^{\check D}$ using soft local time techniques, and then extract the excursions travelling through small boxes $D, U$ and $D', U'$. 
\end{itemize}

With these three coupling results, a certain proportion of bad$(\beta, \gamma)$ boxes in \eqref{eq:bad(beta-gamma)} yields an intersection of independent events on random interlacements (each of which is expressed in terms of $Z_{\ell}^{D_{x}}$ and $Z_{\ell}^{D_{x'}}, {\ell\geq 1}$ for some fixed $x\in\mathcal{C}_{2}$), each of which has small probability given \eqref{eq:def of existence}, \eqref{eq:def of uniqueness} and \eqref{eq:def of ubar} in the strongly-percolative regime. Therefore, proving unlikeliness of surface of $\text{bad}(\beta,\gamma)$ boxes is reduced to a standard large deviation result of sum of Bernoulli variables. 

It is noteworthy that the series of coupling draws inspiration from \cite[Section 5]{Szn17}, and several proofs in this section can even be adapted straightforwardly from those presented in that work. Consequently, we omit some technical parts for the sake of brevity. Instead, we detail on how to adjust the corresponding proofs in \cite{Szn17}, and clarify the correspondence between notation of this work and \cite{Szn17}. We remark that our notation for the constant $L$ 
and the system of boxes $B,D,\check{D},U,\check{U}$ is always in line with that in \cite{Szn17}, except that the requirement that $L$ is of order $\big[(N\log N)^{\frac{1}{d-1}}\big]$ in \cite{Szn17} is now replaced with the choice of $L=[N^{\psi}],\psi\in(1/d,1)$, and the constant $K$ in \cite[(3.9)-(3.10)]{Szn17} is replaced by $\log N$ (see \eqref{eq:def of boxes at origin}). Note that these changes do not affect the proofs, while substituting $\log N$ for $K$ actually facilitates some of the arguments.

The organization of this section is as follows. We first specify our notation, and then carry out the three couplings 
in \Cref{subsec:soft local time}. In \Cref{subsec:proof of prop bad(beta-gamma)} we adopt all the three couplings to prove \Cref{prop:bad(beta-gamma)}. \\

We now introduce notation used in this section. We will focus on the fixed box $\mathsf{B}$, the set of base points $\mathcal{C}_{1}$ and the projection $\pi_{*}$, as in the statement of \Cref{prop:bad(beta-gamma)}. We recall $W_{\ell}^{\check{D}} = W_{\ell}^{\check{D},\check{U}}, {\ell \geq 1}$ as excursions from $\check{D}$ to $\partial\check{U}$ of simple random walk $(X_{n})_{n\geq 0}$, similar as in \eqref{eq:def of return and departure for boxes} and \eqref{eq:def of simple random walk excursions}, and $W_{t}^{\check{D}}  = W_{[t]}^{\check{D}}$ for $t\geq 1$. 
We define a collection of i.i.d.~sequence of excursions $\widetilde Z_{\ell}^{\check D_x}$, ${\ell}\geq 1$, $x\in \mathcal C_1$ such that 
\begin{equation}\label{eq:def of widetildeZ}
\begin{split}
&\text{for each }x\in \mathcal{C}_{1}, \{\widetilde{Z}_{\ell}^{\check{D}_{x}}\}_{\ell\geq 1}\text{ are i.i.d.~excursions having the law as }X_{\cdot \land T_{\check U_x}} \text{ under }P_{\overline e_{\check D_x}},\\
&\text{ and }\{\widetilde{Z}_{\ell}^{\check{D}_{x}}\}_{\ell\geq 1} \text{ are independent as }x \text{ varies over }\mathcal{C}_{1}.
\end{split}
\end{equation}
and let $\widetilde{Z}_{t}^{\check{D}}$ stand for $\widetilde{Z}_{[t]}^{\check{D}}$ when $t\geq 1$, which is in line with the notation $W_{t}^{\check{D}}$.

We then introduce the sparse subset $\mathcal{C}_{2}$ of the set $\mathcal{C}_{1}$ as the maximal subset that satisfies 
\begin{equation}\label{eq:sparsity of C2}
\inf\{d(x,y):x,y\in\mathcal{C}_{2},x\neq y\}\geq 10L\log N.
\end{equation}
Note that when $x\neq y$ belong to $\mathcal{C}_{2}$, the corresponding $\check{U}_{x}$, $\check{U}_{y}$ (recall \eqref{eq:def of boxes at origin} and \eqref{eq:def of translated boxes}) satisfy
\begin{equation}\label{eq:sparsity of boxes in mathcalC2}
d(\check{U}_{x},\check{U}_{y})\geq 5L\log N.
\end{equation}
In addition, by \eqref{eq:size of C1}, \eqref{eq:sparsity of C2} and the fact that $B_{x},x\in\mathcal{C}_{1}$ are disjoint, we have for some $c,C>0$, 
\begin{equation}\label{eq:size of C2}
\begin{split}
\frac{|\mathcal{C}_{2}|}{|\mathcal{C}_{1}|}\geq\frac{c}{\log^{d+1} N}, \quad\text{ and }\quad|\mathcal{C}_{2}|\geq c\cdot C\cdot\left(\frac{N}{L\log^{3}N}\right)^{d}\cdot\left(\frac{1}{\log N}\right)^{d+1}.
\end{split}
\end{equation}

For each subset $\mathcal S$ of $\mathsf B$, we write for the unions of its corresponding boxes:
\begin{equation}\label{eq:def of checkD and checkU C2}
D(\mathcal S)=\bigcup_{x\in \mathcal S}D_{x},
\quad 
\check{D}(\mathcal S):=\bigcup_{x\in\mathcal S}\check{D}_{x}\quad \text{and} \quad \check{U}(\mathcal S):=\bigcup_{x\in\mathcal S}\check{U}_{x},
\end{equation}
which are subsets of $\mathbb{E}$. Depending on context, we may also regard the above unions as subsets of $\mathbb{Z}^{d+1}$.

\subsection{Three couplings}\label{subsec:soft local time}
Following the strategy in Section 5 of \cite{Szn17}, we now provide the coupling between $W_{\ell}^{\check{D}_{x}}$ and $\widetilde{Z}_{\ell}^{\check{D}_{x}}$ for each $x\in\mathcal{C}_{2}$. 

For each $\check D = \check D_x$, $x\in \mathcal C_2$, we let $\left(\big(n_{\check{D}}(0,t)\big)_{t\geq 0},\big\{\widetilde{Z}_{\ell}^{\check{D}}\big\}_{\ell\geq 1}\right)$ stand for independent pairs of independent variables, with $\big(n_{\check{D}}(0,t)\big)_{t\geq 0}$ distributed as a Poisson counting process of intensity $1$, and $\big\{\widetilde{Z}_{\ell}^{\check{D}}\big\}_{\ell\geq 1}$ defined in \eqref{eq:def of widetildeZ}. We also write $n_{\check{D}}(a,b)=n_{\check{D}}(0,b)-n_{\check{D}}(0,a)$.

\begin{proposition}\label{prop:coupling W and widetildeZ}
There exists a coupling $\mathbb{Q}_{W,\widetilde{Z}}$ of the law $P_{0}^{N}$ and the law of $\big((n_{\check{D_x}}(0,t))_{t\geq 0},\{\widetilde{Z}_{\ell}^{\check{D}_x}\}_{\ell\geq 1}\big)$, $x \in \mathcal{C}_{2}$ such that, for every $\eta\in (0, \frac{1}{2})$, $N\geq c_{5}(\psi,\eta)$, $\lambda\in(0,\infty)$ and $\check{D}=\check{D}_{x},x\in\mathcal{C}_{2}$, on the event 
\begin{equation}\label{eq:def of widetildeE}
\begin{split}
E^{\lambda}_{\check D}(W, \widetilde Z):=\big\{& n_{\check{D}}(m,(1+\eta)m)<2\eta m,\,\, (1-\eta)m<n_{\check{D}}(0,m)<(1+\eta)m,\\ &\text{for all }m\geq \lambda\cdot\text{cap}(\check{D})\big\},
\end{split}
\end{equation}
for all $m\geq \lambda\cdot\text{cap}(\check{D})$ we have\footnote{Here, similar to the convention after \eqref{def:exist event}, the sets on the left hand side of \eqref{eq:inclusion checkD_1} and \eqref{eq:inclusion checkD_2} are empty when $(1-\eta)m<1$. We always honour this convention in the rest of this work, see e.g.~\Cref{prop:coupling between widehatZ and W,lem:coupling between widetildeZ and Z}, \Cref{def:widehatgood}, \Cref{prop:coupling W and widetildeW,lem:coupling widetildeW with widetildeZ,prop:coupling W and widetildeZ bias,prop:coupling W and widehatZ bias} and \Cref{sec:couplings}.} 
\begin{align}
&\left\{\widetilde{Z}_{1}^{\check{D}}\ ,\dots,\ \widetilde{Z}_{(1-\eta)m}^{\check{D}}\right\}\subseteq \left\{W_{1}^{\check{D}},\dots,W_{(1+3\eta)m}^{\check{D}}\right\};  \label{eq:inclusion checkD_1} \\
&\left\{W_{1}^{\check{D}},\dots,W_{(1-\eta)m}^{\check{D}}\right\}\subseteq \left\{\widetilde{Z}_{1}^{\check{D}}\ ,\dots,\ \widetilde{Z}_{(1+3\eta)m}^{\check{D}}\right\}.   \label{eq:inclusion checkD_2}
\end{align}
Moreover, for every $\lambda\in(0,\infty)$, the events $E^{\lambda}_{\check D_x}(W, \widetilde Z)$ are independent as $x$ varies over $\mathcal C_2$, and for every $\check{D}=\check{D}_{x},x\in\mathcal{C}_{2}$ (note that all the boxes $\check{D}_{x}$, $x\in \mathbb E$ have the same capacity and the probability of $E_{\check D}^{\lambda}(W, \widetilde Z)$ does not depend on the choice of $\check{D}$)
\begin{equation}\label{eq:prob of bad E}
\limsup_{N\to\infty}\frac{1}{\text{cap}(\check{D})}\log\mathbb{Q}_{W,\widetilde{Z}}\left[E^{\lambda}_{\check D}(W, \widetilde Z)^{c}\right] < -c_{6}(\lambda,\eta)<0.   
\end{equation}
\end{proposition}
\begin{proof}[Proof of \Cref{prop:coupling W and widetildeZ}]\label{remark:technical condition for first excursion}
We begin with an estimate on the hitting distribution of simple random walk, and then use soft local time techniques to construct the coupling. 
Taking $A=\check{D}$, $L=10[N^{\psi}]$ and $K=\log N/15$ in \Cref{prop:hitting distribution two boxes simple}, we get that for any $\eta\in(0,\frac{1}{2})$, there exists a positive constant $c_{5}(\psi,\eta)$ such that for any $\check D = \check D_x$, $x\in \mathcal C_2$, $y\in \check D$ and $z \in\partial\check{U}(\mathcal{C}_{2})$, if $N\geq c_{5}(\psi,\eta)(\geq c_{1}(\eta))$, then 
\begin{equation}\label{eq:hitting distribution}
\left(1-\frac{\eta}{3}\right)\overline{e}_{\check{D}}(y)\leq P_{z}\left[X_{H_{\check{D}(\mathcal C_2)}}=y\mid X_{H_{\check{D}(\mathcal C_2)}}\in\check{D}\right]\leq \left(1+\frac{\eta}{3}\right)\overline{e}_{\check{D}}(y).
\end{equation}

The inclusions \eqref{eq:inclusion checkD_1} and \eqref{eq:inclusion checkD_2} all follow directly from \eqref{eq:hitting distribution} and soft local time techniques (c.f.~\cite[Lemma 2.1]{CGP13}). Here we use the requirement \eqref{eq:large distance from mathsfB} so that \eqref{eq:hitting distribution} applies to the first trajectory $W_{1}^{\check{D}}$. The proof of \eqref{eq:prob of bad E} follows from a union bound and the standard exponential Chebyshev's inequality for Poisson variables.
\end{proof}

Let us now turn to the second coupling, which refines the first coupling as $D,D'$ and $U,U'$ are respectively subsets of $\check{D}$ and $\check{U}$. 

From the infinite sequence of i.i.d.~excursions $\widetilde{Z}_{k}^{\check{D}},\, {k\geq 1}$, we can extract successive excursions $\widetilde{Z}_{\ell}^{D},\,{\ell\geq 1}$ and $\widetilde{Z}_{\ell}^{D'}, \,{\ell\geq 1}$ which are respectively excursions from $D$ to $\partial U$ and $D'$ to $\partial U'$. Note that for a given $\check{D}$, the sequence $\widetilde{Z}_{\ell}^{D},\,{\ell\geq 1}$ is a priori not an independent sequence, and two sequences $\widetilde{Z}_{\ell}^{D},\,{\ell\geq 1}$ and $\widetilde{Z}_{\ell}^{D'},\,{\ell\geq 1}$ are typically mutually dependent. However, under $\mathbb{Q}_{W,\widetilde{Z}}$, the collections $\big\{\widetilde{Z}_{\ell}^{D_{x}},\widetilde{Z}_{\ell}^{D_{x}'}\big\}_{\ell\geq 1}$, as $x$ varies over $\mathcal{C}_{2}$, are independent.
We first estimate the number of excursions $\widetilde{Z}^D$ and $\widetilde Z^{D'}$ extracted from the first $m$ excursions $\widetilde{Z}^{\check D}$, where $m\geq \lambda\cdot\text{cap}(\check{D})$ for some constant $\lambda$.

\begin{lemma}\label{lem:bad F}
Fix $\kappa\in (0,\frac{1}{2})$. For any $B=B_{x},x\in\mathcal{C}_{2}$ and a constant $\lambda\in (0,\infty)$, we define the events 
\begin{align}
F_{B}^{\lambda}(\widetilde Z):=&\bigg\{\text{for all }m\geq \lambda\cdot\text{\rm cap}(\check{D}), \widetilde{Z}_{1}^{\check{D}},\dots,\widetilde{Z}_{m}^{\check{D}}\text{ contain at least }(1-\kappa)m\frac{\text{\rm cap}(D)}{\text{\rm cap}(\check{D})}\\ &\text{ and at most }(1+\kappa)m\frac{\text{\rm cap}(D)}{\text{\rm cap}(\check{D})}\text{ excursions from }D\text{ to }\partial U\bigg\},\quad\text{and}\nonumber \\ 
&F_{B,+}^{\lambda}(\widetilde Z):=F_{B}^{\lambda}(\widetilde Z)\bigcap_{B'\text{ neighbouring }B}F_{B'}^{\lambda}(\widetilde Z). \label{eq:def of widetildeF}
\end{align}
where $D,\check{D},U,\check{U}$ are concentric with $B$ and $D',U'$ are concentric with a $B'$ neighbouring $B$. Noting that the probability of $F_{B,+}^{\lambda}(\widetilde{Z})$ does not depend on the specific choice of $B$, under the coupling $\mathbb{Q}_{W,\widetilde{Z}}$ in \Cref{prop:coupling W and widetildeZ}, for every $\lambda\in (0,\infty)$, it holds that 
\begin{equation}\label{eq:prob of bad F}
\limsup_{N\to\infty}\frac{1}{\text{\rm cap}(\check{D})}\log\mathbb{Q}_{W,\widetilde{Z}}\left[F_{B,+}^{\lambda}(\widetilde Z)^{c}\right] < -c_{7}(\lambda,\kappa)<0.   
\end{equation}
\end{lemma}

\begin{proof}[Proof of \Cref{lem:bad F}]
The proof of \Cref{lem:bad F} follows in a similar way as the proof of \cite[(5.15)]{Szn17} (more specifically, the paragraphs from Lemma 5.2 till the end of proof of Proposition 5.1). The coupling $\mathbb{Q}_{W,\widetilde{Z}}$ corresponds to the coupling $\mathbb{Q}^{\mathcal{C}}$ in \cite{Szn17}. The notation $\widetilde{Z}_{\ell}$ for i.i.d.~excursions is in line with the same notation in \cite{Szn17}, and the number $m_{0}$ in \cite[(5.10)]{Szn17} is now chosen as $\lambda\cdot\text{cap}(\check{D})$, a polynomial function in $L$. Substituting the sufficiently large constant $K$ in \cite{Szn17} by $\log N$ (which tends to infinity) indeed facilitates the proof when $N$ tends to infinity.
\end{proof}

Under the favourable events $E^{\lambda}_{\check D}(W, \widetilde Z)$ and ${F}_{B,+}^{\lambda}(\widetilde Z)$, we can now couple the excursions $W_{\ell}^{D}$, $W_{\ell}^{D'}$ with excursions $\widetilde{Z}_{\ell}^{D}$, $\widetilde{Z}_{\ell}^{D'}$, which are independent as $x$ varies over $\mathcal{C}_{2}$.
\begin{proposition}\label{prop:coupling between widehatZ and W}
Fix $\eta,\kappa\in (0,\frac{1}{2})$ and $\lambda\in(0,\infty)$. Recall the coupling $\mathbb{Q}_{W,\widetilde{Z}}$ and the events $E^{\lambda}_{\check D}(W, \widetilde Z)$, ${F}_{B,+}^{\lambda}(\widetilde Z)$ defined in \Cref{prop:coupling W and widetildeZ} and \Cref{lem:bad F}. 
Let 
\begin{equation}\label{eq:def of event G}
G_B^{\lambda}(W, \widetilde Z):=E^{\lambda}_{\check D}(W, \widetilde Z) \cap F_{B,+}^{\lambda}(\widetilde Z).  
\end{equation}
Then the events $G_{B}^{\lambda}(W, \widetilde Z)$ (where $B$ actually represents some $B_{x}$) are independent as $x$ varies over $\mathcal{C}_{2}$, and 
\begin{equation}\label{eq:prob of bad G}
\limsup_{N\to\infty}\frac{1}{\text{cap}(\check{D})}\log\mathbb{Q}_{W,\widetilde{Z}}\left[G_B^{\lambda}(W, \widetilde Z)^{c}\right] < -c_{8}(\lambda,\eta,\kappa)<0.   
\end{equation}
Moreover, for every $N\geq c_{9}(\psi,\eta)(\geq c_{5}(\psi,\eta))$, under the event $G_B^{\lambda}(W, \widetilde Z)$, for every $\ell\geq\frac{\lambda}{1-\eta}\cdot\text{cap}(D)$, and any $B'$ neighbouring $B$ and its associated $D'$, we have
\begin{equation}\label{eq:coupling widehatZD and WD}    
    \left\{
\begin{array}{rll}
     {\rm (i)}&\;\; \Big\{\widetilde{Z}_{1}^{D},\ \dots, \ \widetilde{Z}_{\ell}^{D}\Big\} \subseteq \Big\{W_{1}^{D},\dots,W_{(1+\zeta)\ell}^{D}\Big\}; \\
     {\rm (ii)}&\;\;\Big\{W_{1}^{D},\dots,W_{\ell}^{D}\Big\} \subseteq \left\{\widetilde{Z}_{1}^{D},\ \dots,\ \widetilde{Z}_{(1+\zeta)\ell}^{D}\right\},
\end{array}
\right.\quad\text{and}
\end{equation}
\begin{equation}\label{eq:coupling widehatZD' and WD'}   
    \left\{
    \begin{array}{rll}
         {\rm (i)}&\;\; \left\{\widetilde{Z}_{1}^{D'},\ \dots,\ \widetilde{Z}_{\ell}^{D'}\right\} \subseteq \left\{W_{1}^{D'},\dots,W_{(1+\zeta)\ell}^{D'}\right\}; \\
        {\rm (ii)}&\;\;\left\{W_{1}^{D'},\dots,W_{\ell}^{D'}\right\} \subseteq \left\{\widetilde{Z}_{1}^{D'},\ \dots,\ \widetilde{Z}_{(1+\zeta)\ell}^{D'}\right\},
    \end{array}
    \right.
\end{equation}
where we set $\zeta$ as 
\begin{equation}\label{eq:def of zeta}
1+\zeta=\frac{1+\kappa}{1-\kappa}\cdot\frac{(1+4\eta)^{2}}{(1-2\eta)^{2}}.
\end{equation}
\end{proposition}

\begin{proof}
The proof follows from combining \Cref{prop:coupling W and widetildeZ} and \Cref{lem:bad F}. The independence of the events $G_{B}^{\lambda}(W, \widetilde Z)$ follows from those of $E^{\lambda}_{\check D}(W, \widetilde Z)$ and ${F}_{B,+}^{\lambda}(\widetilde Z)$. The proof of \eqref{eq:prob of bad G} follows directly from \eqref{eq:prob of bad E} and \eqref{eq:prob of bad F}. The proofs of \eqref{eq:coupling widehatZD and WD} and \eqref{eq:coupling widehatZD' and WD'} are analogous to those of (5.12) and (5.13) in \cite{Szn17} (more specifically, the paragraphs containing (5.17)-(5.19)). Here the excursions $W_{\ell}^{D}$ and $W_{\ell}^{D'}$ in this work correspond to the random interlacements excursions $Z_{\ell}^{D}$ and $Z_{\ell}^{D'}$ in \cite{Szn17}, while the excursions $\widetilde{Z}_{\ell}^{D}$ and $ \widetilde{Z}_{\ell}^{D'}$ extracted from i.i.d.~excursions within the larger box correspond to excursions $\widehat Z_{\ell}^D$ and $\widehat Z_{\ell}^{D'}$ in the same reference. 
The coupling $\mathbb{Q}_{W,\widetilde{Z}}$ and the constants $\eta$, $\kappa$ and $\zeta$ in our statement match the coupling $\mathbb{Q}^{\mathcal{C}}$ and the constants $\delta$, $\kappa$ and $\widehat{\delta}$ respectively in \cite{Szn17}. 
In \eqref{eq:def of event G}, the events $G_B^{\lambda}(W, \widetilde Z), E^{\lambda}_{\check D}(W, \widetilde Z)$ and ${F}_{B,+}^{\lambda}(\widetilde Z)$ respectively correspond to the events $\widetilde{G}_{B}$, $\widetilde{U}_{\check{D}}^{m_{0}}$ and $\widetilde{U}_{\check{D}}^{m_{0}}\setminus \widetilde{G}_{B}$ in \cite[(5.6) and (5.10)]{Szn17}, and the bound $\ell\geq\frac{\lambda}{1-\eta}\cdot\text{cap}(D)$ corresponds to the bound $m\geq m_{0}/(1-\delta)$ in \cite{Szn17}.
\end{proof}

We then give the third coupling of the excursions $\widetilde{Z}_{\ell}^{D},\widetilde{Z}_{\ell}^{D'}$ with the excursions of random interlacements. Recall that we have defined $Z_{\ell}^{D,U},\,{\ell\geq 1}$ in \eqref{eq:def of interlacement excursions} as interlacement excursions from $D$ to $\partial U$. In accordance with our former notation, we simply write $Z_{\ell}^{D}$ in place of $Z_{\ell}^{D,U}$, and let $Z_{t}^{D}$ stand for $Z_{[t]}^{D}$ when $t\geq 1$.

\begin{proposition}\label{lem:coupling between widetildeZ and Z}
There exists a coupling $\mathbb{Q}_{\widetilde{Z},Z}$ between the law of $\big((n_{\check{D}_{x}}(0,t))_{t\geq 0},\{\widetilde{Z}_{\ell}^{\check{D}_{x}}\}_{\ell\geq 1}\big),x\in\mathcal{C}_{2}$ (see \Cref{prop:coupling W and widetildeZ}) and $\PP$, the law of random interlacements such that, for every $\lambda\in(0,\infty)$ and $\widehat{\zeta}\in(0,\frac{1}{2})$, there exist events $H_{B_{x}}^{\lambda }(\widetilde Z, Z)$ that are independent as $x$ varies over $\mathcal C_2$ and a constant $c_{10}(\widehat{\zeta})$ satisfying for each fixed box  $B=B_{x},x\in\mathcal{C}_{2}$, 
\begin{equation}\label{eq:prob of bad H}
\limsup_{N\to\infty}\frac{1}{\text{\rm cap}(\check{D})}\log\mathbb{Q}_{\widetilde{Z},Z}\left[H_B^{\lambda }(\widetilde Z, Z)^c\right] < -c_{11}(\lambda,\widehat{\zeta})<0.
\end{equation}
Moreover, under the event $H_B^{\lambda }(\widetilde Z, Z)$, for every $N\geq c_{10}(\widehat{\zeta})$ and $\ell\geq \lambda\cdot\text{\rm cap}(D)$, and any $B'$ neighbouring $B$ and its associated $D'$, we have 
\begin{equation}\label{eq:coupling widehatZD and ZD}
\left\{
\begin{array}{rll}
     {\rm (i)}&\;\; \left\{\widetilde{Z}_{1}^{D},\dots,\widetilde{Z}_{\ell}^{D}\right\}   \subseteq \left\{Z_{1}^{D},\dots,Z_{(1+\widehat{\zeta})\ell}^{D}\right\}; \\
     {\rm (ii)}&\;\;\Big\{Z_{1}^{D},\dots,Z_{\ell}^{D}\Big\}  \subseteq \left\{\widetilde{Z}_{1}^{D},\dots,\widetilde{Z}_{(1+\widehat{\zeta})\ell}^{D}\right\},
\end{array}
\right.\quad\text{and}
\end{equation}
\begin{equation}\label{eq:coupling widehatZD' and ZD'}  
    \left\{
\begin{array}{rll}
     {\rm (i)}&\;\; \left\{\widetilde{Z}_{1}^{D'},\dots,\widetilde{Z}_{\ell}^{D'}\right\} \subseteq \left\{Z_{1}^{D'},\dots,Z_{(1+\widehat{\zeta})\ell}^{D'}\right\};\\
     {\rm (ii)}&\;\;\left\{Z_{1}^{D'},\dots,Z_{\ell}^{D'}\right\} \subseteq \left\{\widetilde{Z}_{1}^{D'},\dots,\widetilde{Z}_{(1+\widehat{\zeta})\ell}^{D'}\right\}.
\end{array}
\right.
\end{equation}
\end{proposition}

\begin{proof}
\Cref{lem:coupling between widetildeZ and Z} is an adapted version of Proposition 5.1 in \cite{Szn17}, and we explain the adaptations needed here.

For fixed constants $\lambda\in(0,\infty)$ and $\widehat{\zeta}\in(0,\frac{1}{2})$, we replace $K$ in \cite[(3.7),(3.9)]{Szn17} by $\log N$ (see \eqref{eq:def of boxes at origin}), and set $m_{0}$ in \cite[Section 5]{Szn17} as $\lambda\cdot\text{cap}(\check{D})$ (which is a polynomial function of $L$). We fix constants $\eta,\kappa,\widehat{\delta}$  which satisfy the relation (5.14) in the reference, so that $\widehat{\delta}$ is equal to $\widehat{\zeta}$. After fixing all these constants, for any $B=B_{x},x\in\mathcal{C}_{2}$, we then adjust accordingly the definition of event $\widetilde{G}_{B}$ in (5.10) of the reference into the event $\widetilde H_B^{\lambda}(\widetilde Z, Z)$ here. 

We now choose our coupling $\mathbb{Q}_{\widetilde{Z},Z}$  and event $H^{\lambda }_B(\widetilde Z, Z)$ in \Cref{lem:coupling between widetildeZ and Z} as the adapted $\mathbb{Q}^{\mathcal{C}}$ and the adapted $\widetilde{G}_{B}$ (whose definition depends on $\lambda$). The independence of $\widetilde H_B^{\lambda}(\widetilde Z, Z)$ as $B=B_{x}$ varies follows from that of the events $\widetilde{G}_{B}$'s. The estimate \eqref{eq:prob of bad H} and inclusions \eqref{eq:coupling widehatZD and ZD}-\eqref{eq:coupling widehatZD' and ZD'} then follow respectively from the adapted (5.15) and (5.12)-(5.13) in \cite{Szn17}.
Indeed, it is not difficult to check that the minor adaptations do not affect the original proof. In particular, by the bounds (5.20), (5.27) and (5.30) in \cite{Szn17}, the probability that $\widetilde{G}_{B}$ fails decays exponentially in $m_{0}$, which, in our adapted version, means decaying exponentially in $-\text{cap}(\check{D})$ and streched-exponentially in $L$. This bound, which is stronger than the original bound (5.15) in the reference, will be taken as \eqref{eq:prob of bad H}. It is also worth mentioning that the adapted version of the condition (5.11) in \cite{Szn17} now holds as long as $N$ is sufficiently large, since $K$ has been substituted into $\log N$.
\end{proof}

\begin{remark}\label{rem:improve coupling lemma}
Note that in the statement of \Cref{lem:coupling between widetildeZ and Z}, we couple $\{Z_{\ell}^{D_{x}}\}_{\ell\geq 1},\{Z_{\ell}^{D_{x'}}\}_{\ell\geq 1}$ with $\{\widetilde{Z}_{\ell}^{D_{x}}\}_{\ell\geq 1},\{\widetilde{Z}_{\ell}^{D_{x'}}\}_{\ell\geq 1}$ \emph{simultaneously} for all $x\in\mathcal{C}_{2}$ (here $\bigcup_{x\in\mathcal{C}_{2}}\check{U}_{x}\subseteq\mathsf{B}$ is seen as a subset of $\mathbb{Z}^{d+1}$), in the sense that we require that ${H}_{B_x}^{\lambda}(\widetilde Z, Z)$'s are independent as $x$ varies. However, as we shall see, the independence property is unnecessary in the following proofs of this section, but will be useful in \Cref{sec:bad local time}; see the proof of \Cref{prop:irregular}.
\end{remark}

\subsection{Bounding the probability of a surface of bad boxes}\label{subsec:proof of prop bad(beta-gamma)}

In this section, we conclude the proof of \Cref{prop:bad(beta-gamma)}. Note that in view of \Cref{prop:coupling between widehatZ and W}, under the coupling $\mathbb{Q}_{W,\widetilde{Z}}$, if the events  ${G}_{B_x}^{\lambda}(W, \widetilde Z)$'s hold simultaneously for every $x\in\mathcal{C}_{2}$, then the estimate for the probability that $B_x$ is bad$(\beta, \gamma)$ for all $x\in \mathcal C_1$ in \eqref{eq:bad(beta-gamma)} can be expressed in terms of an intersection of independent events (which will then be defined as $\widehat{\text{bad}}(\widehat{\beta},\widehat{\gamma})$), each of which is characterized by the excursions $\widetilde{Z}_{\ell}^{D},\widetilde{Z}_{\ell}^{D'}$ for some $L$-neighbouring $D, D'$. 
The probability of this event can then be bounded from above using \Cref{lem:coupling between widetildeZ and Z} as well as the property of $\overline{u}$ in the strong percolative regime (see \eqref{eq:def of ubar}).

Since the excursions $\widetilde{Z}_{\ell}^{\check{D}_{x},\check{U}_{x}},\,{\ell\geq 1}$ have the same law for different $x$, it suffices to establish the bound for one arbitrary set of boxes. For this reason, we still write $B,D,B',D'$ in place of a specific choice $B_{x},D_{x},B_{x'},D_{x'}$ in the following. Let us first define the ``good'' event with respect of excursions $\widetilde{Z}_{\ell}^{D},\widetilde{Z}_{\ell}^{D'}$. Recall the existence and uniqueness events in \eqref{def:exist event} and \eqref{def:unique event}. 

\begin{definition}[Definition of $\widehat{\text{good}}$ boxes]\label{def:widehatgood}
For any two constants $0<\widehat{\gamma}<\widehat{\beta}$ and an $L$-box $B$ (which is seen as a box in $\mathbb{Z}^{d+1}$ here), we define $\widehat{\text{good}}(\widehat{\beta},\widehat{\gamma})$ event as the intersection of events $\text{Exist}(B, \widetilde{Z}, \widehat{\beta})$ and $\text{Unique}(B, \widetilde{Z}, \widehat{\beta}, \widehat{\gamma})$. 
When $B$ is not $\widehat{\text{good}}(\widehat{\beta},\widehat{\gamma})$, we say $B$ is $\widehat{\text{bad}}(\widehat{\beta},\widehat{\gamma})$.
\end{definition}
In the following, we will consider the complements of $\text{good}$ and $\widehat{\text{good}}$ events in turn with help of the coupling between $Z$ and $\widetilde{Z}$, with parameters changed each step.

We then bound the probability of $\widehat{\text{bad}}(\widehat{\beta},\widehat{\gamma})$ with the help of the coupling \Cref{lem:coupling between widetildeZ and Z}, where $\widehat{\beta}>\widehat{\gamma}$ are two arbitrary constants in $(0,\overline{u})$. Recall that $L=[N^{\psi}]$ for some fixed $1/d<\psi<1$. 

\begin{lemma}\label{lem:widehatgood(beta-gamma)}
Let $B$ be an $L$-box. For any $\widehat{\beta}>\widehat{\gamma}$ in $(0,\overline{u})$, there exists a constant $c_{12}=c_{12}(\psi,\widehat{\beta},\widehat{\gamma})$ such that, 
\begin{equation}\label{eq:prob of widehatbad}
\liminf_{N\to\infty}\frac{1}{\log N}\log\left(-\log\PP\left[B\text{ is }\widehat{\text{\rm bad}}(\widehat{\beta},\widehat{\gamma})\right]\right)> c_{12}>0.    
\end{equation}
\end{lemma}

\Cref{lem:widehatgood(beta-gamma)} is in some sense a baby version of \Cref{prop:bad(beta-gamma)}, whose idea is quite similar to that of Theorem 6.1 in \cite{Szn17}. 
\begin{proof}
Let us fix a sufficiently small positive $\lambda$ and $\widehat{\zeta}\in(0,1/10)$ such that 
\begin{equation}\label{eq:choice of hatzeta}
\lambda<\frac{\widehat{\gamma}}{10}, \quad \overline{u}>(1+\widehat{\zeta})\widehat{\beta}\quad\text{and}\quad\frac{\widehat{\beta}}{1+\widehat{\zeta}}>(1+\widehat{\zeta})\widehat{\gamma}.  
\end{equation}
Recall the coupling $\mathbb{Q}_{\widetilde{Z},Z}$ in \Cref{lem:coupling between widetildeZ and Z} between $\widetilde{Z}$ (i.i.d.~excursions) and $Z$ (interlacements excursions), on the event ${H}_{B}^{\lambda}(\widetilde Z, Z)$, by definitions \eqref{def:exist event}-\eqref{def:fail event}, we have
\begin{equation}\label{eq:fail widetildeZ to Z}
\begin{split}
\text{fail}_1\left(B,\widetilde{Z}, \widehat{\beta}\right)&\overset{\eqref{eq:coupling widehatZD and ZD}{\rm (i)} \text{ and } \eqref{eq:choice of hatzeta}}{\subseteq }\text{fail}_1\left(B,Z, \widehat \beta(1 +\widehat \zeta)\right);\\
\text{fail}_2\left(B,\widetilde Z, \widehat \beta, \widehat \gamma\right) &\overset{\eqref{eq:coupling widehatZD and ZD}, \eqref{eq:coupling widehatZD' and ZD'}(\rm ii)\text{ and }\eqref{eq:choice of hatzeta}}{\subseteq} \text{fail}_2\Big(B,Z, \frac{\widehat{\beta}}{1+\widehat \zeta}, \widehat \gamma(1+\widehat \zeta)\Big).
\end{split}
\end{equation}
Therefore, combining the definition of $\widehat{\text{bad}}$ event \Cref{def:widehatgood}, 
\begin{equation}\label{eq:bad to fail Z}
\PP\left[B\text{ is }\widehat{\text{bad}}(\widehat{\beta},\widehat{\gamma})\right]\leq \mathbb{Q}_{\widetilde{Z},Z}\left[H_{B}^{\lambda}(\widetilde{Z},Z)^c\right]+\PP\left[\text{fail}_1(B,Z, \widehat \beta(1 +\widehat \zeta))\right]+\PP\Big[\text{fail}_2(B,Z, \frac{\widehat \beta}{1+\widehat \zeta}, \widehat \gamma(1+\widehat \zeta))\Big].    
\end{equation}

Now by the same proof of \cite[Theorem 3.3]{Szn17}, we can obtain a stronger version of the original statement, that is, 
\begin{equation}\label{eq:prob of Fail1 widehatZ and Z}
\begin{split}
\liminf_{N\to\infty}\frac{1}{\log N}&\log\Big(-\log\PP\big[\text{fail}_1\big(B,Z, \widehat \beta(1 +\widehat \zeta)\big)\big]\Big)>c(\psi,\widehat \beta,\widehat \gamma)>0; \quad \\
\liminf_{N\to\infty}\frac{1}{\log N}&\log\Big(-\log\PP\Big[\text{fail}_2\Big(B,Z, \frac{\widehat \beta}{1+\widehat \zeta}, \widehat \gamma(1+\widehat \zeta) \Big)\Big]\Big)>c(\psi,\widehat \beta,\widehat \gamma)>0. 
\end{split}
\end{equation}
Indeed, we can use standard exponential Chebyshev's inequality for Poisson variables to bound the first three terms in (3.21) of \cite{Szn17}, and then use \eqref{eq:def of strongly percolate} to bound the last term in (3.21) of \cite{Szn17}, where our choice of $\widehat{\zeta}$ in \eqref{eq:choice of hatzeta} comes into play. Combining \eqref{eq:prob of bad H},\eqref{eq:bad to fail Z} and \eqref{eq:prob of Fail1 widehatZ and Z} then yields \eqref{eq:prob of widehatbad}.
\end{proof}

We now complete the proof of \Cref{prop:bad(beta-gamma)}, that is,  a $d$-dimensional coarse-grained surface of bad$(\beta, \gamma)$ boxes  exists only with very small probability.
\begin{proof}[Proof of \Cref{prop:bad(beta-gamma)}]
Since $\mathcal{C}_{2}$ is a subset of $\mathcal{C}_{1}$, to prove \eqref{eq:bad(beta-gamma)}, it suffices to prove 
\begin{equation}\label{eq:bad(beta-gamma) C2}
\lim_{N\to\infty}\frac{1}{|\mathcal{C}_{1}|\log N}\log P_{0}^{N}\Big[\bigcap_{x\in\mathcal{C}_{2}}\big\{B_x\text{ is bad}(\beta,\gamma)\big\}\Big]=-\infty.    
\end{equation}
For fixed  $\overline{u}-\delta<\gamma<\beta<\overline{u}$,
we choose positive $\eta=\eta(\beta,\gamma)$, $\kappa=\kappa(\beta,\gamma)$,  $\lambda=\lambda(\eta,\gamma)$ and $\zeta$ so that 
\begin{equation}\label{eq:choice of zeta}
\overline{u}>(1+\zeta)\beta,\quad\frac{\beta}{1+\zeta}>(1+\zeta)\gamma,\quad \text{and}\quad\lambda<\frac{(1-\eta)\gamma}{10}.
\end{equation}
In this way, the constants satisfy the requirements of \Cref{prop:coupling between widehatZ and W}.

Now if all the boxes $B_x$, $x\in\mathcal{C}_{2}$ are $\text{bad}(\beta,\gamma)$, then under the coupling $\mathbb{Q}_{W,\widetilde{Z}}$, there exists a subset $\mathcal C_3$ of $\mathcal{C}_{2}$ with cardinality $[|\mathcal{C}_{2}|/3]$, such that the $L$-boxes in $B_{x},{x\in \mathcal C_3}$ either all fail to satisfy ${G}_{B_x}^{\lambda}(W,\widetilde Z)$ (recall \eqref{eq:def of event G}), or are all $\text{bad}(\beta,\gamma)$ while fulfilling ${G}_{B_x}^{\lambda}(W,\widetilde Z)$. For a fixed $\mathcal C_3$, we denote 
\begin{equation}
    \text{bad}_1(\mathcal C_3) := \bigcap_{x\in \mathcal C_3}{G}_{B_x}^{\lambda}(W,\widetilde Z)^c, \quad \text{bad}_2(\mathcal C_3) := \bigcap_{x\in \mathcal C_3} \left(\{B_x \text{ is bad}(\beta, \gamma)\}\cap {G}_{B_x}^{\lambda}(W,\widetilde Z)\right).
\end{equation}
We bound the probabilities of the above two events respectively, and then use a union bound on $\mathcal C_3$ to conclude. 

We first deal with the event $\text{bad}_1(\mathcal C_3)$. By \Cref{prop:coupling between widehatZ and W}, 
combining the estimate \eqref{eq:prob of bad G}, we have for sufficiently large $N\geq N(\psi,\beta,\gamma,\lambda,\zeta)$, for each $B = B_x$, $x\in \mathcal C_3$,
\begin{equation}\label{eq:plug in prob of bad G}
\mathbb{Q}_{W,\widetilde{Z}}\left[{G}_{B}^{\lambda}(W,\widetilde Z)^{c}\right]\leq \exp\left(-c_{8}\text{cap}(\check{D})\right)\leq \exp\left(-cc_{8}L^{d-1}\right)=\exp\left(-cc_{8}N^{(d-1)\psi}\right).
\end{equation}
Note that under the coupling $\mathbb{Q}_{W,\widetilde{Z}}$, the events ${G}_{B_x}^{\lambda}(W,\widetilde Z)$ are independent as $x$ varies over $\mathcal{C}_{2}$. 
Therefore, it follows that for  sufficiently large $N$,
\begin{equation}\label{eq:prob of M1S}
\mathbb{Q}_{W,\widetilde{Z}}\big[\text{bad}_1(\mathcal C_3)\big]\leq \mathbb{Q}_{W,\widetilde{Z}}\left[{G}_{B}^{\lambda}(W,\widetilde Z)^{c}\right]^{|\mathcal C_3|}\overset{\eqref{eq:plug in prob of bad G}}{\leq} e^{-cc_{8}N^{(d-1)\psi}|\mathcal C_3|}\leq e^{-cc_{8}N^{(d-1)\psi}|\mathcal C_2|}. 
\end{equation}

We now turn to $\text{bad}_{2}(\mathcal C_3)$. Under the coupling $\mathbb{Q}_{W,\widetilde{Z}}$, for a fixed box $B=B_{x}, {x\in \mathcal C_3}$, on the event ${G}_{B}^{\lambda}(W,\widetilde Z)$, if $B$ is $\text{bad}(\beta,\gamma)$, then $\text{fail}_1(B,W, \beta)\cup \text{fail}_2(B,W, \beta, \gamma)$ happens. 
Following a similar analysis as the proof of \Cref{lem:widehatgood(beta-gamma)}, combining the coupling \eqref{eq:coupling widehatZD and WD}-\eqref{eq:coupling widehatZD' and WD'} between $\widetilde Z^D, \widetilde Z^{D'}$ and $W^D, W^{D'}$ and the definition of $\lambda$ in \eqref{eq:choice of zeta},
\begin{equation}
\begin{split}
    \text{fail}_1(B,W, \beta) &\overset{\eqref{eq:coupling widehatZD and WD}\text{(ii) and }\eqref{eq:choice of zeta}}{\subseteq }\text{fail}_1(B,\widetilde Z, \beta(1+\zeta)); \\ 
    \text{fail}_2(B,W, \beta, \gamma) &\overset{\eqref{eq:coupling widehatZD and WD},\eqref{eq:coupling widehatZD' and WD'}\text{(i) and }\eqref{eq:choice of zeta}}{\subseteq}  \text{fail}_2\Big(B,\widetilde{Z}, \frac{\beta}{(1 + \zeta)}, \gamma(1+\zeta)\Big).
\end{split}
\end{equation}
Therefore, if $B$ is $\text{bad}(\beta,\gamma)$, then $B$ is either  $\widehat{\text{bad}}(\beta(1+\zeta),\gamma)$ or $\widehat{\text{bad}}(\beta/(1+\zeta),\gamma(1+\zeta))$.
Furthermore, by \Cref{lem:widehatgood(beta-gamma)} and our choice of $\zeta$ in \eqref{eq:choice of zeta}, for sufficiently large $N\geq N(\psi,\beta,\gamma,\lambda,\zeta)$, we have (with $c_{12}'=c_{12}\big(\beta(1+\zeta),\gamma\big)\land c_{12}\big(\beta/(1+\zeta),\gamma(1+\zeta)\big)>0$)
\begin{equation}\label{eq:plug in prob of widehatbad}
\mathbb{Q}_{W,\widetilde{Z}}\left[B\text{ is }\widehat{\text{bad}}(\beta(1+\zeta),\gamma)\text{ or }\widehat{\text{bad}}\Big(\frac{\beta}{(1 + \zeta)},\gamma(1+\zeta)\Big)\right]\leq e^{-N^{c_{12}'}}.
\end{equation}
Therefore, using the independence between $\widetilde{Z}_{\ell}^{\check{D}_{x}}$ for different $x\in\mathcal{C}_{2}$, it follows that for sufficiently large $N$,
\begin{equation}\label{eq:prob of M2S}
\begin{split}
\mathbb{Q}_{W,\widetilde{Z}}[\text{bad}_{2}(\mathcal C_3)]&\leq \mathbb{Q}_{W,\widetilde{Z}}\Big[B\text{ is }\widehat{\text{bad}}(\beta(1+\zeta),\gamma)\text{ or }\widehat{\text{bad}}\Big(\frac{\beta}{(1 + \zeta)},\gamma(1+\zeta)\Big)\Big]^{|\mathcal C_3|}\\
&\leq e^{-N^{c_{12}'}|\mathcal C_3|}\leq e^{-N^{c_{12}'}|\mathcal C_2|}. 
\end{split}
\end{equation}

Combining \eqref{eq:prob of M1S} and \eqref{eq:prob of M2S} yields that for sufficiently large $N$, we have for some sufficiently small positive constant $c(\psi,\beta,\gamma,\lambda,\zeta)$,
\begin{equation}\label{eq:prob of bad(beta-gamma) C2}
\begin{split}
\quad P_{0}^{N}\Big[\bigcap_{x\in\mathcal{C}_{2}}\big\{B_x\text{ is bad}(\beta,\gamma)\big\}\Big]
&\leq \sum_{\mathcal C_3}\left(\mathbb{Q}_{W,\widetilde{Z}}[\text{bad}_{1}(\mathcal C_3)]+\mathbb{Q}_{W,\widetilde{Z}}[\text{bad}_{2}(\mathcal C_3)]\right)\\
&\leq 2^{|\mathcal{C}_{2}|}\left(e^{-cc_{8}N^{(d-1)\psi}|\mathcal{C}_{2}|}+e^{-cN^{c_{12}'}|\mathcal{C}_{2}|}\right)\\
&\leq e^{-N^{c(\psi,\beta,\gamma,\lambda,\zeta)}|\mathcal{C}_{2}|}.
\end{split}
\end{equation}
Plugging \eqref{eq:size of C2} into \eqref{eq:prob of bad(beta-gamma) C2} then gives the result.
\end{proof}

\section{Unlikeliness of surfaces of \texorpdfstring{$\text{poor}(\gamma)$}{large(gamma)} boxes}\label{sec:bad local time}
The main goal of this section is to prove \Cref{prop:bad local time}, that is, to control the probability that there exists a ``$d$-dimensional'' coarse-grained surface of poor$(\gamma)$ boxes for simple random walk. The adaptations to the biased walk case will be presented in \Cref{subsec:adapting bad local time}.

We first recall that $\mathsf{B}$ is a box with side-length $[N/\log^{3}N]$ on the cylinder $\mathbb{E}$ introduced in \Cref{prop:geometric argument}, and define $\mathsf{D}$ as the concentric box of $\mathsf{B}$ with side-length $[N/20]$. We write the successive times of return to $\mathsf{B}$ and departure from $\mathsf{D}$ as (recall \eqref{eq:def of return and departure} for notation) 
\begin{equation}\label{eq:def of return and departure for mathsfB}
R_{1}^{\mathsf{B},\mathsf{D}}<D_{1}^{\mathsf{B},\mathsf{D}}<R_{2}^{\mathsf{B},\mathsf{D}}<D_{2}^{\mathsf{B},\mathsf{D}}<\cdots<R_{k}^{\mathsf{B},\mathsf{D}}<D_{k}^{\mathsf{B},\mathsf{D}}<\cdots, 
\end{equation}
and write $R_{k}^{\mathsf{B}}$ and $D_{k}^{\mathsf{B}}$ as shorthand. Since this section only concerns the (recurrent) simple random walk on $\mathbb{E}$, the stopping times $R_{k}^{\mathsf{B}},D_{k}^{\mathsf{B}}, {k\geq 1}$ are all $P_{0}^{N}$-a.s.~finite. 
We then define the successive excursions from $\mathsf{B}$ to $\partial \mathsf{D}$ in the random walk $(X_{n})_{n\geq 0}$ as $W_{\ell}^{\mathsf{B}} = X_{[R_\ell^{\mathsf{B}},D_\ell^{\mathsf{B}})}$, $\ell \geq 1$ and use $Z_{\ell}^{\mathsf{B}} = Z_{\ell}^{\mathsf{B},\mathsf{D}}$, $\ell \geq 1$ for the excursions from $\mathsf{B}$ to $\partial\mathsf{D}$ (where $\mathsf{B}$ and $\mathsf{D}$ are also seen as subsets of $\mathbb{Z}^{d+1}$) of random interlacements. 
We denote by $N_{\underline{S}_{N}}(\mathsf{B})$ the number of excursions from $\mathsf{B}$ to $\partial \mathsf{D}$ in the trajectory of the simple random walk before time $\underline{S}_{N}$ as in \eqref{eq:def of N_SN}, and 
recall $N_u(\mathsf B) = N_{u}^{\mathsf{B},\mathsf{D}}$ of the number of excursions from $\mathsf{B}$ to $\partial\mathsf{D}$ in the random interlacements $\mathcal{I}^{u}$, as  defined
in \eqref{eq:number of excursions in Iu}.

The proof of \Cref{prop:bad local time} relies heavily on a stochastic domination control of the random walk excursions $W_{\ell}^{\mathsf{B},\mathsf{D}}$ in terms of the corresponding excursions $Z_{\ell}^{\mathsf{B},\mathsf{D}}$ of random interlacements (see \Cref{prop:very strong coupling}). Roughly speaking, the coupling says that, with extremely high probability, the random walk excursions completed before time $\underline{S}_{N}$ are contained in the excursions of random interlacements $\mathcal{I}^{u'}$, where $u'$ is a constant in $(\overline{u}-\delta,\gamma)$ (see \eqref{eq:def of u'}). We often refer to this coupling as the ``very strong'' coupling in this work.

Given this coupling, the proof of \Cref{prop:bad local time} is then reduced to proving a similar claim with regard to random interlacements, that is, the probability that the interlacements $\mathcal I^{u'}$ leave more than $\gamma\cdot\text{cap}(D)$ excursions in a $d$-dimensional subset of boxes $D_{x},\,x\in\mathcal{C}_{2}$ is extremely small; see \Cref{prop:bad local time interlacements}. In order to conclude \Cref{prop:bad local time interlacements}, we introduce the continuous-time random interlacements to benefit from the occupation-time bounds developed in \cite[Theorem 4.2]{Szn09b}. 
\Cref{def:regular theta} is introduced to rule out some atypical events on the exponential clocks when converting discrete-time random interlacements to its continuous-time counterpart, for which the events $\widehat{\text{regular}}(\gamma, \theta)$ and $\widehat{\text{irregular}}(\gamma, \theta)$ are introduced.

The proof of \Cref{prop:bad local time interlacements} is given in \Cref{subsec:continuous-time interlacements}, and can be derived once the following two cases are analyzed:
\begin{itemize}
    \item \Cref{prop:irregular}: the probability that the atypical event $\widehat{\text{irregular}}(\gamma, \theta)$ happens for a positive proportion of boxes on a $d$-dimensional coarse-grained surface is extremely small. Here we will use a similar techinique as in \Cref{sec:bad(beta-gamma)}, based on the ``decoupling" result \Cref{lem:coupling between widetildeZ and Z}.
    \item \Cref{prop:regular}: the probability that $\widehat{\text{regular}}(\gamma, \theta)$ holds but the occupation-time bound fails for a positive proportion of boxes on a $d$-dimensional coarse-grained surface is extremely small. Here we will use the occupation-time bounds in \cite[Section 4]{Szn17}.
\end{itemize}
Note that \Cref{subsec:continuous-time interlacements} only contains results on continuous-time random interlacements, and can be skipped upon first reading. 
In addition, we finally remind the readers that in this section, the $d$-dimensional set of base points $\mathcal C_1$ also satisfies the condition \eqref{eq:C1 disjoint}, which will be useful when analyzing the capacity of the union of boxes $\bigcup_{x\in\mathcal{C}_{1}}D_{x}$. 

\subsection{Reduction to the analysis of interlacements}\label{subsec:very strong coupling}

Recall that $\gamma$ is a fixed constant in $(\overline{u}-\delta,\overline{u})$, and $P_{0}^{N}$ and $\PP$ stands for the law of simple random walk starting from 0 and random interlacements respectively. We first state a stochastic domination between the excursions of random walk and interlacements, whose proof is postponed until \Cref{sec:couplings}, where a more general coupling result \Cref{thm:very strong coupling S} will be established and help us conclude the proof here.
\begin{proposition}\label{prop:very strong coupling}
For every $\delta >0$, define $u'$ in $(\overline{u}-\delta,\gamma)$ as
\begin{equation}\label{eq:def of u'}
u'=u'(\delta,\gamma)=\frac{1}{2}\left(\overline{u}-\delta+\gamma\right).    
\end{equation}
Then one can construct on an auxiliary space $(\underline{\Omega},\underline{\mathcal{F}})$ a coupling $\underline{Q}$ of the cylinder random walk and random interlacements with marginal distributions $P_{0}^{N}$ and $\PP$ respectively, such that there exists a positive constant $c_{13}=c_{13}(\overline{u},\delta,\gamma)$ satisfying 
\begin{equation}\label{eq:very strong coupling mathsfB}
\lim_{N\to\infty}\frac{1}{\text{cap}(\mathsf{B})}\log \underline{Q}\left[\{W_{\ell}^{\mathsf{B}}\}_{\ell\leq N_{\underline{S}_{N}}(\mathsf{B})} \nsubseteq \{Z_{\ell}^{\mathsf{B}}\}_{\ell\leq N_{u'}(\mathsf{B})}\right]< -c_{13}<0.  
\end{equation}
\end{proposition}

We then define the events $\widehat{\text{fine}}(\widehat{\gamma})$ and $\widehat{\text{poor}}(\widehat{\gamma})$, acting as the interlacements counterpart of \Cref{def:bad local time}. Recall that for two non-empty sets $D\subseteq U \subseteq \mathbb Z^{d+1}$, $N_{u}(D)$=$N_{u}^{D,U}$ denotes the number of excursions of $\mathcal I^u$ from $D$ to $\partial U$.

\begin{definition}[Definition of $\widehat{\text{fine}}(\widehat{\gamma})$]\label{def:fine box}
Recall the constant $u'$ defined in \eqref{eq:def of u'}. For each constant $\widehat{\gamma}>u'$, we say an $L$-box $B$, associated with concentric boxes $D$ and $U$, is $\widehat{\text{fine}}(\widehat{\gamma})$, if it satisfies $N_{u'}(D)\leq \widehat{\gamma}\cdot\text{cap}(D)$. In addition, when $B$ is not $\widehat{\text{fine}}(\widehat{\gamma})$, we say it is $\widehat{\text{poor}}(\widehat{\gamma})$.
\end{definition}

In view of \Cref{prop:very strong coupling}, the proof of \Cref{prop:bad local time} can be reduced to  the following proposition regarding interlacements.

\begin{proposition}\label{prop:bad local time interlacements}
Recall the box $\mathsf{B}$ (which is seen as a subset of $\mathbb{Z}^{d+1}$ here), the projection $\pi_{*}\in\{\pi_{i}\}_{1\leq i\leq d+1}$, and the set of base points $\mathcal{C}_{1}$ in \Cref{prop:bad local time}, then
\begin{equation}\label{eq:bad local time interlacements}
\lim_{N\to\infty}\frac{1}{|\mathcal{C}_{1}|\log N}\log \PP\Big[\bigcap_{x\in \mathcal C_1}\big\{B_{x}\text{ is }\widehat{\text{poor}}(\gamma)\big\}\Big]=-\infty.    
\end{equation}
\end{proposition}

We first complete the proof of \Cref{prop:bad local time} assuming \Cref{prop:bad local time interlacements}, while the proof for the latter proposition will be postponed to \Cref{subsec:continuous-time interlacements}. Here, we will rely on the crucial assumption $\psi>1/d$, which ensures that $\text{cap}(\mathsf{B})>|\mathcal{C}_{1}|\log N$.
\begin{proof}[Proof of \Cref{prop:bad local time} given \Cref{prop:very strong coupling,prop:bad local time interlacements}]
Since the side-length of $\mathsf{B}$ is $[N/\log^{3}N]$, by \eqref{eq:capacity of a box}, for some $c>0$ we have 
\begin{equation}\label{eq:capacity of mathsfB}
\text{cap}(\mathsf{B}) \geq c\cdot\left(\frac{N}{\log^{3}N}\right)^{d-1}.
\end{equation}
Then by \eqref{eq:size of C1} we have
\begin{equation}\label{eq:size of C1 in N}
|\mathcal{C}_{1}|=\left[\frac{1}{3}c_{4}\left(\frac{N}{L\log^{3}N}\right)^{d}\right]\leq \left[\frac{1}{3}c_{4}\cdot\frac{N^{d(1-\psi)}}{\log^{3d}N}\right]. 
\end{equation}
Recalling the assumption $\psi>1/d$, it follows that 
\begin{equation}\label{eq:compare C1 and capacity of mathsfB}
|\mathcal{C}_{1}|\log N=o(\text{cap}(\mathsf{B})).
\end{equation}
Therefore, by \Cref{prop:very strong coupling}, \eqref{eq:bad local time} is equivalent to
\begin{equation}\label{eq:from poor to widehatpoor}
\begin{aligned}
\lim_{N\to\infty}\frac{1}{|\mathcal{C}_{1}|\log N}\log \underline{Q}\Big[\bigcap_{x\in\mathcal{C}_{1}}\{B_{x}\text{ is poor}(\gamma)\}\cap \Big\{\{W_{\ell}^{\mathsf{B}}\}_{\ell\leq N_{\underline{S}_{N}}(\mathsf{B})} \subseteq \{Z_{\ell}^{\mathsf{B}}\}_{\ell\leq N_{u'}(\mathsf{B})}\Big\}\Big]=-\infty.   
\end{aligned}
\end{equation}
Note that the events in \eqref{eq:from poor to widehatpoor} imply that all $L$-boxes in $\{B_{x}\}_{x\in\mathcal{C}_{1}}$ are $\widehat{\text{poor}}(\gamma)$, we can conclude by employing \Cref{prop:bad local time interlacements}.
\end{proof}

\subsection{Continuous-time random interlacements}\label{subsec:continuous-time interlacements}
In this subsection we will use the occupation-time bounds on continuous-time random interlacements in \cite{Szn17} to prove \Cref{prop:bad local time interlacements}. This subsection only contains results on continuous-time interlacements, and is the only part of this paper that involves continuous-time random walks.

In this subsection, we always assume that all the simple random walks in random interlacements are continuous-time simple random walks with unit jump rate. More precisely, the random times between jumps of a walk are i.i.d.~exponential variables with expectation $1$, which does not affect the proof of \Cref{prop:bad local time interlacements}. Every random walk excursion will be parameterized via a continuous parameter (e.g.~the parameter $t$ in $Z_{\ell}^{D}(t)$). With a slight abuse of notation, we still keep all notation with respect to the discrete-time random interlacements (e.g., $\PP,N_{u}(D)$, $\widehat{\text{poor}}(\gamma)$).

Following the strategy of \cite{Szn17}, we first define the event concerning the exponential clocks. Recall that $e_{D}$ stands for the equilibrium measure of $D$ in \eqref{eq:equilibrium measure}, and $B$, $D$ and $U$ refer to the concentric boxes defined in \eqref{eq:def of translated boxes}.

\begin{definition}[Definition of $\widehat{\text{regular}}(\gamma,\theta)$]\label{def:regular theta}
For two positive constants $\gamma>\theta>0$, we say an $L$-box $B$ is $\widehat{\text{regular}}(\gamma,\theta)$, if
\begin{equation}
\sum_{1\leq\ell\leq \gamma\cdot\text{cap}(D)}\int_{0}^{T_{U}}e_{D}\left(Z_{\ell}^{D}(t)\right)\,dt\geq \theta\cdot\text{cap}(D).  
\end{equation}
When $B$ is not $\widehat{\text{regular}}(\gamma,\theta)$, we say $B$ is $\widehat{\text{irregular}}(\gamma,\theta)$.
\end{definition}

Note that in the above definition, $\int_{0}^{T_{U}}e_{D}\left(Z_{\ell}^{D}(t)\right)\,dt$ indicates the weighted version of the continuous local time of an excursion $Z_{\ell}^{D}$, and a box is $\widehat{\text{regular}}(\gamma,\theta)$ if it does not accumulate too much local time during the first $\gamma\cdot\text{cap}(D)$ excursions obtained from continuous-time interlacements.

To conclude \Cref{prop:bad local time interlacements}, we split into the following two results as explained in the introduction of this section. The two propositions will be proved in \Cref{subsubsec:irregular,subsubsec:regular} respectively. In the rest of this section, we fix $\theta\in(u',\gamma)$, and recall the sparse set $\mathcal{C}_{2}\subseteq \mathcal{C}_{1}$ as defined in \eqref{eq:sparsity of C2} and \eqref{eq:size of C2}. Additionally, by \eqref{eq:C1 disjoint}, we also add condition to $\mathcal{C}_{2}$ such that
\begin{equation}\label{eq:C2 disjoint}
\pi_{*}(x)\neq \pi_{*}(y),\quad \text{for all different }x,y\in\mathcal{C}_{2}.
\end{equation}

\begin{proposition}\label{prop:irregular}
For each subset $\mathcal C_3$ of $\mathcal{C}_{2}$ with $|\mathcal C_3|=[|\mathcal{C}_{2}|/3]$, 
\begin{equation}\label{eq:irregular}
\lim_{N\to\infty}\frac{1}{|\mathcal{C}_{2}|\log^{d+2}N}\log \PP\Big[\bigcap_{x\in \mathcal C_3}\big\{B_{x}\text{ is }\widehat{\text{irregular}}(\gamma,\theta)\big\}\Big]=-\infty. 
\end{equation}
\end{proposition}
\begin{proposition}\label{prop:regular}
For each subset $\mathcal C_3$ of $\mathcal{C}_{2}$ with $|\mathcal C_3|=[|\mathcal{C}_{2}|/3]$, 
\begin{equation}\label{eq:regular}
\lim_{N\to\infty}\frac{1}{|\mathcal{C}_{2}|\log^{d+2}N}\log \PP\Big[\bigcap_{x\in \mathcal C_3}\big\{B_{x}\text{ is }\widehat{\text{regular}}(\gamma,\theta)\text{ and } \widehat{\text{poor}}(\gamma)\big\}\Big]=-\infty. 
\end{equation}
\end{proposition}

\begin{proof}[Proof of \Cref{prop:bad local time interlacements} assuming \Cref{prop:irregular,prop:regular}]
Since $\mathcal{C}_{2}$ is a subset of $\mathcal{C}_{1}$, to prove \eqref{eq:bad local time interlacements}, it suffices to show
\begin{equation}\label{eq:bad local time C2}
\lim_{N\to\infty}\frac{1}{|\mathcal{C}_{1}|\log N}\log \PP\Big[\bigcap_{x\in\mathcal{C}_{2}}\big\{B_{x}\text{ is }\widehat{\text{poor}}(\gamma)\big\}\Big]=-\infty.   
\end{equation}

Now if all boxes $B_{x}, {x\in\mathcal{C}_{2}}$ are $\widehat{\text{poor}}(\gamma)$, then there exists a subset $\mathcal C_3$ of $\mathcal{C}_{2}$ with cardinality $[|\mathcal{C}_{2}|/3]$, such that the boxes $\{B_{x}\}_{x\in \mathcal C_3}$ are either all $\widehat{\text{irregular}}(\gamma,\theta)$, or all $\widehat{\text{poor}}(\gamma)$ and $\widehat{\text{regular}}(\gamma,\theta)$. 
Similar to in the proof of \Cref{prop:bad(beta-gamma)}, for a given $\mathcal C_3$, we denote by
\begin{equation}
    \widehat{\text{poor}}_{1}(\mathcal C_3) = \hspace{-0.3em} \bigcap_{x\in \mathcal C_3}\hspace{-0.3em}\big\{B_x \text{ is }\widehat{\text{irregular}}(\gamma,\theta)\};\   \widehat{\text{poor}}_{2}(\mathcal C_3) = \hspace{-0.3em} \bigcap_{x\in \mathcal C_3}\hspace{-0.3em}\big\{B_x \text{ is }\widehat{\text{poor}}(\gamma) \text{ and }\widehat{\text{regular}}(\gamma,\theta)\big\}.
\end{equation}
With \Cref{prop:regular,prop:irregular}, for any $C>0$, there exists a large $C'(C,\gamma, \theta)>0$ such that for all $N\geq C'(C,\gamma, \theta)$,
\begin{equation}\label{eq:prob of M1 and M2}
\PP[\widehat{\text{poor}}_{1}(\mathcal C_3)]\leq \exp\left(-C|\mathcal{C}_{2}|\log^{d+2}N\right)\quad \text{and}\quad \PP[\widehat{\text{poor}}_{2}(\mathcal C_3)]\leq \exp\left(-C|\mathcal{C}_{2}|\log^{d+2}N\right).
\end{equation}
Using a union bound on all possible choices of $\mathcal C_3$, we have
\begin{equation}\label{eq:prob of bad local time interlacements C2}
\begin{split}
\quad \PP\Big[\bigcap_{x\in\mathcal{C}_{2}}\{B_{x}\text{ is }\widehat{\text{poor}}(\gamma)\}\Big]
&\leq \sum_{\mathcal C_3}\left(\PP[\widehat{\text{poor}}_{1}(\mathcal C_3)]+\PP[\widehat{\text{poor}}_{2}(\mathcal C_3)]\right)\\
&\leq 2^{|\mathcal{C}_{2}|}\cdot  2\exp\left(-C|\mathcal{C}_{2}|\log^{d+2}N\right)
\leq \exp\left(-cC|\mathcal{C}_{2}|\log^{d+2}N\right).
\end{split}
\end{equation}
The limit \eqref{eq:bad local time C2} then follows from plugging \eqref{eq:size of C2} into \eqref{eq:prob of bad local time interlacements C2} and sending $C$ to infinity.
\end{proof}

\subsubsection{Unlikeliness of surfaces of \texorpdfstring{$\widehat{\text{irregular}}(\gamma, \theta)$}{irr} boxes (Proof of \Cref{prop:irregular})}\label{subsubsec:irregular}
The proof of \Cref{prop:irregular} is similar to the proof of \Cref{prop:bad(beta-gamma)} in \Cref{sec:bad(beta-gamma)}. 
With the ``decoupling" result \Cref{lem:coupling between widetildeZ and Z} in hand, the estimation of the former event can be transformed into that of the occurrence of a series of i.i.d.~events, that is, we can focus on the probability of \emph{one} set of concentric boxes $B\subseteq D\subseteq U$ such that $B$ is $\widehat{\text{irregular}}(\gamma, \theta)$; see \Cref{lem:irregular}.

\begin{lemma}\label{lem:irregular}
Let $B=B_x, x\in \mathcal C_3$ be an $L$-box. For any $\theta\in(0,\gamma)$, there exists a positive constant $c_{14}=c_{14}(\psi,\gamma,\theta)>0$ (recall that $L=[N^{\psi}]$, $1/d<\psi<1$) such that,
\begin{equation}
\liminf_{N\to\infty}\frac{1}{\log N}\log\left(-\log\PP\left[B\text{ \rm is }\widehat{\text{\rm irregular}}(\gamma,\theta)\right]\right)> c_{14}>0.
\end{equation}
\end{lemma}
\begin{proof}[Proof of \Cref{lem:irregular}]
This lemma is proved in \cite[Theorem 3.3]{Szn17}; see the argument between (3.26) and (3.27) in that work for details. The series of concentric boxes $B\subseteq D\subseteq\cdots\subseteq U$ in \eqref{eq:def of translated boxes} correspond to the series of boxes with the same notation in \cite[(3.10)]{Szn17}, and a different choice of $L$ and replacing $K$ in the reference by $\log N$ here do not affect the proof. Our constants $\gamma>\theta$ play the role of $\beta>\gamma$ in \cite[Theorem 3.3]{Szn17}, and the event $\widehat{\text{regular}}(\gamma,\theta)$ corresponds to the event defined in (3.13) within the definition of a $\text{good}(\alpha,\beta,\gamma)$ box in the reference.
\end{proof}

We remark here that in this subsection, \Cref{lem:coupling between widetildeZ and Z} serves the same ``decoupling" purpose as \Cref{prop:coupling between widehatZ and W} in \Cref{sec:bad(beta-gamma)}. Similarly, \Cref{lem:irregular} now plays the role of \Cref{lem:widehatgood(beta-gamma)}, that is, the estimate of the probability for one box to be $\widehat{\text{bad}}(\widehat{\beta}, \widehat{\gamma})$. With the two results above, we can proceed to the remaining part of the proof for \Cref{prop:irregular}. 
\begin{proof}[Proof of \Cref{prop:irregular}]
Given $0<\theta<\gamma$, we choose $\widehat{\zeta}=\widehat{\zeta}(\gamma,\theta)$ and $\lambda=\lambda(\gamma,\theta)$  sufficiently small such that 
\begin{equation}\label{eq:choice of zeta irregular}
\frac{\gamma}{(1+\widehat\zeta)^{2}}>\theta, \quad\text{and}\quad\lambda <\frac{\theta}{10}.
\end{equation}

Now if all boxes in $\{B_{x}\}_{x\in \mathcal C_3}$ are $\widehat{\text{irregular}}(\gamma,\theta)$, then there exists a subset $\mathcal C_4$ of $\mathcal C_3$ with cardinality $[|\mathcal C_3|/3]$ (which is approximately $[|\mathcal{C}_{2}|/9]$), such that the $L$-boxes in $\{B_{x}\}_{x\in \mathcal C_4}$ either fail to satisfy $\widetilde{H}_{B_{x}}^{\lambda}(\widetilde{Z},Z)$,  or are all $\widehat{\text{irregular}}(\gamma,\theta)$ while satisfying $\widetilde{H}_{B_{x}}^{\lambda}(\widetilde{Z},Z)$. Again we define for a fixed set $\mathcal C_4$
\begin{equation}
    \widehat{\text{irregular}}_{1}(\mathcal C_4) = \hspace{-0.3em}\bigcap_{x\in\mathcal C_4}\hspace{-0.3em}\widetilde{H}_{B_x}^{\lambda}(\widetilde{Z},Z)^c;\ \widehat{\text{irregular}}_{2}(\mathcal C_4) = \hspace{-0.5em}\bigcap_{x\in\mathcal C_4}\hspace{-0.3em}\Big(\{B_x \text{ is }\widehat{\text{irregular}}(\gamma,\theta)\}\cap \widetilde{H}_{B_x}^{\lambda}(\widetilde{Z},Z)\Big).
\end{equation}
Again, it suffices to bound the two probabilities from above and take a union bound on all the possible choices of $\mathcal C_4$. 

We first bound $\widehat{\text{irregular}}_{1}(\mathcal C_4)$. According to \eqref{eq:prob of bad H}, for sufficiently large $N\geq C(\psi,\gamma,\theta,\lambda,\widehat{\zeta})(\geq c_{9}(\widehat{\zeta}))$, we have for each $B = B_x$, $x \in \mathcal C_4$, 
\begin{equation}\label{eq:plug in prob of bad H}
\mathbb{Q}_{\widetilde{Z},Z}\left[\widetilde{H}_{B}^{\lambda}(\widetilde{Z},Z)^{c}\right]\overset{\eqref{eq:prob of bad H}}{\leq} \exp\left(-c_{10}\text{cap}(\check{D})\right)\leq \exp\left(-cc_{10}L^{d-1}\right)=\exp\left(-cc_{10}N^{(d-1)\psi}\right).
\end{equation}
By \Cref{lem:coupling between widetildeZ and Z}, under the coupling $\mathbb{Q}_{\widetilde{Z},Z}$, the events $\widetilde{H}_{B_{x}}^{\lambda}(\widetilde{Z},Z)$ are independent as $x$ varies over $\mathcal{C}_{2}$. Therefore, for sufficiently large $N$, by \eqref{eq:plug in prob of bad H},
\begin{equation}\label{eq:prob of M1S1}
\begin{split}
    \mathbb{Q}_{\widetilde{Z},Z}\Big[\widehat{\text{irregular}}_{1}(\mathcal C_4)\Big]
    \leq \mathbb{Q}_{\widetilde{Z},Z}\left[\widetilde{H}_{B}^{\lambda}(\widetilde{Z},Z)^{c}\right]^{|\mathcal C_4|}\leq e^{-cc_{10}N^{(d-1)\psi}|\mathcal C_4|}\leq e^{-cc_{10}N^{(d-1)\psi}|\mathcal C_2|}. 
\end{split}
\end{equation}

Next, we bound $\widehat{\text{irregular}}_{2}(\mathcal C_4)$. Under the coupling $\mathbb{Q}_{\widetilde{Z},Z}$ and on the event $\widetilde{H}_{B_x}^{\lambda}(\widetilde{Z},Z)$, the coupling \eqref{eq:coupling widehatZD and ZD}(i) between $Z_{\ell}^{D_x},\ell\geq 1$ and $\widetilde Z_{\ell}^{D_x},\ell\geq 1$ allows us to transform the event $\widehat{\text{irregular}}(\gamma,\theta)$ using the excursions $\widetilde{Z}_{\ell}^{D_x},{\ell\geq 1}$, where the parameter $\gamma$ is replaced by $\gamma/(1+\widehat{\zeta})$ while $\theta$ remains unchanged. 
After this transformation, thanks to the independence of $\widetilde{Z}_{\ell}^{D_{x}},{\ell\geq 1}$ as $x$ varies over $\mathcal C_4$, we obtain $|\mathcal C_4|$ independent events. 

For each of these independent events, we then apply \eqref{eq:coupling widehatZD and ZD}(ii) to convert it back to the excursions of random interlacements $Z_{\ell}^{D_x},\ell\geq 1$, transforming the event into $\widehat{\text{irregular}}\big(\gamma/(1+\widehat{\zeta})^{2},\theta\big)$.
Combining the above two steps, by \Cref{lem:irregular} and our choice of $\widehat \zeta$ in \eqref{eq:choice of zeta irregular}, for sufficiently large $N\geq C'(\psi,\gamma,\theta,\lambda,\widehat{\zeta})$ and $c_{14}'=c_{14}\big(\gamma/(1+\widehat{\zeta})^{2},\theta\big)>0$, we have 
\begin{equation}\label{eq:prob of M2S1}
\begin{split}
\mathbb{Q}_{\widetilde{Z},Z}\Big[\widehat{\text{irregular}}_{2}(\mathcal C_4)\Big]\leq \mathbb{Q}_{\widetilde{Z},Z}\left[B\text{ is }\widehat{\text{irregular}}(\gamma/(1+\widehat{\zeta})^{2},\theta)\right]^{|\mathcal C_4|}\leq e^{-N^{c_{14}'}|\mathcal C_4|}\leq e^{-N^{c_{14}'}|\mathcal C_2|}.
\end{split}
\end{equation}

Combining \eqref{eq:prob of M1S1} and \eqref{eq:prob of M2S1} yields that for sufficiently large $N$, we have for some sufficient small $c(\psi,\gamma,\theta,\lambda,\widehat{\zeta})>0$,
\begin{equation}\label{eq:prob of irregular C2}
\begin{split}
\quad P_{0}^{N}\Big[\bigcap_{x\in \mathcal C_3}\{B_{x}\text{ is }\widehat{\text{irregular}}(\gamma,\theta)\}\Big]
&\leq \sum_{\mathcal C_4}\left(\mathbb{Q}_{\widetilde{Z},Z}[\widehat{\text{irregular}}_1(\mathcal C_4)]+\mathbb{Q}_{\widetilde{Z},Z}[\widehat{\text{irregular}}_{2}(\mathcal C_4)]\right)\\
&\leq 2^{|\mathcal{C}_{2}|}\Big(e^{-cc_{10}N^{(d-1)\psi}|\mathcal{C}_{2}|}+e^{-cN^{c_{14}'}|\mathcal{C}_{2}|}\Big)\\
&\leq e^{-N^{c(\psi,\gamma,\theta,\lambda,\zeta)}|\mathcal{C}_{2}|},
\end{split}
\end{equation}
and the conclusion follows.
\end{proof}

\subsubsection{Unlikeness of surfaces of \texorpdfstring{$\widehat{\text{poor}}(\gamma)$}{} boxes (Proof of \Cref{prop:regular})}\label{subsubsec:regular}

The proof of \Cref{prop:regular} involves the occupation-time bounds on continuous-time interlacements. Recall that $D(\mathcal C_3)$ denotes the union of boxes $D_x$ for $x\in \mathcal C_3$ (see \eqref{eq:def of checkD and checkU C2}).
\begin{proposition}\label{prop:occupation-time bounds}
Recall $u'<\theta<\gamma$ in $(\overline{u}-\delta,\overline{u})$ and the sets $\mathcal C_3 \subseteq\mathcal{C}_{2}$ in \Cref{prop:regular}. 
There exists a positive constant $c_{15}=c_{15}(\overline{u},\delta,\gamma,\theta)$ such that 
\begin{equation}\label{eq:occupation-time bounds}
\limsup_{N\to\infty}\frac{1}{\text{cap}(D(\mathcal C_3))}\log\PP\Big[\bigcap_{x\in \mathcal C_3}\big\{B_x\text{ is }\widehat{\text{poor}}(\gamma)\text{ and }\widehat{\text{regular}}(\gamma,\theta)\big\}\Big]<-c_{15}<0.
\end{equation}
\end{proposition}
\begin{proof}
This is proved in \cite[Theorem 4.2]{Szn17}. The series of concentric boxes $B\subseteq D\subseteq\cdots\subseteq U$ in \eqref{eq:def of translated boxes} corresponds to the series of boxes with the same notation in \cite[(3.10)]{Szn17}, and a different choice of $L$ and replacing $K$ in the reference by $\log N$ here do not affect the proof. The set $\mathcal{C}_{3}$ corresponds to the set $\mathcal{C}$ in the reference, and they are both sparse in the sense that \eqref{eq:sparsity of C2} and \eqref{eq:sparsity of boxes in mathcalC2} have the same effect as (4.1) and (4.2) in \cite{Szn17}. Our constants $\gamma,\theta,u'$ here correspond to the constants $\beta,\gamma,u$ respectively. The events $\widehat{\text{regular}}(\gamma,\theta)$ and $\widehat{\text{poor}}(\gamma)$ play the role of $\text{good}(\alpha,\beta,\gamma)$ and $N_{u}(D_{z})\geq \beta\text{cap}(D)$ in the reference respectively. 
\end{proof}

With this result, it suffices to compare between the capacity of the union of boxes $D(\mathcal C_3)$ and the cardinality of $\mathcal C_3$, which is proved using the variational characterization of capacity; see \eqref{eq:variational characterization}.
\begin{lemma}\label{lem:capacity of DS}
There exists a positive constant $c_{16}=c_{16}(d)$ such that for each $\mathcal C_3$ defined in \Cref{prop:regular}, 
\begin{equation}\label{eq:capacity of DS}
\liminf_{N\to\infty}\frac{\text{\rm cap}(D(\mathcal C_3))}{|\mathcal C_3|^{\frac{d-1}{d}}L^{d-1}}> c_{16}>0.    
\end{equation}
\end{lemma}
\begin{proof}
Recall that in this section we assume $\mathcal{C}_{1}$ satisfies \eqref{eq:C1 disjoint}, with $\pi_{*}\in\{\pi_{i}\}_{1\leq i\leq d+1}$ denoting the projection. Let $e_{*}$ be the direction vector for the projection $\pi_{*}$. For each box $D_{x}$ with $x\in \mathcal C_3$, let $D_{x,*}$ be one of its two $d$-dimensional boundary faces orthogonal to $e_{*}$, and write $D_{*}(\mathcal C_3)$ as the union of $D_{x,*},x\in \mathcal C_3$. Then by \eqref{eq:def of boxes at origin}, \eqref{eq:C1 disjoint} and \eqref{eq:sparsity of C2}, 
\begin{equation}\label{eq:D*S disjoint}
\pi_{*}(x)\neq \pi_{*}(y),\quad \text{for all different }x,y\in D_{*}(\mathcal C_3).    
\end{equation}

We then use the variational characterization of capacity (see \eqref{eq:variational characterization}) to bound $\text{cap}(D(\mathcal C_3))$. Let $m=2L\cdot\big[|\mathcal C_3|^{\frac{1}{d}}\big]$. By \eqref{eq:D*S disjoint} and \eqref{eq:decay of green function} (an estimate on Green's function), it holds that 
\begin{equation}\label{eq:uniform lower bound on green function}
\max_{x\in D_{*}(\mathcal C_3)}\sum_{y\in D_{*}(\mathcal C_3)}g(x,y)\leq c\sum_{k=1}^{m}k^{1-d}\cdot k^{d-1}\leq cm.    
\end{equation} 
We then take $\nu$ as the uniform measure on $D_{*}(\mathcal C_3)$, and it follows from \eqref{eq:uniform lower bound on green function} that
\begin{equation}\label{eq:lower bound on Enu}
\begin{split}
    E(\nu)&=\sum_{x,y\in D_{*}(\mathcal C_3)}\nu(x)\nu(y)g(x,y)\\
    &\leq \frac{1}{|D_{*}(\mathcal C_3)|}\max_{x\in D_{*}(\mathcal C_3)}\sum_{y\in D_{*}(\mathcal C_3)}g(x,y)\leq \frac{cm}{|D_{*}(\mathcal C_3)|}\leq \frac{c}{|\mathcal C_3|^{\frac{d-1}{d}}L^{d-1}}.
\end{split}
\end{equation}
Plugging \eqref{eq:lower bound on Enu} into \eqref{eq:variational characterization} yields \eqref{eq:capacity of DS}.
\end{proof}

\begin{proof}[Proof of \Cref{prop:regular}]
Comparing \eqref{eq:regular} with \eqref{eq:occupation-time bounds}, it suffices to show that the order of $\text{cap}(D(\mathcal C_3))$ is larger than that of $|\mathcal{C}_{2}|\log^{d+2}N$. 
Now since $\mathcal{C}_{2}$ is a subset of $\mathcal{C}_{1}$, by \eqref{eq:C1 disjoint}, we have
\begin{equation}\label{eq:order of C2 in comparison}
|\mathcal{C}_{2}|\log^{d+2}N\leq C\cdot\left(\frac{N}{L}\right)^{d}\cdot(\log N)^C.    
\end{equation}
Combining \eqref{eq:size of C2}, \Cref{lem:capacity of DS} and $|\mathcal C_3|=[|\mathcal{C}_{2}|/3]$, we have for some $C'=C'(c_{16})>0$,
\begin{equation}\label{order of D(S) in comparison}
\text{cap}(D(\mathcal C_3)) \geq N^{d-1}\cdot(\log N)^{-C'}.
\end{equation}
The conclusion then follows from the assumption $L=[N^{\psi}],\psi\in(1/d,1)$.
\end{proof}

\section{Bounding \texorpdfstring{$T_{N}$}{TN} by \texorpdfstring{$\underline{S}_{N}$}{SN} in the biased walk case}\label{sec:lower bound in biased case}

The main goal of this section is to prove \Cref{prop:underlineS and T biased}, that is, for biased walk with upward drift $N^{-d\alpha}$, $\alpha\in(1/d,\infty)$, disconnection happens after time $\underline S_N$ with high probability. The strategy is identical to that of the simple random walk case, and we will adapt the arguments in \Cref{sec:Geometric,sec:bad(beta-gamma),sec:bad local time} respectively in \Cref{subsec:adapting geometric,subsec:adapting bad(beta-gamma),subsec:adapting bad local time}. Since most of the techniques involved remain the same, we will not lay down full details, but rather focus on necessary adaptations of the proof.
Our main result of this section is as follows.
\begin{proposition}\label{prop:underlineS and T biased}
For every $\alpha\in(1/d,\infty)$ and $\delta>0$, we have 
\begin{equation}\label{eq:underlineS and T biased}
\lim_{N\to \infty}P_{0}^{N,\alpha}[T_{N}\geq \underline{S}_{N}]=1.
\end{equation}
\end{proposition}

\subsection{The geometric argument (Adaptation of \Cref{sec:Geometric})} \label{subsec:adapting geometric}

In this subsection we adapt the geometric argument established in \Cref{sec:Geometric}. 

We still consider the coarse-grained boxes $B_{x}, D_{x},\check{D}_{x},U_{x},\check{U}_{x}$ defined in \eqref{eq:def of boxes at origin} and \eqref{eq:def of translated boxes} with $L=[N^{\psi}]$, which are respectively referred to as $B,D,\check{D},U,\check{U}$ when no confusion arises. In the biased walk case, we further impose an additional requirement on $L$ such that $\psi\in(1/d,\alpha\land 1)$. We remark that the condition $\psi>1/d$ is again used to control the combinatorial complexity of $d$-dimensional bad boxes (see the remark after \eqref{eq:def of translated boxes}), while the condition $\psi<\alpha\land 1$ is required to beat some union bounds (see \eqref{eq:number of C1 bias}) and to facilitate some Radon-Nikodym derivatives estimates (see the proofs of \Cref{prop:coupling W and widetildeW,lem:coupling widetildeW with widetildeZ}).

The analysis of successive random walk excursions between concentric boxes is a core tool of the proof. As the biased random walks on the cylinder in consideration is no longer recurrent, we introduce a sequence of auxiliary biased walks to create infinitely many random walk excursions. Denote the original random walk $(X_{n})_{n\geq 0}$ started at $0$ as $X^{0}$, and let $\{X^{k}\}_{k\geq 1}$ be a sequence of i.i.d~biased random walk with law $P^{N,\alpha}$, each starting from a uniformly random location on $\mathbb{T}\times\{-2M(\alpha)\}$, where 
\begin{equation}\label{eq:def of large constant Calpha}
M(\alpha):=
\begin{cases}
N^{10d},\quad &\alpha\in[1,\infty),\\
\exp\left(\overline{u}\cdot N^{d(1-\alpha)}\right), \quad &\alpha\in(1/d,1).
\end{cases}
\end{equation}
Then by \Cref{prop:asymptotics_alpha>1,prop:asymptotics_alpha=1,prop:asymptotics_alpha<1}, under $P_{0}^{N,\alpha}$, with high probability, $\underline{S}_{N}\leq M(\alpha)/2$.
For simplicity, we still write the product measure of the probability measure with respect to $X^{0}$ and $X^{k},k\geq 1$, as $P_{0}^{N,\alpha}$. 

Given an $L$-box $B$, for every biased random walk $X^{k}$, we define the successive times of return to $D$ and departure from $U$ as in \eqref{eq:def of return and departure for boxes}. However, under $P_{0}^{N,\alpha}$, every walk $X^{k}$ almost surely only contains finitely many excursions from $D$ to $\partial U$. 
By increasing $k$ from zero to infinity, we generate an infinite sequence of excursions  which will be ranked first by the value of $k$, and then by their order of appearance within the trajectory $X^{k}$.  With this criterion, we still write these excursions as $W_{\ell}^{D,U}, {\ell\geq 1}$, and keep the notation $W_{\ell}^{D}$ and $W_{t}^{D}$ (see \eqref{eq:def of simple random walk excursions} and below).

For two fixed constants $\beta>\gamma$ in $(\overline{u}-\delta,\overline{u})$, we retain the definition of $\text{good}(\beta,\gamma)$, $\text{bad}(\beta,\gamma)$, $\text{fine}(\gamma)$, $\text{poor}(\gamma)$, $\text{normal}(\beta,\gamma)$ and $\text{abnormal}(\beta,\gamma)$ boxes defined in \Cref{def:good(beta-gamma),def:bad local time,def:bad boxes}.
Note that here the definition of $\text{good}(\beta,\gamma)$ boxes now involves the artificial walks $X^{k}, {k\geq 1}$, while the definition of $\text{fine}(\gamma)$ only concerns the number of excursion from $D$ to $\partial U$ in the original excursion $X^{0}$ before time $\underline{S}_{N}$. 
In this fashion, \Cref{lem:good connectivity} still holds, and we now state the adapted version of the geometric argument \Cref{prop:geometric argument}. 
\begin{proposition}\label{prop:geometric argument bias}
For all $N\geq c_{3}(\psi)>0$ (where $c_{3}(\psi)$ is the same as in \Cref{prop:geometric argument}), on the event $\{T_{N}\leq\underline{S}_{N}\}$, there exists a box $\mathsf{B}$ with side-length $[N/\log^{3}N]$ and $\mathcal C$ which is a subset of $\mathbb E$ such that for some $\pi_{*}\in\{\pi_{i}\}_{1\leq i\leq d+1}$, \eqref{eq:large distance from mathsfB}-\eqref{eq:bad boxes property} hold. 
Moreover, on the event $\{T_{N}\leq\underline{S}_{N}\leq M(\alpha)/2\}$, we further have
\begin{equation}\label{eq:mathsf B range}
    \pi_{\mathbb{Z}}(\mathsf{B})\subseteq[-M(\alpha),M(\alpha)].
\end{equation}
\end{proposition}

The proof of this proposition remains entirely the same. Then in light of this new geometric argument, the proof of \Cref{prop:underlineS and T biased} can be reduced to two similar results as \Cref{prop:bad(beta-gamma),prop:bad local time}. That is, it suffices to prove that for $\mathsf B$, $\pi_{*}$ in \Cref{prop:geometric argument bias} and $\mathcal{C}_{1}\subseteq\mathcal{C}$ satisfying \eqref{eq:size of C1} and \eqref{eq:C1 disjoint} respectively,
\begin{align}
\lim_{N\to\infty}\frac{1}{|\mathcal{C}_{1}|\log N}&\log P_{0}^{N,\alpha}\Big[\bigcap_{x\in\mathcal{C}_{1}}\{B_{x}\text{ is bad}(\beta,\gamma)\}\Big]=-\infty; \label{eq:bad(beta-gamma) bias}\\
\lim_{N\to\infty}\frac{1}{|\mathcal{C}_{1}|\log N}&\log P_{0}^{N,\alpha}\Big[\bigcap_{x\in\mathcal{C}_{1}}\{B_{x}\text{ is poor}(\gamma)\}\Big]=-\infty.  \label{eq:bad local time bias} 
\end{align}

These two equations will be respectively proved in \Cref{subsec:adapting bad(beta-gamma),subsec:adapting bad local time}, employing the same method as in \Cref{sec:bad(beta-gamma),sec:bad local time}. Note that the condition $\psi>1/d$ is only used in proving \eqref{eq:bad local time bias} to ensure that the denominator $|\mathcal{C}_{1}|\log N$ is smaller than $\text{cap}(\mathsf{B})$, similarly as in \eqref{eq:compare C1 and capacity of mathsfB}.

\begin{proof}[Proof of \Cref{prop:underlineS and T biased} assuming \eqref{eq:bad(beta-gamma) bias} and \eqref{eq:bad local time bias}]
Similar to the proof of \Cref{prop:underlineS and T} assuming \Cref{prop:bad(beta-gamma),prop:bad local time}, by replacing $N^{5d}$ with $M(\alpha)/2$ in \eqref{eq:underlineS and T and N5d}, it suffices to prove that 
\begin{equation}\label{eq:underlineS and T and Malpha}
\lim_{N\to\infty}P_0^{N,\alpha}[T_{N}\leq\underline{S}_{N}\leq M(\alpha)/2]=0.   
\end{equation} 

To achieve this, we use a union bound as in \eqref{eq:union bound on C1}. All the possible ways of selecting a $d$-dimensional set $\mathcal{C}_{1}$ that satisfies \eqref{eq:size of C1} or \eqref{eq:C1 disjoint} and is contained in $\mathsf{B}$ satisfying \eqref{eq:large distance from mathsfB}-\eqref{eq:bad boxes property} and \eqref{eq:mathsf B range} is bounded by
\begin{equation}\label{eq:number of C1 bias}
C\cdot M(\alpha)\cdot N^{(d+1)|\mathcal{C}_{1}|} = Ce^{\overline{u}\cdot N^{d(1-\alpha)}}\cdot N^{(d+1)|\mathcal{C}_{1}|}\leq e^{c|\mathcal{C}_{1}|\log N}.
\end{equation}
Here, the condition $\psi<\alpha$ ensures that $N^{d(1-\alpha)}=o(|\mathcal{C}_{1}|)$.
Combining \eqref{eq:bad(beta-gamma) bias} and \eqref{eq:bad local time bias}, we then finish the proof. 
\end{proof}

\subsection{Unlikeliness of surfaces of \texorpdfstring{$\text{bad}(\beta, \gamma)$}{} boxes (Adaptation of \Cref{sec:bad(beta-gamma)})}\label{subsec:adapting bad(beta-gamma)}
In this subsection we adapt the techniques in \Cref{sec:bad(beta-gamma)} to prove \eqref{eq:bad local time bias}. 
Recall that the corresponding result for simple random walk \Cref{prop:bad(beta-gamma)} is established in \Cref{sec:bad(beta-gamma)} using three coupling results: 
the first coupling \Cref{prop:coupling W and widetildeZ} connects the unbiased excursions $W_{\ell}^{\check{D}}$ with a set of i.i.d.~sequence of unbiased excursions $\widetilde{Z}_{\ell}^{\check{D}}$, the second \Cref{prop:coupling between widehatZ and W} couples between the extracted unbiased excursions $W_{\ell}^{D}, W_{\ell}^{D'}$ with the corresponding extracted excursions $\widetilde{Z}_{\ell}^{D}, \widetilde{Z}_{\ell}^{D'}$, and the third \Cref{lem:coupling between widetildeZ and Z} relates $\widetilde{Z}_{\ell}^{D}, \widetilde{Z}_{\ell}^{D'}$ with the excursions $Z_{\ell}^{D}, Z_{\ell}^{D'}$ of random interlacements. 

To establish a similar relationship between the biased excursions $W$ and excursions of random interlacements $Z$, the second and third couplings remains unchanged, while the first coupling needs adaptation. 
For this, we will introduce a set of i.i.d.~sequence of biased excursions $\widetilde{W}_{\ell}^{\check{D}_{x}}, {\ell\geq 1}$ (see \eqref{item:def of widetildeW}), which serves as the biased walk counterpart to the i.i.d.~sequence $\widetilde Z_{\ell}^{\check D_x},\ell\geq 1$. 

With this, we then split the original coupling between $W$ and $\widetilde{Z}$ into two parts:  the coupling between $W$ and $\widetilde{W}$, and the coupling between $\widetilde{W}$ and $\widetilde{Z}$. The first step is a ``decoupling" step, adapted from \Cref{prop:coupling W and widetildeZ}, while the second step involves calculating Radon-Nikodym derivatives and a Poissonization argument. Combining these two steps yields the counterpart of \Cref{prop:coupling W and widetildeZ}, and the remaining proof proceeds in an identical way as in \Cref{sec:bad(beta-gamma)}.\\

We now clarify our notation and carry out the two steps discussed above. 
Recall the sparse subset $\mathcal{C}_{2}$ of the set $\mathcal{C}_{1}$ as the maximal subset that satisfies \eqref{eq:sparsity of C2}, then \eqref{eq:sparsity of boxes in mathcalC2} and \eqref{eq:size of C2} still hold for $\mathcal{C}_{2}$. We also recall the unions of sets $\check{D}(\mathcal{C}_{2})$ and $\check{U}(\mathcal{C}_{2})$ in \eqref{eq:def of checkD and checkU C2}. 
For $x \in \mathcal{C}_{2}$ and $x'$ an $L$-neighbour of $x$, we write $D, D', U, U', \check D, \check U$ as $D_x, D_{x'}, U_x, U_{x'}, \check D_x, \check U_{x}$ for short, and use the notation $W_{\ell}^{\check{D}}$, $W_{\ell}^{D}$ and $W_{\ell}^{D'}$ again to denote $\ell$-th biased walk excursions from $\check{D}$ to $\check{U}$, from $D$ to $U$, and from $D'$ to $U'$ in the trajectories of $\{X^{k}\}_{k\geq 0}$ respectively. We also define a collection of i.i.d.~sequence of excursions $\widetilde{W}_{\ell}^{\check{D}_x}$, ${\ell}\geq 1$, $x\in \mathcal C_1$ such that 
\begin{equation}\label{item:def of widetildeW}
\begin{split}
&\text{For each }x\in \mathcal{C}_{1}, \widetilde{W}_{\ell}^{\check{D}_{x}}, {\ell\geq 1}\text{ are i.i.d.~excursions having the law as }X_{\cdot \land T_{\check U_x}} \text{ under }P_{\overline e_{\check D_x}}^{N,\alpha},\\
&\text{and }\big\{\widetilde{W}_{\ell}^{\check{D}_{x}}\big\}_{\ell\geq 1} \text{ are independent as }x \text{ varies over }\mathcal{C}_{1}.
\end{split}
\end{equation}
Let $(m_{\check{D}}(0,t))_{t\geq 0}$, $\check{D}=\check{D}_x$, $x \in \mathcal{C}_{2}$ be independent Poisson counting processes of intensity 1, which are independent with $\{\widetilde{W}_{\ell}^{\check{D}}\}_{\ell\geq 1}$ as $\check D$ varies. 
We now adapt \Cref{prop:coupling W and widetildeZ} to couple $W$ with $\widetilde{W}$ in the following proposition.

\begin{proposition}\label{prop:coupling W and widetildeW}
For any fixed $\alpha > 1/d$, there exists a coupling $\mathbb{Q}_{W,\widetilde{W}}^{N, \alpha}$ of the law $P_{0}^{N,\alpha}$ and the law of $\big((m_{\check{D}_x}(0,t))_{t\geq 0},\big\{\widetilde{W}_{\ell}^{\check{D}_x}\big\}_{\ell\geq 1}\big)$,  $x \in \mathcal{C}_{2}$ such that, for every $\eta\in(0,1/2)$, $N\geq c_{17}(\psi,\eta)$, $\lambda\in(0,\infty)$ and $\check{D}=\check{D}_x$, $x \in \mathcal{C}_{2}$, on the event $E_{\check{D}}^{\lambda}(W,\widetilde{W})$ defined in the same way as in \eqref{eq:def of widetildeE} (with the $n$-type Poisson point processes associated with $\widetilde{Z}$-type excursions replaced by the $m$-type Poisson point processes associated with $\widetilde{W}$-type excursions), for all $m\geq \lambda\cdot\text{cap}(\check{D})$, 
\begin{align}
\left\{\widetilde{W}_{1}^{\check{D}},\dots,\widetilde{W}_{(1-\eta)m}^{\check{D}}\right\}&\subseteq \left\{W_{1}^{\check{D}},\dots,W_{(1+3\eta)m}^{\check{D}}\right\}; \label{eq:bias inclusion checkD_1}\\ 
\left\{W_{1}^{\check{D}},\dots,W_{(1-\eta)m}^{\check{D}}\right\}&\subseteq \left\{\widetilde{W}_{1}^{\check{D}},\dots,\widetilde{W}_{(1+3\eta)m}^{\check{D}}\right\}.  \label{eq:bias inclusion checkD_2}
\end{align}
Moreover, for every $\lambda\in(0,\infty)$, \eqref{eq:prob of bad E} still holds for the probability measure $\mathbb{Q}_{W,\widetilde{W}}^{N, \alpha}$, the event $E_{\check{D}}^{\lambda}(W,\widetilde{W})$ and the constant $c_{6}=c_{6}(\lambda,\eta)$.
\end{proposition}
\begin{proof}
We begin with an estimate on the hitting distribution of the biased random walk, and then use soft local time techniques to construct the coupling. 
Taking $A=\check{D}$, $\Delta=N^{-d\alpha}$, $L=10[N^{\psi}]$ and $K=\log N/15$ in \Cref{prop:hitting distribution two boxes bias} (note that $\Delta^{-1}\geq KL(K+L)$ since $\psi<\alpha$ and $d\geq 2$), we have for any $\eta\in(0,1)$, there exists a positive constant $c_{17}(\eta)$ such that for any $\check D = \check D_x$, $x\in \mathcal C_2$, $y\in \check D$ and $z \in\partial\check{U}(\mathcal{C}_{2})$, if $N\geq c_{17}(\psi,\eta)(\geq c_{2}(\eta))$,
\begin{equation}\label{eq:hitting distribution bias}
\left(1-\frac{\eta}{3}\right)\overline{e}_{\check{D}}(y)\leq P_{z}^{N,\alpha}\left[X_{H_{\check{D}(\mathcal C_2)}}=y\mid X_{H_{\check{D}(\mathcal C_2)}}\in\check{D}\right]\leq \left(1+\frac{\eta}{3}\right)\overline{e}_{\check{D}}(y).
\end{equation}
The remaining part of the proof is almost identical to the proof of \Cref{prop:coupling W and widetildeZ} given \eqref{eq:hitting distribution}. Note that tilting the law of simple random walks into the law of biased random walks does not essentially affect \cite[Lemma 2.1]{CGP13}. In addition, we also use the requirement \eqref{eq:large distance from mathsfB} and the assumption that the walks $X^{k},k\geq 1$ all start from $\mathbb{T}\times\{-2M(\alpha)\}$ so that for all $k\geq 0$, \eqref{eq:hitting distribution bias} applies to the first excursion of $X^{k}$.
\end{proof}

Note that the above step also ``decouples" the trajectories between distant concentric boxes $\check D_x, x\in \mathcal C_2$, so that we can focus on a fixed base point $x\in \mathcal C_2$. We now carry out the second step that couples $\widetilde{W}^{\check D}$ with $\widetilde{Z}^{\check D}$, with $\check D = \check D_x$ for a fixed $x\in \mathcal C_2$.

\begin{proposition}\label{lem:coupling widetildeW with widetildeZ}
There exists a coupling $\mathbb{Q}_{\widetilde{W},\widetilde{Z}}^{N, \alpha}$ between the law of $\{\widetilde{W}_{\ell}^{\check{D}}\}_{\ell\geq 1}$ and the law of $\{\widetilde{Z}_{\ell}^{\check{D}}\}_{\ell\geq 1}$ such that, for every $\tau\in(0,\frac{1}{2})$ and $0<\lambda<\widetilde{\lambda}<\infty$, there exists an event $E_{\check{D}}^{\lambda,\widetilde{\lambda}}(\widetilde{W},\widetilde{Z})$  and a positive constant $c_{18}=c_{18}(\alpha,\psi,\lambda,\widetilde{\lambda},\tau)$ satisfying the two following conditions. First, 
\begin{equation}\label{eq:prob of bad E bias}
\liminf_{N\to\infty}\frac{1}{\log N}\log\left(-\log \mathbb{Q}_{\widetilde{W},\widetilde{Z}}^{N, \alpha}\left[E_{\check D}^{\lambda,\widetilde{\lambda}}(\widetilde{W},\widetilde{Z})^{c}\right]\right)> c_{19}(\alpha,\psi,\lambda,\widetilde{\lambda},\tau)>0.   
\end{equation}
Second, under the event $E_{\check{D}}^{\lambda,\widetilde{\lambda}}(\widetilde{W},\widetilde{Z})$, for every $N\geq c_{18}$ and $\ell\in\big(\lambda\cdot\text{cap}(\check{D}),\widetilde{\lambda}\cdot\text{cap}(\check{D})\big)$, we have  
\begin{align}
\left\{\widetilde{Z}_{1}^{\check{D}},\ \dots,\ \widetilde{Z}_{\ell}^{\check{D}}\right\}&\subseteq \left\{\widetilde{W}_{1}^{\check{D}},\dots,\widetilde{W}_{(1+\tau)\ell}^{\check{D}}\right\};   \label{eq:coupling widetildeW and widetildeZ_1}\\
\left\{\widetilde{W}_{1}^{\check{D}},\dots,\widetilde{W}_{\ell}^{\check{D}}\right\}&\subseteq \left\{\widetilde{Z}_{1}^{\check{D}},\ \dots,\ \widetilde{Z}_{(1+\tau)\ell}^{\check{D}}\right\}.   \label{eq:coupling widetildeW and widetildeZ_2}
\end{align}
\end{proposition}
\begin{proof}
Before constructing the coupling $\mathbb{Q}_{\widetilde{W},\widetilde{Z}}^{N,\alpha}$ and the event $E_{\check{D}}^{\lambda,\widetilde{\lambda}}(\widetilde{W},\widetilde{Z})$, we first provide with some preliminaries reminiscent of the proof of \Cref{lem:hitting distribution three boxes bias}.

Recall \Cref{subsec:RN} for the notation $\ell(e),h(e),\text{up}(e),\text{down}(e),p(e),p^{\text{bias}}(e)$. Also recall that the drift is $\Delta=N^{-d\alpha}$ with $d\geq 2$, $L=[N^{\psi}],\psi\in(1/d,\alpha\land 1)$,  $\check{D}=x+[-4L,5L)^{d+1}$ and $\check{U}=x+\left[-L([\log N]+1)+1,L([\log N]+1)-1\right)^{d+1}$.
We define the set of excursions from $\check{D}$ to $\partial\check{U}$ as 
\begin{equation}\label{eq:def of Se}
\begin{split}
\Sigma_{\rm excur}:=\{e= (x_{0},x_{1},\cdots,x_{n})&:\mbox{ for each }0\leq i\leq n-1,|x_{i}-x_{i+1}|_{\infty}=1, \\
&x_{0}\in\partial^{\text{int}}{\check{D}},x_{n}\in\partial\check{U}, \mbox{ and for }1\leq i\leq n-1,  x_{i}\in \check{U}\}. \end{split}
\end{equation}
According to the length of the excursion, we further divide $\Sigma_{\rm excur}$ into
\begin{equation}\label{eq:def of Sshort and Slong}
\Sigma_{\rm short}:=\{e\in \Sigma_{\rm excur}:\ell(e)\leq L^{2}N^{\alpha}\},\quad\text{and}\quad
\Sigma_{\rm long}:=\{e\in \Sigma_{\rm excur}:\ell(e)>L^{2}N^{\alpha}\}.   
\end{equation}
Then since $\psi<\alpha$, we have 
\begin{equation}\label{eq:bound on h and ell excursion}
\ell(e)\leq L^{2}N^{\alpha}\leq N^{3\alpha},\text{ for }e\in\Sigma_{\text{short}},\quad\text{and}\quad h(e)\leq CL\log N\leq CN^{\alpha},\text{ for }e\in\Sigma_{\text{excur}}.
\end{equation}  

For each $e\in \Sigma_{\rm excur}$, we now estimate the the Radon-Nikodym derivative of $e$ under $P_{\overline{e}_{\check{D}}}^{N,\alpha}$ with respect to $P_{\overline{e}_{\check{D}}}^{N}$. Combining \eqref{eq:RN bias to unbias bound} and \eqref{eq:bound on h and ell excursion} yields
\begin{equation}\label{eq:RN bias to unbias upper excursion}
\frac{p^{\rm bias}(e)}{p(e)}\leq \left(\frac{1+N^{-d\alpha}}{1-N^{-d\alpha}}\right)^{CN^{\alpha}} 
\leq 1+CN^{-(d-1)\alpha}\overset{d\geq 2}{\leq} 1+CN^{-\alpha},\quad \mbox{ for all }e\in \Sigma_{\rm excur}.
\end{equation}
In addition, combining \eqref{eq:RN bias to unbias bound} and \eqref{eq:bound on h and ell excursion} shows that for all $e\in \Sigma_{\rm short}$,
\begin{equation}\label{eq:RN bias to unbias lower Sshort excursion}
\frac{p(e)}{p^{\rm bias}(e)}\leq \left(1-N^{-2d\alpha}\right)^{-CN^{3\alpha}}\left(1+CN^{-\alpha}\right)\leq 1+CN^{-(2d-3)\alpha}+CN^{-\alpha}\overset{d\geq 2}{\leq} 1+CN^{-\alpha}. 
\end{equation}

The next ingredient is to bound the measure of $\Sigma_{\rm long}$ under $P_{\overline{e}_{\check{D}}}^{N}$ from above. Note that under $P_{\overline{e}_{\check{D}}}^{N}$, the random walk on $\mathbb{Z}$-direction makes a $+1,0,-1$ move with probability $\frac{1}{2d+2},\frac{d}{d+1}, \frac{1}{2d+2}$ respectively. As in \eqref{eq:uniform second moment three boxes} and \eqref{eq:khasminskii1}, by Kha\'{s}minskii's lemma we have
\begin{equation}
\sup_{x\in\check{D}}E_{x}^{N}\left[\exp\left(\frac{cT_{\check{U}}}{(L\log N)^{2}}\right)\right]\leq C.   
\end{equation}
It also follows from exponential Chebyshev's inequality that 
\begin{equation}\label{eq:Slong measure upper bound}
P_{\overline{e}_{\check{D}}}^{N}\left[T_{\check{U}}>L^{2}N^{\alpha}\right]\leq \frac{E_{x}^{N}\left[\exp\left(\frac{cT_{\check{U}}}{(L\log N)^{2}}\right)\right]}{\exp\left(N^{\alpha}\log^{-2}N\right)}\leq C\exp\left(-cN^{\alpha}/\log^{2}N\right).  
\end{equation}
Combining the estimate \eqref{eq:RN bias to unbias upper excursion} of Radon-Nikodym derivative, 
\begin{equation}\label{eq:Slong measure upper bound bias}
P_{\overline{e}_{\check{D}}}^{N,\alpha}\left[T_{\check{U}}>L^{2}N^{\alpha}\right]\leq (1+CN^{-\alpha})\cdot{\exp\left(-cN^{\alpha}/\log^{2}N\right)}\leq C\exp\left(-cN^{\alpha}/\log^{2}N\right).  
\end{equation}

We now construct the coupling $\mathbb{Q}_{\widetilde{W},\widetilde{Z}}^{N, \alpha}$ between $\widetilde{W}$ and $\widetilde{Z}$. Let $\big(n(0,t)\big)_{t\geq 0}$ and $\big(m(0,t)\big)_{t\geq 0}$ be two Poisson point processes of intensity $1$ with joint law to be determined, and independent from $\widetilde{W}_{\ell}^{\check{D}}$ and $\widetilde{Z}_{\ell}^{\check{D}},\ell\geq 1$. 
We then take $\big(n_{e}(0,t)\big)_{t\geq 0}$, $e\in \Sigma_{\rm excur}$ as $|\Sigma_{\rm excur}|$ i.i.d~Poisson point process of intensity $1$, also independent from $\widetilde{W}_{\ell}^{\check{D}}$ and $\widetilde{Z}_{\ell}^{\check{D}},\ell\geq 1$. Then by the property of Poisson point process, we know that (with $t$ the parameter of processes of point measures)
\begin{align}
\bigg(\sum_{\ell\leq n(0,t)}\delta_{\widetilde{Z}_{\ell}^{\check{D}}}\bigg)_{t\geq 0} &\overset{d}{=}\bigg(\sum_{e\in\Sigma_{\text{excur}}}\sum_{\ell\leq n_{e}(0, p(e)t)}\delta_{e}\bigg)_{t\geq 0},\label{eq:characterization of ZPoisson}\quad\text{and}\\
\bigg(\sum_{\ell\leq m(0,t)}\delta_{\widetilde{W}_{\ell}^{\check{D}}}\bigg)_{t\geq 0}&\overset{d}{=}\bigg(\sum_{e\in\Sigma_{\text{excur}}}\sum_{\ell\leq n_{e}(0,p^{\text{bias}}(e)t)}\delta_{e}\bigg)_{t\geq 0}.\label{eq:characterization of WPoisson}
\end{align}
We take the random variables $\widetilde{W}_{\ell}^{\check{D}},\ell\geq 1$, $\widetilde{Z}_{\ell}^{\check{D}},\ell\geq 1$, $\big(n(0,t)\big)_{t\geq 0}$, $\big(m(0,t)\big)_{t\geq 0}$ and $\big(n_{e}(0,t)\big)_{t\geq 0}$, $e\in \Sigma_{\rm excur}$ altogether into the coupling $\mathbb{Q}_{\widetilde{W},\widetilde{Z}}^{N, \alpha}$ so that the two processes in \eqref{eq:characterization of ZPoisson} and \eqref{eq:characterization of WPoisson} are exactly the same. 

We then construct the event $E_{\check{D}}^{\lambda,\widetilde{\lambda}}(\widetilde{W},\widetilde{Z})$. Consider the following three events. 
\begin{enumerate}[label=(\arabic*)]
\item \label{item:comparison with Poisson} For every integer $\ell\in\big(\lambda \text{cap}(\check{D}),\widetilde{\lambda} \text{cap}(\check{D})\big)$, 
\begin{equation}\label{eq:W widetildeZ comparison with Poisson}
\begin{split}
&n(0,\ell(1+\tau/3))\geq \ell,\quad\text{and}\quad m(0,\ell(1+2\tau/3))\leq\ell(1+\tau);\\
&m(0,\ell(1+\tau/3))\geq \ell,\quad\text{and}\quad n(0,\ell(1+2\tau/3))\leq\ell(1+\tau).
\end{split}
\end{equation}
\item \label{item:Slong empty} For every integer $\ell\in\big(\lambda \text{cap}(\check{D}),\widetilde{\lambda} \text{cap}(\check{D})\big)$, 
\begin{equation}\label{eq:Slong empty}
\Sigma_{\rm long}\cap\left\{\widetilde{Z}_{1}^{\check{D}},\dots,\widetilde{Z}_{\ell}^{\check{D}}\right\}=\Sigma_{\rm long}\cap\left\{\widetilde{W}_{1}^{\check{D}},\dots,\widetilde{W}_{\ell}^{\check{D}}\right\}=\varnothing. \end{equation}
\item \label{item:Sshort Poisson} For every integer $\ell\in\big(\lambda \text{cap}(\check{D}),\widetilde{\lambda} \text{cap}(\check{D})\big)$, 
\begin{align}
\Sigma_{\rm short}\cap\left\{\widetilde{Z}_{1}^{\check{D}}\ ,\dots,\ \widetilde{Z}_{n(0,\ell(1+\tau/3))}^{\check{D}}\right\}&\subseteq \Sigma_{\rm short}\cap\left\{\widetilde{W}_{1}^{\check{D}},\dots,\widetilde{W}_{m(0,\ell(1+2\tau/3))}^{\check{D}}\right\};\label{eq:Sshort Poisson_1}\\
\Sigma_{\rm short}\cap\left\{\widetilde{W}_{1}^{\check{D}},\dots,\widetilde{W}_{m(0,\ell(1+\tau/3))}^{\check{D}}\right\}&\subseteq \Sigma_{\rm short}\cap\left\{\widetilde{Z}_{1}^{\check{D}},\dots,\widetilde{Z}_{n(0,\ell(1+2\tau/3))}^{\check{D}}\right\}. \label{eq:Sshort Poisson_2}
\end{align}
\end{enumerate}
The event $E_{\check{D}}^{\lambda,\widetilde{\lambda}}(\widetilde{W},\widetilde{Z})$ is defined as the intersection of the above three events, and under the event $E_{\check{D}}^{\lambda,\widetilde{\lambda}}(\widetilde{W},\widetilde{Z})$, the two inclusions \eqref{eq:coupling widetildeW and widetildeZ_1} and \eqref{eq:coupling widetildeW and widetildeZ_2} immediately hold. We then argue that the probability estimate \eqref{eq:prob of bad E bias} holds by respectively bounding the probabilities of the three events from above. First, by standard exponential Chebyshev's inequality for Poisson variables, \eqref{eq:capacity of a box} and a union bound, when $N\geq c_{18}$, under $\mathbb{Q}_{\widetilde{W},\widetilde{Z}}$, the probability that \eqref{eq:W widetildeZ comparison with Poisson} does not hold is no more than
\begin{equation}\label{eq:widetildeW widetildeZ Poisson union}
\sum_{\ell=[\lambda\cdot\text{cap}(\check{D})]}^{[\widetilde{\lambda}\cdot\text{cap}(\check{D})]+1}\exp\left(-c'(\lambda,\tau)\text{cap}(\check{D})\right)\leq C'(\alpha,\psi,\lambda,\widetilde{\lambda},\tau)\exp\left(-c'(\lambda,\tau)N^{(d-1)\psi}\right).  
\end{equation}
Second, by \eqref{eq:Slong measure upper bound} and \eqref{eq:Slong measure upper bound bias}, and a union bound, when $N\geq c_{18}$, under $\mathbb{Q}_{\widetilde{W},\widetilde{Z}}$, the probability that \eqref{eq:Slong empty} does not hold is no more than
\begin{equation}\label{eq:Slong empty union}
\sum_{\ell=[\lambda\cdot\text{cap}(\check{D})]}^{[\widetilde{\lambda}\cdot\text{cap}(\check{D})]+1}\exp\left(-cN^{\alpha}/\log^{2}N\right)\leq C'(\alpha,\psi,\lambda,\widetilde{\lambda})\exp\left(-cN^{\alpha}/\log^{2}N\right).  
\end{equation}
Third, given the bounds on Radon-Nikydym derivatives  \eqref{eq:RN bias to unbias upper excursion} and \eqref{eq:RN bias to unbias lower Sshort excursion}, for every large $N\geq c_{18}$ and $e\in \Sigma_{\rm short}$, 
\begin{equation}\label{eq:requirement of c17}
\frac{p^{\rm bias}(e)}{p(e)}+\frac{p(e)}{p^{\rm bias}(e)}\leq 1+CN^{-\alpha}\leq 1+\tau/10.    
\end{equation}
Since $\tau\in(0,1)$ satisfies
\begin{equation}\label{eq:calculation about tau}
\left(1+\frac{\tau}{3}\right)\left(1+\frac{\tau}{10}\right)\leq 1+\frac{2}{3}\tau,   
\end{equation}
plugging \eqref{eq:requirement of c17} into \eqref{eq:characterization of ZPoisson} and \eqref{eq:characterization of WPoisson} shows that under $\mathbb{Q}_{\widetilde{W},\widetilde{Z}}^{N, \alpha}$, \eqref{eq:Sshort Poisson_1} and \eqref{eq:Sshort Poisson_2} almost surely hold true. Therefore, we can conclude our proof by combining the last fact with \eqref{eq:widetildeW widetildeZ Poisson union} and \eqref{eq:Slong empty union}.
\end{proof}

Combining the coupling $\mathbb{Q}_{W, \widetilde{W}}^{N, \alpha}$ in \Cref{prop:coupling W and widetildeW} of $W$ with $\widetilde W$ and the coupling $\mathbb{Q}_{\widetilde{W},\widetilde{Z}}^{N, \alpha}$ in \Cref{lem:coupling widetildeW with widetildeZ} of $\widetilde W$ and $\widetilde Z$ gives the coupling between $W$ and $\widetilde{Z}$ (that is, the biased walk excursions and i.i.d.~simple random walk excursions), a counterpart of \Cref{prop:coupling W and widetildeZ} in the biased walk case. We now state the result as follows. 
\begin{proposition}\label{prop:coupling W and widetildeZ bias}
There exists a coupling $\mathbb{Q}_{W,\widetilde{Z}}^{N,\alpha}$ of the law $P_{0}^{N,\alpha}$, the law of $\big((m_{\check{D}}(0,t))_{t\geq 0},\{\widetilde{W}_{\ell}^{\check{D}}\}_{\ell\geq 1}\big)$, $\check{D}=\check{D}_x$, $x \in \mathcal{C}_{2}$ and the law of $\{\widetilde{Z}_{\ell}^{\check{D}}\}_{\ell\geq 1}$, $\check{D}=\check{D}_x$, $x \in \mathcal{C}_{2}$ satisfying the following conditions. For every constants $0< \lambda<\widetilde{\lambda}<\infty$, $\eta,\tau\in(0,1/2)$ and a box $\check{D}=\check{D}_x$, $x \in \mathcal{C}_{2}$, with the events $E_{\check{D}}^{\lambda}(W,\widetilde{W})$ and $E_{\check{D}}^{\lambda, \widetilde \lambda}(\widetilde{W},\widetilde{Z})$ defined in \Cref{prop:coupling W and widetildeW,lem:coupling widetildeW with widetildeZ}, we define 
\begin{equation}\label{eq:def of event E bias bias}
E_{\check D}^{\lambda, \widetilde \lambda}(W, \widetilde Z):=E_{\check{D}}^{\lambda}(W,\widetilde{W})\cap E_{\check{D}}^{\lambda, \widetilde \lambda}(\widetilde{W},\widetilde{Z}).
\end{equation}
Then for every $\lambda,\widetilde{\lambda}\in(0,\infty)$, the events $E_{\check D}^{\lambda, \widetilde \lambda}(W, \widetilde Z)$ are independent as $\check{D}$ varies, and we have
\begin{equation}\label{eq:prob of bad overlineE}
\liminf_{N\to\infty}\frac{1}{\log N}\log\left(-\log\mathbb{Q}_{W,\widetilde{Z}}^{N,\alpha}\left[E_{\check D}^{\lambda, \widetilde \lambda}(W, \widetilde Z)^{c}\right]\right)> c_{20}(\alpha,\psi,\lambda,\widetilde{\lambda},\eta,\tau)>0.  
\end{equation}
Moreover, if the constants further satisfies $\lambda\cdot (1+3\eta)/(1-\eta)<\widetilde{\lambda}$, then on the event $E_{\check D}^{\lambda, \widetilde \lambda}(W, \widetilde Z)$, for every $N\geq c_{17}(\psi,\eta)$ and $\ell\in\big(\lambda\cdot\text{cap}(\check{D}),\widetilde{\lambda}\cdot\frac{1-\eta}{1+3\eta}\cdot\text{cap}(\check{D})\big)$, we have 
\begin{align}
\left\{\widetilde{Z}_{1}^{\check{D}},\ \dots,\ \widetilde{Z}_{\ell}^{\check{D}}\right\}&\subseteq \left\{W_{1}^{\check{D}},\dots,W_{(1+\widetilde{\eta})\ell}^{\check{D}}\right\};\label{eq:coupling W and widetildeZ bias_1}\\
\left\{W_{1}^{\check{D}},\dots,W_{\ell}^{\check{D}}\right\}&\subseteq \left\{\widetilde{Z}_{1}^{\check{D}},\ \dots,\ \widetilde{Z}_{(1+\widetilde{\eta})\ell}^{\check{D}}\right\}, \label{eq:coupling W and widetildeZ bias_2}
\end{align}
where $\widetilde{\eta}$ is defined through
\begin{equation}\label{eq:def of widetildeeta}
1+\widetilde{\eta}=\frac{1+3\eta}{1-\eta}\cdot(1+\tau).   
\end{equation}
\end{proposition}

\begin{proof}
The proof follows from combining \Cref{prop:coupling W and widetildeW,lem:coupling widetildeW with widetildeZ}. The events $E_{\check{D}}^{\lambda}(W,\widetilde{W})$ are independent as $\check{D}$ varies since the definition of each $E_{\check{D}}^{\lambda}(W,\widetilde{W})$ only concerns the corresponding Poisson process $m_{\check{D}}(0,t)$ (see \eqref{eq:def of widetildeE}), which are independent as $\check{D}$ varies. The independence of the events $E_{\check{D}}^{\lambda,\widetilde{\lambda}}(\widetilde{W},\widetilde{Z})$ follows from the independence of $\{\widetilde{W}_{\ell}^{\check{D}},\widetilde{Z}_{\ell}^{\check{D}}\}_{\ell\geq 1}$ as $\check{D}$ varies. 
Moreover, the probability result in \eqref{eq:prob of bad overlineE} can be deduced using \eqref{eq:prob of bad E} and \eqref{eq:prob of bad E bias}. 
The two inclusions \eqref{eq:coupling W and widetildeZ bias_1} and \eqref{eq:coupling W and widetildeZ bias_2} is derived by combining \eqref{eq:bias inclusion checkD_1} and \eqref{eq:bias inclusion checkD_2} with \eqref{eq:coupling widetildeW and widetildeZ_1} and \eqref{eq:coupling widetildeW and widetildeZ_2}.
\end{proof}

With the above coupling of $W$ and $\widetilde Z$ in the larger box $\check D$ and \Cref{lem:bad F} in hand, we now obtain the counterpart of \Cref{prop:coupling between widehatZ and W} in the biased walk case, which is the coupling of excursions through $D$ and its neighbor $D'$ extracted from $W^{\check D}$ and $\widetilde Z^{\check D}$.
\begin{proposition}\label{prop:coupling W and widehatZ bias}
Set $\widetilde{\eta},\kappa\in (0,\frac{1}{2})$, $0< \lambda<\widetilde{\lambda}<\infty$ and recall the coupling $\mathbb{Q}_{W,\widetilde{Z}}^{N,\alpha}$ and event $E_{\check D}^{\lambda, \widetilde \lambda}(W, \widetilde Z)$ in \Cref{prop:coupling W and widetildeZ bias} and ${F}_{B,+}^{\lambda}(\widetilde Z)$ in  \Cref{lem:bad F}. For an $L$-box $B=B_{x}$ with $x\in\mathcal{C}_{2}$ and the box $D_{x}$ associated with $B_{x}$ as in \eqref{eq:def of translated boxes}, define 
\begin{equation}\label{eq:def of event G bias}
G_{B}^{\lambda, \widetilde \lambda}(W, \widetilde Z):=E_{\check D}^{\lambda, \widetilde \lambda}(W, \widetilde Z) \cap {F}_{B,+}^{\lambda}(\widetilde Z).  
\end{equation}
Then for any $0 < \lambda < \widetilde{\lambda} < \infty$, the events $G_{B}^{\lambda, \widetilde \lambda}(W, \widetilde Z)$ are independent as $x$ varies over $\mathcal{C}_{2}$, and for each $B = B_x$, 
\begin{equation}\label{eq:prob of bad G bias}
\liminf_{N\to\infty}\frac{1}{\log N}\log\left(-\log\mathbb{Q}_{W,\widetilde{Z}}^{N,\alpha}\left[G_{B}^{\lambda, \widetilde \lambda}(W, \widetilde Z)^{c}\right]\right) > c_{21}(\alpha,\lambda,\widetilde{\lambda},\widetilde{\eta},\kappa)>0.   
\end{equation}
Moreover, for any $0 < \lambda < \widetilde{\lambda} < \infty$ and $N\geq c_{22}(\psi,\eta,\tau)(\geq c_{17}(\psi,\eta))$, under the event $G_{B}^{\lambda, \widetilde \lambda}(W, \widetilde Z)$, for every $\ell\in \big(\frac{\lambda}{1-\widetilde{\eta}}\cdot\text{cap}(D),\widetilde{\lambda}(1-\widetilde{\eta})^{2}\cdot\text{cap}(D)\big)$, and any $B'$ neighbouring $B$ and its associated $D'$, 
\begin{equation}\label{eq:coupling widehatZD and WD bias}   
\left\{
\begin{array}{rll}
     {\rm (i)}&\;\; \Big\{\widetilde{Z}_{1}^{D},\ \dots\ ,\widetilde{Z}_{\ell}^{D}\Big\}\subseteq 
     \Big\{W_{1}^{D},\dots,W_{(1+\widetilde{\zeta})\ell}^{D}\Big\}; \\
     {\rm (ii)}&\;\;\Big\{W_{1}^{D},\dots,W_{\ell}^{D}\Big\} \subseteq \Big\{\widetilde{Z}_{1}^{D},\ \dots\ ,\widetilde{Z}_{(1+\widetilde{\zeta})\ell}^{D}\Big\},
\end{array}
\right.\quad\text{and}
\end{equation}
\begin{equation}\label{eq:coupling widehatZD' and WD' bias} 
    \left\{
\begin{array}{rll}
     {\rm (i)}&\;\; \left\{\widetilde{Z}_{1}^{D'},\ \dots\ ,\widetilde{Z}_{\ell}^{D'}\right\} \subseteq \left\{W_{1}^{D'},\dots,W_{(1+\widetilde{\zeta})\ell}^{D'}\right\}; \\
     {\rm (ii)}&\;\;\left\{W_{1}^{D'},\dots,W_{\ell}^{D'}\right\} \subseteq \left\{\widetilde{Z}_{1}^{D'},\ \dots\ ,\widetilde{Z}_{(1+\widetilde{\zeta})\ell}^{D'}\right\},
\end{array}
\right.
\end{equation}
where $\widetilde{\zeta}$ is defined through
\begin{equation}\label{eq:def of widetildezeta}
1+\widetilde{\zeta}=\frac{1+\kappa}{1-\kappa}\cdot\frac{(1+4\widetilde{\eta})^{2}}{(1-2\widetilde{\eta})^{2}}.
\end{equation}    
\end{proposition}
\begin{proof}
The proof of \Cref{prop:coupling W and widehatZ bias} follows in the same way as that of \Cref{prop:coupling between widehatZ and W}. Indeed, given \Cref{prop:coupling W and widetildeZ} and \Cref{lem:bad F}, the original proof of \Cref{prop:coupling between widehatZ and W} only consists of a ``sandwiching" argument, which works no matter the excursions $W$ are biased or not.
\end{proof}

With the coupling \Cref{prop:coupling W and widehatZ bias} in place of \Cref{prop:coupling between widehatZ and W}, we now explain the proof of \eqref{eq:bad(beta-gamma) bias}, which is in a same way as that of \eqref{eq:bad(beta-gamma)} in \Cref{subsec:proof of prop bad(beta-gamma)}.
\begin{proof}[Proof of \eqref{eq:bad(beta-gamma) bias}]
     For a fixed $\mathcal C_3$, we still denote 
\begin{equation*}
    \text{bad}_1(\mathcal C_3) = \bigcap_{x\in \mathcal C_3}{G}_{B_x}^{\lambda, \widetilde{\lambda}}(W,\widetilde Z)^c, \quad \text{bad}_2(\mathcal C_3) = \bigcap_{x\in \mathcal C_3} \left(\{B_x \text{ is bad}(\beta, \gamma)\}\cap {G}_{B_x}^{\lambda, \widetilde{\lambda}}(W,\widetilde Z)\right).
\end{equation*}
Similar as before, if all the boxes $B_x$ for $x\in\mathcal{C}_{2}$ are $\text{bad}(\beta,\gamma)$, then under the coupling $\mathbb{Q}_{W,\widetilde{Z}}^{N,\alpha}$, there exists a subset $\mathcal C_3$ of $\mathcal{C}_{2}$ with cardinality $[|\mathcal{C}_{2}|/3]$, such that at least one of these two events occur. We note that there exist only two essential differences when applying the union bound and bounding the two events above. 

The first difference is that the probability estimate of  $G_{B}^{\lambda,\widetilde{\lambda}}(W, \widetilde{Z})$ (see \eqref{eq:prob of bad G bias}) differs from the event $G_{B}^{\lambda}(W,\widetilde Z)$ in the simple random walk case (see \eqref{eq:prob of bad G}), and therefore the term corresponding to that estimate on the right hand side of \eqref{eq:plug in prob of bad G} will be $e^{-N^{c_{21}}}$. That said, combining the union bound \eqref{eq:prob of M1S} and \eqref{eq:prob of bad(beta-gamma) C2}, we still obtain the desired result. 

The second difference is that the range $\ell\in \big(\frac{\lambda}{1-\widetilde{\eta}}\cdot\text{cap}(D),\widetilde{\lambda}(1-\widetilde{\eta})^{2}\cdot\text{cap}(D)\big)$ in \Cref{prop:coupling W and widehatZ bias} has an upper bound, while the range $\ell\geq \frac{\lambda}{1-\eta}\cdot\text{cap}(D)$ in \Cref{prop:coupling between widehatZ and W} does not.
However, this difference is also inconsequential because for fixed constants $0<\gamma<\beta<\overline{u}$ and $\widetilde{\eta}\in(0,1)$, we can choose $\widetilde{\lambda}$ large enough so that $\widetilde{\lambda}(1-\widetilde{\eta})^{2}>10\overline{u}$, and the proof proceeds smoothly.
\end{proof}

\subsection{Unlikeliness of surfaces of \texorpdfstring{$\text{poor}(\gamma)$}{large(gamma)} boxes (Adaptation of \Cref{sec:bad local time})}\label{subsec:adapting bad local time}
In this short subsection we state without proof the adapted version the stochastic domination \Cref{prop:very strong coupling} of random walk excursions and random interlacements to prove \eqref{eq:bad local time bias}. We recall \Cref{sec:bad local time} for notation. 

\begin{proposition}\label{prop:very strong coupling bias}
Let the box $\mathsf{B}$ be as in \Cref{prop:geometric argument bias}, and the box $\mathsf{D}$ be the concentric box of $\mathsf{B}$ with side-length $[N/20]$. Then one can construct on an auxiliary space $(\underline{\Omega}^{\alpha},\underline{\mathcal{F}}^{\alpha})$ a coupling $\underline{Q}^{N, \alpha}$ of the cylinder random walk and random interlacements with marginal distributions $P_{0}^{N,\alpha}$ and $\PP$ respectively, such that \eqref{eq:very strong coupling mathsfB} still holds for $\underline{Q}^{N,\alpha}$.
\end{proposition}

The proof of \Cref{prop:very strong coupling bias} will be delayed until \Cref{sec:couplings}. 
Now assuming \Cref{prop:very strong coupling bias,prop:bad local time interlacements}, the proof of \eqref{eq:bad local time bias} proceeds in the same fashion as in \eqref{eq:capacity of mathsfB}-\eqref{eq:from poor to widehatpoor}. Note that the condition $\psi>1/d$ is pivotal for deducing \eqref{eq:compare C1 and capacity of mathsfB}. Since \Cref{prop:bad local time interlacements} is purely on random interlacements, the remaining part of the proof of \eqref{eq:bad local time bias} is already contained in \Cref{subsec:continuous-time interlacements}.
\section{Bounding $T_N$ by \texorpdfstring{$\overline{S}_{N}$}{SN}}\label{sec:upper bound}
For the random walk on cylinder with upward drift $N^{-d\alpha}$ along the $\mathbb Z$-direction, we recall the definition of ``record-breaking time'' $\overline{S}_{N}(z)$ of every fixed position $z$ in \eqref{eq:def of overlineSz}. The main goal of this section is to prove that, conditioned on the event that biased random walk succeeds to hit level $z$ for more than $\frac{u_{**} + \delta}{d+1}N^d$ times (in other words, $\overline{S}_{N}(z)<\infty$), with high probability disconnection happens before time $\overline{S}_{N}(z)$, which we state in \Cref{prop:overlineS and T} and \Cref{cor:overlineS and Tz}. 

Without loss of generality, we may assume $z=0$ here. As sketched in \Cref{subsec:sketch}, \Cref{lem:conditional law} gives the conditional distribution of the biased random walk. Precisely, the conditional biased random walk will be ``pulled" towards level $\mathbb{T}\times\{0\}$ with drift $N^{-d\alpha}$ until time $\overline{S}_{N}(0)$, and act normally afterwards. With this, for two fixed mesoscopic box $\overline{B}\subseteq \overline{D}$ on level $\mathbb T \times \{0\}$ (see \eqref{eq:def of overlineB and overlineD} and \eqref{def of ef of overlineB and overlineD x} for formal defintions), 
we establish a stochastic domination control in \Cref{prop:strong coupling} of the conditioned random walk excursions $W_{\ell}^{\overline{B},\overline{D}}$ from $\overline{B}$ to $\partial \overline{D}$ (see \eqref{eq:def of biased random walk excursions overlineB}) in terms of the excursions $Z_{\ell}^{\overline{B},\overline{D}}$ of random interlacements (recall \eqref{eq:def of interlacement excursions}). 

Roughly speaking, the coupling says that, conditioned on the event $\overline{S}_{N}(0)<\infty$, with high probability the excursions $W_{\ell}^{\overline{B},\overline{D}}$ completed before time $\overline{S}_{N}(0)$ contain the excursions $Z_{\ell}^{\overline{B},\overline{D}}$ in $\I^{\widetilde{u}}$, where $\widetilde{u}\in(u_{**},u_{**}+\delta)$ is fixed, and the error term can be $O(N^{-C})$ for any $C>0$ (Here we will just take $C$ as $10d$ for the sake of convenience, but the proof remains the same for all $C$). 
We remark that this coupling is very similar to the ``very strong" coupling appeared in \Cref{prop:very strong coupling,prop:very strong coupling bias}, except that here the requirement on the coupling error is substantially weakened to polynomial in $N$ instead of exponential in $-\text{cap}(\overline{B})$ (note that there are only polynomially many boxes $\overline{B}$ on level $\mathbb{T}\times\{0\}$). It is of independent interest whether one can sharpen \Cref{prop:strong coupling}; see \Cref{remark:strong to very strong} for more discussions. Given this coupling, we can then use the strongly non-percolative property of random interlacements to prove \Cref{prop:overlineS and T}. 

Our main result of this section and its natural corollary are as follows.

\begin{proposition}\label{prop:overlineS and T}
For every $\alpha\in (1/d,\infty]$ and $\delta>0$, we have 
\begin{equation}\label{eq:overlineS and T0}
\lim_{N\to \infty}P_{0}^{N,\alpha}[T_{N}\leq \overline{S}_{N}(0)\mid \overline{S}_{N}(0)<\infty]=1.
\end{equation}
\end{proposition}

\begin{corollary}\label{cor:overlineS and Tz}
For every $\alpha\in (1/d,\infty]$, $\delta>0$ and $z\in\mathbb{Z}$, we have 
\begin{equation}\label{eq:overlineS and Tz}
\lim_{N\to \infty}P_{0}^{N,\alpha}[T_{N}\leq \overline{S}_{N}(z)\mid \overline{S}_{N}(z)<\infty]=1.
\end{equation}    
\end{corollary}

We then provide the conditional law of the biased random walk on the cylinder on the event $\overline{S}_{N}(0)<\infty$. 
\begin{lemma}\label{lem:conditional law}
Conditioned on the event $\overline{S}_{N}(0)<\infty$, the law of the biased random walk $(X_{n})_{n\geq 0}$ under $P_{0}^{N,\alpha}$ is as follows:
\begin{enumerate}[label=(\arabic*)]
\item Before time $\overline{S}_{N}(0)$, the walk has a drift of $N^{-d\alpha}$ pointing toward level $\mathbb{T}\times\{0\}$, that is, the conditioned transition probability of $(X_{n})_{n\geq 0}$ is 
\begin{equation}\label{eq:def of conditioned transition probability}
p'(x_{1},x_{2})=\left\{
\begin{aligned}
&\frac{1+N^{-d\alpha}\cdot \pi_{\mathbb Z}(x_{2}-x_{1})}{2d+2}\mathbbm{1}_{|x_{1}-x_{2}|_{\infty}=1},\quad &\pi_{\mathbb Z}(x_1)>0,\\
&\frac{1}{2d+2}\mathbbm{1}_{|x_{1}-x_{2}|_{\infty}=1},\quad &\pi_{\mathbb Z}(x_1)=0,\\
&\frac{1-N^{-d\alpha}\cdot \pi_{\mathbb Z}(x_{2}-x_{1})}{2d+2}\mathbbm{1}_{|x_{1}-x_{2}|_{\infty}=1},\quad &\pi_{\mathbb Z}(x_1)<0,
\end{aligned} 
\right.
\end{equation}
where we recall $\pi_{\mathbb Z}$ is the projection from $\mathbb E$ to $\mathbb Z$ $($the $(d+1)$-th coordinate$)$.
\item After time $\overline{S}_{N}(0)$, the walk again has an upward drift of $N^{-d\alpha}$.
\end{enumerate}
\end{lemma}
\begin{proof}
The proof follows from using strong Markov property and Doob's $h$-transform (see e.g. \cite[Section 17.6.1]{LP17}), where the absorbing state is $\mathbb{T}\times\{0\}$. We omit the details for the sake of brevity.
\end{proof}

Before moving on to the stochastic domination, we first clarify some notation. Recall that $B(x,r)$ denote the closed $|\cdot|_{\infty}$-ball centered at $x$ with radius $r\geq 0$ in $\mathbb{Z}^{d+1}$ or $\mathbb E$. We write
\begin{equation}\label{eq:def of overlineB and overlineD}
\overline{B}_{0}=B(0,[N^{1/3}]),\quad \overline{D}_{0}=B(0,[N^{2/3}]).
\end{equation}
For every $x\in \mathbb T\times \{0\}$, we also consider the translates of boxes 
\begin{equation}\label{def of ef of overlineB and overlineD x}
\overline B_x = x + \overline{B}_{0}, \quad 
\overline D_x = x+\overline{D}_{0}, \quad
\text{ and thus } 
\overline B_{x}\subseteq \overline D_{x}.
\end{equation}
Note that here the choice of scales such as $N^{1/3}$ and $N^{2/3}$ is rather arbitrary, as long as they meet the conditions of $f(N)$ and $g(N)$ outlined in \eqref{eq:requirement for f and g}. 

We write the successive times of return to $\overline{B}$ and departure from $\overline{D}$ as (recall \eqref{eq:def of return and departure} for notation)
\begin{equation}\label{eq:def of return and departure for overlineB}
R_{1}^{\overline{B},\overline{D}}<D_{1}^{\overline{B},\overline{D}}<R_{2}^{\overline{B},\overline{D}}<D_{2}^{\overline{B},\overline{D}}<\cdots<R_{k}^{\overline{B},\overline{D}}<D_{k}^{\overline{B},\overline{D}}<\cdots,
\end{equation} 
and simply write $R_{k}^{\overline{B}}$ and $
D_{k}^{\overline{B}}$ for short. Note that here there will be only finitely many $R_{\ell}^{\overline{B}}$ and $D_{\ell}^{\overline{B}}$ that are finite.
Therefore, we similarly denote the excursions from $\overline{B}$ to $\partial\overline{D}$ of random walk by 
\begin{equation}\label{eq:def of biased random walk excursions overlineB}
\text{for any }\ell\geq 1,\quad W_{\ell}^{\overline{B}} = W_{\ell}^{\overline{B}, \overline D}:=\left\{
\begin{aligned}
&X_{[R_{\ell}^{\overline{B}},D_{\ell}^{\overline{B}})},\quad &R_{\ell}^{\overline{B}}<\infty;\\
&\varnothing,\quad &R_{\ell}^{\overline{B}}=\infty.
\end{aligned}
\right.
\end{equation}
Recall \eqref{eq:def of overlineSz} for the definition of the random time $\overline S_N(z)$, and similarly as in \eqref{eq:def of N_SN}, we denote the number of excursions from $\overline{B}$ to $\overline{D}$ of the biased random walk before time $\overline{S}_{N}(0)$ by 
\begin{equation}\label{eq:def of N_SN B}
N_{\overline{S}_{N}(0)}(\overline{B}):=\sup\big\{k\geq 0:D_{k}^{\overline{B}}\leq \overline{S}_{N}(0)+1\big\}.
\end{equation}
Recall \eqref{eq:def of interlacement excursions} and \eqref{eq:number of excursions in Iu}. We use 
$Z_{\ell}^{\overline{B}} = Z_{\ell}^{\overline{B},\overline{D}}$ for the $\ell$-th excursion of random interlacements from $\overline{B}$ to $\partial\overline{D}$ (where $\overline{B}$ and $\overline{D}$ are also seen as subsets of $\mathbb{Z}^{d+1}$), and $N_{u}(\overline{B}) = N_{u}^{\overline{B},\overline{D}}$ as the number of excursions in the random interlacements set $\mathcal{I}^{u}$.

We are now ready to state the goal of this section, namely \Cref{prop:strong coupling}. 
\begin{proposition}\label{prop:strong coupling}
For every fixed $\delta > 0$, let 
\begin{equation}\label{eq:def of widetildeu}
\widetilde{u}=u_{**}+\frac{\delta}{2}.    
\end{equation}
Then for any $\alpha\in(1/d,\infty]$, every $x\in\mathbb{T}\times\{0\}$ and boxes $\overline{B}=\overline{B}_{x},$ $\overline{D}=\overline{D}_{x}$ defined in \eqref{def of ef of overlineB and overlineD x}, one can construct on an auxiliary space $(\overline{\Omega}^{\alpha},\overline{\mathcal{F}}^{\alpha})$ a coupling $\overline{Q}_{\overline{B}}^{N,\alpha}$ of the cylinder random walk and random interlacements with marginal distributions $P_{0}^{N,\alpha}\left[\,\cdot\mid\overline{S}_{N}(0)<\infty\right]$ and $\PP$ respectively, such that
\begin{equation}\label{eq:very strong coupling mathrmB}
\lim_{N\to\infty}\frac{1}{\log N}\log\overline{Q}_{\overline{B}}^{N,\alpha}\Big[\big\{Z_{\ell}^{\overline{B}}\big\}_{\ell\leq N_{\widetilde{u}}(\overline{B})} \nsubseteq 
\big\{W_{\ell}^{\overline{B}}\big\}_{\ell\leq N_{\overline{S}_{N}(0)}(\overline{B})}\Big]<-10d<0.  
\end{equation}
\end{proposition}

We first complete the proof of \Cref{prop:overlineS and T} using \Cref{prop:strong coupling}, whose proof will be postponed until \Cref{sec:couplings}.

\begin{proof}[Proof of \Cref{prop:overlineS and T} given \Cref{prop:strong coupling}]
Recall that $S(x,r)$ is the $|\cdot|_{\infty}$ sphere centered at $x$ with radius $r\geq 0$. We set 
\begin{equation}\label{eq:def of R}
R=[N^{1/6}].    
\end{equation}
Then conditioned on $\{\overline{S}_{N}(0)<\infty\}$, the event that disconnection happens after time $\overline{S}_{N}(0)$ implies that there must exist $x\in\mathbb{T}\times\{0\}$ such that
the complement of $X_{[0,\overline{S}_{N}(0)]}$ percolates from $B(x,R)$ to $S(x,2R)$. Taking the union bound, it suffices to show
\begin{equation}\label{eq:upper bound union bound}
\lim_{N\to\infty}\sum_{x\in\mathbb{T}\times\{0\}}P_{0}^{N,\alpha}\left[B(x,R)\xleftrightarrow{\mathbb E\setminus X_{[0,\overline{S}_{N}(0)]}} S(x,2R) \,\middle\vert\, \overline{S}_{N}(0)<\infty\right]=0.
\end{equation}

Using \Cref{prop:strong coupling} for $\overline{B}=\overline{B}_{x}$ and noting that $N^{-10d}\ll |\mathbb{T}\times\{0\}|^{-1}$ (=$N^{-d}$), it then suffices to show that 
\begin{equation}\label{eq:upper bound coupling}
\limsup_{N\to\infty}N^{10d}\sup_{x\in\mathbb{T}\times\{0\}}\overline{Q}_{\overline{B}_{x}}^{N,\alpha}\left[B(x,R)\xleftrightarrow{\overline{B}\setminus\bigcup_{\ell\leq N_{\widetilde{u}}(\overline{B})}\text{range}(Z_{\ell}^{\overline{B}})} S(x,2R)\right]=0,    
\end{equation}
which is equivalent to
\begin{equation}\label{eq:upper bound interlacements}
\limsup_{N\to\infty}N^{10d}\PP\left[B(x,R)\overset{\mathcal{V}^{\widetilde{u}}}\longleftrightarrow S(x,2R)\right]=0,\quad\text{for each }x\in \mathbb T \times \{0\}.
\end{equation}
Finally, the limit \eqref{eq:upper bound interlacements} can be obtained by combining the following inequality
\begin{equation}\label{eq:cater to u**}
\PP\left[B(x,R)\overset{\mathcal{V}^{\widetilde{u}}}\longleftrightarrow S(x,2R)\right]\leq\sum_{y\in S(x,R)}\PP\left[B(y,R/2)\overset{\mathcal{V}^{\widetilde{u}}}\longleftrightarrow S(y,R)\right]
\end{equation}
with the stretched-exponential decay of the connecting probability \eqref{eq:def of strongly non-percolate} in the strongly non-percolative regime 
$\widetilde{u} \in (u_{**}, \infty)$ (see \eqref{eq:def of u**}).
\end{proof}
    
\section{Couplings between random walks on cylinders and random interlacements}\label{sec:couplings}
In this section we establish various couplings of excursions $W_{\ell}^{\B,\D},\ell\geq 1$ between concentric boxes $\B$ and $\D$ of random walk on the cylinder $\mathbb E$ and corresponding excursions $Z_{\ell}^{\B,\D},\ell\geq 1$ of random interlacements, where the concentric boxes $\B$ and $\D$ (both seen as subsets of the cylinder $\mathbb{E}$ as well as $\mathbb{Z}^{d}$) have side-lengths $f(N)$ and $g(N)$ respectively, subject to some rather general conditions (see \eqref{eq:requirement for f and g}). 

Roughly speaking, the main theorem \Cref{thm:very strong coupling S} says that, the set of excursions $W_{\ell}^{\B,\D}$ which collects until the average local time at the level on which $\B$ and $\D$ lie exceeds $u$, can stochastically dominate (resp.~be dominated by) the corresponding excursions $Z_{\ell}^{\B,\D}$ in the random interlacements set $\I^{u_{1}}$ (resp.~$\I^{u_{2}}$) at some suitably adjusted intensity $u_{1}<u<u_{2}$. 
We remark that this theorem itself will not be used (except in the appendix), but it is of independent interest. In addition, our condition \eqref{eq:requirement for f and g} is rather general, and is satisfied by the pairs of concentric boxes considered in various places of this work (e.g.~$\mathsf{B},\mathsf{D}$ in \Cref{sec:bad local time,sec:lower bound in biased case} and $\overline{B},\overline{D}$ defined in \eqref{eq:def of overlineB and overlineD} and \eqref{def of ef of overlineB and overlineD x}). In fact, the couplings needed in this work, namely \Cref{prop:very strong coupling,prop:very strong coupling bias,prop:strong coupling}, can all be seen as variants of this result after minor adaptations (see \Cref{subsec:adapted proof very strong coupling} for more details). 

This stochastic domination control is obtained via a chain of ``very strong" couplings, which is similar to the chain of couplings in \cite{Szn09a,Szn09b}, but is stronger in the following two aspects: First, it significantly improves the coupling error term from polynomial in $N$ to exponential in $N$ (more precisely, $\exp(-c\text{cap}(\mathrm{B}))$), and it is expected that the coupling error here is optimal (see \Cref{remark:optimal}); Second, the stochastic domination control in \cite[Theorem 1.1]{Szn09b} or in \cite{Szn09a} is stated in terms of the trace left by excursions, while here the stochastic domination control is on the set of excursions themselves. 

The organization of this section is as follows. In \Cref{subsec:very strong couplings} we state our main theorem, \Cref{thm:very strong coupling S}. We also introduce return and departure times of concentric cylinders $R_{k}^{\A,\widetilde{\A}},D_{k}^{\A,\widetilde{\A}}$, $k\geq 1$ (see \eqref{eq:def of return and departure for mathrmA}) and use estimates of these random times to reduce the proof of \Cref{thm:very strong coupling S} to a slightly modified version, \Cref{thm:very strong coupling excursion}. In \Cref{subsec:proof of very strong coupling theorem} we outline the mechanism of the proof of \Cref{thm:very strong coupling excursion}, which consists of a chain of ``very strong" couplings, \Cref{prop:horizontal independence,prop:vertical independence,prop:Poissonization,prop:Handling the first excursion,prop:extraction,prop:lower truncation,prop:lower interlacement,prop:upper truncation,prop:upper interlacement}. In \Cref{subsec:proof of horizontal indep etc.} we provide the detailed proofs of \Cref{prop:horizontal independence,prop:vertical independence,prop:Poissonization,prop:Handling the first excursion,prop:extraction,prop:lower truncation,prop:lower interlacement,prop:upper truncation,prop:upper interlacement}, drawing a lot inspirations from the techniques in \cite{Szn09a,Szn09b}. In \Cref{subsec:adapted proof very strong coupling} we display the proofs of \Cref{prop:very strong coupling,prop:very strong coupling bias,prop:strong coupling}. As discussed above, their proofs shall follow the same mechanism, and we only give necessary minor adaptations.

\subsection{``Very strong'' couplings for simple random walk}\label{subsec:very strong couplings}
Throughout this subsection, we assume that our random walk $(X_{n})_{n\geq 0}$ is the simple random walk on the cylinder $\mathbb{E}$ started uniformly from $\mathbb{T}\times\{0\}$. Formally, for any $z\in \mathbb Z$, we define the uniform distribution as follows: 
\begin{equation}\label{eq:def of qz}
q_{z}:=\frac{1}{N^{d}}\sum_{x\in\mathbb{T}\times\{z\}}\delta_{x}.
\end{equation}

We are going to introduce two series of excursions both extracted from the simple random walk. For a fixed point $\mathrm x_c \in \mathbb E$, let $\mathrm{z}_{c} = \pi_{\mathbb Z}(\mathrm x_c)$ refer to the $(d+1)$-th coordinate of $\mathrm x_c$, we will then consider the boxes around point $\mathrm x_c$ and cylinders centered at height $\mathrm z_c$ together in pairs.
Define the boxes in $\mathbb E$ centered at $\mathrm x_c$ by
\begin{equation}\label{def:mathrm box B and D}
    \B = \mathrm x_c + B\left(0, f(N)/2\right), \quad \D = \mathrm x_c + B\left(0, g(N)/2\right),
\end{equation}
with $f(N), g(N)$ satisfying for some $0 < b < 1$ and $c > 0$,
\begin{equation}\label{eq:requirement for f and g}
\liminf_{N\to\infty}\frac{\log f(N)}{\log N}\geq b,\quad g(N)\leq\frac{N}{4},\quad\text{and}\quad\liminf_{N\to\infty}\frac{g(N)}{f(N)\log^{3}N}\geq c.    
\end{equation}
We also define the cylinder centered at height $\mathrm{z}_{c}\in\mathbb{Z}$ by
\begin{equation}\label{eq:def of mathrmA and widetildemathrmA}
\A=\mathbb{T}\times(\mathrm{z}_{c}+\mathrm{I}) \quad\text{and}\quad \widetilde{\A}=\mathbb{T}\times(\mathrm{z}_{c}+\widetilde{\mathrm{I}}),
\end{equation}
where $\mathrm I$ and $\widetilde {\mathrm I}$ are two finite intervals constructed using scales $r_N$ and $h_N$:
\begin{equation}\label{eq:def of I and widetildeI}
\mathrm{I}=[-r_{N},r_{N}] \quad\text{and}\quad\widetilde{\mathrm{I}}=[-h_{N},h_{N}],\quad \mbox{ for } r_{N}=N\text{ and } h_{N}=\left[N(2+\log^{2}N)\right].
\end{equation}
With the above notation, we denote the uniform distribution on $\partial^{\rm int}\A$ by
\begin{equation}\label{eq:def of q}
q=\frac{1}{2}(q_{\mathrm{z}_{c}-r_{N}}+q_{\mathrm{z}_{c}+r_{N}}).
\end{equation}

We then write the successive times of return to $\B$ and departure from $\D$ (resp., from cylinder $\A$ to cylinder $\widetilde{\A}$) as (see notation in \eqref{eq:def of return and departure})
\begin{align}
R_{1}^{\B,\D}<D_{1}^{\B,\D}<R_{2}^{\B,\D}<D_{2}^{\B,\D}<\cdots<R_{k}^{\B,\D}<D_{k}^{\B,\D}<\cdots,\label{eq:def of return and departure for mathrmB}\\
R_{1}^{\A,\widetilde{\A}}<D_{1}^{\A,\widetilde{\A}}<R_{2}^{\A,\widetilde{\A}}<D_{2}^{\A,\widetilde{\A}}<\cdots<R_{k}^{\A,\widetilde{\A}}<D_{k}^{\A,\widetilde{\A}}<\cdots,   \label{eq:def of return and departure for mathrmA}
\end{align}
and write $R_{k}^{\B}$,  $D_{k}^{\B}$, $R_{k}^{\A}$ and $D_{k}^{\A}$ for short respectively. Note that since we only work on recurrent simple random walk on $\mathbb{E}$ in this section, the random times $R_{k}^{\B},D_{k}^{\B}$,$R_{k}^{\A}$, $D_{k}^{\A}, k\geq 1$ are all $P_{q_{0}}^{N}$-a.s.~finite. We denote the successive excursions from $\B$ to $\partial\D$ as well as from $\A$ to $\partial\widetilde{\A}$ of the random walk $(X_{n})_{n\geq 0}$ as
\begin{equation}\label{eq:def of simple random walk excursions mathrm}
W_{\ell}^{\B}:=
X_{[R_\ell^{\B},D_\ell^{\D})} \quad \mbox{ and }\quad
W_{\ell}^{\A}:= X_{[R_\ell^{\A},D_\ell^{\A})}, \mbox{ for }\ell \geq 1. 
\end{equation}

Recall \eqref{eq:def of Sz} for the definition of $S_{N}(\omega,u,z)$, we write 
\begin{equation}\label{eq:def of mathrmS}
\mathrm{S}_{N}(u):=S_{N}(\omega,u,\mathrm{z}_{c}),
\end{equation}
and denote by
\begin{equation}\label{eq:def of N_SN mathrmB}
N_{\mathrm{S}_{N}(u)}(\B):=\sup\left\{k\geq 0:D_{k}^{\B}\leq \mathrm{S}_{N}(u)+1\right\}
\end{equation}
the number of excursions from $\B$ to $\partial\D$ in the trajectory of the simple random walk $(X_{n})_{n\geq 0}$ before time $\mathrm{S}_{N}(u)$. We finally recall \eqref{eq:def of interlacement excursions} and \eqref{eq:number of excursions in Iu} for the excursions from $\B$ to $\partial \D$ (where $\B$ and $\D$ are also seen as subsets of $\mathbb{Z}^{d+1}$) of random interlacements 
$Z_{\ell}^{\B} = Z_{\ell}^{\B,\D}$ and $N_{u}(\B) = N_{u}^{\B,\D}$ of the number of excursions from $\B$ to $\partial\D$ in the random interlacements set $\mathcal{I}^{u}$.

We now state the main theorem of this section. 
\begin{theorem}\label{thm:very strong coupling S}
For three fixed positive constants $u_{1}<u<u_{2}$, one can construct a coupling $Q$ on some auxiliary space of the simple random walk $(X_{n})_{n\geq 0}$ on $\mathbb E$ under $P_{q_{0}}^{N}$ and the random interlacements 
$\I^{u}$ and the excursions under $\PP$, so that there exists a positive constant $c_{23}=c_{23}(u,u_{1},u_{2})$ satisfying
\begin{equation}\label{eq:very strong coupling S}
Q\left[\{Z_{\ell}^{\B}\}_{\ell\leq N_{u_{1}}(\B)}\subseteq \{W_{\ell}^{\B}\}_{\ell\leq N_{\mathrm{S}_{N}(u)}(\B)} \subseteq \{Z_{\ell}^{\B}\}_{\ell\leq N_{u_{2}}(\B)}\right]\geq 1-\frac{1}{c_{23}}\exp\left(-c_{23}\cdot \text{\rm cap}(\B)\right).  
\end{equation}
Note that the excursion $W_{1}^{\B}$ may start from an interior point of $\B$ while all the excursions $Z_{\ell}^{\B},\ell\geq 1$ all start from $\partial^{\text{int}}\B$. In this case the second $\subseteq$ means that the set of excursions $\{W_{\ell}^{\B}\}_{2\leq\ell\leq N_{\mathrm{S}_{N}}(u)(\B)}$ is a subset of the set of excursions $\{Z_{\ell}^{\B}\}_{\ell\leq N_{u_{2}}(\B)}$, and further the first excursion $W_{1}^{\A}$ is part of another random interlacement excursion $Z_{\ell}^{\B}$ in the complement of the previous subset.
\end{theorem}

To prove \Cref{thm:very strong coupling S}, we first establish a good approximation for the random time $\mathrm{S}_{N}(u)$ using the successive times $R_{\ell}^{\A},D_{\ell}^{\A}, \ell \geq 1$ of return and departure of $\A$ and $\widetilde{\A}$.
Recall that in \eqref{eq:def of Zhat}, we defined a non-lazy process $(\widehat{Z}_{n})_{n\geq 0}$ of $(Z_{n})_{n\geq 0}$, the projection of our simple random walk $(X_{n})_{n\geq 0}$ onto $\mathbb{Z}$.  The time-changed process $(\widehat{Z}_{n})_{n\geq 0}$ is now a one-dimensional simple random walk, and $\mathrm{S}_N(u)$ represents the first time when the new process has at least $uN^{d}/(d+1)$ distinct visits to $\mathrm z_c$. 

We can now characterize the random time $\mathrm{S}_{N}(u)$, where we will fix three positive constants $u_{1}<u<u_{2}$ and set $\underline{K}=\underline{K}(N,u,u_{1})$, $K = K(N,u)$ and $\overline{K}=\overline{K}(N,u,u_{2})$ as 
\begin{equation}\label{eq:def of underlineK and overlineK}
\underline{K}=\left[\frac{u_{1}+u}{2(d+1)}\cdot\frac{N^{d}}{h_{N}}\right],\quad K= \left[\frac{u}{d+1}\cdot\frac{N^{d}}{h_{N}}\right]\quad\text{ and }\quad\overline{K}=\left[\frac{u_{2}+u}{2(d+1)}\cdot\frac{N^{d}}{h_{N}}\right].
\end{equation}
\begin{lemma}\label{lem:S and underlineK and overlineK}
There exists a positive constant $c_{24}=c_{24}(u,u_{1},u_{2})$ such that 
\begin{align}
\limsup_{N\to\infty}\frac{h_{N}}{N^{d}}\log P_{0}^{N}\left[D_{\underline{K}}^{\A}\geq \mathrm{S}_{N}(u)\right]\leq -c_{24};  \label{eq:S and underlineK}\\
\limsup_{N\to\infty}\frac{h_{N}}{N^{d}}\log P_{0}^{N}\left[\mathrm{S}_{N}(u)\geq D_{\overline{K}}^{\A}\right]\leq -c_{24}.\label{eq:S and overlineK}
\end{align}
\end{lemma}

\begin{proof}
The proof is similar to those of \cite[Lemma 4.5]{Szn09a} and \cite[Proposition 7.1]{Szn09b}. Suppose $U$ is a Bernoulli random variable with parameter $(h_{N}-r_{N})/h_{N}$ and $V$ is an independent geometric random variable starting from 1 with success probability $h_{N}^{-1}$. Then by the strong Markov property at time $\{R_{k}^{\A}\}_{k\geq 0}$, under $P_{0}^{N}$, for every integer $J\geq 1$, 
\begin{align}
    \sum_{m\geq 0}\mathbbm 1\{\widehat{Z}_{m}=\mathrm{z}_{c},\rho_{m}\leq D_{J}^{\A}\}&\text{ stochastically dominates the sum of }
    J \text{ i.i.d.~copies of }UV, \label{eq:stochastic domination underlineK}\\
    \sum_{m\geq 0}\mathbbm 1\{\widehat{Z}_{m}=\mathrm{z}_{c},\rho_{m}\leq D_{J}^{\A}\}&\text{ is stochastically dominated by the sum of }
    J\text{ i.i.d.~copies of }V.\label{eq:stochastic domination overlineK}
\end{align}
The conclusion then follows from using exponential Chebyshev's inequality, \eqref{eq:def of underlineK and overlineK} and the assumption $u_{1}<u<u_{2}$. We omit the details here.

\end{proof}

In light of \Cref{lem:S and underlineK and overlineK}, our main result \Cref{thm:very strong coupling S} can be reduced to the following theorem.

\begin{theorem}\label{thm:very strong coupling excursion}
Fix three positive constants $u_{1}<u<u_{2}$. We define
\begin{equation}\label{eq:def of NKmathrmB}
N_{K}(\B):=\sup\{\ell\geq 0:D_{\ell}^{\B}\leq  D_{K}^{\A}\},
\end{equation}
where $K=K(N,u)$ is defined in \eqref{eq:def of underlineK and overlineK}. Then one can construct a coupling $Q$ on some auxiliary space of the simple random walk $(X_{n})_{n\geq 0}$ on $\mathbb E$ under $P_{q_{0}}^{N}$ and the random interlacements $\I^{u}$ and the excursions under $\PP$, so that there exists a positive constant $c_{25}=c_{25}(u,u_{1},u_{2})$ satisfying
\begin{equation}\label{eq:very strong coupling excursion}
Q\left[\{Z_{\ell}^{\B}\}_{\ell\leq N_{u_{1}}(\B)}\subseteq \{W_{\ell}^{\B}\}_{\ell\leq N_{K}(\B)} \subseteq \{Z_{\ell}^{\B}\}_{\ell\leq N_{u_{2}}(\B)}\right]\geq 1-\frac{1}{c_{25}}\exp\left(-c_{25}\cdot\text{\rm cap}(\B)\right).  
\end{equation}
Here we adopt the same convention for the second $\subseteq$ as in \Cref{thm:very strong coupling S}.
\end{theorem}
We now complete the proof of \Cref{thm:very strong coupling S} assuming \Cref{thm:very strong coupling excursion}, and the latter will be proved in the next subsection.

\begin{proof}[Proof of \Cref{thm:very strong coupling S} given \Cref{thm:very strong coupling excursion}]
By \eqref{eq:capacity of a box}, \eqref{eq:requirement for f and g} and \eqref{eq:def of I and widetildeI}, we have
\begin{equation}\label{eq:capacity of B upper bound}
\text{cap}(\B)=O\left(f(N)^{d-1}\right)=O\Big(\left(\frac{N}{\log^{3}N}\right)^{d-1}\Big)=o\left(\frac{N^{d}}{h_{N}}\right).    
\end{equation}
Take $\underline{K}$ and $\overline{K}$ as in \eqref{eq:def of underlineK and overlineK}. Then by \Cref{lem:S and underlineK and overlineK}, for some small $c(u,u_{1},u_{2})$, we have 
\begin{equation}\label{eq:transforming S to K}
P_{q_{0}}^{N}\left[\{W_{\ell}^{\B}\}_{\ell\leq N_{\underline{K}}(\B)}\subseteq \{W_{\ell}^{\B}\}_{\ell\leq N_{\mathrm{S}_{N}(u)}(\B)}\subseteq \{W_{\ell}^{\B}\}_{\ell\leq N_{\overline{K}}(\B)}\right]\geq 1-\frac{\exp(-c(u,u_{1},u_{2})\text{cap}(\B))}{c(u,u_{1},u_{2})}.   
\end{equation}
Then by applying \Cref{thm:very strong coupling excursion} on the constants $\underline{K}(N,u,u_1)$ and $\overline{K}(N,u,u_2)$, we get
\begin{align}
Q\left[\{Z_{\ell}^{\B}\}_{\ell\leq N_{u_{1}}(\B)}\subseteq \{W_{\ell}^{\B}\}_{\ell\leq N_{\underline{K}}(\B)}\right]\geq 1-\frac{\exp\left(-c_{25}((u_1+u)/2, u_1, u_2)\left(\text{\rm cap}(\B)\right)\right)}{c_{25}((u_1+u)/2, u_1, u_2)}; \label{eq:very strong coupling excursion underlineK}\\
Q\left[\{W_{\ell}^{\B}\}_{\ell\leq N_{\overline{K}}(\B)} \subseteq \{Z_{\ell}^{\B}\}_{\ell\leq N_{u_{2}}(\B)}\right]\geq 1-\frac{\exp\left(-c_{25}((u+u_2)/2, u_1, u_2)\left(\text{\rm cap}(\B)\right)\right)}{c_{25}((u+u_2)/2, u_1, u_2)}.  \label{eq:very strong coupling excursion overlineK}
\end{align}
The conclusion follows from combining the inequalities \eqref{eq:transforming S to K}-\eqref{eq:very strong coupling excursion overlineK}.
\end{proof}

\begin{remark}\label{remark:optimal}
It is expected that the error terms in \Cref{thm:very strong coupling S,thm:very strong coupling excursion} are optimal. This is because the events in \eqref{eq:very strong coupling S} and \eqref{eq:very strong coupling excursion} are very unlikely to hold if $N_{u_{1}}(\B)$ and $N_{u_{2}}(\B)$ are both abnormally large or small, which has probability exponential in $-\text{cap}(\B)$ since for every $u>0$, $N_{u}(\B)$ is a Poisson random variable with intensity $u\text{cap}(\B)$.   
\end{remark}

\subsection{The chain of couplings}\label{subsec:proof of very strong coupling theorem}
In this subsection we outline the proof of \Cref{thm:very strong coupling excursion}, employing a similar approach as in \cite{Szn09a,Szn09b} and constructing a chain of ``very strong" couplings in \Cref{prop:horizontal independence,prop:vertical independence,prop:Poissonization,prop:Handling the first excursion,prop:extraction,prop:lower truncation,prop:lower interlacement,prop:upper truncation,prop:upper interlacement}. The proofs for these couplings are deferred to \Cref{subsec:proof of horizontal indep etc.}.
Throughout this subsection, we fix the positive constants $u_{1}<u<u_{2}$ and $K$, and further set a positive constant $\xi=\xi(u,u_{1},u_{2})$ as
\begin{equation}\label{eq:def of xi}
\xi=\left(1-\frac{u_{1}}{u}\right)\land\left(\frac{u_{2}}{u}-1\right).    
\end{equation}

Recall that $(Y_{n})_{n\geq 0}$ and $(Z_{n})_{n\geq 0}$ denote the projections of the random walk $(X_{n})_{n\geq 0}$ onto $\mathbb T$ and $\mathbb Z$ respectively, and  $q_{z},z\in\mathbb{Z}$ refers to the uniform distribution of $\mathbb T \times \{z\}$ (see \eqref{eq:def of qz}). Also recall the concentric cylinders $\mathrm{I}$ and $\widetilde{\mathrm{I}}$ in \eqref{eq:def of I and widetildeI}. Now for $z_{1}\in \mathrm I,z_{2}\in \partial \widetilde{\mathrm I}$, we further  define
\begin{equation}\label{eq:def of Pz1z2}
P_{z_{1},z_{2}}^{N}=P_{q_{z_{1}}}^{N}\left[\,\cdot\,\middle\vert Z_{T_{\widetilde{\A}}}=z_{2}\right]. 
\end{equation}

We will first consider the excursions from $\A$ to $\partial\widetilde \A$, and then extract and analyze the excursions from $\B$ to $\partial\D$ from these longer excursions. It takes \Cref{prop:horizontal independence,prop:vertical independence,prop:Poissonization,prop:Handling the first excursion} to handle the excursions from $\A$ to $\partial\widetilde{\A}$, and further extraction and analysis are performed in \Cref{prop:extraction,prop:lower truncation,prop:lower interlacement,prop:upper truncation,prop:upper interlacement}.

We recall \Cref{subsec:rwcylinder} for the notation $\mathscr T_{F}$ for a subset $F$ of $\mathbb E$. 
\begin{proposition}[\textbf{Horizontal Independence}]\label{prop:horizontal independence}
One can construct on some auxiliary space $(\Omega_{1},\mathcal{F}_{1})$ a coupling $Q_{1}$ of the simple random walk $(X_{n})_{n\geq 0}$ on $\mathbb E$ under $P_{q_{0}}^{N}$, a series of Bernoulli random variables $\{g_{\ell}, h_{\ell}\}_{\ell\geq 1}$ and $\mathscr T_{\widetilde \A}$-valued excursions $\{\Y_{\ell}^{\A},\Z_{\ell}^{\A}\}_{\ell\geq 1}$, where under $Q_{1}$,
\begin{itemize}
\item the excursions $W_{\ell}^{\A}$ are cylinder excursions defined in \eqref{eq:def of simple random walk excursions mathrm} (and thus have law $P_{q_{0}}^{N}$);
\item the sequences $\{g_{\ell}\}_{\ell\geq 2}$ and $\{h_{\ell}\}_{\ell\geq 2}$ are two independent sequences of i.i.d.~Bernoulli random variables with respective parameter $1-N^{-d}$ and $N^{-d}$;
\item the collection of random excursions $\big\{\Y_{\ell}^{\A}, \Z_{\ell}^{\A}\big\}_{\ell\geq 2}$ are independent from $\big\{g_{\ell}, h_{\ell}\big\}_{\ell\geq 2}$;
\item given $\big\{Z_{R_{\ell}^{\A}},Z_{D_{\ell}^{\A}}\big\}_{\ell\geq 2}$ (the $\mathbb Z$-height of return and departure points), the variables $\big\{\Y_{\ell}^{\A}\big\}_{\ell\geq 2}$ and $\big\{\Z_{\ell}^{\A}\big\}_{\ell\geq 2}$ are two independent sequences of independent excursions from $\A$ to $\partial\widetilde{\A}$, such that for every $\ell\geq 2$, $\Y_{\ell}^{\A}$ and $\Z_{\ell}^{\A}$ have the same law as that of $W_{\ell}^{\A}$ under $P_{Z_{R_{\ell}^{\A}},Z_{D_{\ell}^{\A}}}^{N}$,
\end{itemize} 
such that for $N\geq c_{26}=c_{26}(d)>0$,
\begin{equation}\label{eq:horizontal independence}
Q_{1}\left[\left\{g_{\ell}\Y_{\ell}^{\A}\right\}_{2\leq\ell\leq K}
\subseteq 
\left\{W_{\ell}^{\A}\right\}_{2\leq\ell\leq K} \subseteq \left\{\Y_{\ell}^{\A},\,h_{\ell}\Z_{\ell}^{\A}\right\}_{2\leq\ell\leq K}\,\middle\vert\,W_{1}^{\A}\right]=1,
\end{equation}
where $0$ times an excursion means empty set and $1$ times an excursion means the excursions itself.
\end{proposition}

Note that the first excursion $W_{1}^{\A}$ is different from other excursions since it starts from $\mathbb{T}\times\{0\}$ instead of $\partial^{\text{int}} \A$, and will be handled separately in \Cref{prop:Handling the first excursion}. We then explain the reasons for introducing two i.i.d.~Bernoulli sequences $\{g_{\ell}\}_{\ell\geq 2}$ and $\{h_{\ell}\}_{\ell\geq 2}$. First and foremost, the sprinkling of $N^{-d}$ on the parameters enable us to obtain a coupling $Q_{1}$ with error smaller than $\exp(-c\text{cap}(\B))$ (which is actually zero here). It is worth mentioning that the methods in \cite[Lemma 2.1]{Szn09a} and \cite[Proposition 2.1]{Szn09b} fail here due to unbearably large error terms. Second, standard large deviation estimates for Bernoulli random variables facilitate further couplings (see the proof of \Cref{prop:Poissonization}).

Through this step, we forget the $\mathbb{T}$-coordinate information of the excursions $\big\{W_{\ell}^{\A}\big\}_{2\leq\ell\leq K}$, since the starting distributions of excursions $\Y_{\ell}^{\A},\Z_{\ell}^{\A}$ are uniform in the $\mathbb{T}$-coordinate. 
Nevertheless, we still take into account the $\mathbb{Z}$-coordinate information of excursions $\big\{W_{\ell}^{\A}\big\}_{2\leq\ell\leq K}$, since the excursions $\Y_{\ell}^{\A}$ and $\Z_{\ell}^{\A}$ are constructed under the conditional laws that depend on levels $Z_{R_{\ell}^{\A}},Z_{D_{\ell}^{\A}}$. In the next step, we will deal with the information with regard to the $\mathbb{Z}$-axis to obtain i.i.d.~excursions.

To prove \Cref{prop:horizontal independence}, we leverage the rapid mixing property of a simple random walk on torus to argue that for every $\ell\geq 2$, the total variation distance between the law of the projection of starting point of $W_{\ell}^{\A}$ on $\mathbb T$ and the uniform distribution on $\mathbb T$ is small. 

\begin{proposition}[\textbf{Vertical Independence}]\label{prop:vertical independence}
One can construct on some auxiliary space $(\Omega_{2},\mathcal{F}_{2})$ a coupling $Q_{2}$ of the random variables $\big\{g_{\ell},h_{\ell},\Y_{\ell}^{\A},\Z_{\ell}^{\A}\big\}_{\ell\geq 2}$ and of $\mathscr T_{\widetilde \A}$-valued excursions $\big\{\widetilde{\Y}_{\ell}^{\A}, \widetilde{\Z}_{\ell}^{\A}\big\}_{\ell\geq 2}$, where under $Q_{2}$,
\begin{itemize}
\item the variables $\big\{g_{\ell},h_{\ell},\Y_{\ell}^{\A},\Z_{\ell}^{\A}\big\}_{\ell\geq 2}$ have the same law as under $Q_{1}$;
\item the excursions $\big\{\widetilde{\Y}_{\ell}^{\A}, \widetilde{\Z}_{\ell}^{\A}\big\}_{\ell\geq 2}$ are independent from Bernoulli variables $\big\{g_{\ell},h_{\ell}\big\}_{\ell\geq 2}$;
\item the excursions $\big\{\widetilde{\Y}_{\ell}^{\A}\big\}_{\ell\geq 2}$ and $\{\widetilde{\Z}_{\ell}^{\A}\}_{\ell\geq 2}$ are two independent sequences of i.i.d excursions from $\A$ to $\partial\widetilde{\A}$ with the same distribution as $X._{\land T_{\widetilde{\A}}}$ under $P_{q}^{N}$ (recall \eqref{eq:def of q} for the definition of the measure $q$). 
\end{itemize}
Furthermore, for $\xi = \xi(u, u_1, u_2)$ in \eqref{eq:def of xi}, letting
\begin{equation}\label{eq:def of widetildeK1 and widetildeK2}
\widetilde{K}_{1}=\left[\big(1-\frac{1}{7}\xi\big)\cdot\frac{u}{d+1}\cdot\frac{N^{d}}{h_{N}}\right]\quad\text{and}\quad\widetilde{K}_{2}=\left[\big(1+\frac{1}{7}\xi\big)\cdot\frac{u}{d+1}\cdot\frac{N^{d}}{h_{N}}\right],
\end{equation}
there exists a positive constant $c_{27}=c_{27}(u,u_{1},u_{2})$ satisfying
\begin{align}
Q_{2}\left[\big\{\widetilde{\Y}_{\ell}^{\A}\big\}_{2\leq\ell\leq \widetilde{K}_{1}}\subseteq \big\{\Y_{\ell}^{\A}\big\}_{2\leq\ell\leq K}\subseteq \big\{\widetilde{\Y}_{\ell}^{\A}\big\}_{2\leq\ell\leq \widetilde{K}_{2}}\right] \geq 1-\frac{1}{c_{27}}\exp\left(-c_{27}K\right);   \label{eq:vertical independence 1}\\
Q_{2}\left[\big\{\widetilde{\Z}_{\ell}^{\A}\big\}_{2\leq\ell\leq \widetilde{K}_{1}}\subseteq \big\{\Z_{\ell}^{\A}\big\}_{2\leq\ell\leq K}\subseteq \big\{\widetilde{\Z}_{\ell}^{\A}\big\}_{2\leq\ell\leq \widetilde{K}_{2}}\right]\geq 1-\frac{1}{c_{27}}\exp\left(-c_{27}K\right).\label{eq:vertical independence 2}
\end{align}
\end{proposition}
Note that the starting distributions of $\big\{\Y_{\ell}^{\A}\big\}_{\ell\geq 2}$ and $\big\{\Z_{\ell}^{\A}\big\}_{\ell\geq 2}$ are i.i.d.~in the $\mathbb{T}$-coordinate but not in the $\mathbb{Z}$-coordinate as explained before, while the starting distributions of excursions $\big\{\widetilde{\Y}_{\ell}^{\A}\big\}_{\ell\geq 2}$ and $\big\{\widetilde{\Z}_{\ell}^{\A}\big\}_{\ell\geq 2}$ are i.i.d.~both in the $\mathbb{T}$-coordinate and $\mathbb{Z}$-coordinate. This proposition allows us to forget the information of the excursions $W_{\ell}^{\A}$ on the $\mathbb Z$-axis, and therefore obtain completely independent sequences of excursions. The proof of \Cref{prop:vertical independence} is based on the observation that $\big\{Z_{R_{\ell}^{\A}},Z_{D_{\ell}^{\A}}\big\}_{\ell \geq 1}$ is a Markov chain, combined with the large deviation estimate.

In the following step, we will couple the excursions $\widetilde{\Z}_{\ell}^{\A}$ and $\widetilde{\Y}_{\ell}^{\A}$ with a Poissonian number of excursions in order to approximate random interlacements. 
\begin{proposition}[\textbf{Poissonization}]\label{prop:Poissonization}
One can construct on an auxiliary space $(\Omega_{3},\mathcal{F}_{3})$ a coupling $Q_{3}$ of the variables $\big\{g_{\ell},h_{\ell},\widetilde{\Y}_{\ell}^{\A},\widetilde{\Z}_{\ell}^{\A}\big\}_{\ell\geq 2}$ and of Poisson variables $\widetilde{J}_{1},\widetilde{J}_{1}'$ and $\mathscr T_{\widetilde \A}$-valued excursions $\big\{\widetilde{\W}_{\ell}^{\A}\big\}_{\ell\geq 1}$, where under $Q_{3}$,
\begin{itemize}
\item the variables $\big\{g_{\ell},h_{\ell},\widetilde{\Y}_{\ell}^{\A},\widetilde{\Z}_{\ell}^{\A}\big\}_{\ell\geq 2}$ have the same law as under $Q_{2}$;
\item the excursions $\big\{\widetilde{\W}_{\ell}^{\A}\big\}_{\ell\geq 1}$ is a sequence of i.i.d excursions from $\A$ to $\partial\widetilde{\A}$ with the same distribution as $X._{\land T_{\widetilde{\A}}}$ under $P_{q}^{N}$ (recall \eqref{eq:def of q} for the definition of the measure $q$);
\item the variables $\widetilde{J}_{1}$ and $\widetilde{J}_{1}'$ are independent from the collection of variables $\big\{g_{\ell}, h_{\ell},\widetilde{\Y}_{\ell}^{\A},\widetilde{\Z}_{\ell}^{\A},\widetilde{\W}_{\ell}^{\A}\big\}_{\ell\geq 1}$. Moreover, $\widetilde{J}_{1}$, $\widetilde{J}_{1}'-\widetilde{J}_{1}$ are two independent Poisson random variables with intensities $\widetilde{\lambda}_1$ and $\widetilde{\lambda}_1' - \widetilde{\lambda}_1$ respectively, where 
\begin{equation}\label{eq:def Poisson intensity}
    \widetilde{\lambda}_1 = \big(1-\frac{3}{7}\xi\big)\cdot \frac{u}{d+1}\cdot \frac{N^{d}}{h_{N}}, \quad \text{and} \quad \widetilde{\lambda}_1' = \big(1 + \frac{3}{7}\xi\big)\cdot \frac{u}{d+1}\cdot \frac{N^{d}}{h_{N}},
\end{equation}
\end{itemize}
such that there exists a positive constant $c_{28}=c_{28}(u,u_{1},u_{2})$ satisfying
\begin{align}
Q_{3}\left[\{\widetilde{\W}_{\ell}^{\A}\}_{\ell\leq \widetilde{J}_{1}}\subseteq \{g_{\ell}\widetilde{\Y}_{\ell}^{\A}\}_{2\leq\ell\leq\widetilde{K}_{1}}\right]\geq 1-\frac{1}{c_{28}}\exp\left(-c_{28}K\right);\label{eq:Poissonization 1}\\
Q_{3}\left[\{\widetilde{\Y}_{\ell}^{\A}, h_{\ell}\widetilde{\Z}_{\ell}^{\A}\}_{2\leq\ell\leq \widetilde{K}_{2}}\subseteq \{\widetilde{\W}_{\ell}^{\A}\}_{\ell\leq\widetilde{J}_{1}'}\right]\geq 1-\frac{1}{c_{28}}\exp\left(-c_{28}K\right).\label{eq:Poissonization 2}
\end{align}
\end{proposition}
The proof involves the standard exponential Chebyshev's inequality on Poisson variables. 
\begin{proposition}[\textbf{Handling the first excursion}]\label{prop:Handling the first excursion}
One can construct on some auxiliary space $(\Omega_{4},\mathcal{F}_{4})$ a coupling $Q_{4}$ of the $\mathscr{T}_{\widetilde{\A}}$-valued excursions $W_{\ell}^{\A}$ and of the $\mathscr{T}_{\widetilde{\A}}$-valued excursions $\{\widetilde{\W}_{\ell}^{\A}\}_{\ell\geq 1}$ and Poisson variables $\widetilde{J}_{1},\widetilde{J}_{1}',\widetilde{J}_{2}$, where under $Q_{4}$,
\begin{itemize}
\item the excursion $W_{1}^{\A}$ is distributed as that under $P_{q_{0}}^{N}$;
\item the excursions $\big\{\widetilde{\W}_{\ell}^{\A}\big\}_{\ell\geq 1}$ and the variables $\widetilde{J}_{1},\widetilde{J}_{1}'$ have the same law as under $Q_{3}$;
\item the variable $\widetilde{J}_{2}-\widetilde{J}_{1}'$ is a Poisson random variable with intensity $\widetilde{\lambda}_2 - \widetilde{\lambda}_1'$ (see \eqref{eq:def Poisson intensity}), independent from $\widetilde{J}_{1},\widetilde{J}_{1}', W_{1}^{\A},\big\{\widetilde{\W}_{\ell}^{\A}\big\}_{\ell\geq 1}$, where
\begin{equation}\label{eq:def Poisson intensity2}
    \widetilde{\lambda}_2 = \big(1 + \frac{4}{7}\xi\big)\cdot \frac{u}{d+1}\cdot \frac{N^{d}}{h_{N}},
\end{equation}
\end{itemize}
such that there exists a positive constant $c_{29}=c_{29}(u,u_{1},u_{2})$ satisfying (where we use $H_{\A}(W_{1}^{A})$, $T_{\widetilde{\A}}(W_{1}^{A})$ respectively for the entrance time into $\A$ and the departure time from $\widetilde{\A}$ of $W_{1}^{\A}$, and use $(W_{1}^{\A})^{\rm tr}$ for the truncated $\mathscr{T}_{\widetilde{\A}}$-valued excursion in $W_{1}^{\A}$ from $H_{\A}(W_{1}^{A})$ to $T_{\widetilde{\A}}(W_{1}^{A})$)
\begin{equation}\label{eq:Handling the first excursion}
Q_{4}\left[(W_{1}^{\A})^{\rm tr}\in \big\{\widetilde{\W}_{\ell}^{\A}\big\}_{\widetilde{J}_{1}'\leq \ell\leq \widetilde{J}_{2}}\right]\geq 1-\frac{1}{c_{29}}\exp\left(-c_{29}K\right).
\end{equation}
Note that $W_{1}^{\A}$ may start from an interior point of $\A$ while all the excursions $\W_{\ell}^{\A},\ell\geq 2$ start from $\partial^{\rm int}\A$. In this case ``$\in$" means that $W_{1}^{\A}$ is part of some excursion $\widetilde{\W}_{\ell}^{\A}$ for some $\widetilde{J}_{1}'\leq\ell\leq\widetilde{J}_{2}$.
\end{proposition}
This is a technical step handling the first excursion. The proof involves the exponential Chebyshev's inequality on Poisson random variables and some simple facts with respect to the one-dimensional simple random walk.

\Cref{prop:horizontal independence,prop:vertical independence,prop:Poissonization,prop:Handling the first excursion} ensure that the excursions $\big\{W_{\ell}^{\A}\big\}_{\ell\leq K}$ can stochastically dominate (resp., be dominated by) a Poisson number of i.i.d.~excursions $\big\{\widetilde{\W}_{\ell}^{\A}\big\}_{\ell\leq \widetilde{J}_{1}}$ (resp., $\big\{\widetilde{\W}_{\ell}^{\A}\big\}_{\ell\leq \widetilde{J}_{2}}$) with the same distribution as $X._{\land T_{\widetilde{\A}}}$ under $P_{q}^{N}$. Here, the Poissonian structure is necessary since it allows us to exploit properties of Poisson point measures.\\

Our next goal is to extract excursions from $\B$ to $\partial\D$ from the excursions $\widetilde{\W}_{\ell}^{\A}, \ell \geq 1$. In the next proposition, we approach this by characterizing the random (multi)sets $\widetilde{\Lambda}_{1}$ and $\widetilde{\Lambda}_{2}$, where
\begin{equation}\label{eq:def of widetildelambda1}
\widetilde{\Lambda}_{1}:=\{\widetilde{\W}_{\ell}^{\A}:\ell\leq \widetilde{J}_{1},\,\,\widetilde{\W}_{\ell}^{\A}\cap\B\neq\emptyset\} \quad\text{and}\quad\widetilde{\Lambda}_{2}:=\{\widetilde{\W}_{\ell}^{\A}:\ell\leq \widetilde{J}_{2},\,\,\widetilde{\W}_{\ell}^{\A}\cap\B\neq\emptyset\}    
\end{equation}
respectively consists of excursions in $\{\widetilde{\W}_{\ell}^{\A}\}_{\ell\leq \widetilde{J}_{1}}$ and $\{\widetilde{\W}_{\ell}^{\A}\}_{\ell\leq \widetilde{J}_{2}}$ that hit $\B$.

\begin{proposition}[\textbf{Extraction}]\label{prop:extraction}
One can construct on some auxiliary space $(\Omega_{5},\mathcal{F}_{5})$ a coupling $Q_{5}$ of a sequence of $\mathscr{T}_{\widetilde{\A}}$-valued excursions $\widetilde{\W}_{\ell}^{\A},\ell\geq 1$, two Poisson variables $\widetilde{J}_{1},\widetilde{J}_{2}$ and two Poisson point measures $\widetilde{\mu}_{1},\widetilde{\mu}_{2}$ on $\mathscr{T}_{\widetilde{\A}}$, where under $Q_{5}$,    
\begin{itemize}
\item the excursions $\{\widetilde{\W}_{\ell}^{\A}\}_{\ell\geq 1}$ and the Poisson variables $\widetilde{J}_{1},\widetilde{J}_{2}$ have the same law as under $Q_{4}$;
\item the Poisson point measures $\widetilde{\mu}_{1}$ and $\widetilde{\mu}_{2}$ satisfy
\begin{align}
\text{intensity of }\widetilde{\mu}_{1}=\Big(1-\frac{3}{7}\xi\Big)u\Big(1-\frac{r_{N}}{h_{N}}\Big)P_{e_{\B,\widetilde{\A}}}^{N}\left[X._{\land T_{\widetilde{\A}}}\in dw\right];   \label{eq:intensity of widetildemu1} \\
\text{intensity of }\widetilde{\mu}_{2}=\Big(1+\frac{4}{7}\xi\Big)u\Big(1-\frac{r_{N}}{h_{N}}\Big)P_{e_{\B,\widetilde{\A}}}^{N}\left[X._{\land T_{\widetilde{\A}}}\in dw\right],    \label{eq:intensity of widetildemu2}
\end{align}
\end{itemize}
such that for every $N$,
\begin{equation}\label{eq:extraction 1}
Q_{5}\left[\widetilde{\Lambda}_{1}=\text{supp}(\widetilde{\mu}_{1})\right]=1, \quad\text{and}\quad
Q_{5}\left[\widetilde{\Lambda}_{2}=\text{supp}(\widetilde{\mu}_{2})\right]=1,
\end{equation}
where $\text{supp}(\widetilde{\mu}_{1})$ and $\text{supp}(\widetilde{\mu}_{2})$ are both random multisets of excursions in $\mathscr{T}_{\widetilde{\A}}$.
\end{proposition}
Thanks to the thinning property of Poisson point measures, proving \eqref{eq:extraction 1} essentially boils down to calculating the probability of a random walk hitting $\B$ when starting from the uniform measure on $\partial^{\rm int}\A$. We refer to \cite[Lemma 1.1]{Szn09a} for an explicit formula.\\ 

Given the above couplings, there are two remaining steps in the proof of \Cref{thm:very strong coupling excursion}. The first is to further extract excursions from $\B$ to $\partial\D$ from the excursions in $\text{supp}(\widetilde{\mu}_{1})$ and $\text{supp}(\widetilde{\mu}_{2})$, and the second is to compare these extracted excursions with corresponding excursions from random interlacements. For the latter, we will use the fact that $\B\subseteq \D$ can be viewed as subsets of both $\mathbb{E}$ and $\mathbb{Z}^{d+1}$. Note that there exists some differences when handling the lower bound and upper bound, and we now show them separately. We denote by $\text{supp}(\widetilde{\mu}_{1})^{\B}$ (resp.~$\text{supp}(\widetilde{\mu}_{2})^{\B}$) the collection of all excursions from $\B$ to $\partial\D$ contained in the random set $\text{supp}(\widetilde{\mu}_{1})$ (resp.~$\text{supp}(\widetilde{\mu}_{2})$).
\begin{proposition}[\textbf{Truncation and comparison with truncated random interlacements}]\label{prop:lower truncation} 
One can construct on an auxiliary space $(\Omega_{6},\mathcal{F}_{6})$ a coupling $Q_{6}$ of two Poisson point measures $\widetilde{\mu}_{1}$ and $\widehat{\mu}_{1}$, where under $Q_{6}$,
\begin{itemize}
\item the Poisson point measure $\widetilde{\mu}_{1}$ has intensity measure as in \eqref{eq:intensity of widetildemu1};
\item the Poisson point measure $\widehat{\mu}_{1}$ satisfies
\begin{equation}\label{eq:intensity of widehatmu1}
\text{intensity of }\widehat{\mu}_{1}=\Big(1-\frac{5}{7}\xi\Big)u P_{e_{\B}}\left[X._{\land T_{\D}}\in dw\right],
\end{equation}
\end{itemize}
such that for $N\geq c_{30}=c_{30}(u,u_{1},u_{2})>0$,
\begin{equation}\label{eq:lower truncation}
Q_{6}\left[\text{supp}(\widehat{\mu}_{1})\subseteq \Big\{\widetilde{\W}._{\land T_{\D}}:\widetilde{\W}\in\text{supp}(\widetilde{\mu}_{1})\Big\}\subseteq \text{supp}(\widetilde{\mu}_1)^{\B}\right]=1. 
\end{equation}
\end{proposition}

Though an excursion $\widetilde{\W}\in\text{supp}(\widetilde{\mu}_{1})$ may contain more than one excursion from $\B$ to $\partial\D$, by considering $\widetilde{\W}._{\land T_{\D}}$ we only take the first one into account. To prove \Cref{prop:lower truncation} we will leverage the property of Poisson point measures, and compare the distribution $e_{\B,\widetilde{\A}}$ with $e_{\B}$ after some sprinkling.
\begin{proposition}[\textbf{Truncated interlacements and interlacements}]\label{prop:lower interlacement}
One can construct on an auxiliary space $(\Omega_{7},\mathcal{F}_{7})$ a coupling $Q_{7}$ of the Poisson point measure $\widehat{\mu}_{1}$ and of the random interlacements under $\PP$, under which $\widehat{\mu}_1$ has intensity measure as in \eqref{eq:intensity of widehatmu1}, such that there exists a positive constant $c_{31}=c_{31}(u,u_{1},u_{2})$ satisfying
\begin{equation}\label{eq:lower interlacement}
Q_{7}\left[\big\{Z_{\ell}^{\B}\big\}_{\ell\leq N_{u_{1}}(\B)}\subseteq \text{supp}(\widehat{\mu}_{1})\right]\geq 1-\frac{1}{c_{31}}\exp\left(-c_{31}\cdot\text{cap}(\B)\right).   
\end{equation}
\end{proposition}
Recall that $Z_{\ell}^{\B}$ stands for the excursions of random interlacements. This proposition offers a stochastic domination between some i.i.d.~excursions from $\B$ to $\D$ and the corresponding excursions of interlacements. We remark that \Cref{prop:lower interlacement} is similar to the couplings in \Cref{sec:bad(beta-gamma)}, where we use soft local time techniques to decouple the excursions in random interlacements. Moreover, with help of this decoupling procedure, it suffices to count the number of excursions in $Z_{\ell}^{\B}, {\ell\leq N_{u_{1}}(\B)}$ and $\text{supp}(\widehat{\mu}_{1})$, where we need to argue that for the simple random walk in $\mathbb{Z}^{d+1}$ started from an $x\in\partial\D$, it hits $\B$ only with small probability, and therefore the truncation on the interlacement trajectories does not decrease the Poisson intensity too much. 
We also refer to \cite[Sections 8 and 9]{Bel13} for a similar idea.

The first inclusion ``$\{Z_{\ell}^{\B}\}_{\ell\leq N_{u_{1}}(\B)}\subseteq \{W_{\ell}^{\B}\}_{\ell\leq N_{K}(\B)}$" in \Cref{thm:very strong coupling excursion} can then be concluded from combining \Cref{prop:horizontal independence,prop:vertical independence,prop:Poissonization,prop:Handling the first excursion,prop:extraction,prop:lower truncation,prop:lower interlacement}, and we now turn to the second inclusion ``$\{W_{\ell}^{\B}\}_{\ell\leq N_{K}(\B)}\subseteq \{Z_{\ell}^{\B}\}_{\ell\leq N_{u_{2}}(\B)}$". 
\begin{proposition}[\textbf{Truncated cylinder excursions and cylinder excursions}]\label{prop:upper truncation}
One can construct on an auxiliary space $(\Omega_{8},\mathcal{F}_{8})$ a coupling $Q_{8}$ of the Poisson point measure $\widetilde{\mu}_{2}$ and of another Poisson point measure $\widehat{\mu}_{2}$, where under $Q_{8}$,
\begin{itemize}
\item the Poisson point measure $\widetilde{\mu}_{2}$ has intensity measure as in \eqref{eq:intensity of widetildemu2};
\item the Poisson point measure $\widehat{\mu}_{2}$ satisfies
\begin{equation}\label{eq:intensity of widehatmu2}
\text{intensity of } \widehat{\mu}_{2}=\Big(1+\frac{6}{7}\xi\Big)u P_{e_{\B,\widetilde{\A}}}^{N}\left[X._{\land T_{\D}}\in dw\right], 
\end{equation}
\end{itemize}
such that there exists a positive constant $c_{32}=c_{32}(u,u_{1},u_{2})$ satisfying
\begin{equation}\label{eq:upper truncation}
Q_{8}\left[\text{supp}(\widetilde{\mu}_{2})^{\B}\subseteq \text{supp}(\widehat{\mu}_{2})\right]\geq 1-\frac{1}{c_{32}}\exp(-c_{32}\cdot\text{cap}(\B)). \end{equation}
\end{proposition}
The idea here is similar to that in \Cref{prop:lower interlacement}, except that we truncate cylinder excursions instead of random interlacements excursions, and we need to use the soft local time techiniques in \Cref{sec:bad(beta-gamma)} for decoupling and argue that for the simple random walk on $\mathbb E$, started from any point $x\in\partial\D$, it hits $\B$ before exiting $\widetilde{\A}$ with small probability.
\begin{proposition}[\textbf{Comparison with random interlacements}]\label{prop:upper interlacement} 
One can construct on an auxiliary space $(\Omega_{9},\mathcal{F}_{9})$ a coupling $Q_{9}$ of the Poisson point measure $\widehat{\mu}_{2}$ and of the random interlacements under $\PP$, under which $\widehat{\mu}_2$ has the intensity as in \eqref{eq:intensity of widehatmu2}, such that for $N\geq c_{33}=c_{33}(u,u_{1},u_{2})>0$, 
\begin{equation}\label{eq:upper interlacement}
Q_{9}\left[\text{supp}(\widehat{\mu}_{2})\subseteq \big\{Z_{\ell}^{\B}\big\}_{\ell\leq N_{u_{2}}(\B)}\right]=1.   
\end{equation}
\end{proposition}

The idea here is similar to that of \Cref{prop:lower truncation} and can be proved by a comparison between $e_{\B,\widetilde{\A}}$ and $e_{\B}$ after sprinkling.

We now complete the proof of \Cref{thm:very strong coupling excursion} using \Cref{prop:horizontal independence,prop:vertical independence,prop:Poissonization,prop:Handling the first excursion,prop:extraction,prop:lower truncation,prop:lower interlacement,prop:upper truncation,prop:upper interlacement}.

\begin{proof}[Proof of \Cref{thm:very strong coupling excursion}]
Let $\widetilde{\mu}_{1}$ and $\widetilde{\mu}_{2}$ be as in \eqref{eq:intensity of widetildemu1} and \eqref{eq:intensity of widetildemu2}. In light of \Cref{prop:horizontal independence,prop:vertical independence,prop:Poissonization,prop:Handling the first excursion,prop:extraction}, it follows that when $N\geq c_{26}$, under some coupling $Q'$ which encompasses all the objects that appeared in couplings $Q_{1}$-$Q_{5}$, we have for some positive constant $c =c(c_{27},c_{28},c_{29})$,
\begin{equation}\label{eq:coupling Q'}
Q'\left[\text{supp}(\widetilde{\mu}_{1})^{\B}\subseteq \big\{W_{\ell}^{\B}\big\}_{\ell\leq N_{K}(\B)}\subseteq \text{supp}(\widetilde{\mu}_{2})^{\B}\right]\geq 1-c^{-1}\exp\left(-cK\right).   
\end{equation}

Now let $\widehat{\mu}_{1}$ and $\widehat{\mu}_{2}$ be as in \eqref{eq:intensity of widehatmu1} and \eqref{eq:intensity of widehatmu2}. Following \Cref{prop:lower truncation,prop:lower interlacement,prop:upper truncation,prop:upper interlacement}, when $N\geq \max(c_{30}, c_{32})$, under some coupling $Q''$ which encompasses all the objects that appeared in couplings $Q_{6}$-$Q_{9}$, we have 
\begin{align}
Q''\left[\big\{Z_{\ell}^{\B}\big\}_{\ell\leq N_{u_{1}}(B)}\subseteq \text{supp}(\widehat{\mu}_{1})\subseteq \text{supp}(\widetilde{\mu}_{1})^{\B}\right]\geq 1-\frac{1}{c_{31}}\exp\left(-c_{31}\text{cap}(\B)\right);  \label{eq:lower coupling}\\
Q''\left[\text{supp}(\widetilde{\mu}_{2})^{\B}\subseteq \text{supp}(\widehat{\mu}_{2})\subseteq \{Z_{\ell}^{\B}\}_{\ell\leq N_{u_{2}}(\B)}\right]\geq 1-\frac{1}{c_{33}}\exp\left(-c_{33}\text{cap}(\B)\right). \label{eq:upper coupling}
\end{align}
Let the coupling $Q$ encompass all the objects that appeared in  $Q'$ and $Q''$. Recalling \eqref{eq:capacity of B upper bound} for the definition of $K$ that implies $\text{cap}(\B)=o(K)$, combining \eqref{eq:coupling Q'}-\eqref{eq:upper coupling} then yields \Cref{thm:very strong coupling excursion}. 
\end{proof}

\subsection{Proofs of 
\texorpdfstring{\Cref{prop:horizontal independence,prop:vertical independence,prop:Poissonization,prop:Handling the first excursion,prop:extraction,prop:lower truncation,prop:lower interlacement,prop:upper truncation,prop:upper interlacement}}{prop:horizontal independence, etc.}}\label{subsec:proof of horizontal indep etc.}
In this subsection we provide the proofs for \Cref{prop:horizontal independence,prop:vertical independence,prop:Poissonization,prop:Handling the first excursion,prop:extraction,prop:lower truncation,prop:lower interlacement,prop:upper truncation,prop:upper interlacement}. 

We first distill \Cref{prop:horizontal independence} into a more general conclusion. For each space $(\mathcal{X}, \mathscr F)$ and two measures $\mu$, $\nu$ on this space, recall the total variation distance between $\mu$ and $\nu$:
\begin{equation}
    \text{TV}(\mu, \nu) = \sup_{A\in \mathscr F}|\mu(A) - \nu(A)|.
\end{equation}

\begin{lemma}\label{lem:horizontal independence}
Suppose $\mathcal{X}$ has $n\geq 100$ elements. Let $\mu$ be the uniform measure on $\mathcal{X}$ and $\mu_{W}$ be the distribution of a random variable $W$ on the same space with 
\begin{equation}\label{eq:requirement fro muw}
\mathrm{TV}(\mu,\mu_{W})\leq 1/n^{4}.    
\end{equation}
Let $\overline{\mu}$ be the distribution of random variables $\Y$, $\Z$ on $\mathcal{X}$ together with Bernoulli random variables $g,h$, such that under $\overline{\mu}$, 
\begin{itemize}
\item the variables $\Y,\Z$ are independent from the variables $g,h$;
\item the variables $\Y,\Z$ are mutually independent and both have distribution $\mu$;
\item the variables $g,h$ are independent Bernoulli variables with parameters $1-1/n$ and $1/n$.
\end{itemize}
Then there exists a coupling of $\mu_{W}$ and $\overline{\mu}$ such that
\begin{equation}\label{eq:horizontal independence coupling}
Q_{\mathcal{X}}(\{g\Y\}\subseteq\{W\}\subseteq\{\Y,h\Z\})=1, 
\end{equation}
where $0$ times an element in $\mathcal{X}$ means the empty set and $1$ times an element in $\mathcal{X}$ means the element itself.
\end{lemma}
\begin{proof}
Without loss of generality we take $\mathcal{X}=\{1,2,\dots,n\}$ and assume that $\mu_{W}(1)\geq \mu_{W}(2)\geq\cdots\geq\mu_{W}(n)$. Since $\mathrm{TV}(\mu_{W},\mu)\leq 1/n^{4}$, it follows that 
\begin{equation}\label{eq:bound on muW}
\frac{1}{n}+\frac{1}{n^{4}}\geq\mu_{W}(1)\geq \frac{1}{n}\geq\mu_{W}(n)\geq \frac{1}{n}-\frac{1}{n^{4}}.    
\end{equation}

We first couple $W,\Y,\Z$ together and show there exists a coupling $Q_{\mathcal{X}}$ of $W,\Y,\Z$ such that 
\begin{equation}\label{eq:property1 of QX}
Q_{\mathcal{X}}\left[W=\Y\text{ or }W=\Z\right]=1,\quad\mbox{ and }\quad\ 
Q_{\mathcal{X}}\left[\Y=j,W=\Z=i\right]\leq\frac{1}{n^{4}},\mbox{ for all }j\neq i.
\end{equation}
This coupling is constructed via a $n$-step algorithm, which is described as follows. 

We first set $n^2$ non-negative variables $\textbf{m}_{i,j},1\leq i,j\leq n$ and two non-negative variables $\textbf{b}$, $\textbf{q}$, and let $\textbf{i}$ be a variable taking values in $\mathcal X$. We remark that here $\textbf{b}$, $\textbf{q}$ and $\textbf{i}$ are short for ``balance'', ``quota'' and ``index'' respectively.

Initially, we set $\textbf{m}_{i,j} = 1/n^2$, $\forall 1\leq i,j\leq n$, $\textbf{b} = 1/n - \mu_W(n)$, $\textbf{q} = 0$, and $\textbf{i} = n$.

For each $1\leq i \leq n$, at the $i$-th step of the algorithm, let
\begin{equation}\label{eq:W equals Y}
Q_{\mathcal{X}}\left[W=\Y=i,\Z=j\right]=\textbf{m}_{i,j},\mbox{ for all }1\leq j\leq n,\quad \mbox{ and }\quad\mathbf{q} = \mu_{W}(i)-\sum_{j=1}^{n}\textbf{m}_{i,j}.
\end{equation}
If $\textbf{q}=0$, end the $i$-th step.
Otherwise, 
\begin{itemize}
    \item if $\textbf{b}\geq \textbf{q}$, then let 
    \begin{equation}\label{eq:W equals Z enough}
    Q_{\mathcal{X}}\left[\Y=\textbf{i},W=\Z=i\right]=\textbf{q},    
    \end{equation}
    and update $\textbf{m}_{\textbf{i},i}$ to $\textbf{m}_{\textbf{i},i}-\textbf{q}$, update $\textbf{b}$ to $\textbf{b}-\textbf{q}$, and end the $i$-th step;
    \item if $\textbf{b}<\textbf{q}$, then let 
    \begin{equation}\label{eq:W equals Z not enough}
        Q_{\mathcal{X}}\left[\Y=\textbf{i},W=\Z=i\right]=\textbf{b},    
    \end{equation}
    and update $\textbf{m}_{\textbf{i},i}$ to $\textbf{m}_{\textbf{i},i}-\textbf{b}$, update $\textbf{q}$ to $\textbf{q}-\textbf{b}$, update $\textbf{i}$ to $\textbf{i}-1$, update $\textbf{b}$ to $1/n-\mu_{W}(\textbf{i})$ (with the updated $\textbf{i}$), and then again check whether $\textbf{b}\geq \textbf{q}$. We proceed this process until $\textbf{q}=0$ (which will happen since $\mu_{W}$ and $\overline{\mu}$ are both probability measures), and then end the $i$-th step.
\end{itemize}

For every $1\leq i\leq n$, summing all the collected equations \eqref{eq:W equals Y}-\eqref{eq:W equals Z not enough} in which $W=i$ yields $\mu_{W}(i)$, and similarly we see that $\Y$ and $\Z$ are mutually independent, both having distribution $\mu$. The condition \eqref{eq:property1 of QX} is satisfied thanks to \eqref{eq:bound on muW}. 

We now expand the above coupling $Q_{\mathcal{X}}$ so that it encompasses the all random variables $W,\Y,\Z,g,h$, where $g,h$ are two Bernoulli random variables that are independent from $\Y,\Z$ and have respective parameters $1-1/n$ and $1/n$. The only additional requirement is that under $Q_{\mathcal{X}}$, 
\begin{equation}\label{eq:property3 of QX}
Q_{\mathcal{X}}\left[\Y=j,W=\Z=i, g=0,h=1\right]=Q_{\mathcal{X}}\left[\Y=j,W=\Z=i\right]\leq\frac{1}{n^{4}},\quad\forall j\neq i. \end{equation}
Note that such $Q_{\mathcal{X}}$ exists thanks to \eqref{eq:property1 of QX}. It then suffices to argue that the final coupling $Q_{\mathcal{X}}$, which features properties \eqref{eq:property1 of QX} and \eqref{eq:property3 of QX}, satisfies \eqref{eq:horizontal independence coupling}. Indeed, by \eqref{eq:property1 of QX}, $W\neq\Y$ implies $W=\Z$. Then by \eqref{eq:property3 of QX}, combining $W\neq \Y$ and $W=\Z$ further implies $g=0$ and $h=1$, and the conclusion follows.
\end{proof}

Let us denote the uniform distribution on torus by
\begin{equation}\label{eq:def of qY}
q_{\mathbb{T}}:=\frac{1}{N^{d}}\sum_{x\in\mathbb{T}}\delta_{x}.   
\end{equation}
In order to use the lemma above to prove our horizontal independence result \Cref{prop:horizontal independence}, we first explain the following standard lemma on simple random walks and Markov chain mixing.
Here we recall the stopping times $R_{\ell}^{\A}$ and $D_{\ell}^{\A}$ in \eqref{eq:def of return and departure for mathrmA} where $\A$ and $\widetilde{\A}$ are cylinders defined in \eqref{eq:def of mathrmA and widetildemathrmA}. 
\begin{lemma}\label{lem:horizontal mixing}
There exists a positive constant $c_{34}$ such that, for any $x\in\partial\widetilde{\A}$ and the simple random walk $(X_{n})_{n\geq 0}=\{(Y_n,Z_n)\}_{n\geq 0}$ on $\mathbb{E}$ started at $x$, 
\begin{equation}\label{eq:horizontal mixing}
\mathrm{TV}(Y_{R_{1}^{\A}},q_{\mathbb{T}})\leq \frac{1}{c_{34}}\exp(-c_{34}\log^{2}N).
\end{equation}
\end{lemma}
\begin{proof}
Since this lemma is quite standard, we only sketch the proof and omit the details. We consider $(\widehat{Y}_{n})_{n\geq 0}$ and $(\widehat{Z}_{n})_{n\geq 0}$ as the non-lazy skeleton chain of $(Y_{n})_{n\geq 0}$ and $(Z_{n})_{n\geq 0}$, and use $R_{\widehat{Y}}$ and $R_{\widehat{Z}}$ to denote the number of steps $(Y_{n})_{n\geq 0}$ and $(Z_{n})_{n\geq 0}$ move before time $R_{1}^{\A}$.
Now by standard reflection principle and Azuma-Hoeffding's inequality on the martingale $(\widehat{Z}_{n})_{n\geq 0}$,
\begin{equation}\label{eq:bound for RwidehatZ2}
    P_{x}^{N}\left[R_{\widehat{Z}}\leq [N^{2}\log^{2}N]\right]\leq \frac{1}{c_{34}}\exp\left(-c_{34}\cdot\left(N\log^{2}N\right)^{2}/({N^{2}\log^{2}N})\right)=\frac{1}{c_{34}}e^{-c_{34}\log^{2}N}.
\end{equation}
Since $R_{\widehat{Y}}$ equals a sum of $R_{\widehat{Z}}$ i.i.d.~geometric random variables on $\{0,1,2,\dots\}$ with parameter $1/(d+1)$, by exponential Chebyshev's inequality,
\begin{equation}\label{eq:bound for RwidehatY}
P_{x}^{N}\left[R_{\widehat{Y}}\leq [N^{2}\log^{2}N]\right]\leq \frac{1}{c_{34}}\exp(-c_{34}\log^{2}N).    
\end{equation}

Finally, it follows from standard estimates with respect to the spectral gap of torus and Markov chain mixing (e.g.~\cite[Section 12]{LP17}) that when $R_{\widehat{Y}}\geq N^{2}\log^{2}N$, we have
\begin{equation}\label{eq:mixing and spectral gap}
\mathrm{TV}(Y_{R_{1}^{\A}},q_{\mathbb{T}})=\mathrm{TV}(\widehat{Y}_{R_{\widehat{Y}}},q_{\mathbb{T}})\leq N^{d}\exp(-CN^{-2}R_{\widehat{Y}})=\frac{1}{c_{34}}\exp\left(-c_{34}\log^{2}N\right).   
\end{equation}
Combining \eqref{eq:bound for RwidehatY} and \eqref{eq:mixing and spectral gap} then yields \eqref{eq:horizontal mixing}.
\end{proof}

The proof of horizontal independence follows from applying \Cref{lem:horizontal independence,lem:horizontal mixing} inductively. 
\begin{proof}[Proof of \Cref{prop:horizontal independence}]
We take $\mathcal{X}$ in \Cref{lem:horizontal independence} as the torus $\mathbb{T}$, and construct $Q_{1}$ given $\{Z_{R_{\ell}^{\A}},Z_{D_{\ell}^{\A}}\}_{\ell\geq 2}$.  We also take $N\geq c_{26}=c_{26}(d)$ so that $e^{-c_{34}N^{2}\log^{2}N}/c_{34}\leq N^{-4d}$ (recall \eqref{eq:horizontal mixing}) and $N\geq 100$. 

When $\ell=2$, by \Cref{lem:horizontal mixing} and the strong Markov property at time $D_{1}^{\A}$, we have $\text{TV}(Y_{R_{2}^{\A}},q_{\mathbb{T}})\leq N^{-4d}$. It then follows from \Cref{lem:horizontal independence} that we can couple the starting point $x_W,x_{\Y},x_{\Z}$ of $W_{2}^{\A},\Y_{2}^{\A},\Z_{2}^{\A}$ such that
\begin{equation}\label{eq:starting point coupling}
\{g_{2}x_{\Y}\}\subseteq \{x_W\}\subseteq \{x_{\Y},h_{2}x_{\Z}\}.  
\end{equation}
Then we couple $W_{2}^{\A},\Y_{2}^{\A},\Z_{2}^{\A},g_{2},h_{2}$ together in the following way. If $x_W=x_{\Y}$, then let $W_{2}^{\A}$ equal $\Y_{2}^{\A}$ and let $\Z_{2}^{\A}$ move independently. Otherwise, by \eqref{eq:starting point coupling}, $x_W=x_{\Z}$. We then let $W_{2}^{\A}$ equal $\Z_{2}^{\A}$ and let $\Y_{2}^{\A}$ move independently. In either case $W_{2}^{\A},\Y_{2}^{\A},\Z_{2}^{\A}$ are required to be identically distributed as $P_{Z_{R_2^{\A}},Z_{D_2^{\A}}}^{N}$. We then have
\begin{equation}\label{eq:horizontal coupling for 2}
\{g_{2}\Y_{2}^{\A}\}\subseteq \{W_{2}^{\A}\}\subseteq \{\Y_{2}^{\A},h_{2}\Z_{2}^{\A}\}.  
\end{equation}
Now for a general $3\leq\ell\leq K$, we couple $\W_{\ell}^{\A},\Y_{\ell}^{\A},\Z_{\ell}^{\A},g_{\ell},h_{\ell}$ conditioned on all the random variables with subscripts smaller than $\ell$ in the same manner as how we couple $W_{2}^{\A},\Y_{2}^{\A},\Z_{2}^{\A},g_{2},h_{2}$ conditioned on $W_{1}^{\A}$, again using \Cref{lem:horizontal independence,lem:horizontal mixing}. Proceeding $(K-1)$ steps of coupling and we finally obtain $Q_{1}$.

Note that under $Q_{1}$, conditioned on the $\mathbb Z$-height of return and departure points $\big\{Z_{R_{\ell}^{\A}},Z_{D_{\ell}^{\A}}\big\}_{\ell\geq 2}$, the correlations in the sequence $\{W_{\ell}^{\A}\}_{\ell \geq 2}$ only comes from $\big\{Y_{R_{\ell}^{\A}},Y_{D_{\ell}^{\A}}\big\}_{\ell\geq 2}$ (the $\mathbb T$-coordinate of return and departure points). Therefore, in the above paragraph, $Y_{R_{\ell+1}^{\A}}$ and the $\ell$-th step of coupling both have correlations with the former $(\ell-1)$ steps of coupling, but the marginal law of $(\Y_{\ell+1}^{\A},\Z_{\ell+1}^{\A},g_{\ell+1},h_{\ell+1})$ does not. In other words, conditioned on the former $(\ell-1)$ steps of coupling, the marginal conditioned law of $(\Y_{\ell+1}^{\A},\Z_{\ell+1}^{\A},g_{\ell+1},h_{\ell+1})$ remains unchanged. This guarantees the conditional independence of $\Y^{\A},\Z^{\A},g,h$-type variables given $\{Z_{R_{\ell}^{\A}},Z_{D_{\ell}^{\A}}\}_{\ell\geq 2}$ under $Q_{1}$. 
\end{proof}

We move on to the proof of \Cref{prop:vertical independence}, which is similar to that of \cite[Proposition 3.1]{Szn09b}. 
\begin{proof}[Proof of \Cref{prop:vertical independence}]
We only prove \eqref{eq:vertical independence 1}, and \eqref{eq:vertical independence 2} follows similarly. Let 
\begin{equation}\label{eq:def of Gamma}
\Gamma:=\left\{(r_{N},h_{N}),(r_{N},-h_{N}),(-r_{N},h_{N}),(-r{N},-h_{N})\right\}.
\end{equation}
For each $(z_{1},z_{2})\in\Gamma$, we define a sequence of i.i.d.~$\mathscr{T}_{\widetilde{\A}}$-valued random variables $\{\zeta_{\ell}^{z_{1},z_{2}}\}_{\ell\geq 1}$ with same distribution as $X._{\land T_{\widetilde{\A}}}$ under $P_{z_{1},z_{2}}^{N}$. We further require that the four sequences $\{\zeta_{\ell}^{z_{1},z_{2}}\}_{\ell\geq 1}$ are independent from each other, and are also independent from Bernoulli variables $\{g_{\ell},h_{\ell}\}_{\ell\geq 1}$. 
We also define a sequence of i.i.d.~random variables $\left\{(R_{\ell},D_{\ell})\right\}_{\ell\geq 2}$ on $\Gamma$ with distribution
\begin{equation}\label{eq:def of iid R and D}
P[(R_{\ell},D_{\ell}) = (r_{N}, \pm h_{N})] = \frac{(h_{N}\pm r_{N})}{2h_{N}},\quad  P[(R_{\ell},D_{\ell}) = (-r_{N}, \pm h_{N})] = \frac{(h_{N}\mp r_{N})}{2h_{N}}, 
\end{equation}
and is independent from $\big\{Z_{R_{\ell}^{\A}},Z_{D_{\ell}^{\A}}\big\}_{\ell\geq 1}$ and all the other random variables that have appeared. 

We now define two counting functions that for $(z_{1},z_{2})\in\Gamma,k\geq 2,$
\begin{align}
N_{k}(z_{1},z_{2})&=\big|\{2\leq \ell\leq k:\big(Z_{R_{\ell}^{\A}},Z_{D_{\ell}^{\A}}\big)=(z_{1},z_{2})\big|,\label{eq:def of couting N}\\
M_{k}(z_{1},z_{2})&=\big|\{2\leq \ell\leq k:(R_{\ell},D_{\ell})=(z_{1},z_{2})\big|.\label{eq:def of counting M}
\end{align}
We take our coupling $Q_{2}$ as a coupling of all the random variables $\{g_{\ell},h_{\ell}\}_{\ell\geq 1}$, $\{\Y_{\ell}^{\A}\}_{\ell\geq 1}$, $\{\widetilde{\Y}_{\ell}^{\A}\}_{\ell\geq 1}$, $\{\zeta_{\ell}^{z_{1},z_{2}}\}_{\ell\geq 1},(z_{1},z_{2})\in\Gamma$, and $\{(R_{\ell},D_{\ell})\}_{\ell\geq 1}$ (and also $\Z$-type variables when proving \eqref{eq:vertical independence 2}) so that
\begin{equation}\label{eq:coupling between Y and widetildeY}
\Y_{\ell}^{\A}=\zeta_{N_{\ell}(Z_{R_{\ell}^{\A}},Z_{D_{\ell}^{\A}})}^{Z_{R_{\ell}^{\A}},Z_{D_{\ell}^{\A}}}\quad \text{and}\quad \widetilde{\Y}_{\ell}^{\A}=\zeta_{M_{\ell}(R_{\ell},D_{\ell})}^{R_{\ell},D_{\ell}}.\end{equation}
Under this coupling, the proof of \eqref{eq:vertical independence 1} essentially boils down to the proving the relations between $M_{\widetilde{K}_{1}}$, $N_{K}$ and $M_{\widetilde{K}_{2}}$. Indeed, it suffices to show that there exist $c=c(u, u_1, u_2)$, $c' =c'(u, u_1, u_2)$ and $\widetilde{c}=\widetilde{c}(u, u_1, u_2)$such that for any $(z_{1},z_{2})\in\Gamma$, 
\begin{align}
&Q_{2}\Big[M_{\widetilde{K}_{1}}(z_{1},z_{2})\geq \frac{1}{4}\Big(1+\frac{1}{100}\xi\Big)\widetilde{K}_{1}\Big]\leq c^{-1}\exp(-c\widetilde{K}_{1})=\widetilde{c}^{-1}\exp\left(-\widetilde{c}K\right);  \label{eq:large deviation MK1}\\
&Q_{2}\Big[M_{\widetilde{K}_{2}}(z_{1},z_{2})\leq \frac{1}{4}\Big(1-\frac{1}{100}\xi\Big)\widetilde{K}_{2}\Big]\leq c^{-1}\exp(-c\widetilde{K}_{2})= \widetilde{c}^{-1}\exp\left(-\widetilde{c}K\right);  \label{eq:large deviation MK2}\\
&Q_{2}\left[\frac{1}{4}\Big(1+\frac{1}{100}\xi\Big)\widetilde{K}_{1}\leq N_{K}(z_{1},z_{2})\leq \frac{1}{4}\Big(1-\frac{1}{100}\xi\Big)\widetilde{K}_{2}\right]\geq 1-c'^{-1}\exp(-c'K). \label{eq:large deviation NK}
\end{align}

Now for any $(z_{1},z_{2})\in\Gamma$, since $r_{N}=o(h_{N})$ as $N$ tends to infinity, it follows from standard exponential Chebyshev's inequality on sum of i.i.d.~Bernoulli random variables that \eqref{eq:large deviation MK1} and \eqref{eq:large deviation MK2} hold for $N\geq C(u,u_{1},u_{2})$ and any $(z_{1},z_{2})\in\Gamma$.
In addition, \eqref{eq:large deviation NK} is a large deviation bound for a Markov chain on a finite space since the sequence $\{(Z_{R_{\ell}},Z_{D_{\ell}})\}_{2\leq\ell\leq K}$ actually forms a Markov chain with four states that has invariant distribution same to that of $(R_{\ell},D_{\ell})$. Using Sanov's theorem for the pair empirical measure of Markov chains \cite[Theorem 3.1.13]{DZ10}, the upper bound in \eqref{eq:large deviation NK} is a result of \cite[(3.35)]{Szn09b} for $N\geq c'(u,u_{1},u_{2})$ and any $(z_{1},z_{2})\in \Gamma$, and the lower bound in \eqref{eq:large deviation NK} then follows since if some $(z_{1},z_{2})\in \Gamma$ violates the lower bound, then there exists another $(z_{1}',z_{2}')\in\Gamma$ that violates the upper bound (with $\xi/100$ replaced by $\xi/300$). 
\end{proof}

The proof of \Cref{prop:Poissonization} is similar to that of \cite[Proposition 4.1]{Szn09b}.
\begin{proof}[Proof of \Cref{prop:Poissonization}]
We first reorder the excursions $\widetilde{\Y}_{\ell}^{\A},\widetilde{\Z}_{\ell}^{\A},\ell\geq 2$ into a new sequence $\big\{\widetilde{\X}_{\ell}^{\A}\big\}_{\ell\geq 1}$ in a specific fashion described below, which is dependent on $\{g_{\ell}\}_{\ell\leq \widetilde{K}_{1}}$ and $\{h_{\ell}\}_{\ell\leq \widetilde{K}_{2}}$. We first collect excursions $\widetilde{\Y}_{\ell}^{\A}$ with $2\leq\ell\leq \widetilde{K}_{1}$ and $g_{\ell}=1$ in increasing order of subscript, and then gather up the remaining excursions in $\big\{\widetilde{\Y}_{\ell}^{\A}\big\}_{2\leq \ell\leq \widetilde{K}_{2}}$ in increasing order of subscript. After this two steps, we collect excursions $\widetilde{\Z}_{\ell}^{\A}$ such that $\ell\leq\widetilde{K}_{2}$ and $h_{\ell}=1$ by its original order, and then put the remaining $\widetilde{\Z}_{\ell}^{\A}$'s with $\ell\leq \widetilde{K}_{2}$ in increasing order of subscript. Finally we gather those $\widetilde{\Y}_{\ell}^{\A}$'s, $\widetilde{\Z}_{\ell}^{\A}$'s with $\ell>\widetilde{K}_{2}$ in increasing order of subscript. 

With this, we now construct our coupling $Q_{3}$. Conditioned on $\{g_{\ell}\}_{\ell\geq \widetilde{K}_{1}}$ and $\{h_{\ell}\}_{\ell\leq \widetilde{K}_{2}}$, for every positive integer $k$, we require $\widetilde{\W}_{k}^{\A}$ to be equal to $\widetilde{\X}_{k}^{\A}$. Note that since $\{g_{\ell},h_{\ell}\}_{\ell\geq 1}$ is independent from $\big\{\widetilde{\Y}_{\ell}^{\A},\widetilde{\Z}_{\ell}^{\A}\big\}_{\ell\geq 1}$, the law of $\big\{\widetilde{\W}_{\ell}^{\A}\big\}_{\ell\geq 1}$ is indeed that of a sequence of i.i.d.~excursions from $\A$ to $\partial\widetilde{\A}$ with the same distribution as $X._{\land T_{\widetilde{\A}}}$ under $P_{q}^{N}$. 
Thanks to the reordering, \eqref{eq:Poissonization 1} and $\eqref{eq:Poissonization 2}$ can now be implied by
\begin{equation}\label{eq:proof of Poissonization 1}
Q_{3}\Big[\sum_{\ell=2}^{\widetilde{K}_{1}}g_{\ell}\geq \widetilde{J}_{1}\Big]\geq 1-\frac{1}{c_{28}}\text{exp}(-c_{28}K), \text{ and } Q_{3}\Big[\widetilde{K}_{2}+\sum_{\ell=2}^{\widetilde{K}_{1}}h_{\ell}\leq \widetilde{J}_{1}'\Big]\geq 1-\frac{1}{c_{28}}\text{exp}(-c_{28}K). 
\end{equation}
Recall \eqref{eq:def of widetildeK1 and widetildeK2} for the definition of $\widetilde{K}_{1}$ and $\widetilde{K}_{2}$, that $\{g_{\ell}\}_{\ell\geq 1}$ and $\{h_{\ell}\}_{\ell\geq 1}$ are two sequences of i.i.d.~Bernoulli variables with respective parameter $1-N^{-d}$ and $N^{-d}$, and \eqref{eq:def Poisson intensity} for the intensities of Poisson random variables $\widetilde{J}_{1}$ and $\widetilde{J}_{1}'$ respectively. The above inequalities then follow from standard exponential Chebyshev's inequalities on Bernoulli and Poisson random variables.
\end{proof}

We then move on to the proof of \Cref{prop:Handling the first excursion}, which is essentially identical to the proof of \cite[(3.11)]{Szn09b}; see \cite[(3.23) and (3.24)]{Szn09b}. We still include a proof for completeness.
\begin{proof}[Proof of \Cref{prop:Handling the first excursion}]
Recall \eqref{eq:def of mathrmA and widetildemathrmA} and \eqref{eq:def of I and widetildeI} for the definition of cylinders $\A,\widetilde{\A}$ and intervals $\mathrm{I},\widetilde{\mathrm{I}}$. It suffices to consider the case $|\mathrm{z}_{c}|\leq r_{N}$, where we construct our coupling $Q_{4}$ as follows. 

We call an excursions in $\widetilde{\W}_{\ell}^{\A}$ ``good'' if it passes $\mathbb{T}\times\{0\}$ before leaving $\widetilde{\A}$. 
Now if there exists at least one ``good'' excursion in $\big\{\widetilde{\W}_{\ell}^{\A}\big\}_{\widetilde{J}_{1}'\leq \ell\leq \widetilde{J}_{2}}$, we let $W_{1}^{\A}$ be the part of the first good excursion after reaching level $\mathbb{T}\times\{0\}$, under which circumstance the inclusion in \eqref{eq:Handling the first excursion} holds. Otherwise we let $W_{1}^{\A}$ run freely. 

With this construction, it suffices to prove that, with probability larger than $1-c_{29}^{-1}\text{exp}(-c_{29}K)$, there exists at least one good excursion in $\big\{\widetilde{\W}_{\ell}^{\A}\big\}_{\widetilde{J}_{1}'\leq \ell\leq \widetilde{J}_{2}}$. Note that by a standard calculation on one-dimensional simple random walk and the assumption $|\mathrm{z}_{c}|\leq r_{N}$, the chance for each excursion $\widetilde{\W}_{\ell}^{\A}$ (which has the same law as $X._{\land T_{\widetilde{\A}}}$ under $P_{q}^{N}$) to reach level $\mathbb{T}\times\{0\}$ before leaving $\widetilde{\A}$ is larger than $1/4$, and the conclusion thus follows from an exponential Chebyshev's inequality on the Poisson random variable $\widetilde{J}_{2}-\widetilde{J}_{1}'$. 
\end{proof}

The proof of \Cref{prop:extraction} follows from combining the thinning property of Poisson point processes with \cite[Lemma 1.1]{Szn09a}.
\begin{proof}[Proof of \Cref{prop:extraction}]
By \cite[Lemma 1.1]{Szn09a}, for $\K\subseteq \mathbb{T}\times(\mathrm{z}_{c}-r_{N},\mathrm{z}_{c}+r_{N})$ and each $x\in \K$, 
\begin{equation}
P_{q}^{N}\left[H_{\K}<T_{\widetilde{\A}},X_{H_{\K}}=x\right]=(d+1)\frac{h_{N}-r_{N}}{N^{d}}e_{\K,\widetilde{\A}}(x),\label{eq:hitting distribution of Pq}
\end{equation}
and as an application of the Markov property,
\begin{equation}
P_{q}^{N}\left[H_{\K}<T_{\widetilde{\A}},X_{H_{\K}+}.\in dw\right]=(d+1)\frac{h_{N}-r_{N}}{N^{d}}P_{e_{\K,\widetilde{\A}}}^{N}(dw).    \label{eq:Markov property of Pq}
\end{equation}
Recall that $\widetilde{J}_{1}$ and $\widetilde{J}_{2}-\widetilde{J}_{1}$ are two independent Poisson random variables with intensities $\widetilde{\lambda}_1$ and $\widetilde{\lambda}_2 - \widetilde{\lambda}_1$ respectively; see \eqref{eq:def Poisson intensity} and \eqref{eq:def Poisson intensity2}. Taking $\K=\B$ in \eqref{eq:hitting distribution of Pq} and \eqref{eq:Markov property of Pq}, combining the thinning property of Poisson point processes yields the conclusion.
\end{proof}

The proofs of \Cref{prop:lower truncation,prop:upper interlacement} are very similar. In the following, we use $\succeq$ and $\preceq$ to denote the relation of stochastic domination.
\begin{proof}[Proofs of \Cref{prop:lower truncation,prop:upper interlacement}] 
Note that although each excursion $\widetilde{\W}\in\text{supp}(\widetilde{\mu}_{1})$ may travel back and forth between $\B$ and $\partial\D$ multiple times, the truncated version $\widetilde{\W}._{\land T_{\D}}$ only takes the first excursion into account, and thus the second inclusion in \Cref{prop:lower truncation} is trivial.
Similarly, the variable $N_{u_{2}}(\B)$ appeared in \Cref{prop:upper interlacement} is a Poissonian sum of geometric random variables starting from 1, where each geometric variable represents the number of excursions in one random walk that hits $\B$, and we bound every geometric variable from below by 1, that is, we only consider the first excursion. 

Thanks to the property of Poisson point measures, to prove \eqref{eq:lower truncation} and \eqref{eq:upper interlacement}, it suffices to prove the following two stochastic domination relations with respect to the intensity measures:
\begin{equation}\label{eq:lower domination and upper domination}
\Big(1-\frac{5}{7}\xi\Big)e_{\B}(\cdot)\preceq\Big(1-\frac{3}{7}\xi\Big)\Big(1-\frac{r_{N}}{h_{N}}\Big)e_{\B,\widetilde{\A}}(\cdot)\quad\text{and}\quad
\Big(1+\frac{6}{7}\xi\Big)e_{\B,\widetilde{\A}}(\cdot)
\preceq \left(1+\xi\right)e_{\B}(\cdot).
\end{equation}
Running the same techniques as the proof of \cite[Lemma 4.4]{Szn09a}, for all $N\geq c_{30}=c_{30}(u.u_{1}.u_{2})>0$ and $x\in\partial^{\text{int}} \B$, 
\begin{equation}\label{eq:lower equilibrium}
e_{\B,\widetilde{\A}}(x)\geq e_{\B}(x)\Big(1-c\frac{f(N)^{d-1}\log^{2}N}{N^{d-1}}\Big)\overset{\eqref{eq:requirement for f and g}, d\geq 2}{\geq }e_{\B}(x)\Big(1-\frac{c}{\log N}\Big)\geq \frac{(1-\frac{5}{7}\xi)}{(1-\frac{3}{7}\xi)(1-\frac{r_{N}}{h_{N}})}.
\end{equation}
Similarly, by running the same techniques as the proof of \cite[Proposition 6.1]{Szn09b}, for all $N\geq c_{32}=c_{32}(u.u_{1}.u_{2})>0$ and $x\in\partial^{\text{int}} \B$, 
\begin{equation}\label{eq:upper equilibrium}
e_{\B,\widetilde{\A}}(x)\leq e_{\B}(x)\Big(1+c'\frac{f(N)^{d-1}}{N^{d-1}}\Big)\overset{\eqref{eq:requirement for f and g}, d\geq 2}{\leq }e_{\B}(x)\Big(1+\frac{c'}{\log N}\Big)\leq \frac{(1+\xi)}{1+\frac{6}{7}\xi}.
\end{equation}
The conclusion then follows.
\end{proof}

The proofs of \Cref{prop:lower interlacement,prop:upper truncation} are very similar.
\begin{proof}[Proofs of \Cref{prop:lower interlacement,prop:upper truncation}]
We first prove \Cref{prop:lower interlacement}. We first define two Poisson random variables $\widehat J_0, \widehat J_1$ satisfying
\begin{equation}\label{eq:intensity of widehatJ1}
\text{intensity of }\widehat{J}_{0} = u_{1}\text{cap}(\B)<\Big(1-\frac{5}{7}\xi\Big)u\cdot\text{cap}(\B)=\text{intensity of }\widehat{J}_{1}.
\end{equation}
Note that 
the variable $N_{u_{1}}(\B)$ can be stochastically dominated by the sum of $\widehat J_0$ i.i.d.~geometric random variables supported on $\{1, 2, \cdots\}$ with success probability $1-\sup_{x\in\partial \D}P_{x}[H_{\B}<\infty]$.
Moreover, let $\big\{\widetilde{Z}_{\ell}^{\B}\big\}_{\ell\geq 1}$ denote a sequence of i.i.d.~excursions with the same law as $X._{\land T_{\D}}$ under $P_{\overline{e}_{\B}}$, then
the set $\big\{\widetilde{Z}_{\ell}^{\B}\big\}_{\ell\leq\widehat{J}_{1}}$ is equal in distribution with $\text{supp}(\widehat{\mu}_{1})$.

By applying \Cref{prop:hitting distribution two boxes simple} with $A=\B$ and $U=\D$, we can show that when $N\geq C(\xi)(\geq c_{1}(\xi/10^{3}))$, for every $x\in\partial\D$ and $y\in\partial^{\text{int}}\B$, 
\begin{equation}\label{eq:hitting distribution Zd simple}
\Big(1-\frac{1}{10^{3}}\xi\Big)\overline{e}_{\B}(y)\leq P_{x}[X_{H_{\B}}=y\mid H_{\B}<\infty]\leq \Big(1+\frac{1}{10^{3}}\xi\Big)\overline{e}_{\B}(y).  \end{equation}
Then by soft local time techniques, similar to the proof of \Cref{prop:coupling W and widetildeZ}, there exists a coupling $Q_{7}$, an event $E^{\B}$ and a sufficiently small constant $c_{31}$ with
\begin{equation}\label{eq:probability of EB}
Q_{7}\left[E^{\B}\right]\geq 1-\exp(-10c_{31}\text{cap}(\B)),    
\end{equation}
so that on $E^{\B}$, for all $m\geq u_{1}\text{cap}(\B)/10$,
\begin{equation}\label{eq:inclusion coupling 2}
\Big\{Z_{1}^{\B},\dots,Z_{m}^{\B}\Big\}\subseteq \left\{\widetilde{Z}_{1}^{\B},\dots,\widetilde{Z}_{(1+\frac{1}{20}\xi)m}^{\B}\right\}. \end{equation}
By this coupling, to conclude \eqref{eq:lower interlacement}, it suffices to bound $N_{u_{1}}(\B)$ from above and bound $\widehat{J_{1}}$ from below.  
Using standard hitting probability estimates of simple random walk, by $d\geq 2$ and \eqref{eq:requirement for f and g} we obtain that 
\begin{equation}\label{eq:upper hitting}
\sup_{x\in\D}P_{x}[H_{\B}<\infty]\leq c\frac{f(N)^{d-1}}{g(N)^{d-1}}\leq c\frac{1}{\left(\log^{3}N\right)^{d-1}}\leq \frac{c}{\log N}.    
\end{equation}
The bound \eqref{eq:lower interlacement} then follows from \eqref{eq:intensity of widehatJ1}, the success probability \eqref{eq:upper hitting} of the aforementioned geometric variables, and standard exponential Chebyshev's inequalities on Poisson random variables and geometric random variables.

We then explain \Cref{prop:upper truncation}, which is very similar to the statements above. Indeed, combining \eqref{eq:lower equilibrium}, \eqref{eq:upper equilibrium} and \eqref{eq:hitting distribution Zd simple} yields that for some $C'= C'(\xi)\geq c_{1}(\xi/10^{3})$, when $N\geq C'$, for every $x\in \partial\D$ and $y\in \partial^{\text{int}}\B$,
\begin{equation}\label{eq:hitting distribution torus simple}
\Big(1-\frac{1}{10^{2}}\xi\Big)\overline{e}_{\B,\widetilde{\A}}(y)\leq P_{x}^{N}[X_{H_{\B}}=y\mid H_{\B}<\infty]\leq \Big(1+\frac{1}{10^{2}}\xi\Big)\overline{e}_{\B,\widetilde{\A}}(y).
\end{equation}
Following the similar steps as in \eqref{eq:probability of EB} and \eqref{eq:inclusion coupling 2}, it suffices to establish the bounds with respect to the number of excursions in $\text{supp}(\widetilde{\mu}_{2})^{\B}$ and $\text{supp}(\widehat{\mu}_{2})$.
Note that the former can be stochastically dominated by a Poissonian sum of i.i.d.~geometic random variables supported on $\{1, 2, \cdots\}$ with success probability~$1-\sup_{x\in\D}P_{x}^{N}[H_{\B}<T_{\widetilde{\A}}]$, and the latter is a Poisson random variable. Using the same techniques as in the proof of \cite[(4.16)]{Szn09a} or \cite[Lemma 5.22]{Szn09b}, by $d\geq 2$ and \eqref{eq:requirement for f and g} we have
\begin{equation}\label{eq:lower hitting}
\sup_{x\in\D}P_{x}^{N}[H_{\B}<T_{\widetilde{\A}}]\leq c'\frac{f(N)^{d-1}\log^{2}N}{g(N)^{d-1}}\leq c\frac{\log^{2}N}{\left(\log^{3}N\right)^{d-1}}\leq \frac{c'}{\log N},    
\end{equation}
and the conclusion similarly follows from exponential Chebyshev's inequalities.
\end{proof}

\subsection{Adapted proofs for \texorpdfstring{\Cref{prop:very strong coupling,prop:very strong coupling bias,prop:strong coupling}}{prop:very strong coupling,prop:very strong coupling bias,prop:strong coupling}}\label{subsec:adapted proof very strong coupling}
In this subsection we adapt the proof of \Cref{thm:very strong coupling S} to prove the couplings appeared in \Cref{prop:very strong coupling,prop:very strong coupling bias,prop:strong coupling} for simple random walk, biased random walk and conditional biased walk respectively in \Cref{subsubsec:simple,subsubsec:bias upper,subsubsec:bias lower}. We will frequently compare between our goals with \Cref{thm:very strong coupling S} to highlight the subtle differences as well as necessary adaptions. 

\subsubsection{The proof of \Cref{prop:very strong coupling}}\label{subsubsec:simple}
Let us first begin with the proof of \Cref{prop:very strong coupling}. We will take $\B = \mathsf B$ and $\D = \mathsf D$ in \Cref{thm:very strong coupling S}, where the side-lengths of $\mathsf B$ and $\mathsf D$ have been set as $[N/\log^3 N]$ and $[N/20]$ in \Cref{sec:bad local time}, which satisfy the requirements in \eqref{eq:requirement for f and g}. Therefore, this proposition can be seen as a special case of the upper bound of \Cref{thm:very strong coupling S}, except that the walk now starts from the origin instead of the uniform distribution on $\mathbb{T}\times\{0\}$. 

Consequently, it suffices to adapt the result \Cref{prop:Handling the first excursion} concerning the first excursion. However, with the distribution of $W_{1}^{\A}$ governed by $P_0^{N}$ instead of $P_{q_{0}}^{N}$, \eqref{eq:Handling the first excursion} may no longer be true if $0\in\A$ since with uniformly positive probability none of the excursions in $\big\{\widetilde{\W}_{\ell}^{\A}\big\}_{\widetilde{J}_{1}'\leq \ell\leq \widetilde{J}_{2}}$ hits $0$. 

The idea is that we only care about the traces of the excursion after hitting $\mathsf{B}$. 
For an excursion $W$ in $\{W^{\A}_1, \widetilde{\W}^{\A}_{\ell}, \ell\geq 1\}$, we use 
$H_{\mathsf{B}}(W)$, $T_{\widetilde{\A}}(W)$ for the entrance time of $W$ into $\mathsf B$ and the departure time of $W$ from $\widetilde{\A}$ respectively. We also denote by 
\begin{equation}
(W)^{\text{tr}}:=W_{[H_{\mathsf{B}}(W),T_{\widetilde{\A}}(W)]}\quad\text{and}\quad (W)_{H_{\mathsf B}}:=(W)_{H_{\mathsf B}(W)}
\end{equation}
for the $\mathscr{T}_{\widetilde{\A}}$-valued truncated excursion of $W$ from $H_{\mathsf B}(W)$ to $T_{\widetilde{\A}}(W)$, and for the point where excursion $W$ enters $\mathsf B$.
Then it actually suffices to prove the following weaker version of \eqref{eq:Handling the first excursion} that under some coupling $Q_{4}'$, there exists $c_{35} = c_{35}(u,u_{1},u_{2})>0$ such that 
\begin{equation}\label{eq:Handling the first excursion adaped simple}
Q_{4}'\left[(W_{1}^{\A})^{\text{tr}}\in \big\{(\widetilde{\W}_{\ell}^{\A})^{\text{tr}}\big\}_{\widetilde{J}_{1}'\leq \ell\leq \widetilde{J}_{2}}\right]\geq 1-\frac{1}{c_{35}}\exp\left(-c_{35}\text{cap}(\mathsf{B})\right).
\end{equation}
Here, the requirement \eqref{eq:large distance from mathsfB} plays an important role since by a similar argument in \eqref{eq:lower hitting}, the probability that $W_1^{\A}$ hits box $\B$ before exiting cylinder $\widetilde{\A}$ is small, and conditioned on this rare event, the hitting distributions for $W_{1}^{\A}$ and $\widetilde{\W}_{\ell}^{\A},\ell\geq 1$ on $\mathsf{B}$ are all approximately $e_{\mathsf{B},\widetilde{\A}}$ (see \eqref{eq:hitting distribution torus simple}).
\begin{proof}[Proof of adaptation of \Cref{prop:Handling the first excursion} \rm{(SRW case)}]
Thanks to \eqref{eq:hitting distribution torus simple} and \eqref{eq:lower hitting}, when $N\geq C(\xi)$, 
\begin{equation}\label{eq:Handling the first excursion stochastic domination simple}
P_{0}^{N}\left[\left(W_{1}^{\A}\right)_{H_{\mathsf{B}}}=\cdot,H_{\mathsf B}(W_{1}^{\A})<T_{\widetilde{\A}}(W_{1}^{\A})\right]\preceq \frac{C}{\log N}\left(1+\frac{1}{10^{2}}\xi\right)\overline{e}_{\mathsf{B},\widetilde{\A}}(\cdot).   
\end{equation}
Taking $\K=\mathsf{B}$ in \eqref{eq:hitting distribution of Pq}, we then obtain that $\big\{(\widetilde{\W}_{\ell}^{\A})_{H_{\mathsf{B}}}\big\}_{\widetilde{J}_{1}'\leq \ell\leq \widetilde{J}_{2}}$ is equal to a sequence of i.i.d.~variables with distribution $\overline{e}_{\mathsf{B},\widetilde{\A}}$, where the size of this sequence is a Poisson random variable $\widetilde{J}$ with
\begin{equation}\label{eq:intensity of widetildeJ}
\text{intensity of }\widetilde J = \frac{\xi}{7}\cdot u\left(1-\frac{r_{N}}{h_{N}}\right)\sum_{x\in\partial\mathsf{B}}e_{\mathsf{B},\widetilde{\A}}(x).    
\end{equation}

We now construct the coupling $Q_{4}'$ between the excursions $W_{1}^{\A}$ and $\big\{\widetilde{\W}_{\ell}^{\A}\big\}_{\widetilde{J}_{1}'\leq \ell\leq \widetilde{J}_{2}}$. If there exists an excursion in $\big\{\widetilde{\W}_{\ell}^{\A}\big\}_{\widetilde{J}_{1}'\leq \ell\leq \widetilde{J}_{2}}$ that hits $\mathsf B$, we denote $\widetilde{\W}$ by the first excursion. In this case, $(\widetilde{\W})_{H_{\mathsf{B}}}$ has distribution $\overline{e}_{\mathsf{B},\widetilde{\A}}$, and we can construct a coupling so that $(W_{1}^{\A})_{H_{\mathsf{B}}}=x\in\partial\mathsf{B}$ implies $(\widetilde{\W})_{H_{\mathsf{B}}}=x$, which is possible when $N\geq C(\xi)$ thanks to \eqref{eq:Handling the first excursion stochastic domination simple}. We then let $W_{1}^{\A}$ and $\widetilde{\W}$ run together after hitting $\mathsf{B}$ at the same point, and let all the other excursions in $\big\{\widetilde{\W}_{\ell}^{\A}\big\}_{\widetilde{J}_{1}'\leq \ell\leq \widetilde{J}_{2}}$ run freely. Under this coupling, we shall see that the event in \eqref{eq:Handling the first excursion adaped simple} holds as long as $\widetilde{J}\geq 1$. Therefore, we only need to prove that 
\begin{equation}\label{eq:large deviation on widetildeJ}
Q_{4}'[\widetilde{J}\geq 1]\geq 1-\frac{1}{c_{35}}\exp\left(-c_{35}\text{cap}(\mathsf{B})\right),  
\end{equation}
which holds by combining \eqref{eq:lower equilibrium}, \eqref{eq:intensity of widetildeJ} and the fact that $\widetilde{J}$ is a Poisson random variable.
\end{proof}

\subsubsection{The proof of \Cref{prop:very strong coupling bias}}\label{subsubsec:bias upper}
We move on to the proof of \Cref{prop:very strong coupling bias}, the coupling in the biased random walk case. We still take $\B = \mathsf B$ and $\D = \mathsf D$ in \Cref{thm:very strong coupling S} as in \Cref{subsubsec:simple}, which satisfy the requirements in \eqref{eq:requirement for f and g}. Then this proposition can be seen as a special case of the upper bound of \Cref{thm:very strong coupling S}, except that the walk now has an upward drift and starts from the origin instead of the uniform distribution on $\mathbb{T}\times\{0\}$. 
In the following, we use $P_{q_{z}}^{N,\alpha}$, $P_{q}^{N,\alpha}$, $P_{Z_{R_{\ell}},Z_{D_{\ell}}}^{N,\alpha}$ as the biased walk counterparts of measures $P_{q_{z}}^{N}$, $P_{q}^{N}$ and $P_{Z_{R_{\ell}},Z_{D_{\ell}}}^{N}$. 

The adapted proof contains three steps:
\begin{enumerate}[label=\textbf{Step \arabic*}]
\item \label{item:step1} Adapt \Cref{prop:horizontal independence,prop:vertical independence,prop:Poissonization}, where all the excursions are now considered under the law of biased walks instead of simple random walks.
\item \label{item:step2} Adapt \Cref{prop:Handling the first excursion}. Prove a stochastic domination control of a Poissonian number of biased random walk excursions in terms of a Poissonian number of simple random walk excursions. 
Then prove an adaptation of \eqref{eq:Handling the first excursion adaped simple} with $(W_{1}^{\A})^{\rm tr}$ extracted from the first excursions of biased walk, while $\{(\widetilde{\W}_{\ell})^{\rm tr}\}_{\ell\geq 1}$ being a sequence of i.i.d.~simple walk excursions. 
This step relies on the calculation of Radon-Nikodym derivatives, similar to those in the proof of \Cref{lem:hitting distribution three boxes bias} and \Cref{lem:coupling widetildeW with widetildeZ}.
\item \label{item:step3} Combine the above two steps with a slightly changed version of \eqref{eq:extraction 1} in \Cref{prop:extraction} and \Cref{prop:upper truncation,prop:upper interlacement} (with all the excursions still having the law of simple walks) to conclude.
\end{enumerate}

We first complete \ref{item:step1}, which proceeds in a similar way as \Cref{prop:horizontal independence,prop:vertical independence,prop:Poissonization}. In the similar proof structure, we first state the adaptations in these propositions: 
\begin{itemize}
\item The excursions $W_{\ell}^{\A}, {\ell\geq 1}$ are now extracted from the walk under law $P_{0}^{N,\alpha}$ instead of $P_{q_0}^{N}$. \label{item:adaptW}
\item Given $\big\{Z_{R_{\ell}^{\A}},Z_{D_{\ell}^{\A}}\big\}_{\ell\geq 2}$, for every $\ell\geq 2$, the independent excursions $\Y_{\ell}^{\A}$ and $\Z_{\ell}^{\A}$ have the same conditional law as that of $W_{\ell}^{\A}$ under $P_{Z_{R_{\ell}^{\A}},Z_{D_{\ell}^{\A}}}^{N,\alpha}$ instead of under $P_{Z_{R_{\ell}^{\A}},Z_{D_{\ell}^{\A}}}^{N}$. \label{item:adaptYZ}
\item The i.i.d.~excursions $\widetilde{\Y}_{\ell}^{\A}$, $\widetilde{\Z}_{\ell}^{\A}$ and $\widetilde{\W}_{\ell}^{\A}, {\ell\geq 2}$ have the same distribution as $X._{\land T_{\widetilde{\A}}}$ under $P_{q}^{N,\alpha}$ instead of $P_{q}^{N}$.  \label{item:adapt widetildeYZ}
\end{itemize}
We first show that the claims of \Cref{prop:horizontal independence,prop:vertical independence,prop:Poissonization} remain valid after these modifications. 
\begin{proof}[Proof of the modified \Cref{prop:horizontal independence,prop:vertical independence,prop:Poissonization}]
The proofs remain roughly the same, and we only point out necessary changes. 

In the proof of the modified \Cref{prop:horizontal independence}, we keep \Cref{lem:horizontal independence} and show that \Cref{lem:horizontal mixing} still holds in the biased walk case, after which the proof of the adapted \Cref{prop:horizontal independence} can be completed in the same way. The only minor change appears in the proof of \eqref{eq:bound for RwidehatZ2}, since $(\widehat{Z}_{n})_{n\geq 0}$ now has the law of a biased random walk instead of a simple random walk. However, for an excursion $e$ which travels from $\mathbb{T}\times\{h_{N}\}$ to $\mathbb{T}\times\{r_{N}\}$ or from $\mathbb{T}\times\{-h_{N}\}$ to $\mathbb{T}\times\{-r_{N}\}$, its height $h(e)$ (recall \Cref{subsec:RN}) is no larger than $h_{N}-r_{N}$. Therefore, by \eqref{eq:RN bias to unbias bound} and $\alpha>1/d$ we have
\begin{equation}\label{eq:bound for RwidehatZ bias}
\begin{split}
P_{x}^{N,\alpha}\left[R_{\widehat{Z}}\leq[N^{2}\log^{2}N]\right]&\leq P_{x}^{N}\left[R_{\widehat{Z}}\leq[N^{2}\log^{2}N]\right]\cdot\left(\frac{1+N^{-d\alpha}}{1-N^{-d\alpha}}\right)^{(h_{N}-r_{N})/2}\\
&\leq\frac{\exp(-c_{34}\log^{2}N)}{c_{34}}\left(\frac{1+N^{-d\alpha}}{1-N^{-d\alpha}}\right)^{CN\log^{2}N}\leq\frac{\exp(-c_{34}\log^{2}N)}{c_{34}}.
\end{split}
\end{equation}

In the proof of the modified \Cref{prop:vertical independence}, the distribution of $(R_{\ell},D_{\ell})$ should be slightly changed to remain to be the invariant distribution of $(Z_{R_{\ell}^{\A}},Z_{D_{\ell}^{\A}}),{2\leq\ell\leq K}$. Indeed, we have
\begin{equation}\label{eq:11 bias}
P^{\alpha}[(R_{\ell},D_{\ell})=(r_{N},h_{N})]=\frac{\left(\frac{1-N^{-d\alpha}}{1+N^{-d\alpha}}\right)^{h_{N}}-\left(\frac{1-N^{-d\alpha}}{1+N^{-d\alpha}}\right)^{r_{N}}}{\left(\frac{1-N^{-d\alpha}}{1+N^{-d\alpha}}\right)^{h_{N}}-\left(\frac{1-N^{-d\alpha}}{1+N^{-d\alpha}}\right)^{-h_{N}}}.
\end{equation}
However, since $N^{-d\alpha}<N^{-1}$ and $r_{N}<h_{N}<2N\log^{2}N$, it follows that as $N$ tends to infinity, the right hand side of \eqref{eq:11 bias} equals $(h_{N}-r_{N})(1+o(1))/2h_{N}$,
which is approximately equal to the simple walk case. Similar estimates on the probabilities of other elements in $\Gamma$ (recall \eqref{eq:def of Gamma}) under $P^{N,\alpha}$ guarantees that the invariant distribution of $(Z_{R_{\ell}^{\A}},Z_{D_{\ell}^{\A}}),{2\leq\ell\leq K}$ remains asymptotically identical after replacing $P_{0}^{N}$ with $P_{0}^{N,\alpha}$. Therefore, the argument of the original proof of \Cref{prop:vertical independence} is still valid.

Finally, the proof of the modified \Cref{prop:Poissonization} remains unchanged.
\end{proof}

We then move on to \ref{item:step2}. We first show that a Poissonian number of biased random walk excursions can be stochastically dominated by a Poissonian number of simple random walk excursions. Recall the notation $(W)^{\text{tr}}=W_{[H_{\mathsf{B}}(W),T_{\widetilde{\A}}(W)]}$ for the truncated excursion of $W$ from $H_{\mathsf B}(W)$ to $T_{\widetilde{\A}}(W)$. 

\begin{proposition}\label{prop:Poissonian coupling upper}
One can construct on an auxiliary space $(\Omega_{10},\mathcal{F}_{10})$ a coupling $Q_{10}$ of the $\mathscr{T}_{\widetilde{\A}}$-valued excursions $\{\widetilde{\W}_{\ell}^{\B}\}_{\ell \geq 1}$, $\{(\overline{\W}_{\ell}^{\A})^{\text{tr}}\}_{\ell\geq 1}$ and Poisson random variables $\widetilde{J}^{\B}$, $\overline{J}^{\text{tr}}$,
under which
\begin{itemize}
\item the excursions $\widetilde{\W}_{\ell}^{\B}, {\ell\geq 1}$ is a sequence of i.i.d.~excursions from $\B$ to $\partial\widetilde{\A}$ with the same distribution as $X._{\land T_{\widetilde{\A}}}$ under $P_{\overline{e}_{\B,\widetilde{\A}}}^{N,\alpha}$;
\item the excursions $(\overline{\W}_{\ell}^{\A})^{\text{tr}}, {\ell\geq 1}$ is a sequence of i.i.d.~excursions from $\B$ to $\partial\widetilde{\A}$ with the same distribution as $X._{\land T_{\widetilde{\A}}}$ under $P_{\overline{e}_{\B,\widetilde{\A}}}^{N}$;
\item the Poisson variables $\widetilde{J}^{\B}$, $\overline{J}^{\text{tr}}$ are independent from $\{\widetilde{\W}_{\ell}^{\B}\}_{\ell\geq 1}$, $\{(\overline{\W}_{\ell}^{\A})^{\text{tr}}\}_{\ell\geq 1}$, and satisfy
\begin{align}
\text{intensity of }\widetilde{J}^{\B}&:=\widetilde{\lambda}^{\B}=\frac{1}{14}\xi\cdot u\left(1-\frac{r_{N}}{h_{N}}\right)\sum_{x\in\partial\mathsf{B}}e_{\mathsf{B},\widetilde{\A}}(x)\label{eq:intensity of widetildeJB}\\
\text{intensity of }\overline{J}^{\text{tr}}&=\overline{\lambda}^{\text{tr}}:=\frac{1}{7}\xi\cdot u\left(1-\frac{r_{N}}{h_{N}}\right)\sum_{x\in\partial\mathsf{B}}e_{\mathsf{B},\widetilde{\A}}(x)\label{eq:intensity of overlineJtr}.    
\end{align}
\end{itemize}
such that for all $N\geq c_{36}=c_{36}(\alpha,u,u_{1},u_{2})>0$, 
\begin{equation}\label{eq:Poissonian coupling upper}
Q_{10}\left[\{\widetilde{\W}_{\ell}^{\B}\}_{\ell\leq\widetilde{J}^{\B}}\subseteq\{(\overline{\W}_{\ell}^{\A})^{\text{tr}}\}_{\ell\leq\overline{J}^{\text{tr}}}\right]=1.
\end{equation}
\end{proposition}
\begin{proof}
The proof follows the same idea as that of \Cref{lem:coupling widetildeW with widetildeZ}, and we only sketch the proof. Recall \Cref{subsec:RN} for the notation $p(e)$ and $p^{\text{bias}}(e)$. We define the set of excursions from $\B$ to $\partial\widetilde{\A}$ as $\Sigma_{\rm excur}$.
Take $\big(n_{e}(0,t)\big)_{t\geq 0}$, $e\in \Sigma_{\rm excur}$ as $|\Sigma_{\rm excur}|$ i.i.d.~Poisson point process of intensity $1$, which are all independent from the excursions $\widetilde{\W}_{\ell}^{\B},\ell\geq 1$ and $(\overline{\W}_{\ell}^{\A})^{\text{tr}},\ell\geq 1$. Then by properties of Poisson point process, we can take the coupling $Q_{10}$ under which
\begin{equation}
\sum_{\ell\leq \widetilde{J}^{\B}}\delta_{\widetilde{\W}_{\ell}^{\B}}=\sum_{e\in\Sigma_{\text{excur}}}\sum_{\ell\leq n_{e}(0,p(e)\widetilde{\lambda}^{\B})}\delta_{e};\label{eq:characterization of widetildeW Poisson}\quad\text{and}\quad \sum_{\ell\leq \overline{J}^{\text{tr}}}\delta_{(\overline{\W}_{\ell}^{\A})^{\text{tr}}} =\sum_{e\in\Sigma_{\text{excur}}}\sum_{\ell\leq n_{e}(0,p^{\text{bias}}(e)\overline{\lambda}^{\text{tr}})}\delta_{e}.   
\end{equation}
Since for every $e\in\Sigma_{\text{excur}}$, $h(e)\leq 2h_{N}$, by \eqref{eq:RN bias to unbias bound} and $\alpha>1/d$, for all $N\geq c_{36}(\alpha,u,u_{1},u_{2})>0$, 
\begin{equation}\label{eq:RN bias to unbias upper bound cylinder}
\frac{p^{\text{bias}}(e)}{p(e)}\leq \left(\frac{1+N^{-d\alpha}}{1-N^{-d\alpha}}\right)^{h_{N}}=\left(\frac{1+N^{-d\alpha}}{1-N^{-d\alpha}}\right)^{CN\log^{2}N}\leq \frac{\frac{1}{7}\xi}{\frac{1}{14}\xi},    
\end{equation}
and the conclusion follows from combining \eqref{eq:characterization of widetildeW Poisson}-\eqref{eq:RN bias to unbias upper bound cylinder}.
\end{proof}

We then prove the adaptation of \eqref{eq:Handling the first excursion adaped simple} on the first excursion. Taking $\overline{J}_{2}-\overline{J}_{1}'$ as a Poisson variable that is independent from $\overline{J}_{1}'$ and has parameter $\overline{\lambda}_2- \overline{\lambda}_1'$ where
\begin{equation}
    \overline{\lambda}_2 = \overline{\lambda}_{1}'+\overline{\lambda}^{\text{tr}}\overset{\eqref{eq:intensity of overlineJtr}}{=} \overline{\lambda}_{1}'+\frac{1}{7}\xi\cdot\frac{u}{d+1}\cdot \frac{N^d}{h_N}=\big(1+\frac{5}{7}\xi\big)\cdot \frac{u}{d+1}\cdot \frac{N^d}{h_N},
\end{equation}
we will show that under some coupling $Q_{4}^{\alpha}$, there exists $c_{37}=c_{37}(\alpha,u,u_{1},u_{2}) > 0$ such that
\begin{equation}\label{eq:Handling the first excursion adaped simple and bias}
Q_{4}^{\alpha}\left[(W_{1}^{\A})^{\rm{tr}}\in \big\{(\overline{\W}_{\ell}^{\A})^{\rm{tr}}\big\}_{\overline{J}_{1}'\leq \ell\leq \overline{J}_{2}}\right]\geq 1-\frac{1}{c_{37}}\exp\left(-c_{37}\text{cap}(\mathsf{B})\right).
\end{equation}
Note that here $(W_{1}^{\A})^{\rm{tr}}$ is truncated from a biased walk excursion $W_{1}^{\A}$ with law $P_{0}^{N,\alpha}$, while $\big\{(\overline{\W}_{\ell}^{\A})^{\rm{tr}}\big\}_{\ell\geq 1}$ are truncated from $\big\{\overline{\W}_{\ell}^{\A}\big\}_{\ell\geq 1}$, which are i.i.d.~simple walk excursions with the same distribution as $X._{\land T_{\widetilde{\A}}}$ under $P_{q}^{N}$. 
\begin{proof}[Proof of \eqref{eq:Handling the first excursion adaped simple and bias}]
Let $\{\widetilde{\W}_{\ell}^{\B}\}_{\ell\geq 1}$ be a sequence of i.i.d.~excursions with distribution $X._{\land T_{\widetilde{\A}}}$ under $P_{\overline{e}_{\mathsf{B},\widetilde{\A}}}^{N,\alpha}$. Since all the excursion $e$ from $\A$ to $\partial\widetilde{\A}$ that may hit $\mathsf{B}$ has height $h(e)<2h_{N}<CN\log^{2}N$, combining \eqref{eq:RN bias to unbias bound} with \eqref{eq:Handling the first excursion stochastic domination simple} yields that when $N\geq C'(\alpha,\xi)$, we have
\begin{equation}\label{eq:Handling the first excursion stochastic domination bias}
\begin{split}
P_{0}^{N,\alpha}\left[\left(W_{1}^{\A}\right)_{H_{\mathsf{B}}}=\cdot,H_{\mathsf{B}}(W_{1}^{\A})<T_{\widetilde{\A}}(W_{1}^{\A})\right]&\preceq\frac{C}{\log N}\left(1+\frac{\xi}{10^{2}}\right)\left(\frac{1+N^{-d\alpha}}{1-N^{-d\alpha}}\right)^{h(e)}\overline{e}_{\mathsf{B},\widetilde{\A}}(\cdot)\\
&\preceq \frac{C}{\log N}\left(1+\frac{\xi}{10^{2}}\right)\overline{e}_{\mathsf{B},\widetilde{\A}}(\cdot).
\end{split}
\end{equation}
A similar argument to the last paragraph of the proof of \eqref{eq:Handling the first excursion adaped simple} then shows that, under some coupling $Q_{4,1}$, for $\widetilde{J}^{\B}$ defined in \eqref{eq:intensity of widetildeJB}, the event $\{\widetilde{J}^{\B}\geq 1\}$ implies $\big\{(W_1^{\A})^{\rm{tr}}\in \big\{\widetilde{\W}_{\ell}^{\mathsf{\mathsf B}}\big\}_{\ell\leq \widetilde{J}^{\B}}\big\}$. 
Combining \eqref{eq:intensity of widetildeJB}, \eqref{eq:lower equilibrium} and exponential Chenyshev's inequality on Poisson random variable further gives a positive constant $c_{37}=c_{37}(\alpha,u,u_{1},u_{2})$ such that
\begin{equation}\label{eq:Handling the first excursion adapted bias}
Q_{4,1}\left[(W_{1}^{\A})^{\rm{tr}}\in\{\widetilde{\W}_{\ell}^{\mathsf{\mathsf B}}\}_{\ell\leq \widetilde{J}^{\B}}\right]\geq Q_{4,1}[\widetilde{J}^{\B}\geq 1]\geq 1-\frac{1}{c_{37}}\exp\left(-c_{37}\text{cap}(\mathsf{B})\right).
\end{equation}

Taking $\K=\mathsf{B}$ in \eqref{eq:Markov property of Pq}, we obtain that $\big\{(\overline{\W}_{\ell}^{\A})^{\rm{tr}}\big\}_{\overline{J}_{1}'\leq \ell\leq \overline{J}_{2}}$ has the same law as a sequence of i.i.d.~excursions with distribution as $X._{\land T_{\widetilde{\A}}}$ under $P_{\overline{e}_{\mathsf{B},\widetilde{\A}}}^{N}$, where the size of this sequence is a Poisson random variable with intensity as in \eqref{eq:intensity of overlineJtr}. Then by \Cref{prop:Poissonian coupling upper}, it follows that when $N\geq C(\alpha,\xi)$, we can construct a coupling $Q_{4,2}$ such that
\begin{equation}\label{eq:coupling widetildeW with overlineW bias}
Q_{4,2}\left[\big\{\widetilde{\W}_{\ell}^{\mathsf{\B}}\big\}_{\ell\leq \widetilde{J}^{\B}}\subseteq \big\{(\overline{\W}_{\ell}^{\A})^{\rm{tr}}\big\}_{\overline{J}_{1}'\leq \ell\leq \overline{J}_{2}}\right]=1.   
\end{equation}
The conclusion then follows from combining \eqref{eq:Handling the first excursion adapted bias} and \eqref{eq:coupling widetildeW with overlineW bias}. 
\end{proof}

We finally complete \ref{item:step3}. Thanks to \ref{item:step2} in which we stochastically dominate all the biased walk exucrsions with the simple walk excursions $\{\overline{\W}_{\ell}^{\A}\}_{\ell\geq 1}$, we can directly combine \ref{item:step1} and \ref{item:step2} with a slightly adapted version of \eqref{eq:extraction 1} (with $(1+\frac{4}{7}\xi)$ substituted into $(1+\frac{5}{7}\xi)$ in \eqref{eq:intensity of widetildemu2}) and \Cref{prop:upper truncation,prop:upper interlacement} to conclude the proof of \Cref{prop:very strong coupling bias}.

\subsubsection{The proof of \Cref{prop:strong coupling}}\label{subsubsec:bias lower}
We eventually come to the proof of \Cref{prop:strong coupling}, that is, the ``strong'' coupling between excursions under conditional measures and interlacements. 
We take $\B = \overline{B}$, $\D = \overline{D}$ in \Cref{thm:very strong coupling S} with side-lengths $[N^{1/3}]$ and $[N^{2/3}]$ respectively (recall \eqref{eq:def of overlineB and overlineD}), which satisfy the requirements for $f(N)$ and $g(N)$ in \eqref{eq:requirement for f and g}. Therefore, \Cref{prop:strong coupling} can be seen as a weaker case of the lower bound of \Cref{thm:very strong coupling S}, except that the walk now starts from the origin instead of the uniform distribution on $\mathbb{T}\times\{0\}$ and we require the error term to be a negative polynomial of $N$ instead of $\exp(-c\text{cap}(\overline{B}))$. 

Under this setting, we can adapt the statements and proofs of the domination-from-below parts of \Cref{prop:horizontal independence,prop:vertical independence,prop:Poissonization} in the same way as in \ref{item:step1} in \Cref{subsubsec:bias upper}. There is no need to adapt \Cref{prop:Handling the first excursion} concerning the first excursion $W_1^{\A}$ since we can simply abandon the first one when giving a stochastic lower bound for $\{W_{\ell}^{\A}\}_{\ell\geq 1}$. Therefore, the only missing part now is the following proposition transforming i.i.d.~biased excursions into i.i.d.~simple excursions, with which we can directly use a slightly adapted version of \eqref{eq:extraction 1} (with $(1-\frac{3}{7}\xi)$ substituted into $(1-\frac{4}{7}\xi)$ in \eqref{eq:intensity of widetildemu1}) and \Cref{prop:lower truncation,prop:lower interlacement} to complete the proof.

\begin{proposition}\label{prop:Poissonian coupling lower}
One can construct on an auxiliary space $(\Omega_{11},\mathcal{F}_{11})$ a coupling $Q_{11}$ of the $\mathscr{T}_{\widetilde{A}}$-valued excursions $\{\widetilde{\W}_{\ell}^{\A}\}_{\ell\geq 1}$ and $\{\overline{\W}_{\ell}^{\A}\}_{\ell \geq 1}$ and two Poisson random variables $\widetilde{J}_{1}$, $\overline{J}_{1}$ 
under which
\begin{itemize}
\item the excursions $\widetilde{\W}_{\ell}^{\A},{\ell\geq 1}$ is a sequence of i.i.d.~excursions from $\A$ to $\partial\widetilde{\A}$ with the same distribution as $X._{\land T_{\widetilde{\A}}}$ under $P_{q}^{N,\alpha}$;
\item the excursions $\overline{\W}_{\ell}^{\A},{\ell\geq 1}$ is a sequence of i.i.d.~excursions from $\A$ to $\partial\widetilde{\A}$ with the same distribution as $X._{\land T_{\widetilde{\A}}}$ under $P_{q}^{N}$;
\item the Poisson variables $\widetilde{J}_{1}$, $\overline{J}_{1}$ are independent from $\{\widetilde{\W}_{\ell}^{\A}\}_{\ell\geq 1}$ and $\{\overline{\W}_{\ell}^{\A}\}_{\ell\geq 1}$, and have respective intensities $\widetilde{\lambda}_1$ (see \eqref{eq:def Poisson intensity}) and $\overline{\lambda}_1$, where 
\begin{equation}
    \overline{\lambda}_1 = \big(1-\frac{4}{7}\xi\big)\cdot \frac{u}{d+1}\cdot \frac{N^{d}}{h_{N}},
\end{equation}
\end{itemize}
such that for all $N\geq c_{38}=c_{38}(\alpha,u,u_{1},u_{2})>0$, 
\begin{equation}\label{eq:Poissonian coupling lower}
Q_{11}\left[\{\overline{\W}_{\ell}^{\A}\}_{\ell\leq\overline{J}_{1}}\subseteq\{\widetilde{\W}_{\ell}^{\A}\}_{\ell\leq\widetilde{J}_{1}}\right]\geq 1-N^{-10d}. 
\end{equation}
\end{proposition}
\begin{proof}
The proof follows the same idea as that of \Cref{lem:coupling widetildeW with widetildeZ}, and we only sketch the proof. Recall \Cref{subsec:RN} for the notation $\ell(e),h(e),\text{up}(e),\text{down}(e),p(e),p^{\text{bias}(e)}$. We denote by $\Sigma_{\text{excur}}$ the set of excursions from $\A$ to $\partial\widetilde{\A}$, and further divide $\Sigma_{\text{excur}}$ into
\begin{equation}\label{eq:def of Sshort and Slong cylinder}
\Sigma_{\rm short}:=\{e\in \Sigma_{\rm excur}:\ell(e)\leq N^{1+d\alpha}\},\quad\text{and}\quad
\Sigma_{\rm long}:=\{e\in \Sigma_{\rm excur}:\ell(e)>N^{1+d\alpha}\}.   
\end{equation} 
Take $\big(n_{e}(0,t)\big)_{t\geq 0}$, $e\in \Sigma_{\rm excur}$ as $|\Sigma_{\rm excur}|$ i.i.d.~Poisson point process of intensity $1$ that are all independent from the excursions $\overline{\W}_{\ell}^{\A},\ell\geq 1$ and $\widetilde{\W}_{\ell}^{\A},\ell\geq 1$. Then by properties of Poisson point process, we can construct the coupling $Q_{11}$ so that
\begin{equation}
\sum_{\ell\leq \widetilde{J}_{1}}\delta_{\widetilde{\W}_{\ell}^{\A}}=\sum_{e\in\Sigma_{\text{excur}}}\sum_{\ell\leq n_{e}(0,p(e)\widetilde{\lambda}_1)}\delta_{e};\quad\text{and}\quad
\sum_{\ell\leq \overline{J}_{1}'}\delta_{\overline{\W}_{\ell}^{\A}}=\sum_{e\in\Sigma_{\text{excur}}}\sum_{\ell\leq n_{e}(0,p^{\text{bias}}(e)\overline{\lambda}_1)}\delta_{e}.
\end{equation}

Note that under $Q_{11}$, the event in \eqref{eq:Poissonian coupling lower} holds once we have the following two events:  
\begin{align}\label{eq:Slong empty cylinder}
\Sigma_{\rm long}\cap\left\{\widetilde{\W}_{1}^{\A},\dots,\widetilde{\W}_{\widetilde{J}_{1}}^{\A}\right\}&=\Sigma_{\rm long}\cap\left\{\overline{\W}_{1}^{\A},\dots,\overline{\W}_{\overline{J}_{1}}^{\A}\right\}=\varnothing;\\ \label{eq:Sshort Poisson cylinder}
\Sigma_{\rm short}\cap\left\{\widetilde{\W}_{1}^{\A},\dots,\widetilde{\W}_{\widetilde{J}_{1}}^{\A}\right\}&\subseteq \Sigma_{\rm short}\cap\left\{\overline{\W}_{1}^{\A},\dots,\overline{\W}_{\overline{J}_{1}}^{\A}\right\}.\end{align}
We bound the probability that \eqref{eq:Slong empty cylinder} or \eqref{eq:Sshort Poisson cylinder} does not hold from above in the same way as we do to \ref{item:Slong empty} and \ref{item:Sshort Poisson} in the proof of \Cref{lem:coupling widetildeW with widetildeZ} (the first using union bound and Kha\'{s}minskii's lemma, the second using \eqref{eq:RN bias to unbias bound}), and the conclusion follows.
\end{proof}

\begin{remark}\label{remark:strong to very strong}
It is an interesting question whether one can improve \Cref{prop:Poissonian coupling lower} so that the error term is exponential in $-\text{cap}(\overline{B})$, which will also imply an error term of the same order in \Cref{prop:strong coupling}, making it a ``very strong" coupling too. The main obstacle in the above proof is that the probability that \eqref{eq:Slong empty cylinder} does not hold is much larger than $\exp(-\text{cap}(\overline{B}))$, while on the other hand, with high probability \eqref{eq:Sshort Poisson cylinder} does not hold if one replaces $\Sigma_{\text{short}}$ by $\Sigma_{\text{long}}$ since $p(e)/p^{\text{bias}}(e)$ explodes when $e\in\Sigma_{\text{long}}$. A possible solution is to slice every excursion in $\Sigma_{\text{long}}$ into a union of excursions in $\Sigma_{\text{short}}$, and then stochastically dominate these extra simple walk excursions in $\Sigma_{\text{short}}$ by biased walk excursions in $\Sigma_{\text{short}}$ using Radon-Nikodym derivative estimates. In this case, for every $e\in\Sigma_{\text{long}}$, the number of sliced excursions in $\Sigma_{\text{short}}$ is linear in $\ell(e)$, while $p(e)$ decays stretched-exponentially in $\ell(e)$ thanks to Kha\'{s}minskii's lemma and exponential Chebyshev's inequality, and it is likely that these extra biased walk excursions in $\Sigma_{\text{short}}$ only lead to a tiny increase in the Poisson intensity. 
\end{remark}

\section{Denouement}\label{sec:denouement}

In this section, we conclude the proof of \Cref{thm:biased walk aymptotics}. Note that \Cref{thm:sharp lower bound} is indeed a corollary of \Cref{thm:biased walk aymptotics}. We also incorporate a remark discussing very strong bias case ($\alpha\leq 1/d$) and mention some open problems. 
Recall that we have split the main result into six inequalities \eqref{eq:lower bound large alpha}-\eqref{eq:upper bound small alpha} in the sketch of proofs in \Cref{subsec:sketch}.

\begin{proof}[Proof of lower bounds \eqref{eq:lower bound large alpha}-\eqref{eq:lower bound small alpha}]
    We first consider the case $\alpha\geq 1$. For every $\delta > 0$ and $\underline S_N = \underline S_N(\omega, \delta)$, it follows from \Cref{prop:underlineS and T,prop:underlineS and T biased} that
    \begin{equation}
    \begin{split} \liminf_{N\to\infty}P_{0}^{N,\alpha}\left[T_{N}\geq sN^{2d}\right]
        &\geq \liminf_{N\to\infty}P_{0}^{N,\alpha}\left[\underline{S}_{N}\geq sN^{2d} \right] - \limsup_{N\to\infty}P_{0}^{N,\alpha}[T_N < \underline S_N]\\
        &=\liminf_{N\to\infty}P_{0}^{N,\alpha}\left[\underline{S}_{N}\geq sN^{2d} \right].
    \end{split}
    \end{equation}
    Then \Cref{prop:asymptotics_alpha>1} gives \eqref{eq:lower bound large alpha}, and  \Cref{prop:asymptotics_alpha=1} gives \eqref{eq:lower bound alpha=1}. 
    For the strong bias case, repeating the above argument while replacing $sN^{2d}$ with exponential term $\exp\left(\frac{\overline u - 2\delta}{d + 1}N^{d(1-\alpha)}\right)$, together with \Cref{prop:asymptotics_alpha<1} yields \eqref{eq:lower bound small alpha}.
\end{proof}

We now turn to the upper bounds with the help of \Cref{prop:overlineS and T}. 
We begin with the case $\alpha>1$. To finalize the comparison of $\overline S_N$ with $T_N$, it is essential to replace $\overline S_N  = \inf\limits_{z\in \mathbb Z} \overline S_N(\omega, \delta, z)$ (see \eqref{eq:def of underlineS and overlineS}) by the truncated version $\inf\limits_{z = \lfloor \ell/L N^d\rfloor, |\ell|\leq L^2} \overline S_N(\omega, 3\delta/4, z)$ for some $L\geq 1$ that will eventually tend to infinity, with the parameter $\delta$ adjusted as well. We omit the proof as it has already been contained in \cite{Szn09a} (under a slightly different setting).
\begin{lemma}[\cite{Szn09a}, (4.31)]\label{lem:overline S confine region small bias}
    For every $\delta > 0$, 
    \begin{equation}\label{eq:overline S confine region small bias}
    \limsup_{L\rightarrow \infty}\limsup_{N\rightarrow \infty} P_{0}^{N}\left[\inf\limits_{z = \lfloor \ell/L N^d\rfloor, |\ell|\leq L^2} \overline S_N(\omega, \frac{3}{4}\delta, z) > \inf\limits_{z\in \mathbb Z} \overline S_N(\omega, \delta, z)\right] = 0.
\end{equation}
\end{lemma}

With this, we are finally able to conclude the main theorem when $\alpha>1$.
\begin{proof}[Proof of  \eqref{eq:upper bound large alpha}]
    For every $\delta > 0$ and $L>0$, it follows from \Cref{prop:overlineS and T} that
    \begin{equation}\label{eq:proof upper bound process1}
    \begin{split}
        \limsup_{N\to\infty}P_{0}^{N,\alpha}\left[{T_N}\geq s{N^{2d}}\right] \leq\limsup_{N\to\infty}P_{0}^{N,\alpha}\left[\inf\limits_{z = \lfloor \ell/L N^d\rfloor, |\ell|\leq L^2} \overline S_N(\omega, \frac{3}{4}\delta, z) \geq s N^{2d}\right].
        \end{split}
    \end{equation}
    Using a similar procedure as \Cref{lem:asympt_srw_to_biasedwalk}, we can further bound the last term from above using estimates regarding Radon-Nikodym derivates of measures with and without bias. In this way, for every fixed $L>0$,
    \begin{equation}
        \limsup_{N\rightarrow \infty} P_{0}^{N,\alpha}\left[{T_N}\geq s{N^{2d}}\right] \leq \limsup_{N\rightarrow \infty}P_{0}^{N}\left[\inf\limits_{z = \lfloor \ell/L N^d\rfloor, |\ell|\leq L^2} \overline S_N(\omega, \frac{3}{4}\delta, z) \geq s N^{2d}\right].
    \end{equation}
    Take $L$ to infinity and the result then follows from \eqref{eq:overline S confine region small bias} and \eqref{eq:asymp_for_overlineS_alpha>1}.
\end{proof}

For $\alpha=1$, we still consider the truncated version of $\overline{S}_{N}$, namely $\inf\limits_{z = \lfloor \ell/L N^d\rfloor, |\ell|\leq L^2} \overline S_N(\omega, 3\delta/4, z)$.
\begin{proof}[Proof of \eqref{eq:upper bound alpha=1}]
    We again use \eqref{eq:proof upper bound process1}. Now for every $\delta > 0$, it suffices to show
    \begin{equation}\label{eq:proof of upper bound1}
        \limsup_{L\rightarrow \infty}\limsup_{N\rightarrow \infty}P_{0}^{N,1}\left[\inf\limits_{z = \lfloor \ell/L N^d\rfloor, |\ell|\leq L^2} \overline S_N(\omega, \frac{3}{4}\delta, z) \geq s N^{2d}\right] \leq \WW\left[\zeta^{1/\sqrt{d+1}}\left(\frac{u_{**}+\delta}{\sqrt{d+1}}\right)\geq s\right].
    \end{equation}
    
    Similarly as in \eqref{eq:weak conv proof1}, by \eqref{eq:limit of rho_k}, for every $0<\widetilde s < s$, the left side of \eqref{eq:proof of upper bound1} is no larger than
    \begin{equation}
        \begin{split}
            \limsup_{L\rightarrow \infty}\limsup_{N\rightarrow \infty} P_{0}^{N,1}\left[\inf\limits_{z = \lfloor \ell/L N^d\rfloor, |\ell|\leq L^2} \overline S_N(\omega, \frac{3}{4}\delta, z) \geq \rho_{\widetilde sN^{2d}/(d+1)}\right].
        \end{split}
    \end{equation}
    Following the same notation as in \Cref{lem:weak converge 1d rw} with $N$ replaced by $N^d$, the limit equals
    \begin{equation}
    \begin{split}
        &\quad\quad\limsup_{L\rightarrow \infty} \limsup_{N\rightarrow \infty}\mathsf P_0^{N^{-d}}\left[\inf_{z = \lfloor \ell/L N^d\rfloor, |\ell|\leq L^2}S\Big(\frac{u_{**} + 3\delta/4}{d + 1}, z\Big)\geq \frac{\widetilde s}{d + 1}N^{2d}\right]\\
        &\overset{\eqref{eq:weak convergence truncated}}{\leq} \WW\left[\zeta^{1}\left(\frac{u_{**}+\delta}{d + 1}\right)\geq \frac{\widetilde s}{d + 1}\right] = \WW\left[\zeta^{\frac{1}{\sqrt{d+1}}}\left(\frac{u_{**}+\delta}{\sqrt{d+1}}\right)\geq \widetilde{s}\right],
    \end{split}
    \end{equation}
    where the last equality holds thanks to scaling property of drifted Brownian motion. Taking $\widetilde s \rightarrow s$ and using the continuity of local time of drifted Brownian motion then conclude the proof.
    
\end{proof}

Finally, we turn to the strong bias case.

\begin{proof}[Proof of \eqref{eq:upper bound small alpha}]
For convenience, we write
\begin{equation}
K_{**} = \exp\left(\frac{u_{**}+2\delta}{d+1}\cdot N^{d(1-\alpha)}\right), \quad J=\left[\frac{K_{**}}{N^{6d}}\right],\quad\text{and}\quad\ell_{**} = \frac{u_{**} + \delta}{d + 1}N^d. 
\end{equation}

    We now derive the upper bound of the disconnection time through exponential trials for the random walk over disjoint time intervals of length $N^{6d}$. In other words, 
        \begin{equation}\label{eq:upper bound into indep parts}
            P_{0}^{N,\alpha}\left[T_N \geq  K_{**}\right] \leq P_{0}^{N,\alpha}\left[T_N \geq N^{6d}\right]
            \cdot 
            \prod_{i = 1}^{J-1}
            P_{0}^{N,\alpha}\left[T_N \geq (i + 1)N^{6d}\mid T_N \geq iN^{6d}\right].
        \end{equation}
    For $1\leq i\leq J-1$, using Markov property on the time $iN^{6d}$ and the fact that the historical trace of the walk helps disconnection, we have
    \begin{equation}\label{eq:cond prob using markov}
        P_{0}^{N,\alpha}\left[T_N\geq (i+1)N^{6d}\mid T_N\geq iN^{6d}\right]
        \leq P_{0}^{N,\alpha}\left[T_N \geq N^{6d}\right].
    \end{equation}
    Therefore, \eqref{eq:upper bound into indep parts} and \eqref{eq:cond prob using markov} lead to 
    \begin{equation}\label{eq:exp products}
        P_{0}^{N,\alpha}\left[\log T_N \geq \frac{u_{**} + 2\delta}{d + 1}N^{d(1-\alpha)}\right] \leq \left(1 - P_{0}^{N,\alpha}\left[T_N < N^{6d}\right]\right)^J.
    \end{equation}
    
    We now bound $P_{0}^{N,\alpha}\left[T_N < N^{6d}\right]$ from below. Note that 
    \begin{equation}\label{eq:lower bound disconnection time}
    \begin{split}
        P_{0}^{N,\alpha}\left[T_N < N^{6d}\right]
        &\geq P_{0}^{N,\alpha}\left[T_N< \overline S_N(0) \mid \overline S_N(0) < N^{6d}\right]\cdot P_{0}^{N,\alpha}\left[\overline S_N(0) < N^{6d}\right] = \text{\rom{1}}\cdot \text{\rom{2}}.
    \end{split}
    \end{equation}
    To bound the two probabilities from below, we define 
    \begin{equation}\label{eq:def typical event}
    \mathcal G = \big\{Z_{[N^{6d}, \infty)}\cap \mathbb T \times \{0\} = \varnothing\big\}, \text{ and } M_{n}=Z_{n}-\frac{N^{-d\alpha}}{d+1}n.    
    \end{equation}
    Combining the relation
    \begin{equation}\label{eq:relation in time N6d}
        \mathcal G \cap \left\{\overline S_N(0) < \infty\right\} \subseteq \left\{\overline S_N(0) < N^{6d}\right\}\subseteq \left\{\overline S_N(0) < \infty\right\},
    \end{equation}
    we have 
    \begin{align}
        &\text{\rom{1}}\geq  P_{0}^{N,\alpha}\left[T_N< \overline S_N(0)\mid \overline S_N(0) < \infty\right] - \frac{P_{0}^{N,\alpha}\left[\mathcal G^c\right]}{P_{0}^{N,\alpha}\left[\overline S_N(0)<\infty\right]},\text{ and }\label{eq:cond disc prob lower bound}\\  
        & \text{\rom{2}}\geq P_{0}^{N,\alpha}\left[\overline S_N(0) < \infty\right]-P_{0}^{N,\alpha}\left[\mathcal{G}^{c}\right].\label{eq:complete before N46}
    \end{align}

    Note that $(M_n)_{n\geq 0}$ is a martingale, using Azuma-Hoeffding's inequality gives
    \begin{equation}\label{eq:AH inequality on probability of G}
    \begin{split}
    P_{0}^{N,\alpha}\left[\mathcal{G}^{c}\right]&\leq\sum_{n\geq N^{6d}}P_{0}^{N,\alpha}\left[Z_{n}=0\right] \leq \sum_{n\geq N^{6d}}P_{0}^{N,\alpha}\left[M_{n}-M_{0}\geq \frac{N^{-d\alpha}}{d+1}n\right]\\
    &\leq \sum_{n\geq N^{6d}}e^{-cnN^{-2d\alpha}}\leq Ce^{-N^{3d}}.
    \end{split}
    \end{equation}
    In addition, by \eqref{eq:probability of multiple hits} we have
    \begin{equation}\label{eq:overlineS finite}
    P_{0}^{N,\alpha}\left[\overline{S}_{N}(0)<\infty\right]=\left(1-N^{-d\alpha}\right)^{\ell_{**}-1}=\exp\left(-O\left(N^{d(1-\alpha}\right)\right).
    \end{equation}

    Therefore, when $N$ tends to infinity, the right hand side of \eqref{eq:cond disc prob lower bound} is larger than $1/2$ thanks to \Cref{prop:overlineS and T}, \eqref{eq:AH inequality on probability of G} and \eqref{eq:overlineS finite}. Again by \eqref{eq:AH inequality on probability of G} and \eqref{eq:overlineS finite}, the right hand side of \eqref{eq:complete before N46} is larger than $1/2\cdot P_0^{N, \alpha}[\overline S_N(0) < \infty]$ as $N$ goes to infinity.
    Combining \eqref{eq:lower bound disconnection time}, \eqref{eq:cond disc prob lower bound} and \eqref{eq:complete before N46} yields for $N\geq C(\alpha)$,
    \begin{equation}\label{eq:disconnection before N6d}
    P_{0}^{N,\alpha}\left[T_N < N^{6d}\right]\geq \frac{1}{4}P_{0}^{N,\alpha}\left[\overline S_N(0) < \infty\right]\overset{\eqref{eq:overlineS finite}}{=}\frac{1}{4}\left(1-N^{-d\alpha}\right)^{\ell_{**}-1}. 
    \end{equation}
    Plugging \eqref{eq:disconnection before N6d} into \eqref{eq:exp products}, and we finally obtain that when $N\geq C(\alpha)$,
    \begin{equation*}
    \begin{split}
        &\quad P_{0}^{N,\alpha}\left[\log T_N \geq \frac{u_{**} + 2\delta}{d + 1}N^{d(1-\alpha)}\right]
        \leq \left(1 - \frac{1}{4} \left(1 - \frac{1}{N^{d\alpha}}\right)^{\ell_{**} - 1}\right)^J\\
        &\leq \exp \left(-\frac{c}{N^{6d}}\exp\left(\frac{u_{**} + 2\delta}{d+1}N^{d(1-\alpha)} - (1+o(1))\frac{u_{**} + \delta}{d + 1}N^d\cdot N^{-d\alpha}\right)\right),
        \end{split}
    \end{equation*}
    which vanishes as $N\rightarrow \infty$, and then concludes the proof for \eqref{eq:upper bound small alpha}.
\end{proof}
\begin{proof}[Proofs of \Cref{thm:sharp lower bound,thm:biased walk aymptotics}]
Given \eqref{eq:lower bound large alpha}-\eqref{eq:upper bound small alpha}, the proof of \Cref{thm:biased walk aymptotics} follows from sending $\delta$ to zero and using the continuity of Brownian local time. \Cref{thm:sharp lower bound} is a corollary of \Cref{thm:biased walk aymptotics} by taking $\alpha=\infty$.
\end{proof}

\begin{remark}\label{rem:finalremark} We now briefly comment on the difficulties to extend the current scheme to the study of the asymptotics of the disconnection time in the presence of a super strong bias (i.e., when $\alpha \leq 1/d$), and then discuss some open questions and possible future directions.

{\bf 1)} 
Although not explicitly spelled out in this work, it is not hard to see (at least at a heuristically level) that in the strong bias regime ($1/d<\alpha<1$), to achieve disconnection only takes an order of $N^{d(1+\alpha)}$ steps, spanning an order of $N^{d\alpha}$ layers in terms of height. We believe that these asymptotics do extend to $\alpha\leq 1/d$ and the upper bound in \eqref{eq:Windisch} is tight. However, in this regime,  the random walk fails to mix horizontally (i.e., in $\mathbb T$-direction) within such a short period. This strongly suggests a phase transition at $\alpha=1/d$, resulting in a different pre-factor (or even a different scaling factor) from that of \eqref{eq:distribution limit small alpha}. It is very plausible that random interlacements still come into play in this regime, but some new ideas have to be involved to obtain the correct order of asymptotic disconnection time.  

{\bf 2)} As a starting point, we propose the following ``toy'' model: consider a biased random walk on $\mathbb E=\mathbb T \times \mathbb Z$, with an upward (downward resp.) drift of strength $N^{-d\alpha}$ for $x\in \mathbb T \times \mathbb Z^{-}$ ($\mathbb T\times \mathbb Z^+$ resp.), $0\leq \alpha \leq 1$, for which the disconnection should happen at time scale $N^{d(1+\alpha)}$. What can we say on the precise asymptotics of the disconnection time? What about the large deviation probability that the disconnection happens at a proportion of the expected disconnection time? We believe that answering these questions is vital to the understanding of the disconnection problem with super strong bias. 

{ \bf3)} The work \cite{vermesi2008intersection} considers intersection exponents for biased random walk on cylinders with a fixed base. It is very natural to also consider non-intersection events for biased walks on cylinders with large bases and the behaviour of the walks conditioned on non-intersection.
\end{remark}

\newpage
\section{Tables of symbols}\label{sec:symbol}

We summarize in the following tables various symbols that appear in this paper. We begin with lattice sets.
\vspace{6mm}
\begin{table}[htbp]
\renewcommand{\arraystretch}{1.3}
\setlength{\tabcolsep}{10mm}
\caption{Subsets of $\mathbb E $ and $\mathbb Z^{d+1}$}\label{table:Various sets}
\centering
\begin{tabular}{@{\hspace{0.3cm}}p{3.8cm} @{\hspace{0.3cm}}p{2.2cm}@{\hspace{0.5cm}} p{9.5cm}@{\hspace{0.2cm}}}
\toprule[1.5pt]
Symbol & Location & Description\\
\midrule[1pt]
$B_x, D_x, \check D_x, U_x, \check U_x$ & \eqref{eq:def of boxes at origin},\eqref{eq:def of translated boxes} & ``Matryoshka'' of boxes with $B_x\subseteq D_x\subseteq \check{D}_{x}\subseteq U_{x}\subseteq \check{U}_{x}.$\\
$B',D'$ ($B_{x'},D_{x'}$) & \eqref{def:exist event} & Concentric boxes with $x'=x+Le$ for some $|e|=1$.\\
$\mathsf B$ & Prop.~\ref{prop:geometric argument}, \ref{prop:geometric argument bias} & A fixed box with side-length $[N/\log^3N]$ with conditions \eqref{eq:vertical component of mathsfB} (or \eqref{eq:mathsf B range}) and \eqref{eq:large distance from mathsfB}.\\
$\mathsf D$ & Sect.~\ref{sec:bad local time} & The concentric box of $\mathsf B$ with side-length $[N/20]$. \\
$\overline B, \overline D$ ($\overline B_x, \overline D_x$) & \eqref{eq:def of overlineB and overlineD}, \eqref{def of ef of overlineB and overlineD x} & Mesoscopic boxes with $\overline B_x \subset \overline D_x$ centered around $\mathbb T\times \{0\}$. \\
$\B$, $\D$ & \eqref{def:mathrm box B and D}, \eqref{eq:requirement for f and g} & Mesoscopic boxes centered at $\mathrm x_c$ with $\B\subset \D$.\\
$\A$, $\widetilde{\A}$ & \eqref{eq:def of mathrmA and widetildemathrmA} & Cylinders centered at height $\mathrm z_c$ with $\A \subseteq \widetilde{\A}$.\\
$\mathcal C$ & Prop.~\ref{prop:geometric argument} & Set of base points satisfying \eqref{eq:d-dim surface of bad boxes} and \eqref{eq:bad boxes property}. \\
$\mathcal C_1$ & Prop.~\ref{prop:bad(beta-gamma)}, \ref{prop:bad local time} & A fixed subset of $\mathcal C$ with condition  \eqref{eq:size of C1} or \eqref{eq:C1 disjoint}.\\
$\mathcal C_2$ & \eqref{eq:sparsity of C2} & A sparse subset of $\mathcal C_1$.\\
$D(\mathcal S), \check D(\mathcal S), \check U(\mathcal S)$ & \eqref{eq:def of checkD and checkU C2} & Union of $D_x, \check D_x, \check U_x$ for $x\in \mathcal S$ respectively.\\
\bottomrule[1.5pt]
\end{tabular}
\end{table}

We now introduce different types of excursions involved in the couplings. 

\vspace{5mm}
\begin{table}[H]
\renewcommand{\arraystretch}{1.3}
\setlength{\tabcolsep}{10mm}
\caption{Excursion types}\label{table:excursion types}
\centering
\begin{tabular}{@{\hspace{0.2cm}}p{1.4cm}@{\hspace{0.3cm}}p{1.4cm}@{\hspace{0.3cm}} p{12.8cm}@{\hspace{0.2cm}}}
\toprule[1.5pt]
Type& Example & Description and Location\\
\midrule[1pt] 
$W$-type& $W_{\ell}^{\check{\mathsf{B}}}$ & Cylinder walk excursions successively extracted from $(X_n)_{n\geq 0}$ (or $\{X^{k}\}_{k\geq 1}$ in Sect.~\ref{sec:lower bound in biased case}) under law $P_{0}^{N}$ (Sect.~\ref{sec:Geometric}-\ref{sec:bad local time}), $P_{0}^{N,\alpha}$ (Sect.~\ref{sec:lower bound in biased case}-\ref{sec:upper bound}), or $P_{q_{0}}^{N}$ and $P_{q_{0}}^{N,\alpha}$ (Sect.~\ref{sec:couplings}).\\ 
$\widetilde{W}$-type &$\widetilde{W}_{\ell}^{\check{D}}$ & For a fixed $\check{D}$, $\{\widetilde{W}_{\ell}^{\check{D}}\}_{\ell\geq 1}$ are i.i.d.~biased walk excursions with law $P_{\overline{e}_{\check{D}_{x}}}^{N,\alpha}$, and $\{\widetilde{W}_{\ell}^{\check{D}_{x}}\}_{\ell\geq 1}$ are independent as $x$ varies over $\mathcal{C}_{1}$ (see Sect.~\ref{subsec:adapting bad(beta-gamma)}).\\
$\widetilde{Z}$-type &$\widetilde{Z}_{\ell}^{D}$& For a fixed $\check{D}$, $\{\widetilde{Z}_{\ell}^{\check{D}}\}_{\ell\geq 1}$ are i.i.d.~SRW excursions with law $P_{\overline{e}_{\check{D}_{x}}}^{N}$, and $\{\widetilde{Z}_{\ell}^{\check{D}_{x}}\}_{\ell\geq 1}$ are independent as $x$ varies over $\mathcal{C}_{1}$. Excursions $\{\widetilde{Z}_{\ell}^{D_{x}}\}_{\ell\geq 1}$ and $\{\widetilde{Z}_{\ell}^{D_{x'}}\}_{\ell\geq 1}$ are successively extracted from $\{\widetilde{Z}_{\ell}^{\check{D}_{x}}\}_{\ell\geq 1}$ (see Sect.~\ref{sec:bad(beta-gamma)} and \ref{subsec:adapting bad(beta-gamma)}).\\
$Z$-type & $Z_{\ell}^{\mathrm{B}}$ & SRW excursions extracted from random interlacements (see Sect.~\ref{sec:bad(beta-gamma)}-\ref{sec:couplings}).\\
$\W$-type & $\widetilde{\W}_{\ell}^{\B},\overline{\W}_{\ell}^{\A}$ & \\
$\Y$-type & $\Y_{\ell}^{\A}$, $\widetilde{\Y}_{\ell}^{\A}$ &Cylinder excursions with law depending on the context (see \Cref{sec:couplings}).\\ 
$\Z$-type & $\Z_{\ell}^{\A}$, $\widetilde{\Z}_{\ell}^{\A}$ & \\
\bottomrule[1.5pt]
\end{tabular}
\end{table}

\newpage

For two concentric sets $U\subseteq V$ and an excursion type $X$, we use $X^{U}$ as a shorthand of $X^{U,V}$ to denote successive $X$-type excursions from $U$ to $\partial V$. In addition, we may omit the subscript $U$ if $U$ is indexed by a set of base points when no confusion arises (e.g.~use $W_{\ell}^{D}$ for $W_{\ell}^{D_x, U_x}$). In \Cref{table:short forms} we list the shorthand notation, with the concentric sets displayed in \Cref{table:Various sets}.

\vspace{5mm}
\begin{table}[htbp]
\renewcommand{\arraystretch}{1.3}
\setlength{\tabcolsep}{10mm}
\caption{Shorthands}\label{table:short forms}
\centering
\begin{tabular}{@{\hspace{0.3cm}}p{2.5cm} @{\hspace{0.3cm}}p{1.5cm}@{\hspace{0.5cm}} p{4cm}@{\hspace{0.2cm}} @{\hspace{0.5cm}} p{3cm}}
\toprule[1.5pt]
Concentric sets & Short form & Example & Location\\
\midrule[1pt] 
$\check{D}_{x},\check{U}_{x}$ & $\check{D}$ & $W_{\ell}^{\check{D}}=W_{\ell}^{\check{D}_{x},\check{U}_{x}}$ & \Cref{prop:coupling W and widetildeZ} \\
$D_{x},U_{x}$ & $D$ & $\widetilde{Z}_{\ell}^{D}=\widetilde{Z}_{\ell}^{D_{x}.U_{x}}$ &\Cref{prop:coupling between widehatZ and W} \\
$D_{x'},U_{x'}$ & $D'$ & $\widetilde{Z}_{\ell}^{D'}=\widetilde{Z}_{\ell}^{D_{x'},U_{x'}}$ & \Cref{prop:coupling between widehatZ and W} \\
$\mathsf{B},\mathsf{D}$ & $\mathsf{B}$ & $Z_{\ell}^{\mathsf{B}}=Z_{\ell}^{\mathsf{B},\mathsf{D}}$ & \Cref{prop:very strong coupling} \\
$\overline{B}_{x},\overline{D}_{x}$ & $\overline{B}$ & $W_{\ell}^{\overline{B}}=W_{\ell}^{\overline{B}_{x},\overline{D}_{x}}$ & \Cref{prop:strong coupling} \\
$\B,\D$ & $\B$ & $Z_{\ell}^{\B}=Z_{\ell}^{\B,\D}$ & \Cref{thm:very strong coupling S}\\
$\A,\widetilde{\A}$ & $\A$ & $\widetilde{\W}_{\ell}^{\A}=\widetilde{\W}_{\ell}^{\A,\widetilde{\A}}$ & \Cref{prop:Poissonization}\\
\bottomrule[1.5pt]
\end{tabular}
\end{table}

We now introduce some important random quantities. 
For boxes $U\subseteq V$ (subsets of both $\mathbb{E}$ and $\mathbb{Z}^{d+1}$) and a record-breaking time $S$ we use $N_{{S}}(U)$ for the number of excursions of random walk from $U$ to $\partial V$ before time ${S}$, and $N_{u}(U)=N_{u}^{U,V}$ for the number of excursions (abbreviated as {\it \# exc.} below) in $\I^{u}$ from $U$ to $\partial V$. 
\vspace{5mm}
\begin{table}[htbp]
\renewcommand{\arraystretch}{1.3}
\setlength{\tabcolsep}{5mm}
\caption{Random variables}
\centering
\begin{tabular}{@{\hspace{0.2cm}}p{3.4cm} @{\hspace{0.3cm}}p{2.2cm}@{\hspace{0.3cm}} p{9.5cm}@{\hspace{0.2cm}}}
\toprule[1.5pt]
Symbol & Location & Description\\
\midrule[1pt]
$S_N(z)$ ($S_N(\omega, u, z)$) & \eqref{eq:def of Sz} & Record-breaking time when the local time at height $z$ exceeds $uN^d/(d+1)$.\\
$\underline S_N(z)$ ($\underline S_N(\omega, \delta, z)$) & \eqref{eq:def of underlineSz} & Record-breaking time $S_N(\omega, \overline{u}-\delta, z)$.\\
$\overline S_N(z)$ ($\overline S_N(\omega, \delta, z)$) & \eqref{eq:def of underlineSz} & Record-breaking time $S_N(\omega, u^{**}+\delta, z)$.\\
$\underline S_N$ $\quad \ $($\underline S_N(\omega, \delta)$) & \eqref{eq:def of underlineS and overlineS} & The infimum of $\underline S_N(z)$ over $z\in\mathbb{Z}$.\\
$\overline S_N$ $\quad \ $($\overline S_N(\omega, \delta)$) & \eqref{eq:def of underlineS and overlineS} & The infimum of $\overline S_N(z)$ over $z\in\mathbb{Z}$.\\
$\mathrm S_N(u)$ ($S_{N}(w,u,\mathrm{z}_{c})$) & \eqref{eq:def of mathrmS} & Record-breaking time w.r.t.\ a specific height $\mathrm{z}_{c}$\\
$N_{\underline{S}_{N}}(D)$ & \eqref{eq:def of N_SN}  
& \# exc.~of RW from $D$ to $\partial U$ before $\underline{S}_{N}$.\\
$N_{\underline{S}_{N}}(\mathsf{B})$ & Prop.~\ref{prop:very strong coupling},\,\ref{prop:very strong coupling bias} & \# exc.~of RW from $\mathsf{B}$ to $\partial\mathsf D$ before $\underline{S}_{N}$.\\
$N_{\underline{S}_{N}}(\overline{B})$ & Prop.~\ref{prop:strong coupling} & \# exc.~of RW from $\overline{B}$ to $\overline{D}$ before $\underline{S}_{N}$.\\
$N_{\mathrm{S}_{N}(u)}(\B)$ &Thm.~\ref{thm:very strong coupling S} & \# exc.~of RW from $\B$ to $\partial \D$ before $\mathrm{S}_{N}(u)$.\\
$N_{K}(\B)$ &Thm.~\ref{thm:very strong coupling excursion} & \# exc.~of RW from $\B$ to $\partial \D$ before  $D_{K}^{\A,\widetilde{\A}}$.\\
$N_{u}(\mathsf{B})$ ($N_{u}^{\mathsf{B},\mathsf{D}}$) & Prop.~\ref{prop:very strong coupling}, \ref{prop:very strong coupling bias} & \# exc.~in $\mathcal I^u$ from $\mathsf{B}$ to $\partial\mathsf{D}$.\\
$N_{u}(\mathsf{B})$ ($N_{u}^{\overline{B},\overline{D}}$) & Prop.~\ref{prop:strong coupling} & \# exc.~in $\mathcal I^u$ from $\overline{B}$ to $\partial\overline{D}$.\\
$N_{u}(\B)$ ($N_{u}^{\B,\D}$) & Prop.~\ref{thm:very strong coupling S}, \ref{thm:very strong coupling excursion} & \# exc.~in $\mathcal I^u$ from $\B$ to $\partial\D$.\\
\bottomrule[1.5pt]
\end{tabular}
\end{table}

\newpage

The following table lists the measures that appear in this paper.

\vspace{10mm}
\begin{table}[H]
\renewcommand{\arraystretch}{1.3}
\setlength{\tabcolsep}{10mm}
\caption{Measures}
\centering
\begin{tabular}{@{\hspace{0.2cm}}p{4.5cm}@{\hspace{0.2cm}} p{10.8cm}@{\hspace{0.2cm}}}
\toprule[1.5pt]
Symbol & Description and Location\\
\midrule[1pt]
$P_{x}^{N}$ & Simple random walk on $\mathbb{E}=\mathbb{T}\times(\mathbb{Z}/N\mathbb{Z})^{d}$ started from $x$.\\
$P_{x}^{N,\alpha}$ & Random walk with upward drift $N^{-d\alpha}$ on $\mathbb{E}=\mathbb{T}\times(\mathbb{Z}/N\mathbb{Z})^{d}$ started from $x$.\\
$P_{x}$ & Simple random walk on $\mathbb{Z}^{d+1}$ started from $x$.\\
$P_{x}^{\Delta}$ & Random walk with upward drift $\Delta$ on $\mathbb{Z}^{d+1}$ started from $x$.\\
$P_{\mu}^{N}(P_{\mu}^{N,\alpha},P_{\mu},P_{\mu}^{\Delta})$ & The measure (not essentially a probability measure) $\Sigma_{x\in\mathbb{E}}\mu(x)P_{x}^{N}$ (similar for other three) with initial distribution $\mu$ on $\mathbb{E}$ (or $\mathbb{Z}^{d+1}$). \\
$\PP$ & Law of random interlacements (and also continuous-time random interlacements, see \Cref{subsec:continuous-time interlacements}).\\
$\mathbb{W}$ & The Wiener measure. Mainly appears in \Cref{sec:one-dimensional walks,sec:denouement}.\\
$\mathsf{P}_{x}^{\Delta}$ & One-dimensional random walk started from $x$ with drift $\Delta$, see \Cref{subsec:random walk preliminary}.\\
$\mathbb{Q}_{W,\widetilde{Z}}$ (resp.~$\mathbb{Q}_{\widetilde{Z},Z}$) & The coupling between excursions of $W$-type and $\widetilde{Z}$-type (resp.~$\widetilde{Z}$-type and $Z$-type) in the $\alpha=\infty$ case, see \Cref{prop:coupling W and widetildeZ} to \ref{lem:coupling between widetildeZ and Z}.\\
$\mathbb{Q}_{W,\widetilde{W}}^{N,\alpha}$ (resp.~$\mathbb{Q}_{\widetilde{W},\widetilde{Z}}^{N,\alpha},\mathbb{Q}_{W,\widetilde{Z}}^{N,\alpha}$) & The coupling between excursions of $W$-type and $\widetilde{W}$-type (resp.~$\widetilde{W}$-type and $\widetilde{Z}$-type, ${W}$-type and $\widetilde{Z}$-type) for  $\alpha<\infty$, see \Cref{subsec:adapting bad(beta-gamma)}.\\
$\underline{Q}$ & The ``very strong" coupling of the laws $P_{0}^{N}$ and $\PP$ in \Cref{prop:very strong coupling} (focusing on excursions from $\mathsf{B}$ to $\partial\mathsf{D}$).\\
$\underline{Q}^{N,\alpha}$ & The ``very strong" coupling of the laws $P_{0}^{N,\alpha}$ and $\PP$ in \Cref{prop:very strong coupling bias} (focusing on excursions from $\mathsf{B}$ to $\partial\mathsf{D}$).\\
$\overline{Q}_{\overline{B}}^{N,\alpha}$ & The ``strong" coupling of the laws $P_{0}^{N,\alpha}$ and $\PP$ in \Cref{prop:strong coupling} (focusing on excursions from $\overline{B}$ to $\partial\overline{D}$).\\
$Q$ & The very strong coupling of the laws $P_{0}^{N}$ and $\PP$ in \Cref{thm:very strong coupling S,thm:very strong coupling excursion} (focusing on excursions from $\B$ to $\D$).\\
$P_{q_{z}}^{N},P_{q}^{N},P_{z_{1},z_{2}}^{N}$  & The probability measures defined by \eqref{eq:def of qz}, \eqref{eq:def of q} and \eqref{eq:def of Pz1z2}.\\
$P_{q_{z}}^{N,\alpha},P_{q}^{N,\alpha},P_{z_{1},z_{2}}^{N,\alpha}$ & The biased walk counterparts of $P_{q_{z}}^{N},P_{q}^{N},P_{z_{1},z_{2}}^{N}$.\\
$Q_{1}$-$Q_{9}$ & The chain of couplings for proving \Cref{thm:very strong coupling excursion}, see \Cref{prop:horizontal independence,prop:vertical independence,prop:Poissonization,prop:Handling the first excursion,prop:extraction,prop:lower truncation,prop:lower interlacement,prop:upper truncation,prop:upper interlacement}.\\
$Q_{\mathcal{X}}$ & The perfect coupling in \Cref{lem:horizontal independence}.\\
$Q_{4}',Q_{4}^{\alpha},Q_{10},Q_{11}$ & Couplings appearing in \Cref{subsec:adapted proof very strong coupling} (adapted proofs of \Cref{prop:very strong coupling,prop:very strong coupling bias,prop:strong coupling}) that are respectively defined in \eqref{eq:Handling the first excursion adaped simple},\eqref{eq:Handling the first excursion adaped simple and bias} and \Cref{prop:Poissonian coupling lower,prop:Poissonian coupling upper}. Here $Q_{4}'$ and $Q_{4}^{\alpha}$ are adapted versions of $Q_{4}$ in \Cref{prop:Handling the first excursion}.\\
\bottomrule[1.5pt]
\end{tabular}
\end{table}

\newpage
We also list different events describing properties of boxes and their complements. The box $B$ in the following description refers to the $L$-box that we will take into consideration.

\vspace{10mm}
\begin{table}[htbp]
\renewcommand{\arraystretch}{1.3}
\setlength{\tabcolsep}{10mm}
\caption{Events}
\centering
\begin{tabular}{@{\hspace{0.2cm}}p{2.6cm} @{\hspace{0.5cm}}p{1.6cm}@{\hspace{0.5cm}} p{7cm}@{\hspace{0.5cm}}p{1.8cm}}
\toprule[1.5pt]
Event & Location & Description & Complement \\
\midrule[1pt]
$\text{Exist}(B,X, a)$& \eqref{def:exist event}& Existence of large clusters for the complement of excursion type $X$ in $B$. &$\text{fail}_1(B,X, a)$\\
$\text{Unique}(B,X, a, b)$& \eqref{def:unique event} & Uniqueness of large clusters for the complement of excursion type $X$ in $B$.&$\text{fail}_2(B,X, a, b)$\\
$\mbox{good}(\beta, \gamma)$ & Def.~\ref{def:good(beta-gamma)} & ``Strongly-percolative'' property for the complement of random walk in $B$. & $\mbox{bad}(\beta, \gamma)$\\
$\widehat{\text{good}}(\widehat{\beta},\widehat{\gamma})$ & Def.~\ref{def:widehatgood}~ & ``Strongly-percolative'' property for the complement of $\widetilde Z_\ell^{D}$ excursions in $B$. & $\widehat{\text{bad}}(\widehat{\beta},\widehat{\gamma})$\\
$\mbox{fine}(\gamma)$ & Def.~\ref{def:bad local time} & The local time of random walk before time $\underline S_N$ in $B$ is not too large. & $\mbox{poor}(\gamma)$ \\
$\widehat{\text{fine}}(\widehat \gamma)$ & Def.~\ref{def:fine box} & The local time of random interlacements $\mathcal I^{u'}$ in $B$ is not too large. &$\widehat{\text{poor}}(\widehat{\gamma})$ \\
normal$(\beta, \gamma)$ & Def.~\ref{def:bad boxes} & The intersection of good$(\beta, \gamma)$ and fine$(\gamma)$. & abnormal$(\beta, \gamma)$\\
$\widehat{\text{regular}}(\gamma,\theta)$  & Def.~\ref{def:regular theta} & The weighted local time in the excursions of continuous-time interlacements is not too small. & $\widehat{\text{irregular}}(\gamma,\theta)$\\

\bottomrule[1.5pt]
\end{tabular}
\end{table}

\bibliographystyle{plain}
\bibliography{reference}
\appendix
\section{Sketch of an alternative  proof for \Cref{thm:sharp lower bound}}\label{sec:appendix}

In this section we give a simpler proof for the sharp lower bound on the disconnection time of simple random walk. Recall the definition of $\underline{S}_{N}$ in \eqref{eq:def of underlineS and overlineS}. We shall sketch an alternative proof that $T_{N}\leq \underline{S}_{N}$ holds with high probability, which, combined with \Cref{sec:one-dimensional walks}, yields \Cref{thm:sharp lower bound} as the proof of \eqref{eq:lower bound large alpha} in \Cref{sec:denouement}. We remark here that this alternative proof does not work for disconnection time of general biased random walks with drift $N^{-d\alpha}$ for all $\alpha \in (1/d, 1)$; see \Cref{remark:alternative proof not work}.

The basic idea is to use the coupling between simple random walk on cylinder and random interlacements (see \Cref{thm:very strong coupling S,thm:very strong coupling excursion} and \Cref{prop:very strong coupling}). Consider coarse-grained boxes of side-length, say, $[\sqrt{N}]$ on $\mathbb E = \mathbb T\times\mathbb{Z}$. For any such box $B$ and neighbouring boxes $B'$, 
\begin{equation}\label{eq:vacant set cylinder}
\Big(\bigcup_{B'\text{ neighbouring B}}B'\cup B\Big)\setminus X_{[0,\underline{S}_{N}]}\text{ stochastically dominates }\Big(\bigcup_{B'\text{ neighbouring B}}B'\cup B\Big)\cap \V^{u}
\end{equation}
for some $u\leq \overline{u}-\delta/2$, and consequently there exist connected components with diameter of order $\sqrt{N}$ in $B$ and each $B'$ as guaranteed by existence property of $\overline{u}$, which also enjoy good connectivity as guaranteed by local uniqueness property of $\overline{u}$. We then construct a consecutive coarse-grained path of boxes from $\mathbb{T}\times(-\infty,-N^{2d+1}]$ to $\mathbb{T}\times[N^{2d+1},\infty)$, and then connect together the connected components with a large diameter in vacant set of these boxes to form a path of vertices in $\mathbb{E}\setminus X_{[0,\underline{S}_{N}]}$.

The main problem in the above idea is that the local uniqueness property is not monotone. Although a coupling can help locate a large connected component in a box $B$, the large connected components determined by different couplings may not coincide with each other, thus the gluing procedure may fail. To overcome this difficulty, we use the following observation: Suppose the average local time of $X_{[0,\underline{S}_{n}]}$ in a box $B$ is $u(\leq \overline{u}-\delta)$, then with high probability, there exists a unique connected component in $B$ with size larger than $\frac{3}{4}\theta(u)|B|$, where $\theta(u)$ is the percolation function of $\V^{u}$, that is,
\begin{equation}\label{eq:def of thetau}
\theta(u):=\PP[0\overset{\V^{u}}{\longleftrightarrow}\infty],    
\end{equation}
and $|B|$ denotes the size of $B$. We will argue that every coupling between $B\setminus X_{[0,\underline{S}_{N}]}$ and vacant set of random interlacements can help distinguish this unique largest connected component, and then connect largest componenets within a coarsed-grained path of boxes to form a path of vertices in $\mathbb{E}\setminus X_{[0,\underline{S}_{N}]}$. This time, since different couplings actually determine the same connected component in each box, the gluing procedure works.

To prove the key observation we shall use the following property on $\V^u$: With high probability, the number of points in $B$ that lies in a connected component of $B\cap\V^{u}$ with diameter larger than $N^{1/4}$ is between $\frac{3}{4}\theta(u)|B|$ and $\frac{5}{4}\theta(u)|B|$. Now by repeatedly using local uniqueness property of $\V^{u}$ with $u<\overline{u}$ in boxes with side-length $[N^{1/4}]$ lying in $B$, with error term stretched-exponential in $N$, all the connected sets in $B\cap\V^{u}$ whose diameters are larger than $N^{1/4}$ and whose distances to $\partial B$ are larger than $N^{1/4}$ shall lie in the same connected component of $B\cap \V^{u}$. Combining the last two facts with the coupling between $B\setminus X_{[0,\underline{S}_{n}]}$ and $B\cap\V^{u}$ ensures the existence of a connected component in $B$ with size larger than $\frac{3}{4}\theta(u)|B|$, and the uniqueness follows from the coupling and the fact that $2\cdot\frac{3}{4}>\frac{5}{4}$. 

We remark here that similar ideas have appeared in the literature. Indeed, the existence part draws inspirations from the proof of Theorem 1.4 in \cite{TW11} (see \Cref{remark:largest connected component} for more discussions), while the idea using $2\cdot\frac{3}{4}>\frac{5}{4}$ in the uniqueness part also appears in \cite{RS13}. 

We now give a more detailed proof of \Cref{thm:sharp lower bound}, in which we cite and adapt many results in \cite{RS13}. We remark here that a large part of the notation, including the $E$-type and $F$-type events defined in \eqref{eq:def of good increasing event}-\eqref{eq:def of family of bad decreasing}, the event $\overline{H}^{*}(x,N)$ defined in the paragraph above \eqref{eq:union bound of overlineH}, notation regarding the renormalized lattice defined in \eqref{eq:def of L0 and l0}-\eqref{eq:lambda xn}, etc., is in line with the notation in \cite{RS13} for sake of coherence and readers' convenience. However, these symbols should not be confused with these in \Cref{sec:intro,sec:preliminary,sec:one-dimensional walks,sec:Geometric,sec:bad(beta-gamma),sec:bad local time,sec:lower bound in biased case,sec:upper bound,sec:couplings,sec:denouement}.\\

We only need to prove \Cref{prop:underlineS and T} in the simple random walk case. We now introduce a planar strip of disjoint coarsed-grained boxes in the cylinder $\mathbb{E}$. Let $L=[\sqrt{N}]$ and $B_{0}=[0,L)^{d+1}$. We remark here that the choice of $L$ is arbitrary as long as it is of order $N^{c}$ with $c\in(0,1)$. For any $x\in\mathbb{E}$, let $B_{x}=B_{0}+x$, and we call these boxes $L$-boxes. Recall the unit vectors $e_{1},\dots,e_{d+1}$ in \Cref{sec:preliminary}. We say that two $L$-boxes $B_{x_{1}}$ and $B_{x_{2}}$ are neighbours if $x_{2}=x_{1}+Le_{i}$ for some $1\leq i\leq d+1$, and say $B_{x_{1}}$ and $B_{x_{2}}$ are $*$-neighbours if $|(x_{2}-x_{1})/L|_{\infty}=1$, with $|\cdot|_{\infty}$ denoting the $L^{\infty}$-norm. We write
\begin{equation}\label{eq:def of strip}
V:=\left\{ae_{1}+ze_{d+1}:a\in L\mathbb{Z}\cap\left[-\frac{1}{4}N,\frac{1}{4}N\right]\text{ and }z\in L\mathbb{Z}\cap\left[-N^{7d},N^{7d}\right]\right\},    
\end{equation}
where the choice of $1/4$ and $7$ is rather arbitrary, and write
\begin{equation}\label{eq:def of BS}
B(V):=\bigcup_{x\in V}B_{x}.    
\end{equation}
The left and right sides (resp.~up and down sides) of $B(V)$ are two longest line segments in the boundary of $B(V)$ that are parallel with $e_{d+1}$ (resp.~$e_{1}$). The boxes $B_{x},x\in V$ will play a central role in the following analysis.

We say an $L$-box $B=B_{x},x\in V$ is a \textit{nice} box if the largest connected of $B\setminus X_{[0,\underline{S}_{N}]}$ is connected with the largest connected component of $B'\setminus X_{[0,\underline{S}_{N}]}$ in $\mathbb{E}\setminus X_{[0,\underline{S}_{N}]}$, for any $B'$ neighbouring $B$. Then a path of neighbouring nice boxes yields a path of vertices starting from the first box and ending at the last box in $\mathbb{E}\setminus X_{[0,\underline{S}_{N}]}$. If a box $B$ is not nice, we say $B$ is \textit{nasty}. 

We now take $D_{x}=B(x,\frac{N}{2\log^{3}N})$ and $U_{x}=B(x,N/\log^{3}N)$ for all $x\in V$, and denote by $\overline{H}^{*}(x,N)$ the event that $B_{x}$ is connected to the outside of $D_{x}$ by a path of $*$-neighbouring nasty $L$-boxes in $B(V)$ (This notation is in line with that in \cite{RS13}, and should not be confused with the $H$-type events introduced in \Cref{lem:coupling between widetildeZ and Z}). Then by planar duality \cite[Proposition 2.1]{Kes82}, if $T_{N}\leq \underline{S}_{N}\leq N^{5d}$, then there exists a path of $*$-neighbouring nasty $L$-boxes in $B(V)$ starting from the left side of $B(V)$ and ending at the right side of $B(V)$ \footnote[1]{A corollary of \cite[Proposition 2.1]{Kes82} is that the outer boundary of a finite connected set in $\mathbb{Z}^{2}$ is $*$-connected in $\mathbb{Z}^{2}$. Here we will use this corollary on the connected component of nice boxes in $B(V)$ that contains the up side of $B(V)$, where connection of boxes is determined by the ``neighbouring" relation.}. Under this circumstance, there exists an $x\in V$ such that $\overline{H}^{*}(x,N)$ occurs. Therefore, by a union bound on $x$, to prove \Cref{prop:underlineS and T} it suffices to prove that, for every $N$ and some absolute constants $C$ and $c$,
\begin{equation}\label{eq:union bound of overlineH}
\sup_{x\in V\atop U_{x}\subseteq B(V)}P_{q_{0}}^{N}\left[\overline{H}^{*}(x,N)\right]\leq C\exp\left(N^{-c}\right),    
\end{equation}
where we recall that $q_{0}$ is defined in \eqref{eq:def of qz} as the uniform distribution on $\mathbb{T}\times\{0\}$. Note that we can change the law from $P_{0}^{N}$ into $P_{q_{0}}^{N}$ thanks to the symmetry of torus. (We cannot do this in \Cref{prop:very strong coupling,prop:very strong coupling bias} because the position of the random box $\mathsf{B}$ depends on the starting point of the walk.)

Now for the percolation function $\theta(\cdot)$ in \eqref{eq:def of thetau}, we have
\begin{equation}\label{eq:def of u*}
0<\overline{u}-\delta<\overline{u}\overset{\eqref{eq:unique critial parameter}}{=}u_{*}=\inf\{u\geq 0,\theta(u)=0\}\in (0,\infty).    
\end{equation}
By \cite[Corollary 1.2]{Tei09}, $\theta(\cdot)$ is continuous on $[0,u_{*})$, and therefore there exists a positive integer $k$ and absolute constants $u_{i}$, $0\leq i\leq k+2$ satisfying
\begin{equation}\label{eq:def of ui}
\begin{aligned}
0=u_{0}<u_{1}<u_{2}<\dots<u_{k-1}&\leq \overline{u}-\delta<u_{k}<u_{k+1}<u_{k+2}\leq \overline{u}-\frac{\delta}{2},\\
\theta(u_{10})\geq \frac{2}{3},&\quad\theta(u_{k+2})>0,\quad\text{and}\\
\theta(u_{i+2})\geq \frac{10^{4}-1}{10^{4}}&\cdot\theta(u_{i-2}),\quad\text{ for all }3\leq i\leq k.
\end{aligned}
\end{equation}
Then to prove \eqref{eq:union bound of overlineH} it suffices to prove the following proposition.

\begin{proposition}\label{prop:union bound of overlineH and uk}
Consider an arbitrary $x=ae_{1}+ze_{d+1}$ such that $x\in V$ and $U_{x}\subseteq B(V)$. Recall the definition of $S_{N}(\omega,u,z)$ as in \eqref{eq:def of Sz}. For $1\leq i\leq k$, define
\begin{equation}\label{eq:def of Rk}
R_{i}:=S_{N}(\omega,u_{i},z),   
\end{equation}
that is, the first time when the ``average local time" of level $\mathbb{T}\times\{z\}$ equals $u_{i}$ (Note that almost surely $\underline{S}_{N}\leq R_{k}$). Then for every $N$ and some absolute constants $C$ and $c$ we have
\begin{equation}\label{eq:union bound of overlineH and u5}
P_{q_{0}}^{N}\left[\overline{H}^{*}(x,N),0\leq \underline{S}_{N}\leq R_{3}\right]\leq C\exp\left(-N^{c}\right); 
\end{equation}
\begin{equation}\label{eq:union bound of overlineH and uk}
P_{q_{0}}^{N}\left[\overline{H}^{*}(x,N),R_{i-1}\leq\underline{S}_{N}\leq R_{i}\right]\leq C\exp\left(-N^{c}\right),\quad\text{for all }3\leq i\leq k. 
\end{equation}
\end{proposition}

\begin{proof}
We first prove \eqref{eq:union bound of overlineH and uk} using \Cref{thm:very strong coupling S} and the mechanism in \cite{RS13}, combining ideas from \cite{TW11}. The proof of \eqref{eq:union bound of overlineH and u5} follows a similar fashion. Fix $3\leq i\leq k$. It follows from \Cref{thm:very strong coupling S} that one can construct a coupling $Q_{i}$ between simple random walk on cylinder and random interlacements satisfying
\begin{equation}\label{eq:very strong coupling ui}
Q_{i}\left[U_{x}\cap\V^{u_{i+1}} \subseteq U_{x}\setminus X_{[0,R_{i}]} \subseteq U_{x}\setminus X_{[0,R_{i-1}]} \subseteq U_{x}\cap\V^{u_{i-2}}\right]\geq 1-C\exp\left(-N^{c}\right).    
\end{equation}
In the following, we assume the coupling $Q_{i}$ and view $U_{x}$ as a subset of both $\mathbb{E}$ and $\mathbb{Z}^{d+1}$, depending on the context.

Now let 
\begin{equation}\label{eq:def of L0 and l0}
L_{0}=L=[\sqrt{N}]\quad\text{and}\quad\ell(d)=300\cdot 4^{d+1}.   
\end{equation}
We also choose positive integer $\ell_{0}\geq 10\ell(d)$ and set
\begin{equation}\label{eq:def of Ln}
L_{n}=\ell_{0}^{n}L_{0},\quad\text{for all }n\geq 1.    
\end{equation}
The renormalized lattice of $V$ is then defined as
\begin{equation}\label{eq:renormalized lattice of S}
V_{n}:=V\cap L_{n}\mathbb{Z}^{d+1},\quad\text{for all }n\geq 0,
\end{equation}
and we set 
\begin{equation}\label{eq:lambda xn}
\Lambda_{y,n}:=V_{n-1}\cap (y+[0,L_{n})^{d+1}),\quad\text{for all }n\geq 1.   
\end{equation}
Note that the above definitions are similar to those in \cite[Section 3.3]{RS13}, except that $d+1$ here plays the role of $d$ in \cite{RS13}.

We now define a family of bad increasing events (resp.~bad decreasing events) with respect to the interlacements set and bound the probability of bad increasing events (resp.~bad decreasing events) from above as in Section 4.2 (resp.~Section 4.1) of \cite{RS13}. 
Note that the monotonicity of $E$-type or $F$-type events here is different from that in \cite{RS13} since \cite{RS13} focuses on the connectivity of $\I^{u}$ while we care about the connectivity of $\V^{u}$, but this does not affect the proof.

For $K\subseteq U_{x}$ and $y\in K$, we say $y$ is \emph{long-connected} in $K$ if $y$ lies in a connected subset of $K$ with diameter larger than $N^{1/4}$. For $B_{y}\subseteq B(V)\cap U_{x}$ and $u\in\mathbb{R}$, write the ``good" decreasing event $E_{y}^{u_{i+2}}(\V^{u})$ and the ``good" increasing event $F_{y}^{u_{i-2}}(\V^{u})$ as
\begin{align}
E_{y}^{u_{i+2}}(\V^{u}):=\Big\{\text{there are at least }\frac{3}{4}\cdot\theta(u_{i+2})|B_{y}|\text{ long-connected points in }\V^{u}\cap B_{y}\Big\}; \label{eq:def of good decreasing event}  \\  
F_{y}^{u_{i-2}}(\V^{u}):=\Big\{\text{there are at most }\frac{5}{4}\cdot\theta(u_{i-2})|B_{y}|\text{ long-connected points in }\V^{u}\cap B_{y}\Big\}.   \label{eq:def of good increasing event} 
\end{align}
Then by the continuity of $\theta(\cdot)$ and the ergodic theorem (see also the proofs of Lemmas 4.5 and 4.2 in \cite{RS13} for similar arguments), there exists a small positive constant $\eta=\eta(u_{i+2},u_{i-2},\delta)$ so that
\begin{equation}\label{eq:good event with high probability}
\lim_{N\to\infty}\PP\left[E_{0}^{u_{i+2}}(\V^{u_{i+2}/(1-\eta)})\right]=1,\quad\text{and}\quad\lim_{N\to\infty}\PP\left[F_{0}^{u_{i-2}}(\V^{u_{i-2}/(1+\eta)})\right]=1.
\end{equation}
The bad increasing seed event $\overline{E}_{y,0}^{u_{i+2}}(\V^{u})$ is defined as the complement of $E_{y}^{u_{i+2}}(\V^{u})$, and the family of cascading bad increasing events are defined via
\begin{equation}\label{eq:def of family of bad increasing}
\overline{E}_{y,n}^{u_{i+2}}(\V^{u}):=\bigcup_{y_{1},y_{2}\in \Lambda_{y,n};|y_{1}-y_{2}|_{\infty}>\frac{L_{n}}{\ell(d)}}\overline{E}_{y_{1},n-1}^{u_{i+2}}(\V^{u})\cap\overline{E}_{y_{2},n-1}^{u_{i+2}}(\V^{u}),\quad\text{for all }n\geq 1.
\end{equation}
The bad decreasing seed event $\overline{F}_{y,0}^{u_{i-2}}(\V^{u})$ is defined as the complement of $F_{y}^{u_{i-2}}(\V^{u})$, and the family of cascading bad decreasing events are defined via
\begin{equation}\label{eq:def of family of bad decreasing}
\overline{F}_{y,n}^{u_{i-2}}(\V^{u}):=\bigcup_{y_{1},y_{2}\in \Lambda_{y,n};|y_{1}-y_{2}|_{\infty}>\frac{L_{n}}{\ell(d)}}\overline{F}_{y_{1},n-1}^{u_{i-2}}(\V^{u})\cap\overline{F}_{y_{2},n-1}^{u_{i-2}}(\V^{u}),\quad\text{for all }n\geq 1.
\end{equation}
It now follows from \eqref{eq:good event with high probability} and the decoupling inequality (c.f~\cite[Theorem 3.4]{Szn12} or \cite[Corollary 3.4]{RS13}) that (see also the proofs of Corollaries 4.6 and 4.3 in \cite{RS13} for similar arguments) for every $n\geq 0$,
\begin{equation}\label{eq:bad event small probability}
\lim_{N\to\infty}\PP\left[\overline{E}_{0,n}^{u_{i+2}}(\V^{u_{i+2}})\right]\leq 2^{-2^{n}},\quad\text{and}\quad\lim_{N\to\infty}\quad\PP\left[\overline{F}_{0,n}^{u_{i-2}}(\V^{u_{i-2}})\right]\leq 2^{-2^{n}}.   
\end{equation}

For $y\in D_{x}$, let
\begin{equation}\label{eq:def of Bhat}
L' = [L-\sqrt{L}], \text{ and }
\widehat{B}_{y}:=y+[0,L')^{d+1}
\end{equation}
as a subset of $B_{y}$. We then define the event $J_{y}$ for each $y\in B(x,\frac{2N}{3\log^{3}N})$ as
\begin{equation}\label{eq:def of event Jy}
\begin{aligned}
J_{y}:=\big\{&\text{every connected subset of }\V^{u_{i+2}}\cap \widehat{B}_{y}\text{ with diameter larger than }\sqrt{L}\\ &\text{ is connected in }\V^{u_{i+1}}\cap B_{y}\big\},  
\end{aligned}
\end{equation}
and also set
\begin{equation}\label{eq:def of event J}
J(x):=\bigcap_{y\in B(x,\frac{2N}{3\log^{3}N})}J_{y}.    
\end{equation}
Then since $u_{i+1}<u_{i+2}<\overline{u}$, by \eqref{eq:def of ubar} we have
\begin{equation}\label{eq:prob of J}
\PP\left[J(x)^{c}\right]\leq C\exp(-N^{c}).    
\end{equation}

Now for some $3\leq i\leq k$, under the coupling $Q_{i}$, assume that $N\geq C(\overline{u},\delta)$ and that $R_{i-1}\leq\underline{S}_{N}\leq R_{i}$. For an $L$-box $B_{y}\subseteq B(V)\cap U_{x}$, suppose that $J(x)\cap E_{y}^{u_{i+2}}(\V^{u_{i+2}})\cap F_{y}^{u_{i-2}}(\V^{u_{i-2}})$ occurs, and denote by $\mathcal{C}_{y}$ the largest connected component in $B_{y}\setminus X_{[0,\underline{S}_{N}]}$. Then since $J(x)\cap E_{y}^{u_{i+2}}$ holds, it follows that 
\begin{equation}\label{eq:size of mathcalCy lower bound}
|\mathcal{C}_{y}|\geq\frac{11}{16}\cdot\theta(u_{i+2})|B_{y}|.    
\end{equation}
In addition, since $ F_{y}^{u_{i-2}}(\V^{u_{i-2}})$ holds, it follows from \eqref{eq:def of ui} that $\mathcal{C}_{y}$ is the only connected component in $B_{y}\setminus X_{[0,\underline{S}_{N}]}$ that satisfies \eqref{eq:size of mathcalCy lower bound}. Therefore, if $J(x)$ holds and $B_{y}\subseteq B(V)\cap U_{x}$ is a nasty box, then one of nine $y'\in V$ with $|y'-y|_{\infty}\leq L$ does not satisfy $ E_{y}^{u_{i+2}}(\V^{u_{i+2}})\cap F_{y}^{u_{i-2}}(\V^{u_{i-2}})$. Then by the same argument as in the proof of Lemma 5.2 given Corollaries 4.3 and 4.6 and Lemma 4.7 in \cite{RS13}, thanks to \eqref{eq:bad event small probability}, we have
\begin{equation}\label{eq:union bound of overlineH and uk given J}
P_{q_{0}}^{N}\left[\overline{H}^{*}(x,N),R_{i-1}\leq\underline{S}_{N}\leq R_{i},J(x)\right]\leq C\exp\left(-N^{c}\right),
\end{equation}
and \eqref{eq:union bound of overlineH and uk} then follows from combining \eqref{eq:prob of J} and \eqref{eq:union bound of overlineH and uk given J}.

The proof of \eqref{eq:union bound of overlineH and u5} follows in a similar fashion. This time we fix a coupling $Q_{3}$ that satisfies \begin{equation}\label{eq:very strong coupling u3}
Q_{3}\left[U_{x}\cap\V^{u_{4}} \subseteq U_{x}\setminus X_{[0,R_{3}]}\right]\geq 1-C\exp\left(-N^{c}\right),    
\end{equation}
and consider the ``good" increasing event $E_{y}^{u_{5}}(\V^{u})$ as
\begin{equation}\label{eq:def of good increasing event u5}
E_{y}^{u_{5}}(\V^{u}):=\left\{\text{there are at least }\frac{2}{3}|B_{y}|\text{ long-connected points in }\V^{u}\cap B_{y}\right\}.
\end{equation}
Then by the assumption that $\theta(u_{10})\geq 2/3$ in \eqref{eq:def of ui}, $E_{0}^{u_{5}}(\V^{u_{5}})$ still holds with high probability. In addition, under the coupling $Q_{3}$, assuming that $N\geq C(\overline{u},\delta)$ and that $\underline{S}_{N}\leq R_{3}$, when $J(x)\cap E_{y}^{u_{5}}(\V^{u_{5}})$ occurs, it holds that 
\begin{equation}\label{eq:size of mathcalCy lower bound u3}
|\mathcal{C}_{y}|\geq\frac{5}{9}\cdot|B_{y}|,   
\end{equation}
and $\mathcal{C}_{y}$ is again the unique connected component in $X_{[0,\underline{S}_{N}]}$ that satisfies \eqref{eq:size of mathcalCy lower bound u3} thanks to \eqref{eq:def of ui}. The other necessary ingredients of the proof of \eqref{eq:union bound of overlineH and uk} remain roughly the same as those of \eqref{eq:union bound of overlineH and uk}, except that we do not need the family of bad increasing events now.
\end{proof}

\begin{remark}[Failure for general biased walk]\label{remark:alternative proof not work}
The above proof does not work for disconnection time of biased random walks with large drift $N^{-d\alpha}$, because the error term in \eqref{eq:union bound of overlineH} explodes when we apply a union bound on $x$. More precisely, when the drift is $N^{-d\alpha},\alpha\in(0,1)$, the disconnection time $T_{N}$ is expected to be of order $\exp(cN^{d(1-\alpha)})$, therefore there will be $\exp(cN^{d(1-\alpha)})$ terms in the union bound, requiring the error term in \eqref{eq:union bound of overlineH} to be smaller than $\exp(-cN^{d(1-\alpha)})$. However, both the decoupling inequality and \eqref{eq:def of ubar} only provide stretched exponential decay, and hence the method fails for $\alpha\in(0,1-1/d)$.
\end{remark}

\begin{remark}[Size of the largest connected component in vacant sets of random walks and random interlacements]\label{remark:largest connected component}
In \cite{TW11}, the author obtains some results on the size of largest (in terms of volume) connected components of $\mathbb{T}_N\setminus X_{[0,uN^{d}]}$, with $\mathbb{T}_N$ a torus with side-length $N$ and $(X_{n})_{n\geq 0}$ being a simple random walk on torus $\mathbb{T}$, and further introduces a critical parameter $\widehat{u}$ of random interlacements on $\mathbb{Z}^{d},d\geq 3$ (see \cite[(1.11), Definition~2.4]{TW11}), which is also conjectured to be equal to $u_{*}$; see \cite[Section 1.3]{DCGR+23b} for more discussions. It is shown in \cite[Theorems 3.2 and 3.3]{Tei11} that $\widehat{u}>0$ for every $d\geq 5$, and by definition, for every ``strongly supercritical" $u\in [0,\widehat{u})$, the vacant set $\V^{u}$ enjoys similar but slightly stronger existence property and local uniqueness property as those with respect of  $\overline{u}$ in \eqref{eq:def of existence} and \eqref{eq:def of uniqueness} (and hence $\widehat{u}\leq\overline{u}$). 

We denote by $\mathcal{C}_{\text{max}}^{u}$ the largest connected components in $\mathbb{T}_{N}\setminus X_{[0,uN^{d}]}$. We now have the following results. 
\begin{theorem}[\cite{TW11}, Theorem 1.4]\label{thm:size of largest component TW11}
If $u\in[0,\widehat{u})$, then for $\theta(u)$ as defined in \eqref{eq:def of thetau} and every $\varepsilon>0$,
\begin{equation}\label{eq:size of largest component TW11}
\lim_{N\to\infty}P\left[\left|\frac{|\mathcal{C}_{\rm{max}}^{u}|}{N^{d}}-\theta(u)\right|>\varepsilon\right]=0.   
\end{equation}
\end{theorem}

\begin{theorem}[\cite{DCGR+23b}, Theorem 1.1]\label{thm:size of largest component DCGR+23b}
For every $u\in[0,u_{*})$ and every $\varepsilon>0$,
\begin{equation}\label{eq:size of largest component DCGR+23b}
\lim_{N\to\infty}P\left[\frac{|\mathcal{C}_{\rm{max}}^{u}|}{N^{d}}>\theta(u)-\varepsilon\right]=1.   
\end{equation}
\end{theorem}
The idea for \eqref{eq:size of largest component TW11} is similar to that of \Cref{prop:union bound of overlineH and uk}, where  a coupling between random walks and random interlacements is used in combination with the estimate on the number of ``long-connected" points (see the paragraph above \eqref{eq:def of good decreasing event} for definition) via the ergodic theorem as well as repeated use of local uniqueness property of $\overline{u}$ (or $\widehat{u}$ in \cite{TW11}) in small boxes. By the same arguments, we can slightly improve \Cref{thm:size of largest component TW11,thm:size of largest component DCGR+23b} into \Cref{thm:size of largest component}. 
\begin{proposition}\label{thm:size of largest component}
The claim \eqref{eq:size of largest component TW11} holds for every $u\in [0,u_{*})$ thanks to \eqref{eq:unique critial parameter}.
\end{proposition}
We remark that a result similar to \Cref{thm:size of largest component} still holds if one replace the torus $\mathbb{T}$ and simple random walk $(X_{n})_{n\geq 0}$ on $\mathbb{T}_{N}$ by a box with side length $N$ in $\mathbb{Z}^{d},d\geq 3$ and random interlacements on $\mathbb{Z}^{d}$.  

It is a natural question whether one can further improve these results. 
\begin{conjecture}\label{conjecture:improvement}
For every $u\in[0,u_{*})$ and every $\varepsilon>0$, there exist positive constants $C=C(u,\varepsilon)$ and $c=c(u,\varepsilon)$ such that
\begin{equation}\label{eq:improvement largest}
P\left[\left|\frac{|\mathcal{C}_{\text{max}}^{u}|}{N^{d}}-\theta(u)\right|>\varepsilon\right]\leq C\exp\left(-N^{c}\right).    
\end{equation}
Moreover, similar results still hold if one replaces the torus $\mathbb{T}_{N}$ and simple random walk $(X_{n})_{n\geq 0}$ on $\mathbb{T}$ by a box with side length $N$ in $\mathbb{Z}^{d},d\geq 3$ and random interlacements on $\mathbb{Z}^{d}$.  
\end{conjecture}
Combined with the ideas in the alternative proof of \Cref{thm:sharp lower bound} sketched in this section, proving \Cref{conjecture:improvement} (or a weaker version with any error term smaller than $N^{-2d}$) would give rise to an even more concise proof of \Cref{thm:sharp lower bound} that does not involve the decoupling inequality for bootstrapping, in which we simply glue the largest connected components within a straight tunnel of boxes to form a path of vertices in $\mathbb{E}\setminus X_{[0,\underline{S}_{N}]}$. 
As a final remark, for the same reason as  \Cref{remark:alternative proof not work}, even with the help of \eqref{eq:improvement largest}, the techniques in this section is still not sufficient to derive sharp asymptotics for biased walks with a large drift $N^{-d\alpha},\,\alpha\in(0,\alpha_0)$, for some $\alpha_0\in(1/d,1)$.
\end{remark}

\end{document}